\documentclass{memo-l}

\usepackage{amssymb}
\usepackage{array}
\usepackage{stmaryrd}
\usepackage{enumitem}
\usepackage{tikz}
\usetikzlibrary{calc}
\usetikzlibrary{arrows}

\newcommand{\Z}{\mathbb{Z}}

\newcommand{\F}{\mathbb{F}}
\newcommand{\D}{\mathbb{D}}
\newcommand{\Sph}{\mathbb{S}}

\newcommand{\lamalg}{\scalebox{1.25}{$\amalg$}}

\newcommand\acts{\mathrel{\rotatebox[origin=c]{-90}{$\circlearrowright$}}}
\newcommand\xqed[1]{
  \leavevmode\unskip\penalty9999 \hbox{}\nobreak\hfill
  \quad\hbox{#1}}
\newcommand\customendremark{\xqed{$\vert\vert$}}

\DeclareFontFamily{OT1}{pzc}{}
\DeclareFontShape{OT1}{pzc}{m}{it}{<-> s * [1.13] pzcmi7t}{}
\DeclareMathAlphabet{\mathpzc}{OT1}{pzc}{m}{it}
\DeclareFontFamily{U}{mathx}{\hyphenchar\font45}
\DeclareFontShape{U}{mathx}{m}{n}{<-> mathx10}{}
\DeclareSymbolFont{mathx}{U}{mathx}{m}{n}
\DeclareMathAccent{\widebar}{0}{mathx}{"73}

\newtheorem{Theorem}{Theorem}[chapter]
\newtheorem{Proposition}[Theorem]{Proposition}
\newtheorem{Lemma}[Theorem]{Lemma}
\newtheorem{Corollary}[Theorem]{Corollary}

\theoremstyle{definition}
\newtheorem{Definition}[Theorem]{Definition}
\newtheorem{Remark}[Theorem]{Remark}
\newtheorem{Example}[Theorem]{Example}

\numberwithin{section}{chapter}
\numberwithin{equation}{chapter}

\makeindex

\begin{document}

\frontmatter

\title{Stabilizations of $\mathbb{E}_\infty$ Operads and $p$-Adic Stable Homotopy Theory}

\author{Montek Singh Gill}


\email{monteksgill@gmail.com}


\date{November 3, 2022, and, in revised form, July 7, 2024.}

\subjclass[2010]{Primary 55P42, 18N70; Secondary 55S10, 55S20.}

\keywords{operads, algebras, cochain complexes, cohomology operations, stable homotopy theory}

\begin{abstract}
We study differential graded operads and $p$-adic stable homotopy theory. We first construct a new class of differential graded operads, which we call the stable operads. These operads are, in a particular sense, stabilizations of $\mathbb{E}_\infty$ operads. We provide an application of these stable operads to $p$-adic stable homotopy theory. It is well-known that cochains on spaces yield examples of algebras over $\mathbb{E}_\infty$ operads. We show that, in the stable case, cochains on spectra yield examples of algebras over our stable operads. Moreover, a result of Mandell says that, endowed with the $\mathbb{E}_\infty$ algebraic structure, cochains on spaces provide algebraic models of $p$-adic homotopy types. We show that, endowed with the algebraic structure encoded by our stable operads, spectral cochains provide algebraic models for $p$-adic stable homotopy types.
\end{abstract}

\maketitle

\tableofcontents

\mainmatter


\chapter{Introduction}

\section{Description of the Results}

In this work, we cover two main topics:

\begin{itemize}
	\item[(1)] Stable differential graded operads, which are a new class of operads, and which are, in a particular sense, stable analogues of $\mathbb{E}_\infty$ operads. (Henceforth, ``differential graded'' will often be abbreviated as ``dg''\index{dg}.)
	\item[(2)] Algebraic models of $p$-adic stable homotopy types using these operads.
\end{itemize}

First, let us recall the notion of an $\mathbb{E}_\infty$ cochain operad and how it gives rise to algebraic models of $p$-adic homotopy types. Let $\mathpzc{E}^\dagger$ be a model for the $\mathbb{E}_\infty$ cochain operad (throughout this work, we have used the symbol $\dagger$ to distinguish dg objects, such as modules, operads, etc., with differentials of degree $+1$ from corresponding such objects with differentials of degree $-1$). Given a cochain complex $X$, the structure of an algebra over $\mathpzc{E}^\dagger$ encodes a homotopy coherent commutative, associative and unital multiplication. If we take the cohomology of the complex (killing the higher homotopies, so to speak), we find that $\mathrm{H}^\bullet(X)$ inherits a (graded) commutative algebra structure in the traditional sense. In fact, the cohomology inherits even more structure. It posseses certain cohomology operations $P^s$, $s \in \Z$, which satisfy an instability condition, and as a result becomes an unstable module over $\mathcal{B}$, the algebra of generalized Steenrod operations. \\

A particular case of such $\mathpzc{E}^\dagger$-algebras are the cochains $\mathrm{C}^\bullet(X)$ on spaces $X$. In this case, the algebra structure on the cohomology is given by the cup product, while the operations are the Steenrod operations. While the cochains, as a dg module, might not remember the homotopy type of a space, in~\cite{Mandell}, Mandell demonstrated that if we take cochains with coefficients in $\overline{\F}_p$, the cochains functor
\[
\text{C}^\bullet(-;\overline{\F}_p) \colon \mathsf{Spc}^{\text{op}} \to \widebar{\mathpzc{E}}^\dagger\text{-}\mathsf{Alg}
\]
as a functor to $\widebar{\mathpzc{E}}^\dagger$-algebras (where $\widebar{\mathpzc{E}}^\dagger$ is the $\mathbb{E}_\infty$ cochain operad over $\overline{\F}_p$), induces a full embedding of the homotopy category of spaces into the derived category of $\widebar{\mathpzc{E}}^\dagger$-algebras when we restrict to connected nilpotent $p$-complete spaces of finite $p$-type. In fact, there is a retraction from $\widebar{\mathpzc{E}}^\dagger$-algebras to spaces given by the derived mapping space of functions to the constant algebra $\overline{\F}_p$ (the cochains on a point). This latter part is implicit in Mandell’s paper, and is worked out in~\cite{French}. This implies that all higher homotopical information for spaces, such as Toda brackets, is captured in $\widebar{\mathpzc{E}}^\dagger$-algebras, via the $\overline{\F}_p$-cochains. Thus, while rational homotopy types admit algebraic models via CDGAs, when working $p$-adically, we can use $\widebar{\mathpzc{E}}^\dagger$-algebras.  \\ 

Now we move onto stable operads and stable homotopy types. First of all, we show that the operad $\mathpzc{E}^\dagger$ possesses a stabilization map
\[
\Psi \colon \Sigma \mathpzc{E}^\dagger \to \mathpzc{E}^\dagger
\]
from its operadic suspension to itself. Using the maps $\Sigma^k\Psi \colon \Sigma^{k+1}\mathpzc{E}^\dagger \to \Sigma^k\mathpzc{E}^\dagger$, for $k \ge 0$, via an inverse limit, we produce a new operad, denoted $\mathpzc{E}_{\text{st}}^\dagger$, which is our stable operad. In fact, we have a new class of operads, the stable operads, in the sense that we are able to perform the above construction for multiple models of the $\mathbb{E}_\infty$ operad (we have a stable Barratt-Eccles operad, a stable McClure-Smith operad, and also a stable Eilenberg-Zilber operad, though the Eilenberg-Zilber is not quite an $\mathbb{E}_\infty$ operad). The stable operad $\mathpzc{E}_{\text{st}}^\dagger$ appears to be of independent interest outside of its application, which we discuss below, to $p$-adic stable homotopy theory; for example,  due to its homotopy additivity, which we also discuss below. While $\mathbb{E}_\infty$ operads are characterized intrinsically as cofibrant replacements of the commutative operad, we unfortunately do not provide an intrinsic definition of a stabilized $\mathbb{E}_\infty$ operad (one could speculate a definition such as a cofibrant replacement of the constant operad in a model category of additive operads, where the constant operad is that for which the associated monad is the identity). After having constructed several stabilizations in Chapter 2, by ``the'' in ``the stable operad'', we shall generally mean the stabilization of the Barratt-Eccles operad. \\

Having constructed the stable operad, first, we demonstrate that one has homotopical control over $\mathpzc{E}_{\text{st}}^\dagger$ and the corresponding category of algebras $\mathpzc{E}_{\text{st}}^\dagger\text{-}\mathsf{Alg}$ in the following sense.

\begin{Theorem}\label{thm:monadweakequiv}
\textit{The monad $\mathbf{E}^\dagger_{\mathbf{st}}$ associated to $\mathpzc{E}_{\emph{st}}^\dagger$ preserves quasi-isomorphisms.}
\end{Theorem}

\begin{Theorem}\label{thm:semimodelstr}
\textit{The category $\mathpzc{E}_{\text{st}}^\dagger\text{-}\mathsf{Alg}$ admits a Quillen semi-model structure where the weak equivalences and fibrations are the quasi-isomorphisms and degreewise epimorphisms.}
\end{Theorem}

(See Definition~\ref{def:category_semi_model_str} for the definition of a Quillen semi-model structure, a weakening of the more well-known notion of a Quillen model structure.) Next, we develop a theory of cohomology operations for algebras over $\mathpzc{E}_{\text{st}}^\dagger$. We once again get operations $P^s$ for $s \in \Z$, but they now no longer satisfy the instability condition.

\begin{Theorem}\label{thm:homops}
\textit{We have the following:}
\begin{itemize}
	\item[(i)] \textit{The cohomologies of $\mathpzc{E}_{\emph{st}}^\dagger$-algebras possess natural operations $P^s$ for $s \in \Z$ which satisfy the Adem relations.}
	\item[(ii)] \textit{Given a cochain complex $X$, we have a natural isomorphism}
\[
\emph{H}^\bullet(\mathbf{E}^\dagger_{\normalfont{\textbf{st}}} X) \cong \widehat{\mathcal{B}} \otimes \emph{H}^\bullet(X).
\]
\end{itemize}
\end{Theorem}

Here $\widehat{\mathcal{B}}$ is a certain completion, with respect to a filtration by excess, of the algebra $\mathcal{B}$. Note that, in the unstable case of $\mathpzc{E}^\dagger$, in (ii) above, we would not only have to tensor with $\mathcal{B}$ to account for the cohomology operations, but also enforce the instability condition and also take a polynomial algebra to add in products. In the case of $\mathpzc{E}_{\text{st}}^\dagger$, the instability of the operations and the products are no longer present. \\

Next, we justify the ``stable'' in ``stable operad''. This should of course be a statement about homotopy coherent, or $\infty$-, additivity, and this is exactly what we demonstrate. In particular, we demonstrate that the monad $\mathbf{E}_{\textbf{st}}^\dagger$ is homotopy coherent, or $\infty$-, additive, in the following sense.

\begin{Theorem}\label{thm:stab}
\textit{We have the following:}
\begin{itemize}
	\item[(i)] \textit{For dg modules $X$ and $Y$, we have a natural quasi-isomorphism}
\[
\mathbf{E}^\dagger_{\normalfont{\textbf{st}}}(X \oplus Y) \sim \mathbf{E}^\dagger_{\normalfont{\textbf{st}}}(X) \oplus \mathbf{E}^\dagger_{\normalfont{\textbf{st}}}(Y).
\]
	\item[(ii)] \textit{More generally, for cofibrant $\mathpzc{E}_{\emph{st}}^\dagger$-algebras $A$ and $B$, we have a natural quasi-isomorphism}
\[
A \amalg B \sim A \oplus B.
\]
\end{itemize}
\end{Theorem}

In (i) and (ii) above, note that there are canonical maps $\mathbf{E}^\dagger_{\normalfont{\textbf{st}}}(X) \oplus \mathbf{E}^\dagger_{\normalfont{\textbf{st}}}(Y) \to \mathbf{E}^\dagger_{\normalfont{\textbf{st}}}(X \oplus Y)$ and $A \oplus B \to A \amalg B$, and it is exactly these maps which yield the quasi-isomorphisms. Also, for comparison, in the unstable case, in (ii), we have $A \otimes B$ in place of $A \oplus B$. \\

Finally, we move onto the application to $p$-adic stable homotopy types. For this, we need to fix a model for spectra. We take the classical sequential model in the sense of Bousfield-Friedlander, with the exception that, rather than the ordinary suspension $- \wedge \Sph^1$, we use the Kan suspension of based simplicial sets. We then define an appropriate, and concrete in the sense that we get dg modules, notion of spectral cochains and then prove the following, providing another sense in which $\mathpzc{E}_{\text{st}}^\dagger$ is a stable analogue of $\mathpzc{E}^\dagger$.

\begin{Theorem}\label{thm:specopaction}
\textit{Given any spectrum $E$, the spectral cochains $\emph{C}^\bullet(E)$ naturally form an algebra over $\mathpzc{E}_{\emph{st}}^\dagger$.}
\end{Theorem}

Finally then, we get algebraic models for $p$-adic stable homotopy types in the following sense; in the statement, the $\overline{\F}_p$ cochains functor $\widebar{\mathrm{C}}{}^\bullet$ is constructed from the $\F_p$ cochains functor $\mathrm{C}^\bullet$ simply by tensoring with $\overline{\F}_p$, and similarly, the operad $\widebar{\mathpzc{E}}_{\text{st}}^\dagger$, an operad over $\overline{\F}_p$, is constructed from $\mathpzc{E}_{\text{st}}^\dagger$ by tensoring with $\overline{\F}_p$.

\begin{Theorem}\label{thm:modelsstable}
\textit{The spectral cochains functor}
\[
\widebar{\mathrm{C}}{}^\bullet \colon \mathsf{Sp}^{\mathrm{op}} \to \widebar{\mathpzc{E}}_{\mathrm{st}}^\dagger\text{-}\mathsf{Alg}
\]
\textit{induces a full embedding of the stable homotopy category into the derived category of $\widebar{\mathpzc{E}}_{\emph{st}}$-algebras when we restrict to bounded below $p$-complete spectra of finite $p$-type. In fact, there is a retraction from $\widebar{\mathpzc{E}}_{\mathrm{st}}^\dagger$-algebras to spectra given by a derived mapping space of functions to the cochains on the sphere spectrum.}
\end{Theorem}

We mentioned above that rational homotopy types can be modelled by commutative DGAs, and also that $p$-adic homotopy types can be modelled by $\mathbb{E}_\infty$ DGAs. It is also well-known that rational stable homotopy types can be modelled by dg modules. Our result for $p$-adic stable homotopy types then completes the following picture.

\begin{center}
\begin{tabular}{m{3.5cm} | m{3cm}  m{3cm}} 
\textbf{Algebraic models of homotopy types} & unstable & stable \tabularnewline \hline
\tabularnewline [-1em]
rational & commutative DGAs & dg modules \tabularnewline
\tabularnewline [-1em]
$p$-adic & $\mathbb{E}_\infty$ DGAs & $\widebar{\mathpzc{E}}_{\text{st}}^\dagger$ dg algebras
\end{tabular}
\end{center}

Note that the commutative DGAs in the table above are exactly the algebras over $\mathpzc{Comm}^\dagger$, the commutativity cochain operad. In each arity, this commutativity operad is simply the ground field concentrated in degree zero. Moreover, dg modules are exactly the algebras over $\mathpzc{Const}^\dagger$, the constant cochain operad. In arity one, this constant operad is the ground field concentrated in degree zero; in all other arities, it is zero. We demonstrate that, in fact, the constant operad is a stabilization of the commutativity operad, in a manner exactly analogous to the stabilization of $\mathpzc{E}^\dagger$ to get $\mathpzc{E}_{\text{st}}^\dagger$. Now, as is well-known, $\mathpzc{E}^\dagger$ is a cofibrant replacement of $\mathpzc{Comm}^\dagger$. We demonstrate that $\mathpzc{E}^\dagger_{\text{st}}$ possesses a map $\mathpzc{E}_{\text{st}}^\dagger \to \mathpzc{Const}^\dagger$ which is a degreewise epimorphism and a quasi-isomorphism, and so an acyclic fibration in the projective model structure on dg operads. We, however, unfortunately have not verified that $\mathpzc{E}_{\text{st}}^\dagger$ is cofibrant (for this to be true, we may need to first restrict to additive operads; that is, to operads whose associated monads are additive). With this caveat mentioned, and assuming one can find the right model category in which $\mathpzc{E}_{\text{st}}^\dagger$ is indeed cofibrant, we get the following picture.

\begin{center}
\begin{tikzpicture}
\node [] (A) {
\begin{tabular}{m{3.5cm} | m{3cm}  m{3cm}}
\textbf{Operad for algebraic models} & unstable & stable \tabularnewline \hline
\tabularnewline [-1em]
rational & commutativity operad & constant operad \tabularnewline
\tabularnewline [-1em]
$p$-adic & $\mathbb{E}_\infty$ operad & stabilized $\mathbb{E}_\infty$ operad
\end{tabular}
};

\node [above of = A, yshift = 5mm, xshift = 15mm] (B) {};
\node [right of = B, xshift = 0mm] (C) {};

\draw [->]  (B) -- (C) node[midway, anchor = south] {stabilize operad};

\node [left of = A, xshift = -45mm] (D) {};
\node [below of = D, yshift = 0mm] (E) {};

\draw [->]  (D) -- (E) node[midway, anchor = east] {\rotatebox{90}{cofibrant}\rotatebox{90}{replacement}};
\end{tikzpicture}
\end{center}

We can make a remark on the structure that is captured by the operad $\mathpzc{E}_{\text{st}}^\dagger$. In both the unstable and stable $p$-adic cases, the models arise via cochains, with coefficients in $\overline{\F}_p$. Let us consider just $\F_p$ cochains. Given a space $X$, the mod $p$ cochains are given, as a spectrum, by $\mathrm{F}(\Sigma^\infty_+X, \text{H}\F_p)$, where $\mathrm{F}(-,-)$ denotes the mapping spectrum. This object carries the following structure:
\begin{itemize}
	\item[(i)] It is an $\text{H}\F_p$-module via ``pointwise'' scalar multiplication.
	\item[(ii)] It possesses an action, via postcomposition, by $\mathrm{F}(\text{H}\F_p, \text{H}\F_p)$.
	\item[(iii)] It is a ring spectrum via the multiplication of $\text{H}\F_p$.
\end{itemize}
These manifest in a dg context as follows:
\begin{itemize}
	\item[(i$^\prime$)] The cochains can be modelled as a dg module.
	\item[(ii$^\prime$)] The cohomology inherits an action by $\mathcal{B}/(1-P^0) \cong \mathcal{A}$.
	\item[(iii$^\prime$)] The cochains form an $\mathbb{E}_\infty$ dg algebra. 
\end{itemize}
In fact, we shall see that (ii') is a consequence of (iii'). Now let us consider cochains on a spectrum $E$. The cochains are given, as a spectrum, by $\mathrm{F}(E, \text{H}\F_p)$. This object carries the structure described in (i) and (ii) above. It no longer posseses a ring structure as, although the multiplication of $\text{H}\F_p$ is still present, to define a ``pointwise'' multiplication, one needs a diagonal map, which general spectra, unlike spaces and their suspension spectra, do not possess. The structure that is still present manifests, respectively, in our work as (i$^\prime$) above together with the following modified version of (ii$^\prime$): the cohomology inherits an action by $\widehat{\mathcal{B}}/(1-P^0) \cong \mathcal{A}$. We shall see that these operations in (ii$^\prime$) are a consequence of the $\mathpzc{E}_{\text{st}}^\dagger$-algebra structure, and so one can say that it is primarily these operations which this operad serves to encode.

\section{Acknowledgments}

The author would like to thank Igor Kriz, under whose supervision this project was completed. We also thank Paul Goerss, Mike Mandell and Peter May, who made many helpful comments on earlier versions of this work.

\section{Notations, Terminology and Conventions}\label{sec:nots_convs}

We list here some notations, terminology and conventions which are used throughout the work. We do this for ease of reference and because setting them in place now will allow us to be precise in our statements and constructions in later parts of the work. \\

\textbf{General:}
\begin{itemize}
	\item For each integer $d \ge 0$, $[d]$ denotes the poset $0 \to 1 \to \cdots \to d$.
	\item For each integer $d \ge 0$, $(d)$ denotes the set $\{1,\dots,d\}$; $(0)$ is the empty set.
	\item For $d \ge 0$, $\Sigma_d$ denotes the symmetric group on $d$ letters; $\Sigma_0$ is the trivial group, where the unique element is thought of as representing the unique isomorphism on the empty set.
	\item The ends of proofs are marked by $\square$. The ends of remarks are marked by $\vert\vert$.
\end{itemize}

\textbf{In relation to simplicial sets:}
\begin{itemize}
	\item The category of simplicial sets will be denoted by $\mathsf{Spc}$, and that of based simplicial sets by $\mathsf{Spc}_*$.
	\item For each integer $d \ge 0$, $\Delta_d$ denotes the standard $d$-dimensional simplex as a simplicial set.
	\item Given a simplicial set $S$, $S_d^{\text{nd}}$ denotes the non-degenerate $d$-simplices of $S$.
	\item The disjoint basepoint functor, the left adjoint to the forgetful functor from $\mathsf{Spc}_*$ to $\mathsf{Spc}$, will be denoted by $(-)_+$. 
\end{itemize}

\textbf{In relation to (co)chain complexes:}
\begin{itemize}
	\item The term \textit{chain complex} will refer to a graded module equipped with a differential of degree $-1$ and the term \textit{cochain complex} will refer to a graded module equipped with a differential of degree $+1$; in addition, we shall let \textit{differential graded module} refer to either of these two possibilities. The phrase \textit{differential graded} is often shortened to \textit{dg}\index{dg}. By default, all dg modules are unbounded. The category of chain complexes over $k$, where $k$ is a field, is denoted by $\mathsf{Ch}_k$, and the category of cochain complexes over $k$ is denoted by $\mathsf{Co}_k$; the symbol $\mathsf{DG}_k$ will denote either of these two possibilities.
	\item Given a (co)chain complex $X$ over $k$, $X^\vee$ denotes the dual\index{dual!(co)chain complex} (co)chain complex of $X$. Thus $X^\vee := \text{F}(X,k[0])$ where $\text{F}$ denotes the internal hom of (co)chain complexes and $k[0]$ is the (co)chain complex with $k$ concentrated in degree $0$. This yields contravariant functors $\mathsf{Ch}_k \to \mathsf{Ch}_k$ and $\mathsf{Co}_k \to \mathsf{Co}_k$, which are both denoted by $(-)^\vee$. More concretely, unravelling the definition, we find that, given a dg module $X$ and a fixed degree $d$, $(X^\vee)_d$ consists of module maps $X_{-d} \to k$ in the chain case, and $X^{-d} \to k$ in the cochain case. The negative sign here ensures that the induced differential on $X^\vee$ has the same degree as the differential of the original dg module $X$. Note that if $X$ is concentrated in non-negative degrees, $X^\vee$ will be concentrated in non-positive degrees, and vice versa.
	\item Given a chain complex $X$ over $k$, $X^\dagger$ denotes the associated cochain complex\index{associated!(co)chain complex}, where $(X^\dagger)^p = X_{-p}$; similarly, if $X$ is a cochain complex over $k$, $X^\dagger$ denotes the associated chain complex, where $(X^\dagger)_p = X^{-p}$. These yield an inverse pair of functors $\mathsf{Ch}_k \to \mathsf{Co}_k$ and $\mathsf{Co}_k \to \mathsf{Ch}_k$, both of which are denoted by $(-)^\dagger$. As an example, note that, given any space, the cochains on the space are constructed from the chains by first dualizing via $(-)^\vee$, and then reindexing via $(-)^\dagger$.
	\item Given dg modules $X$ and $Y$, the tensor product\index{tensor product of (co)chain complexes} $X \otimes Y$ is defined, as usual, by letting $x \otimes y$ have degree $|x| + |y|$, and the differential follows the standard sign convention:
\[
\partial (x \otimes y) = \partial x \otimes y + (-1)^{|x|} x \otimes \partial y.
\]
This tensor product on dg modules is always endowed with the symmetry $X \otimes Y \to Y \otimes X$ defined by
\[
x \otimes y \mapsto (-1)^{|x||y|}(y \otimes x).
\]
The internal hom\index{internal hom of (co)chain complexes} $\text{F}(X,Y)$ is defined, in degree $d$, as the collection of degree $d$ graded module maps $f \colon X_\bullet \to Y_{\bullet + d}$ (no compatibility with the differnetial is required for these maps), and the differential on $\text{F}(X,Y)$ is given by
\[
\partial f = \partial \circ f - (-1)^{|f|} (f \circ \partial).
\]
	\item Let $k$ denote a field. We shall often denote the chain or cochain complex
\[
\cdots \leftarrow 0 \leftarrow \underset{\text{deg.} \: d}{k} \leftarrow 0 \leftarrow \cdots
\]
\[
\cdots \rightarrow 0 \rightarrow \underset{\text{deg.} \: d}{k} \rightarrow 0 \rightarrow \cdots
\]
by $\Sph^d$, or $k[d]$, and refer to it as a \textit{sphere complex}\index{sphere complex}, and the chain or cochain complex
\[
\cdots \leftarrow 0 \leftarrow \underset{\text{deg.} \: d-1}{k} \overset{\text{id}}\leftarrow \underset{\text{deg.} \:d}{k} \leftarrow 0 \leftarrow \cdots
\]
\[
\cdots \rightarrow 0 \rightarrow \underset{\text{deg.} \: d-1}{k} \overset{\text{id}}\rightarrow \underset{\text{deg.} \:d}{k} \rightarrow 0 \rightarrow \cdots
\]
by $\D^n$, and refer to it as a \textit{disk complex}\index{disk complex}.
	\item Given a dg module $X$ and $n \in \Z$, we let $X[n]$ be the dg module defined by setting $X[n]_d = X_{d-n}$\index{shifted (co)chain complex}. Note that, if we let $k$ denote a ground field, $X[1] \cong X \otimes k[1]$\index{suspension!of (co)chain complexes} and $X[-1] \cong \text{F}(k[1],X)$\index{looping!of (co)chain complexes}.
	\item All (co)chains on spaces or spectra are taken to be normalized\index{normalized (co)chains} (that is, degenerate simplices are always modded out). All (co)chains on a based space, unless explicitly mentioned otherwise, are reduced\index{reduced (co)chains}, in addition to being normalized (that is, the basepoint is always modded out).
\end{itemize}

\textbf{In relation to operads and their (co)algebras:}	
\begin{itemize}
	\item \index{chain and cochain operads}A \textit{chain operad} is an operad in $\mathsf{Ch}_k$. A \textit{cochain operad} is an operad in $\mathsf{Co}_k$. A \textit{differential graded operad}, or \textit{dg operad}, is an operad in either one of $\mathsf{Ch}_k$ and $\mathsf{Co}_k$.
	\item The category of chain operads is denoted by $\mathsf{Op}(\mathsf{Ch}_k)$ and the category of cochain operads is denoted by $\mathsf{Op}(\mathsf{Co}_k)$.
	\item All operads are dg operads, and are symmetric.
	\item The notations, in the sense of the typeface, for operads and their corresponding monads and free algebra functors will follow the following rule: if $\mathpzc{P}$ denotes an operad, the corresponding monad and free algebra functor will both be denoted by $\mathbf{P}$.
	\item Given a dg operad $\mathpzc{P}$, the categories of $\mathpzc{P}$-algebras and $\mathpzc{P}$-coalgebras, respectively, are denoted by $\mathpzc{P}\text{-}\mathsf{Alg}$ and $\mathpzc{P}\text{-}\mathsf{Coalg}$.
	\item We have seen that we have a reindexing operator $(-)^\dagger$ on dg modules. We also have such an operator on dg operads. When this operation is applied aritywise to an operad, we in fact get another operad (an easy check shows that the reindexing operation is compatible with all the structure data in an operad). As such, if $\mathpzc{P}$ is an operad in $\mathsf{Op}(\mathsf{Ch}_k)$, we let $\mathpzc{P}^{\dagger}$ denote the \textit{associated operad}\index{associated!(co)chain operad} in $\mathsf{Op}(\mathsf{Co}_k)$, where $\mathpzc{P}^\dagger (n) := \mathpzc{P}(n)^\dagger$. Similarly, if $\mathpzc{Q}$ is an operad in $\mathsf{Op}(\mathsf{Co}_k)$, wet let $\mathpzc{Q}^{\dagger}$ denote the \textit{associated operad} in $\mathsf{Op}(\mathsf{Ch}_k)$, where again $\mathpzc{Q}^\dagger (n) := \mathpzc{Q}(n)^\dagger$. These yield an inverse pair of functors $\mathsf{Op}(\mathsf{Ch}_k) \to \mathsf{Op}(\mathsf{Co}_k)$ and $\mathsf{Op}(\mathsf{Co}_k) \to \mathsf{Op}(\mathsf{Ch}_k)$, both of which are denoted by $(-)^\dagger$. Note: one can also apply the dualization operator $(-)^\vee$ aritywise to an operad, which yields, under suitable finiteness hypotheses, a co-operad, but we shall have no need for this.
	\item We have seen that we have a reindexing operator $(-)^\dagger$ on dg modules and dg operads. We also have such an operator on (co)algebras over dg operads\index{associated!(co)algebra}. In fact, given a dg operad $\mathpzc{P}$ and a (co)algebra $A$ over $\mathpzc{P}$, an easy check shows that $A^\dagger$ is canonically a (co)algebra over $\mathpzc{P}^\dagger$. Moreover, we also have a dualization operator $(-)^\vee$ on coalgebras over dg operads\index{dual!(co)algebra}: if $\mathpzc{P}$ is a dg operad and if $A$ is a $\mathpzc{P}$-coalgebra, then $A^\vee$ is canonically a $\mathpzc{P}$-algebra (see, e.g., Lemma 2.1 in~\cite{iLucio}).
\end{itemize}


\chapter{Stabilizations of $\mathbb{E}_\infty$ Operads}\label{sec:stabilizations}

In this section, we shall construct stable analogues of the Eilenberg-Zilber, McClure-Smith and Barratt-Eccles operads. Note that only the latter two constitute stabilizations of $\mathbb{E}_\infty$ operads. In order to construct actions of the latter two on spectral cochains, however, as we will do later in Section~\ref{subsec:spec_cochains}, it is convenient to also have a stable analogoue of the Eilenberg-Zilber operad. Prior to constructing these stabilizations, we first need to discuss some general aspects of dg operads, their algebras and some basic constructions on simplicial sets, which we shall also need in later chapters. For reference throughout this chapter: we set that $p$ is to denote an unspecified but fixed prime, and, when considering the aforementioned operads, the ground field will be taken to be $\F_p$.

\section{Kan Suspensions and Moore Loopings}

As is standard, we let $\Delta$ denote the simplex category. We also let $\mathsf{Spc}$ denote the category of spaces, by which we mean simplicial sets, and let $\mathsf{Spc}_*$ denote the category of based spaces, by which we mean based simplicial sets. For each $d \ge 0$, we let $\Delta_d$ denote the standard $d$-simplex. Given a based simplicial set, there exists more than one possible choice for a suspension functor. The most obvious one is perhaps $- \wedge \Sph^1$, where $\Sph^1 = \Delta_1/\partial\Delta_1$, but we will use a different one, the Kan suspension, which is weakly equivalent to $- \wedge \Sph^1$. Similarly, rather than $\text{F}(\Sph^1,-)$ for loopings, we will use a different, but weakly equivalent, looping functor, the Moore looping. We reserve the standard suspension and loops notations, $\Sigma$ and $\Omega$, for the Kan suspensions and Moore loopings. \\

In order to define the Kan suspension, we first recall a cone\index{cone on simplicial set} functor $\text{C}(-)$ for based simplicial sets (see Chapter 3, Section 5 in~\cite{GoerssJardine} for details on this construction). Let $S$ be a based simplicial set as above. In degree $d$, we have that
\[
\text{C}(S)_d = S_d \vee S_{d-1} \vee \dots \vee S_{0}.
\]
Moreover, the action of the simplicial operators is as follows. Consider some map $\theta \colon [d] \to [e]$ in $\Delta$. We want a function $S_e \vee S_{e-1} \vee \dots \vee S_{0} \to S_d \vee S_{d-1} \vee \dots \vee S_{0}$. Let $i \in \{0,1,\dots,e\}$. Our function will be a based one, so that we need to define, for each such $i$, a map $S_i \to S_d \vee S_{d-1} \vee \dots \vee S_{0}$. Consider the final $i+1$ elements of $[e]$. If the preimage under $\theta$ of these elements is empty, the desired map is to be just the constant one at the basepoint. Otherwise, we form the restricted map with source the preimage of the final $i+1$ elements of $[e]$ and target these final $i+1$ elements of $[e]$ and then reindex so that we have a map
\begin{equation}\label{eq:theta_i}
\theta(i) \colon [j] \to [i]
\end{equation}
for some $j \in \{0,1,\dots,d\}$. The desired map is then set to be $\theta(i)^* \colon S_i \to S_j$ followed by the inclusion into $S_d \vee S_{d-1} \vee \dots \vee S_{0}$.

\begin{Example}\label{examp:cones_Delta_k_+}
For any $d \ge 0$, we have an isomorphism of based simplicial sets
\[
\Delta_{d+1} \overset{\cong}\longrightarrow \text{C}(\Delta_{d+})
\]
where $\Delta_{d+1}$ is based at $0$. The map is as follows. Consider some $\theta \colon [e] \to [d+1]$ in $(\Delta_{d+1})_e$. We have that $\text{C}(\Delta_{d+})_e = (\Delta_{d})_e \amalg \cdots \amalg (\Delta_{d})_0 \amalg *$. If $\theta$ doesn't map anything to the final $d+1$ elements of $[d+1]$, that is, if it maps everything to $0$, then we send it to $*$. Otherwise, we get some new map $\theta(d) \colon [j] \to [d]$ (the notation here is as in (\ref{eq:theta_i})), for some $j \in \{0,1\dots,e\}$ and $\theta$ is mapped to this element of $\text{C}(\Delta_{d+})_e$. An easy check shows that this does indeed define a map, in fact an isomorphism, of based simplicial sets.
\end{Example}

Note that, given a based simplicial set $S$, we have a canonical inclusion map
\begin{equation}\label{eq:X_to_cone}
i \colon S \to \text{C}(S)
\end{equation}
which, in degree $d$, is just the inclusion $S_d \to S_d \vee S_{d-1} \vee \cdots \vee S_0$ of the $S_d$ summand (this is a map of based simplicial sets because the simplicial operators act on the wedge sums ``summand-wise''). With this map, we are able to define the Kan suspension for based simplicial sets.

\begin{Definition}\label{def:kan_susps}
Given a based simplicial set $S$, its \textit{Kan suspension}\index{suspension!Kan suspension of simplicial sets}, denoted $\Sigma S$, is defined by setting
\[
\Sigma S := \text{C}(S)/S
\]
where the inclusion $S \to \text{C}(S)$ is as above.
\end{Definition}

Thus, given a based simplicial set $S$ and $d \ge 0$, we have
\[
(\Sigma S)_d \cong \left\{ \begin{array}{ll}
S_{d-1} \vee \cdots \vee S_0 & \text{if $d \ge 1$} \\
* & \text{if $d = 0$.}
\end{array}
\right.
\]
In particular, for the case of $S_+$, where $S$ is now an unbased simplicial set, we have
\[
(\Sigma S_+)_d \cong \left\{ \begin{array}{ll}
S_{d-1} \amalg \cdots \amalg S_0 \amalg * & \text{if $d \ge 1$} \\
* & \text{if $d = 0$.}
\end{array}
\right.
\]

\begin{Remark}
\index{suspension!relation between Kan suspension and $- \wedge \Sph^1$}The relation between the Kan suspension and the more usual smash suspension $- \wedge \Sph^1$, where $\Sph^1 = \Delta_1/\partial\Delta_1$, is that, for based simplicial sets $S$, there is a natural weak equivalence
\[
S \wedge \Sph^1 \to \Sigma S.
\]
For a proof, see Proposition 2.17 in~\cite{MarcStephan}. \customendremark
\end{Remark}

We now record some simple facts and definitions regarding the Kan suspension, which will be useful for us later. Given any simplicial set $S$, we let $S_d^{\text{nd}}$ denote the collection of non-degenerate $d$-simplices of $S$. Note that, for any $S$, $S_0^{\text{nd}} = S_0$.

\begin{Proposition}\label{prop:nd_susp}
Let $S$ be a based simplicial set. We have
\[
(\Sigma S)_d^{\emph{nd}} \cong \left\{ \begin{array}{ll}
S_{d-1}^{\emph{nd}} & \emph{if $d \ge 2$} \\
S_0 \smallsetminus * & \emph{if $d = 1$} \\
* & \emph{if $d = 0$}.
\end{array}
\right.
\]
In particular, for the case of $S_+$, where $S$ is now an unbased simplicial set, we have
\[
(\Sigma S_+)_d^{\emph{nd}} \cong \left\{ \begin{array}{ll}
S_{d-1}^{\emph{nd}} & \emph{if $d \ge 1$} \\
* & \emph{if $d = 0$}.
\end{array}
\right.
\]
\end{Proposition}

\begin{proof}
The first part follows from a laborious but easy direct check using the definition above of the simplicial operators on cones. For the second part, note that $(S_+)_{d-1}^{\text{nd}} = S_{d-1}^{\text{nd}}$ for $d \ge 2$ and $(S_+)_0 \smallsetminus * = S_0 = S_0^{\text{nd}}$.
\end{proof}

\begin{Proposition}\label{prop:susp_monos}
The Kan suspension $\Sigma$ preserves monomorphisms.
\end{Proposition}

\begin{proof}
This is immediate from the fact that induced maps act ``summand-wise''.
\end{proof}

\begin{Definition}\label{def:susp_simps}
\index{suspension!of a simplex}Given a based simplicial set $S$ and a simplex $s \colon \Delta_d \to S$ of $S$, of dimension $d$, let $\Sigma s \colon \Delta_{d+1} \to \Sigma S$ denote the corresponding simplex of dimension $d+1$, given by inclusion into the first wedge summand, of $\Sigma S$. We call $\Sigma s$ the \textit{suspension} of $s$.
\end{Definition}

\begin{Proposition}\label{prop:susp_simps_maps}
Let $S$ and $T$ be based simplicial sets, $f \colon S \to T$ a based map and $s$ a simplex of $S$. We have that $(\Sigma f)(\Sigma s) = \Sigma (f(s))$.
\end{Proposition}

\begin{proof}
This follows immediately from the fact that $\Sigma f$ acts ``summand-wise''.
\end{proof}

\begin{Proposition}\label{prop:faces_of_susp_simps}
Let $S$ be a based simplicial set and $s$ a $d$-simplex of $S$. Then we have
\[
d_i(\Sigma s) = \left\{
\begin{array}{ll}
\Sigma (d_{i-1}s) & i = 1,\dots,d+1 \\
* & i = 0.
\end{array}
\right.
\]
\end{Proposition}

Note that the $\Sigma$'s here are used in the sense of Definition~\ref{def:susp_simps}, not as summation symbols.

\begin{proof}
Let $s \in S_d$ and consider some $d_i \colon S_d \to S_{d-1}$, $d^i \colon [d-1] \to [d]$. Consider the map $[d] \to [d+1]$ achieved by adjoining $0 \mapsto 0$ at the beginning (that is, we send $0$ to $0$ and otherwise, $i$, for $i \ge 1$, to $d^i(i-1)+1$) and note that this is exactly $d^{i+1}$. Using this relation between $d^{i+1}$ and $d^i$, and the definition of the action of simplicial operators on cones and suspensions, we have that $d_{i+1}(\Sigma s) = \Sigma (d_i s)$. As for the $d_0$ case, again, this follows from the definition of the action of the simplicial operators on cones and suspensions.
\end{proof}

\begin{Proposition}\label{prop:n_r_chains_susp_X}
Let $S$ be a based simplicial set. We have a natural isomorphism of chain complexes
\[
\Phi \colon \emph{C}_\bullet(\Sigma S) \overset{\cong}\longrightarrow \emph{C}_\bullet(S)[1].
\]
\end{Proposition}

The chains here may be taken to have any desired coefficients, and are, as always throughout this work, normalized (and of course also reduced as our simplicial set is based).

\begin{proof}
This follows from Propositions~\ref{prop:nd_susp} and~\ref{prop:faces_of_susp_simps}. The isomorphism is given by $\Sigma s \mapsto s$.
\end{proof}

We now move on to a discussion of Moore loopings, which constitute the loops functor which is right adjoint to the Kan suspension defined above. For more detail on this loops functor, see, for example, Chapter 2, Section 6 of~\cite{Wu}. It is defined as follows.

\begin{Definition}\label{def:moore_loopings}
Let $S$ be a based simplicial set. The \textit{Moore looping}\index{looping!Moore looping of simplicial sets} of $S$ is defined by setting, for each $d \ge 0$
\[
(\Omega S)_d := \{s \in S_{d+1} \mid d_1 \cdots d_{d+1}(s) = *, d_0(s) = *\}.
\]
We of course also need actions of the simplicial operators $d_i^d \colon (\Omega S)_d \to (\Omega S)_{d-1}$ and $s_i^d \colon (\Omega S)_d \to (\Omega S)_{d+1}$. These are given by, respectively,  $d^{d+1}_{i+1}$ and $s^{d+1}_{i+1}$ (from the simplicial structure maps of $S$). One can check that, with these definitions, the required simplicial identities do indeed hold.
\end{Definition}

\begin{Remark}
We can give a more general description of the action of the simplicial operators for Moore loopings. Given a map $\theta \colon [d] \to [e]$ in $\Delta$, to act on an element of $(\Omega X)_e$, we first abut $0 \mapsto 0$ at the beginning to get a map $\hat\theta \colon [d+1] \to [e+1]$, where by ``abut $0 \mapsto 0$'', we mean construct the map which sends $0$ to $0$ and otherwise, for $i \ge 1$, sends $i$ to $\theta(i-1) + 1$. Next, we act by using the map $\hat\theta^*$ from the simplicial structure of $S$. \customendremark
\end{Remark}

Prior to discussing the adjunction with the Kan suspension, as we did with the suspensions, we compute the non-degenerates in Moore loopings.

\begin{Proposition}\label{prop:nd_loops}
Let $S$ be a based simplicial set. Given any $d \ge 0$, we have that
\[
(\Omega S)_d^{\emph{nd}} \cong \left\{ \begin{array}{ll}
S_{d+1}^{\emph{nd}} \cap (\Omega S)_d & \emph{if $d \ge 1$} \\
(S_1^{\emph{nd}} \cup *) \cap (\Omega S)_0 & \emph{if $d = 0$}.
\end{array}
\right.
\]
\end{Proposition}

\begin{proof}
This follows from an easy direct check using the definition above of the action of the simplicial operators on loopings.
\end{proof}

Finally, we record that we have an adjunction with $\Sigma$ and $\Omega$ and record facts about the unit and counit of this adjunction.

\begin{Proposition}\label{prop:SigmaOmegadj}
We have the following:
\begin{itemize}
	\item[(i)] The Kan suspensions and Moore loopings constitute an adjunction where $\Sigma \dashv \Omega$.
	\item[(ii)] For all based simplicial sets $S$, the unit $S \to \Omega\Sigma S$ is an isomorphism.
	\item[(iii)] For all based simplicial sets $S$, the counit $\Sigma \Omega S \to S$ is a monomorphism.
\end{itemize}
\end{Proposition}

\begin{proof}
(i): The necessary verifications are straightforward. For a written account, see Proposition 2.14 in~\cite{MarcStephan}; our loops functor is dual to the one used there, but an entirely analogous argument carries through (the duality here is in the sense of mirror reflections of simplicial objects -- see Section 2.6.3 in~\cite{Wu} for details). Under the natural hom-isomorphism
\[
\mathsf{Spc}_*(\Sigma S, T) \cong \mathsf{Spc}_*(S, \Omega T)
\]
given a map $f \colon \Sigma S \to T$, the corresponding adjoint map $g \colon S \to \Omega T$ is given by the composites
\[
S_d \hookrightarrow S_d \vee S_{d-1} \vee \cdots \vee S_1 \vee S_0 = (\Sigma S)_{d+1} \to T_{d+1}
\]

(ii): To demonstrate this, we explicitly describe the unit of adjunction. It is given by maps $S \to \Omega \Sigma S$. We have
\begin{align*}
(\Omega \Sigma S)_d &= \{s \in (\Sigma S)_{d+1} \mid d_0(s) = d_1 \cdots d_{d+1}(s) = *\} \\
&= \{s \in S_d \vee \cdots \vee S_0 \mid d_0(s) = d_1 \cdots d_{d+1}(s) = *\}.
\end{align*}
Using the definition of the action of the simplicial operators on suspensions, we find that on each $S_d, \dots, S_0$, the action by $d_0$ is the identity, so that the elements that go to $*$ under $d_0$ are exactly those in $S_d$. Moreover, the condition $d_1 \cdots d_{d+1}(s) = *$ is automatic for all simplices since $(\Sigma T)_0 = *$ for any $T$ (one can also directly check that, for $s \in S_d$, $d_1\cdots d_{d+1}(s) = d_0\cdots d_d(s)$). Thus $(\Omega \Sigma S)_d = S_d$. One can check that the unit of adjunction is then just the identity on $S_d$ and hence an isomorphism. \\

(iii): It suffices (by, for example, the Eilenberg-Zilber lemma expressing degenerate simplices uniquely as iterated degeneracies of non-degenerate simplices) to show that the counit preserves non-degenerate simplices and that it is injective when restricted to the non-degenerate simplices. In dimension $d = 0$, this is clear since $(\Sigma T)_0 = *$ for any $T$. Let $d \ge 1$. We have that
\[
(\Sigma \Omega S)_d = (\Omega S)_{d-1} \vee \cdots \vee (\Omega S)_0.
\]
By Proposition~\ref{prop:nd_susp}, the non-degenerate simplices are exactly the elements which lie in the first summand, $(\Omega S)_{d-1}$ (excluding the basepoint if $d = 1$). Moreover, an easy check shows that the counit, restricted to this summand, is simply the inclusion into $S_d$. This map is then certainly injective on the non-degenerate simplices. It remains to show that the non-degenerate simplices are preserved, and this follows by Proposition~\ref{prop:nd_loops}, which tells us that a non-degenerate element in $(\Omega S)_{d-1}$ is necessarily non-degenerate in $S_d$ (except possibly in the case $d = 1$, where the element may also be the basepoint, but as mentioned just above, in the case $d = 1$, the basepoint is to be excluded).
\end{proof}

\section{Operadic (De)suspensions}\label{sec:opsusp}

We now discuss the second ingredient which we need prior to constructing stabilizations of $\mathbb{E}_\infty$ operads, that of operadic suspensions and desuspensions. For this purpose, we shall need to explicitly distinguish the case of chain complexes and cochain complexes, i.e., of chain operads and cochain operads. In either case, the operadic suspension and desuspension are endofunctors, which map a chain operad to another chain operad and a cochain operad to another cochain operad. If $X$ is a (co)chain complex and $n \in \Z$, we let $X[n]$ be the (co)chain complex where $X[n]_d = X_{d-n}$ (or, to be notationally more consistent in the cochain case, $X[n]^d = X^{d-n}$). We also let $k[n]$ denote the (co)chain complex with $k$ in degree $n$ and zeros elsewhere, where $k$ is an arbitrary ground field. We let $\mathpzc{End}_{X}$ denote the endomorphism operad on $X$, the operad having $\mathpzc{End}_X(n) = \mathrm{F}(X^{\otimes n}, X)$, where $\mathrm{F}(-,-)$ denotes the internal hom of (co)chain complexes. Finally, given (co)chain operads $\mathpzc{P}$ and $\mathpzc{Q}$, we let $\mathpzc{P} \otimes_k \mathpzc{Q}$ denote the operad where $(\mathpzc{P} \otimes_k \mathpzc{Q})(n) = \mathpzc{P}(n) \otimes_k \mathpzc{Q}(n)$ (the rest of the operad structure is inherited, in a straightforward manner, from the operad structures of $\mathpzc{P}$ and $\mathpzc{Q}$).

\begin{Definition}\label{def:operadic_susp_ch}
Given a chain operad $\mathpzc{P}$ over $k$:
\begin{itemize}
	\item The \textit{operadic suspension}\index{suspension!of operads} $\Sigma\mathpzc{P}$ is defined by setting $\Sigma \mathpzc{P} := \mathpzc{P} \otimes_k \mathpzc{End}_{k[1]}$.
	\item The \textit{operadic desuspension}\index{desuspension!of operads} $\Sigma^{-1}\mathpzc{P}$ is defined by setting $\Sigma^{-1} \mathpzc{P} := \mathpzc{P} \otimes_k \mathpzc{End}_{k[-1]}$.
\end{itemize}
Given a cochain operad $\mathpzc{P}$ over $k$:
\begin{itemize}
	\item The \textit{operadic suspension} $\Sigma\mathpzc{P}$ is defined by setting $\Sigma \mathpzc{P} := \mathpzc{P} \otimes_k \mathpzc{End}_{k[-1]}$.
	\item The \textit{operadic desuspension} $\Sigma^{-1}\mathpzc{P}$ is defined by setting $\Sigma^{-1} \mathpzc{P} := \mathpzc{P} \otimes_k \mathpzc{End}_{k[1]}$.
\end{itemize}
\end{Definition}

Note that we have $\mathpzc{End}_{k[1]}(n)_d = \{\text{maps $k[1]^{\otimes n} \to k[1]$ in $\mathsf{Gr}_k$ of degree $d$}\}$ where $\mathsf{Gr}_k$ denotes the category of $\Z$-graded $k$-modules. Thus $\mathpzc{End}_{k[1]}(n)_d$ is zero if $d \neq 1-n$ and is $k$ otherwise, so that $\mathpzc{End}_{k[1]}(n) = k[1-n]$. As a result, we find that:

\begin{center}
\begin{tabular}{m{6cm}  m{6cm}}
\hline
Chain operad $\mathpzc{P}$ & Cochain operad $\mathpzc{P}$ \tabularnewline \hline
\tabularnewline [-1em]
$(\Sigma \mathpzc{P})(n) \cong \mathpzc{P}(n)[1-n]$ & $(\Sigma \mathpzc{P})(n) \cong \mathpzc{P}(n)[n-1]$ \tabularnewline
\tabularnewline [-1em]
$(\Sigma^{-1} \mathpzc{P})(n) \cong \mathpzc{P}(n)[n-1]$ & $(\Sigma^{-1} \mathpzc{P})(n) \cong \mathpzc{P}(n)[1-n]$
\end{tabular}
\end{center}

As in Section~\ref{sec:nots_convs}, we have a reindexing operator $(-)^\dagger$ between chain and cochain operads. We now discuss how this construction behaves with respect to operadic (de)suspensions.

\begin{Proposition}\label{prop:opreindexsusp}
Let $\mathpzc{P}$ be a chain or cochain operad over $k$. We have that
\[
(\Sigma \mathpzc{P})^{\dagger} = \Sigma (\mathpzc{P}^{\dagger})
\quad
\text{and}
\quad
(\Sigma^{-1} \mathpzc{P})^{\dagger} = \Sigma^{-1} (\mathpzc{P}^{\dagger}).
\]
\end{Proposition}

\begin{proof}
An easy direct check one degree at a time. For example, in the case where $\mathpzc{P}$ is a chain operad, in operadic degree $n$ and chain degree $d$, we have:
\begin{align*}
(\Sigma \mathpzc{P})^{\dagger}(n)^d &= ((\Sigma \mathpzc{P})(n)^\dagger)^{d} \\
&= (\Sigma \mathpzc{P})(n)_{-d} \\
&= (\mathpzc{P}(n)[1-n])_{-d} \\
&= \mathpzc{P}(n)_{-d+n-1} \\
&= (\mathpzc{P}(n)^\dagger)^{d-n+1} \\
&= \mathpzc{P}^\dagger(n)^{d-n+1} \\
&= (\mathpzc{P}^\dagger(n)[n-1])^d \\
&= (\Sigma(\mathpzc{P}^\dagger))(n)^d
\end{align*}
Note that, for either identity, the notion of suspension on one side is that for chain operads, while on the other is for cochain operads.
\end{proof}

Next, we discuss (co)algebras over (de)suspended operads. Let $\mathpzc{P}$ be a dg operad over $k$. We wish to discuss the relation between (co)algebra structures over $\mathpzc{P}$ and (co)algebra structures over the (de)suspensions $\Sigma^r\mathpzc{P}$, where $r \in \Z$. Once again, we must explicitly distinguish between the chain and cochain cases.

\begin{Proposition}\label{prop:chainopsuspalg}
Let $\mathpzc{P}$ be a chain operad over $k$. Then, for each $r \in \Z$, we have functors as follows
\[
\mathpzc{P}\text{-}\mathsf{Alg} \overset{(\cdot)[r]}\longrightarrow \Sigma^r\mathpzc{P}\text{-}\mathsf{Alg}
\qquad
\mathpzc{P}\text{-}\mathsf{Coalg} \overset{(\cdot)[-r]}\longrightarrow \Sigma^r\mathpzc{P}\text{-}\mathsf{Coalg}.
\]
On the other hand, if $\mathpzc{P}$ is a cochain operad over $k$, then, we have functors as follows
\[
\mathpzc{P}\text{-}\mathsf{Alg} \overset{(\cdot)[-r]}\longrightarrow \Sigma^r\mathpzc{P}\text{-}\mathsf{Alg}
\qquad
\mathpzc{P}\text{-}\mathsf{Coalg} \overset{(\cdot)[r]}\longrightarrow \Sigma^r\mathpzc{P}\text{-}\mathsf{Coalg}.
\]
\end{Proposition}

\begin{proof}
We shall describe the case of algebras over a chain operad; the other three cases are analogous. Let $A$ be an algebra over a chain operad $\mathpzc{P}$. We wish to show that $A[r]$ is an algebra over $\Sigma^r\mathpzc{P}$. To do so, we must construct maps
\[
(\Sigma^r\mathpzc{P})(n) \otimes_{\Sigma_n} A[r]^{\otimes n} \to A[r].
\]
These are as follows
\[
(\Sigma^r \mathpzc{P})(n) \otimes_{\Sigma_n} A[r]^{\otimes n} = \mathpzc{P}(n)[r-rn] \otimes_{\Sigma_n} A^{\otimes n}[rn] = (\mathpzc{P}(n) \otimes_{\Sigma_n} A^{\otimes n})[r] \to A[r].
\]
\end{proof}

Finally, we discuss free algebras over (de)suspended operads. Given a dg operad $\mathpzc{P}$, and a (de)suspension $\Sigma^r\mathpzc{P}$, we shall let $\Sigma^r\mathbf{P}$ denote the monad and free algebra functor associated to $\Sigma^r\mathpzc{P}$. Note though that we are not (de)suspending the monad, but rather the operad.

\begin{Proposition}\label{prop:freealgsuspop}
Let $\mathpzc{P}$ denote a chain operad over $k$. Then, for a chain complex $X$, we have
\[
(\Sigma^r \mathbf{P}) X = \mathbf{P}(X[-r])[r].
\]
On the other hand, if $\mathpzc{P}$ is a cochain operad over $k$, for a cochain complex $X$, we have
\[
(\Sigma^r \mathbf{P}) X = \mathbf{P}(X[r])[-r].
\]
\end{Proposition}

\begin{proof}
We describe the case of a chain operad; the cochain case is analogous. We have
\begin{align*}
(\Sigma^r \mathbf{P}) X &= \bigoplus_{n \ge 0} (\Sigma^r \mathpzc{P})(n) \otimes_{\Sigma_n} X^{\otimes n} \\
&= \bigoplus_{n \ge 0} \mathpzc{P}(n)[r-rn] \otimes_{\Sigma_n} X^{\otimes n} \\
&= \bigoplus_{n \ge 0} (\mathpzc{P}(n) \otimes_{\Sigma_n} X^{\otimes n}[-rn])[r] \\
&= \left(\bigoplus_{n \ge 0} \mathpzc{P}(n) \otimes_{\Sigma_n} X[-r]^{\otimes n}\right)[r] \\
&= \mathbf{P}(X[-r])[r].
\end{align*}
\end{proof}

\section{The Stable Eilenberg-Zilber Operad}\label{subsec:stable_ez_operad}

Let us recall the definition of the Eilenberg-Zilber operad\index{Eilenberg-Zilber operad}, in both a chain and a cochain version. Let $\mathsf{Gr}_{\F_p}$ denote the category with objects the $\Z$-graded $\F_p$-modules and morphisms the homogeneous maps. Let also $\mathsf{Spc}$ denote the category of spaces, by which we mean simplicial sets. Fix some $n \ge 0$. For each $d \in \Z$, we have that
\[
\mathpzc{Z}(n)_d := \left\{
\begin{array}{ll}
\text{degree $d$ natural transformations from $\mathrm{C}_{\bullet}(-)$ to $\mathrm{C}_{\bullet}(-)^{\otimes n}$} \\
\text{where $\mathrm{C}_{\bullet}(-)$ and $\mathrm{C}_{\bullet}(-)^{\otimes n}$ are viewed as functors $\mathsf{Spc} \to \mathsf{Gr}_{\F_p}$}
\end{array}
\right\}.
\]
Thus $\mathpzc{Z}(n)_d$ consists of the degree $d$, $n$-ary co-operations on chains; note that, as always throughout this work, the chains here are normalized. Given such a natural transformation $\alpha = \{\alpha_S \colon \text{C}_\bullet(S) \to \text{C}_\bullet(S)^{\otimes n}\}$ of degree $d$, we get a natural transformation $\partial\alpha$ of degree $d-1$ by setting, for a simplicial set $S$ and a non-degenerate simplex $s \in S$, $(\partial\alpha)_S(s)= \partial\alpha_S(s) - (-1)^{d}\alpha_S\partial(s)$. An easy check shows that this gives us a differential $\partial \colon \mathpzc{Z}(n)_d \to \mathpzc{Z}(n)_{d-1}$. The operad identity in $\mathpzc{Z}(1)$ is the identity transformation. The symmetric group $\Sigma_n$ acts on $\mathpzc{Z}(n)$ by permuting the tensor factors. Finally, the operad composition maps are analogous to those that occur in co-endomorphism operads.

\begin{Definition}\label{def:ez_op}
The \textit{Eilenberg-Zilber chain operad} is the operad $\mathpzc{Z}$ defined above; the \textit{Eilenberg-Zilber cochain operad} is the operad $\mathpzc{Z}^\dagger$.
\end{Definition}

Here $(-)^\dagger$ refers to the reindexing construction -- see Section~\ref{sec:nots_convs}.

\begin{Remark}\label{rmk:cochainEZ}
An easy check shows that, since we are working over a field and since co-operations on chains and operations on cochains are determined by their values on the standard simplices, we can equivalently define the Eilenberg-Zilber cochain operad by setting
\[
\mathpzc{Z}^\dagger(n)_d := \left\{
\begin{array}{ll}
\text{degree $d$ natural transformations from $\mathrm{C}^{\bullet}(-)^{\otimes n}$ to $\mathrm{C}^{\bullet}(-)$} \\
\text{where $\mathrm{C}^{\bullet}(-)$ and $\mathrm{C}^{\bullet}(-)^{\otimes n}$ are viewed as functors $\mathsf{Spc}^{\text{op}} \to \mathsf{Gr}_{\F_p}$}
\end{array}
\right\}
\]
and setting the rest of the operad data in a manner analogous to that above for the chain operad. Thus $\mathpzc{Z}^\dagger(n)_d$ consists of the degree $d$, $n$-ary operations on cochains. \customendremark
\end{Remark}

In order to stabilize the Eilenberg-Zilber operad, we first need to introduce basepoints. Thus, we alter the above construction slightly, and consider instead the following operad, consisting of co-operations on chains on based simplicial sets.

\begin{Definition}
The \textit{reduced Eilenberg-Zilber chain operad}\index{reduced Eilenberg-Zilber operad}, denoted $\mathpzc{Z}_*$, is the operad constructed in the same manner as the Eilenberg-Zilber chain operad $\mathpzc{Z}$ except that the chains are to be taken on based simplicial sets (and so are of course reduced, as well as being normalized as always). Thus, for example, we have the following
\[
\mathpzc{Z}_*(n)_d := \left\{
\begin{array}{ll}
\text{degree $d$ natural transformations from $\mathrm{C}_{\bullet}(-)$ to $\mathrm{C}_{\bullet}(-)^{\otimes n}$} \\
\text{where $\mathrm{C}_{\bullet}(-)$ denotes normalized reduced chains and where} \\
\text{$\mathrm{C}_{\bullet}(-)$ and $\mathrm{C}_{\bullet}(-)^{\otimes n}$ are viewed as functors $\mathsf{Spc}_* \to \mathsf{Gr}_{\F_p}$}
\end{array}
\right\}.
\]
The \textit{reduced Eilenberg-Zilber cochain operad} is then defined to be $\mathpzc{Z}_*^\dagger$.
\end{Definition}

Now, we note that the operad $\mathpzc{Z}_{*}$ admits a stabilization map
\[
\Psi \colon \Sigma\mathpzc{Z}_{*} \to \mathpzc{Z}_{*}
\]
where $\Sigma$ here denotes the operadic suspension of chain operads, found in Definition~\ref{def:operadic_susp_ch}. This map, or the slight variant of it in the case of $\mathpzc{Z}$, has also been noted, for example, in~\cite{BergerFresse}, \cite{Smirnov1} and \cite{Smirnov2}. There are many ways in which one can describe this map; we give a description via the Kan suspension which will be useful for us later when we discuss spectral cochains. To do this, we need to specify maps $\mathpzc{Z}_{*}(n)[1-n] \to \mathpzc{Z}_{*}(n)$. Consider some natural transformation $\alpha$ in $\mathpzc{Z}_{*}(n)[1-n]_d$. This is a natural transformation of graded modules $\text{C}_\bullet(-) \to \text{C}_\bullet(-)^{\otimes n}$ of degree $d+n-1$ over based simplicial sets. By precomposition with the Kan suspension, we get a natural transformation $\text{C}_\bullet(\Sigma -) \to \text{C}_\bullet(\Sigma -)^{\otimes n}$ over based simplicial sets $X$. From Proposition~\ref{prop:n_r_chains_susp_X}, we have a natural isomorphism $\text{C}_\bullet(\Sigma -) \cong \text{C}_\bullet(-)[1]$. Thus, by pre and postcomposition, we get a natural transformation of graded modules $\text{C}_\bullet(-)[1] \to (\text{C}_\bullet(-)[1])^{\otimes n}$ of degree $d+n-1$, or equivalently a natural transformation $\text{C}_\bullet(-) \to \text{C}_\bullet(-)^{\otimes n}$ of degree $(d+n-1) + 1 - n = d$, which is to say an element, say $\alpha'$, of $\mathpzc{Z}_{*}(n)_{d}$. This correpondence $\alpha \mapsto \alpha'$ gives us a map $(\Sigma\mathpzc{Z}_{*})(n) \to \mathpzc{Z}_{*}(n)$. Moreover, one can check that this map is a chain map, and then that, assembling over $n$, we get an operad map $\Psi \colon \Sigma\mathpzc{Z}_{*} \to \mathpzc{Z}_{*}$, as desired. \\

We are now able to define our stabilization of the Eilenberg-Zilber operad. For each $k \ge 0$, we have a map, namely $\Sigma^k\Psi$, from $\Sigma^{k+1}\mathpzc{Z}_*$ to $\Sigma^k\mathpzc{Z}_*$.

\begin{Definition}\label{def:stable_ez_operad}
The \textit{stable Eilenberg-Zilber chain operad}\index{stable Eilenberg-Zilber operad}, denoted $\mathpzc{Z}_{\text{st}}$, is the operad defined as follows:
\[
\mathpzc{Z}_{\text{st}} := \underset{\longleftarrow}{\text{lim}}(\cdots \overset{\Sigma^2\Psi}\longrightarrow \Sigma^2\mathpzc{Z}_{*} \overset{\Sigma\Psi}\longrightarrow \Sigma\mathpzc{Z}_{*} \overset{\Psi}\longrightarrow \mathpzc{Z}_{*}).
\]
The \textit{stable Eilenberg-Zilber cochain operad} is the operad $\mathpzc{Z}^\dagger_{\text{st}}$.
\end{Definition}

\begin{Remark}\label{rmk:stable_ez_cochain}
By Proposition~\ref{prop:opreindexsusp}, the operadic suspension $\Sigma$ and the reindexing operator $(-)^\dagger$ commute, so that we also have maps $\Sigma^{k+1}\mathpzc{Z}_*^\dagger \to \Sigma^{k}\mathpzc{Z}_*^\dagger$, and moreover, $(-)^\dagger$ clearly also commutes with inverse limits, so that one can also analogously construct the stable Eilenberg-Zilber cochain operad as the inverse limit
\[
\mathpzc{Z}^\dagger_{\text{st}} := \underset{\longleftarrow}{\text{lim}}(\cdots \overset{\Sigma^2\Psi^\dagger}\longrightarrow \Sigma^2\mathpzc{Z}^\dagger_{*} \overset{\Sigma\Psi^\dagger}\longrightarrow \Sigma\mathpzc{Z}^\dagger_{*} \overset{\Psi^\dagger}\longrightarrow \mathpzc{Z}^\dagger_{*}).
\] \customendremark
\end{Remark}

\begin{Remark}
The Eilenberg-Zilber chain and cochain operads are defined using natural transformations on (co)chains on spaces. The stable Eilenberg-Zilber operads also admit a similar interpretation, but using spectral (co)chains. See Remark~\ref{stable_EZ_operad_as_spectral_ops}. \customendremark
\end{Remark}

Consider the chain complex $\mathpzc{Z}_{\text{st}}(n)$ in operadic degree $n$ of the stable Eilenberg-Zilber chain operad. Since limits of operads are formed termwise, $\mathpzc{Z}_{\text{st}}(n)$ is equivalent to the limit, in chain complexes, of the diagram
\[
\cdots \longrightarrow \Sigma^2\mathpzc{Z}_{*}(n) \longrightarrow \Sigma\mathpzc{Z}_{*}(n) \longrightarrow \mathpzc{Z}_{*}(n).
\]
That is, it is the limit of
\[
\cdots \longrightarrow \mathpzc{Z}_*(n)[2-2n] \longrightarrow \mathpzc{Z}_*(n)[1-n] \longrightarrow \mathpzc{Z}_*(n).
\]
In degree $d \in \Z$ then, we have
\begin{equation}\label{eq:EZst_elts}
\mathpzc{Z}_{\text{st}}(n)_d \subseteq \prod_{k \ge 0} \mathpzc{Z}_{*}(n)[k-kn]_d = \prod_{k \ge 0} \mathpzc{Z}_{*}(n)_{d+kn-k}.
\end{equation}
More specifically, an element of $\mathpzc{Z}_{\text{st}}(n)$ in degree $d$ is a sequence $(\alpha_0,\alpha_1,\alpha_2,\dots)$ where $\alpha_0$ is a degree $d$ chain co-operation $\text{C}_\bullet(-) \to \text{C}_\bullet(-)^{\otimes n}$, $\alpha_1$ is a degree $d+n-1$ chain co-operation $\text{C}_\bullet(-) \to \text{C}_\bullet(-)^{\otimes n}$ such that $\Psi(\alpha_1) = \alpha_0$, and so on. Similarly, the cochain complex $\mathpzc{Z}^\dagger_{\text{st}}(n)$ is the limit, in cochain complexes, of the diagram
\[
\cdots \longrightarrow \mathpzc{Z}^\dagger_*(n)[2n-2] \longrightarrow \mathpzc{Z}^\dagger_*(n)[n-1] \longrightarrow \mathpzc{Z}^\dagger_*(n)
\]
and, recalling what was said in Remark~\ref{rmk:cochainEZ}, an element of $\mathpzc{Z}_{\text{st}}^\dagger(n)$ in degree $d$ is a sequence $(\alpha_0,\alpha_1,\alpha_2,\dots)$ where $\alpha_0$ is a degree $d$ cochain operation $\text{C}^\bullet(-)^{\otimes n} \to \text{C}^\bullet(-)$, $\alpha_1$ is a degree $d-n+1$ cochain operation $\text{C}^\bullet(-)^{\otimes n} \to \text{C}^\bullet(-)$ such that $\Psi(\alpha_1) = \alpha_0$, and so on.

\section{The Stable McClure-Smith Operad}\label{subsec:stable_MS_operad}

In this section, as we did for the Eilenberg-Zilber operad, we stabilize the McClure-Smith operad, as always in both a chain form and a cochain form. First, we need to recall the definition of the McClure-Smith operad. This operad was used by McClure and Smith in~\cite{McClureSmith}, and is also discussed in~\cite{BergerFresse}, where it is called the surjection operad. In~\cite{McClureSmith}, Theorem 2.15 and Remark 2.16(a) demonstrate that this operad is an $\mathbb{E}_\infty$ operad, so that the stabilization in this section will constitute our first stabilization of an $\mathbb{E}_\infty$ operad. \\

In order to recall the definition of the McClure-Smith operad, we require some preliminaries. For each integer $n \ge 0$, let $(n)$ denote the set $\{1,\dots,n\}$, where $(0) := \varnothing$. Given a set map $f \colon (m) \to (n)$, we will often view it as, and denote it by, the ordered sequence $(f(1),\dots,f(m))$. Suppose given $n \ge 0$ and a surjection $f \colon (m) \to (n)$. Say that $f$ is \textit{degenerate}\index{degenerate surjection} if there exists some $l \in (m)$ such that $f(l) = f(l+1)$, and otherwise \textit{non-degenerate}; that is, call $f$ degenerate exactly when the sequence $(f(1),\dots,f(m))$ contains two equal adjacent entries. For each $n \ge 0$, let $\text{S}(n)$ be the graded $\F_p$-module freely generated by maps $f \colon (-) \to (n)$, where, if the source is $(m)$, the assigned degree is $m-n$. Let $\text{N}(n)$ denote the sub graded module generated by the non-surjective maps and $\text{D}(n)$ the sub graded module generated by the degenerate surjections. For each $n \ge 0$, set
\[
\mathpzc{M}(n) := \text{S}(n)/(\text{N}(n)+\text{D}(n)).
\]

\begin{Remark}
Taken over $n \ge 0$, the above graded modules will be the underlying graded modules of the McClure-Smith chain operad. It is clear that $\mathpzc{M}(n)$ is the graded $\F_p$-module freely generated by the non-degenerate surjections $f \colon (-) \to (n)$, where, as above, if the source is $(m)$, the assigned degree is $m-n$. Moreover, $\mathpzc{M}(n)_d$ is zero if $d < 0$, and otherwise is freely generated by the non-degenerate surjections $(n+d) \to (n)$. \customendremark
\end{Remark}

We now endow the $\mathpzc{M}(n)$ with differentials. To keep track of signs, given a surjection $f \colon (m) \to (n)$ and $i \in (m)$, we set
\[
\tau_f(i) = \#\{j \in (m) \mid f(j) < f(i) \: \text{or} \: (f(j) = f(i) \: \text{and} \: j \le i)\}.
\]
The differential of $\mathpzc{M}(n)$ is then defined as follows. Given a non-degenerate surjection $f \colon (m) \to (n)$, denoted also by $(f(1),\dots,f(m))$, we have that
\[
\partial(f) = \sum_{i=1}^m (-1)^{\tau_f(i)-f(i)} (f(1),\dots,\widehat{f(i)},\dots,f(m)).
\]
Here, for $i \in (m)$, the term $(f(1),\dots,\widehat{f(i)},\dots,f(m))$ on the righthand side denotes the map $(m-1) \to (n)$ whose images are, in order, exactly those that appear in the sequence $(f(1),\dots,f(i-1),f(i+1),\dots,f(m))$; if, upon omitting $f(i)$, this resulting map is no longer a surjection or is a degenerate surjection, that term is taken to be zero. The verification that this indeed defines a well-defined differential may be found in Proposition 2.18 of~\cite{McClureSmith}. \\

It remains to describe the operadic identity, symmetric group actions and composition data. The identity in $\mathpzc{M}(1)$ is the identity on $(1)$ and the action of $\Sigma_n$ on $\mathpzc{M}(n)$ is by postcomposition of surjections onto $(n)$. As for the composition maps, we shall specify this in the form of maps $\circ_r \colon \mathpzc{M}(n) \otimes \mathpzc{M}(m) \to \mathpzc{M}(n+m-1)$, for $n,m \ge 0$ and $r=1,\dots,n$. Let $f \colon (N) \to (n)$ and $g \colon (M) \to (m)$ be non-degenerate surjections. We need to define a composite $f \circ_r g$, which will be zero or a non-degenerate surjection $(N+M-1) \to (n+m-1)$. This composite can be described algorithmically as follows:
\begin{itemize}
	\item In the sequence $(f(1),\dots,f(N))$, let $t$ be the number of occurences of $r$. Let these occurences be given by $f(i_1),\dots,f(i_t)$.
	\item Fix a choice of $t+1$ entries $1 = j_0 \le j_1 \le \cdots \le j_{t-1} \le j_t = M$ inside $(M)$, where the first is $1$ and the final $M$. In the sequence $(f(1),\dots,f(N))$, replace $f(i_1)$ by the subsequence $(g(j_{0}),\dots,g(j_1))$, $f(i_2)$ by the subsequence $(g(j_{1}),\dots,g(j_2))$, and so on, with the final replacement being that of $f(i_t)$ by the subsequence $(g(j_{t-1}),\dots,g(j_t))$. Note that the resulting sequence has length $N-t+M+t-1 = N+M-1$. Now alter this sequence as follows: (i) increase each entry $g(j)$ which has been entered by $r-1$ (ii) increase those entries $f(i)$ which remain and where $f(i) > r$ by $M-1$.
	\item The resulting sequence gives a map $f_{(j_0,\dots,j_t)} \colon (N+M-1) \to (n+m-1)$; if it is not a surjection or is a degenerate surjection, replace it with zero.
	\item The composite $f \circ_r g$ is then the sum of all the resulting maps $f_{(j_0,\dots,j_t)}$, the sum being taken over the tuples $(j_0, j_1, \dots j_t)$.
\end{itemize}
The verification that the above algorithmic procedure yields well-defined maps $\mathpzc{M}(n) \otimes \mathpzc{M}(m) \to \mathpzc{M}(n+m-1)$ for $r=1,\dots,n$, and yields an operad structure on the chain complexes $\mathpzc{M}(n)$, may be found as Proposition 1.2.7 in~\cite{BergerFresse}.

\begin{Definition}
The \textit{McClure-Smith chain operad}\index{McClure-Smith operad} is the operad consisting of the chain complexes $\mathpzc{M}(n)$ and structural data specified above. The \textit{McClure-Smith cochain operad} is the operad $\mathpzc{M}^{\dagger}$.
\end{Definition}

Now we shall describe the stabilizations. We first discuss the case of the chain operad. As with the Eilenberg-Zilber chain operad, we first note that $\mathpzc{M}$ admits a stabilization map
\[
\Psi \colon \Sigma \mathpzc{M} \to \mathpzc{M}.
\]
This map was also observed in~\cite{BergerFresse}. Here we have once again used the symbol $\Psi$, just as in the stabilization of $\mathpzc{Z}_{*}$. The context will always make it clear which map we intend by this symbol. Now, to define this map, we need to define, for each $n \ge 0$, a map $(\Sigma \mathpzc{M})(n) \to \mathpzc{M}(n)$, which is to say a map $\mathpzc{M}(n)[1-n] \to \mathpzc{M}(n)$. Consider a non-degenerate surjection $f \in \mathpzc{M}(n)[1-n]_d$. This is a non-degenerate surjection $f \colon (m) \to (n)$ where $m-n = d+n-1$ and so $m = d+2n-1$. We define $\Psi(f)$ algorithmically as follows:
\begin{itemize}
	\item If $(f(1),\dots,f(n))$ is not a permutation of $(1,\dots,n)$, $\Psi(f)$ is zero.
	\item If $(f(1),\dots,f(n))$ is a permutation of $(1,\dots,n)$, $\Psi(f)$ is represented by the map $(n+d) \to (n)$ given by the sequence $(f(n),\dots,f(d+2n-1))$, together with a sign which is the sign of $(f(1),\dots,f(n))$ as a permutation of $(1,\dots,n)$.
\end{itemize}

One can check that the above algorithmic procedure does indeed yield an operad map $\Psi \colon \Sigma \mathpzc{M} \to \mathpzc{M}$. See Remark 3.2.10 in~\cite{BergerFresse}. 

\begin{Definition}
The \textit{stable McClure-Smith chain operad}\index{stable McClure-Smith operad}, denoted $\mathpzc{M}_{\text{st}}$, is the operad defined as follows
\[
\mathpzc{M}_{\text{st}} := \underset{\longleftarrow}{\text{lim}}(\cdots \overset{\Sigma^2\Psi}\longrightarrow \Sigma^2\mathpzc{M} \overset{\Sigma\Psi}\longrightarrow \Sigma\mathpzc{M} \overset{\Psi}\longrightarrow \mathpzc{M}).
\]
The \textit{stable McClure-Smith cochain operad} is defined to be $\mathpzc{M}_{\text{st}}^\dagger$.
\end{Definition}

\begin{Remark}\label{rmk:stable_ms_cochain}
By Proposition~\ref{prop:opreindexsusp}, the operadic suspension $\Sigma$ and the reindexing operator $(-)^\dagger$ commute, so that we also have maps $\Sigma^{k+1}\mathpzc{M}^\dagger \to \Sigma^{k}\mathpzc{M}^\dagger$, and moreover, $(-)^\dagger$ clearly also commutes with inverse limits, so that one can also analogously construct the stable McClure-Smith cochain operad as the inverse limit
\[
\mathpzc{M}_{\text{st}}^\dagger := \underset{\longleftarrow}{\text{lim}}(\cdots \overset{\Sigma^2\Psi^\dagger}\longrightarrow \Sigma^2\mathpzc{M}^\dagger \overset{\Sigma\Psi^\dagger}\longrightarrow \Sigma\mathpzc{M}^\dagger \overset{\Psi^\dagger}\longrightarrow \mathpzc{M}^\dagger).
\] \customendremark
\end{Remark}

Next, we also note a comparison between the stabilization maps for the McClure-Smith and Eilenberg-Zilber operads. We note that there is a map, in fact an embedding, from the McClure-Smith operad to the Eilenberg-Zilber operad. To describe this map, we need some preliminaries. Given a finite linearly ordered set $L$, an \textit{overlapping partition} $\mathcal{A}$ of $L$ with $m$ pieces is a collection of subsets $A_1,\dots, A_m$ of $L$ with the following properties: (i) $\cup_i A_i = L$, (ii) if $i < j$, then each element of $A_i$ is $\leq$ each element of $A_j$, (iii) for $i < m$, $A_i \cap A_{i+1}$ has exactly one element.

\begin{Definition}\label{defn:seqcoop}
Suppose given $n \ge 0$ and a surjection $f \colon (m) \to (n)$. We then get a natural transformation, over simplicial sets, $\langle f \rangle \colon \text{C}_\bullet(-) \to \text{C}_\bullet(-)^{\otimes n}$, where, given a simplicial set $S$ and $\sigma \colon \Delta_p \to S$, we have that
\begin{equation}\label{eqn:AW_seq_ops}
\langle f \rangle (\sigma) = \sum_{\substack{\text{overlapping partitions} \\ \text{$(A_1, \dots, A_m)$ of $\{0,1,\dots,p\}$} \\ \text{with $m$ pieces}}} \bigotimes_{i=1}^n \sigma|\amalg_{f(j)=i}A_j.
\end{equation}
We call these natural transformations, the \textit{sequence co-operations}. Note that since the righthand side above has degree $\sum_{i=1}^n (|f^{-1}(i)|-1) = m - n$, the sequence operation $\langle f \rangle$ is homogeneous of degree $m-n$.
\end{Definition}

Note that, given $n \ge 0$ and a surjection $f \colon (m) \to (n)$, if $f$ is degenerate, $\langle f \rangle$ is the zero transformation. To see this, note that, if the surjection is degenerate, one of the tensor factors in the righthand side of (\ref{eqn:AW_seq_ops}) receives a repeated coordinate and so is zero as the chains are normalized. As a result, for each $n \ge 0$, we have a map $\text{AW}_n \colon \mathpzc{M}(n) \to \mathpzc{Z}(n)$. As in~\cite{BergerFresse}, we use the notation ``$\text{AW}_n$'' because the sequence operations generalize the classical Alexander-Whitney diagonal operation. In~\cite{McClureSmith}, Theorem 2.15 shows that the $\text{AW}_n$ are injective and that they yield an operad map $\text{AW} \colon \mathpzc{M} \to \mathpzc{Z}$. Thus the McClure-Smith operad $\mathpzc{M}$ embeds into the Eilenberg-Zilber operad $\mathpzc{Z}$. The $\text{AW}_n$ are also quasi-isomorphisms. To see this, note that, for each $n \ge 0$, $\mathpzc{M}(n)$ and $\mathpzc{Z}(n)$ have the homology of a point; for the former, see Theorem 2.15 in~\cite{McClureSmith}, and for the latter, see Proposition 3.2 in~\cite{MayUnpublished}. Thus we need only check that $\mathrm{AW}_n$ maps a generator of the degree zero homology of $\mathpzc{M}(n)$ to a generator of the degree zero homology of $\mathpzc{Z}(n)$. If, for each $n \ge 0$, we let $\varepsilon_n$ denote the map $\mathpzc{Z}(n) \to \mathbb{F}_p[0]$ which sends a degree zero natural transformation $\alpha = \{\alpha_S \colon \text{C}_\bullet(S) \to \text{C}_\bullet(S)^{\otimes n}\}$ to the element of $\F_p$ corresponding to $\alpha_{\Delta_0} \in \mathrm{Hom}_{\mathrm{Gr}_{\mathbb{F}_p}}(\text{C}_\bullet(\Delta_0), \text{C}_\bullet(\Delta_0)^{\otimes n}) \cong \mathrm{Hom}_{\mathrm{Gr}_{\mathbb{F}_p}}(\F_p[0], \F_p[0]) \cong \F_p$, Proposition 3.2 in~\cite{MayUnpublished} shows that in fact, for each $n \ge 0$, $\varepsilon_n$ is a quasi-isomorphism. The degree zero homology of $\mathpzc{M}(n)$ is generated by $\mathrm{id}_{(n)}$, the identity on $(n)$, and an easy check shows that $(\varepsilon_n \circ \mathrm{AW}_n)(\mathrm{id}_{(n)}) = 1$, which gives us the desired result. Note that, upon applying $(-)^\dagger$, we also get a quasi-isomorphic embedding $\text{AW}^\dagger \colon \mathpzc{M}^\dagger \to \mathpzc{Z}^\dagger$ of cochain operads. \\

Now, given a surjection $f \colon (m) \to (n)$, the sequence co-operations, as defined in Definition~\ref{defn:seqcoop}, in fact yield a natural transformation $\text{C}_\bullet(-) \to \text{C}_\bullet(-)^{\otimes n}$ not only over simplicial sets, but also over based simplicial sets. To see this, let $S$ be a based simplicial set and let $\Delta_0 \to S$ classify the basepoint $*_S \in S_0$ of $S$. Then, since every piece of an overlapping partition of $[0]$ is just $\{0\}$, the only case in which each tensor factor in (\ref{eqn:AW_seq_ops}) will be non-degenerate is if the input surjection is the identity on $(n)$. In this case the image is $*_S \otimes \cdots \otimes *_S$, and this tensor is zero in $\text{C}_\bullet(X)^{\otimes n}$ as the chains are reduced chains. As such, an easy check shows us that we then have an operad map
\[
\text{AW} \colon \mathpzc{M} \to \mathpzc{Z}_*
\]
which we denote by the same symbol, $\text{AW}$, as earlier. Upon applying $(-)^\dagger$, we also get a map between the corresponding cochain operads
\[
\text{AW}^\dagger \colon \mathpzc{M}^\dagger \to \mathpzc{Z}_*^\dagger.
\]

Moreover, these maps are compatible with the stabilizations of the two operads in the sense that the following squares commute:
\begin{center}
\begin{tikzpicture}[node distance = 1.5cm]
\node [] (A) {$\Sigma \mathpzc{M}$};
\node [right of = A,xshift=1.5cm] (B) {$\mathpzc{M}$};
\node [below of = A] (C) {$\Sigma \mathpzc{Z}_{*}$};
\node [below of = B] (D) {$\mathpzc{Z}_{*}$};

\draw [->] (A) -- (B) node[midway,anchor=south]{$\Psi$};
\draw [->] (A) -- (C) node[midway,anchor=east]{$\Sigma \text{AW}$};
\draw [->] (B) -- (D) node[midway,anchor=west]{$\text{AW}$};
\draw [->] (C) -- (D) node[midway,anchor=north]{$\Psi$};
\end{tikzpicture}
\qquad
\begin{tikzpicture}[node distance = 1.5cm]
\node [] (A) {$\Sigma \mathpzc{M}^\dagger$};
\node [right of = A,xshift=1.5cm] (B) {$\mathpzc{M}^\dagger$};
\node [below of = A] (C) {$\Sigma \mathpzc{Z}_{*}^\dagger$};
\node [below of = B] (D) {$\mathpzc{Z}_{*}^\dagger$};

\draw [->] (A) -- (B) node[midway,anchor=south]{$\Psi$};
\draw [->] (A) -- (C) node[midway,anchor=east]{$\Sigma \text{AW}^\dagger$};
\draw [->] (B) -- (D) node[midway,anchor=west]{$\text{AW}^\dagger$};
\draw [->] (C) -- (D) node[midway,anchor=north]{$\Psi$};
\end{tikzpicture}
\end{center}

A tedious but not too difficult diagram chase confirms this -- see also Remark 3.2.10 in~\cite{BergerFresse}. As a result, we have commutative diagrams as follows:
\begin{center}
\begin{tikzpicture}[node distance = 1.5cm]
\node [] (A) {$\Sigma \mathpzc{M}$};
\node [right of = A,xshift=1.5cm] (B) {$\mathpzc{M}$};
\node [below of = A] (C) {$\Sigma \mathpzc{Z}_{*}$};
\node [below of = B] (D) {$\mathpzc{Z}_{*}$};
\node [left of = A,xshift=-1.5cm] (E) {$\Sigma^2 \mathpzc{M}$};
\node [left of = E,xshift=-1.5cm] (F) {$\cdots$};
\node [left of = C,xshift=-1.5cm] (G) {$\Sigma^2 \mathpzc{Z}_{*}$};
\node [left of = G,xshift=-1.5cm] (H) {$\cdots$};

\draw [->] (A) -- (B) node[midway,anchor=south]{$\Psi$};
\draw [->] (A) -- (C) node[midway,anchor=west]{$\Sigma \text{AW}$};
\draw [->] (B) -- (D) node[midway,anchor=west]{$\text{AW}$};
\draw [->] (C) -- (D) node[midway,anchor=south]{$\Psi$};
\draw [->] (F) -- (E) node[midway,anchor=south]{$\Sigma^2\Psi$};
\draw [->] (E) -- (A) node[midway,anchor=south]{$\Sigma\Psi$};
\draw [->] (H) -- (G) node[midway,anchor=south]{$\Sigma^2\Psi$};
\draw [->] (G) -- (C) node[midway,anchor=south]{$\Sigma\Psi$};
\draw [->] (E) -- (G) node[midway,anchor=west]{$\Sigma^2 \text{AW}$};
\end{tikzpicture}
\end{center}
\begin{center}
\begin{tikzpicture}[node distance = 1.5cm]
\node [] (A) {$\Sigma \mathpzc{M}^\dagger$};
\node [right of = A,xshift=1.5cm] (B) {$\mathpzc{M}^\dagger$};
\node [below of = A] (C) {$\Sigma \mathpzc{Z}_{*}^\dagger$};
\node [below of = B] (D) {$\mathpzc{Z}_{*}^\dagger$};
\node [left of = A,xshift=-1.5cm] (E) {$\Sigma^2 \mathpzc{M}^\dagger$};
\node [left of = E,xshift=-1.5cm] (F) {$\cdots$};
\node [left of = C,xshift=-1.5cm] (G) {$\Sigma^2 \mathpzc{Z}_{*}^\dagger$};
\node [left of = G,xshift=-1.5cm] (H) {$\cdots$};

\draw [->] (A) -- (B) node[midway,anchor=south]{$\Psi$};
\draw [->] (A) -- (C) node[midway,anchor=west]{$\Sigma \text{AW}^\dagger$};
\draw [->] (B) -- (D) node[midway,anchor=west]{$\text{AW}^\dagger$};
\draw [->] (C) -- (D) node[midway,anchor=south]{$\Psi$};
\draw [->] (F) -- (E) node[midway,anchor=south]{$\Sigma^2\Psi$};
\draw [->] (E) -- (A) node[midway,anchor=south]{$\Sigma\Psi$};
\draw [->] (H) -- (G) node[midway,anchor=south]{$\Sigma^2\Psi$};
\draw [->] (G) -- (C) node[midway,anchor=south]{$\Sigma\Psi$};
\draw [->] (E) -- (G) node[midway,anchor=west]{$\Sigma^2 \text{AW}^\dagger$};
\end{tikzpicture}
\end{center}
From these, we get induced maps
\begin{equation}\label{eq:can_MSst_to_EZst}
\text{AW}_{\text{st}} \colon \mathpzc{M}_{\text{st}} \to \mathpzc{Z}_{\text{st}}
\qquad
\text{AW}_{\text{st}}^\dagger \colon \mathpzc{M}_{\text{st}}^\dagger \to \mathpzc{Z}_{\text{st}}^\dagger
\end{equation}
between the stable Eilenberg-Zilber and stable McClure-Smith, chain and cochain, operads.

\section{The Stable Barratt-Eccles Operad}\label{subsec:stable_barratt_eccles_operad}

We have now constructed stabilizations of both the Eilenberg-Zilber and McClure-Smith operads. In this section, we present a third stabilization, that of the Barratt-Eccles operad. This will constitute a second stabilization of an $\mathbb{E}_\infty$ operad. First, we recall the definition of the Barratt-Eccles operad. In spaces, by which we mean simplicial sets, there exists an operad $\mathpzc{E}_{\text{spc}}$ where
\[
\mathpzc{E}_{\text{spc}}(n) = \text{E}\Sigma_n.
\]
Here $\text{E}\Sigma_n$ denotes the total space of the universal $\Sigma_n$-bundle; in particular, in simplicial degree $d$, we have that $(\text{E}\Sigma_n)_d = \Sigma_n^{\times (d+1)}$. This is the operad originally defined by Barratt and Eccles in~\cite{BarrattEccles}. We are interested in dg operads and so we take chains on this operad.

\begin{Definition}\label{def:BE_op}
The \textit{Barratt-Eccles chain operad}\index{Barratt-Eccles operad}, denoted $\mathpzc{E}$, is the dg operad over $\F_p$ defined by
\[
\mathpzc{E}(n) = \text{C}_\bullet(\text{E}\Sigma_n).
\]
Moreover, the \textit{Barratt-Eccles cochain operad} is then the operad $\mathpzc{E}^\dagger$.
\end{Definition}

Here $\text{C}_\bullet$ denotes normalized chains, and we get a dg operad upon taking these because the normalized chains functor is symmetric monoidal. The chains here are of course taken with coefficients in $\F_p$, so that we get a dg operad over $\F_p$. Note that the Barratt-Eccles chain operad is concentrated in non-negative degrees, while the Barratt-Eccles cochain operad is concentrated in non-positive degrees. \\

As in the previous sections, we begin the stabilization by noting a stabilization map
\[
\Psi \colon \Sigma \mathpzc{E} \to \mathpzc{E}
\]
which will yet again be denoted by $\Psi$ just as in the previous two stabilizations (the context will always make it clear which map we intend by this symbol). In order to define this map, we need to define, for each $n \ge 0$, a map $(\Sigma \mathpzc{E})(n) \to \mathpzc{E}(n)$, which is to say a map $\mathpzc{E}(n)[1-n] \to \mathpzc{E}(n)$. Consider a tuple $(\rho_0,\dots,\rho_{d+n-1})$ in $\mathpzc{E}(n)[1-n]_d = \mathpzc{E}(n)_{d+n-1}$, where each $\rho_i$ is a permutation in $\Sigma_n$. Then we define the image $\Psi((\rho_0,\dots,\rho_{d+n-1}))$ algorithmically as follows:
\begin{itemize}
	\item If $(\rho_0(1),\dots,\rho_{n-1}(1))$ is not a permutation of $(1,\dots,n)$, then the image $\Psi((\rho_0,\dots,\rho_{d+n-1}))$ is zero.
	\item If $(\rho_0(1),\dots,\rho_{n-1}(1))$ is a permutation of $(1,\dots,n)$, $\Psi((\rho_0,\dots,\rho_{d+n-1}))$ is the tuple $(\rho_{n-1},\dots,\rho_{d+n-1}) \in \mathpzc{E}(n)_{d}$, together with a sign which is given by the sign of $(\rho_0(1),\dots,\rho_{n-1}(1))$ as a permutation of $(1,\dots,n)$.
\end{itemize}

One can check that the above algorithmic procedure does indeed yield an operad map $\Psi \colon \Sigma \mathpzc{E} \to \mathpzc{E}$ (see Proposition 3.2.9 in~\cite{BergerFresse}).

\begin{Definition}
The \textit{stable Barratt-Eccles chain operad}\index{stable Barratt-Eccles operad}, denoted $\mathpzc{E}_{\text{st}}$, is the operad defined as follows
\[
\mathpzc{E}_{\text{st}} := \underset{\longleftarrow}{\text{lim}}(\cdots \overset{\Sigma^2\Psi}\longrightarrow \Sigma^2\mathpzc{E} \overset{\Sigma\Psi}\longrightarrow \Sigma\mathpzc{E} \overset{\Psi}\longrightarrow \mathpzc{E}).
\]
The \textit{stable Barratt-Eccles cochain operad} is the operad $\mathpzc{E}_{\text{st}}^\dagger$.
\end{Definition}

\begin{Remark}\label{rmk:stable_be_cochain}
By Proposition~\ref{prop:opreindexsusp}, the operadic suspension $\Sigma$ and the reindexing operator $(-)^\dagger$ commute, so that we also have maps $\Sigma^{k+1}\mathpzc{E}^\dagger \to \Sigma^{k}\mathpzc{E}^\dagger$, and moreover, $(-)^\dagger$ clearly also commutes with inverse limits, so that one can also analogously construct the stable Barratt-Eccles cochain operad as the inverse limit
\[
\mathpzc{E}_{\text{st}}^\dagger := \underset{\longleftarrow}{\text{lim}}(\cdots \overset{\Sigma^2\Psi^\dagger}\longrightarrow \Sigma^2\mathpzc{E}^\dagger \overset{\Sigma\Psi^\dagger}\longrightarrow \Sigma\mathpzc{E}^\dagger \overset{\Psi^\dagger}\longrightarrow \mathpzc{E}^\dagger).
\] \customendremark
\end{Remark}

Next, we note a relation between the Barratt-Eccles and McClure-Smith operads. The two chain operads are related via a map
\[
\text{TR} \colon \mathpzc{E} \to \mathpzc{M}.
\]
This map can be described algorithmically as follows:
\begin{itemize}
	\item Consider some tuple $(\rho_0,\dots,\rho_d) \in \mathpzc{E}(n)_d = \text{C}_d(\text{E}\Sigma_n)$, where the $\rho_i$ are permutations of $(n)$.
	\item Let $r_0,\dots,r_d$ be any positive integers such that $r_0+ \cdots +r_d = n+d$. Note that each $r_i$ is necessarily $\le n$; moreover, each $r_0 + \cdots + r_i$ is necessarily $\le n+i$. Form a sequence $(\rho_0(1),\dots,\rho_0(r_0))$ of length $r_0$ using the first $r_0$ entries of the sequence given by $\rho_0$. Next, form a sequence of length $r_1$ using the first $r_1$ entries of the sequence given by $\rho_1$, but skipping any entry which has already occured as a non-final entry of a previous sequence (there are $r_0-1$ such entries, and we have $n-(r_0-1)-r_1 = n-r_0-r_1+1 \ge 0$). Now repeat this process to construct sequences of length $r_2,\dots,r_d$.
	\item Concatenate the $d+1$ sequences constructed in the previous point to construct an indexed sequence of length $r_0+\cdots+r_d = n+d$. This yields a map $f_{(r_0,\dots,r_d)} \colon (n+d) \to (n)$. If this map is not a surjection or is a non-degenerate surjection, replace it by zero. 
	\item The image of $(\rho_0,\dots,\rho_d)$ under $\text{TR}_n$ is the sum $\sum f_{(r_0,\dots,r_d)}$ over the tuples $(r_0,\dots,r_d)$.
\end{itemize}
The verification that the above algorithmic procedure defines a well-defined map of operads may be found as Assertion 1.3.3 in~\cite{BergerFresse}. Moreover Lemma 1.3.4 and Lemma 1.6.1 in the same work show that $\mathrm{TR}$ is both onto and a quasi-isomorphism in each operadic degree. Thus we see that the McClure-Smith chain operad $\mathpzc{M}$ is a quotient of the Barratt-Eccles chain operad $\mathpzc{E}$. Moreover, upon reindexing, we also get a map between the corresponding cochain operads
\[
\text{TR}^\dagger \colon \mathpzc{E}^\dagger \to \mathpzc{M}^\dagger.
\]
This map is of course also a surjective quasi-isomorphism in each operadic degree, so that the McClure-Smith cochain operad $\mathpzc{M}^\dagger$ is a quotient of the Barratt-Eccles cochain operad $\mathpzc{E}^\dagger$. \\

Finally, we note that the stabilization maps for the Barratt-Eccles and McClure-Smith operads are compatible in the sense that the following squares commute (see Remark 3.2.10 in~\cite{BergerFresse}).
\begin{center}
\begin{tikzpicture}[node distance = 1.5cm]
\node [] (A) {$\Sigma \mathpzc{E}$};
\node [right of = A,xshift=1.5cm] (B) {$\mathpzc{E}$};
\node [below of = A] (C) {$\Sigma \mathpzc{M}$};
\node [below of = B] (D) {$\mathpzc{M}$};

\draw [->] (A) -- (B) node[midway,anchor=south]{$\Psi$};
\draw [->] (A) -- (C) node[midway,anchor=east]{$\Sigma \text{TR}$};
\draw [->] (B) -- (D) node[midway,anchor=west]{$\text{TR}$};
\draw [->] (C) -- (D) node[midway,anchor=north]{$\Psi$};
\end{tikzpicture}
\qquad
\begin{tikzpicture}[node distance = 1.5cm]
\node [] (A) {$\Sigma \mathpzc{E}^\dagger$};
\node [right of = A,xshift=1.5cm] (B) {$\mathpzc{E}^\dagger$};
\node [below of = A] (C) {$\Sigma \mathpzc{M}^\dagger$};
\node [below of = B] (D) {$\mathpzc{M}^\dagger$};

\draw [->] (A) -- (B) node[midway,anchor=south]{$\Psi$};
\draw [->] (A) -- (C) node[midway,anchor=east]{$\Sigma \text{TR}^\dagger$};
\draw [->] (B) -- (D) node[midway,anchor=west]{$\text{TR}^\dagger$};
\draw [->] (C) -- (D) node[midway,anchor=north]{$\Psi$};
\end{tikzpicture}
\end{center}

Combining the above commutative squares with the ones earlier, in the previous section, which compared the stabilization maps for the McClure-Smith and Eilenberg-Zilber operads, we have the following commutative diagrams:
\begin{center}
\begin{tikzpicture}[node distance = 1.5cm]
\node [] (A) {$\Sigma \mathpzc{M}$};
\node [right of = A,xshift=1.5cm] (B) {$\mathpzc{M}$};
\node [below of = A] (C) {$\Sigma \mathpzc{Z}_{*}$};
\node [below of = B] (D) {$\mathpzc{Z}_{*}$};
\node [left of = A,xshift=-1.5cm] (E) {$\Sigma^2 \mathpzc{M}$};
\node [left of = E,xshift=-1.5cm] (F) {$\cdots$};
\node [left of = C,xshift=-1.5cm] (G) {$\Sigma^2 \mathpzc{Z}_{*}$};
\node [left of = G,xshift=-1.5cm] (H) {$\cdots$};
\node [above of = B] (I) {$\mathpzc{E}$};
\node [above of = A] (J) {$\Sigma \mathpzc{E}$};
\node [above of = E] (K) {$\Sigma^2 \mathpzc{E}$};
\node [above of = F] (L) {$\cdots$};

\draw [->] (A) -- (B) node[midway,anchor=south]{$\Psi$};
\draw [->] (A) -- (C) node[midway,anchor=west]{$\Sigma \text{AW}$};
\draw [->] (B) -- (D) node[midway,anchor=west]{$\text{AW}$};
\draw [->] (C) -- (D) node[midway,anchor=south]{$\Psi$};
\draw [->] (F) -- (E) node[midway,anchor=south]{$\Sigma^2\Psi$};
\draw [->] (E) -- (A) node[midway,anchor=south]{$\Sigma\Psi$};
\draw [->] (H) -- (G) node[midway,anchor=south]{$\Sigma^2\Psi$};
\draw [->] (G) -- (C) node[midway,anchor=south]{$\Sigma\Psi$};
\draw [->] (E) -- (G) node[midway,anchor=west]{$\Sigma^2 \text{AW}$};
\draw [->] (L) -- (K) node[midway,anchor=south]{$\Sigma^2\Psi$};
\draw [->] (K) -- (J) node[midway,anchor=south]{$\Sigma\Psi$};
\draw [->] (J) -- (I) node[midway,anchor=south]{$\Psi$};
\draw [->] (K) -- (E) node[midway,anchor=west]{$\Sigma^2 \text{TR}$};
\draw [->] (J) -- (A) node[midway,anchor=west]{$\Sigma \text{TR}$};
\draw [->] (I) -- (B) node[midway,anchor=west]{$\text{TR}$};
\end{tikzpicture}
\end{center}

\begin{center}
\begin{tikzpicture}[node distance = 1.5cm]
\node [] (A) {$\Sigma \mathpzc{M}^\dagger$};
\node [right of = A,xshift=1.5cm] (B) {$\mathpzc{M}^\dagger$};
\node [below of = A] (C) {$\Sigma \mathpzc{Z}_{*}^\dagger$};
\node [below of = B] (D) {$\mathpzc{Z}_{*}^\dagger$};
\node [left of = A,xshift=-1.5cm] (E) {$\Sigma^2 \mathpzc{M}^\dagger$};
\node [left of = E,xshift=-1.5cm] (F) {$\cdots$};
\node [left of = C,xshift=-1.5cm] (G) {$\Sigma^2 \mathpzc{Z}_{*}^\dagger$};
\node [left of = G,xshift=-1.5cm] (H) {$\cdots$};
\node [above of = B] (I) {$\mathpzc{E}^\dagger$};
\node [above of = A] (J) {$\Sigma \mathpzc{E}^\dagger$};
\node [above of = E] (K) {$\Sigma^2 \mathpzc{E}^\dagger$};
\node [above of = F] (L) {$\cdots$};

\draw [->] (A) -- (B) node[midway,anchor=south]{$\Psi$};
\draw [->] (A) -- (C) node[midway,anchor=west]{$\Sigma \text{AW}^\dagger$};
\draw [->] (B) -- (D) node[midway,anchor=west]{$\text{AW}^\dagger$};
\draw [->] (C) -- (D) node[midway,anchor=south]{$\Psi$};
\draw [->] (F) -- (E) node[midway,anchor=south]{$\Sigma^2\Psi$};
\draw [->] (E) -- (A) node[midway,anchor=south]{$\Sigma\Psi$};
\draw [->] (H) -- (G) node[midway,anchor=south]{$\Sigma^2\Psi$};
\draw [->] (G) -- (C) node[midway,anchor=south]{$\Sigma\Psi$};
\draw [->] (E) -- (G) node[midway,anchor=west]{$\Sigma^2 \text{AW}^\dagger$};
\draw [->] (L) -- (K) node[midway,anchor=south]{$\Sigma^2\Psi$};
\draw [->] (K) -- (J) node[midway,anchor=south]{$\Sigma\Psi$};
\draw [->] (J) -- (I) node[midway,anchor=south]{$\Psi$};
\draw [->] (K) -- (E) node[midway,anchor=west]{$\Sigma^2 \text{TR}^\dagger$};
\draw [->] (J) -- (A) node[midway,anchor=west]{$\Sigma \text{TR}^\dagger$};
\draw [->] (I) -- (B) node[midway,anchor=west]{$\text{TR}^\dagger$};
\end{tikzpicture}
\end{center}

From this, our canonical maps in (\ref{eq:can_MSst_to_EZst}) now extend to the following sequences of maps
\[
\mathpzc{E}_{\mathrm{st}} \overset{\mathrm{TR}_{\text{st}}}{\xrightarrow{\hspace{12mm}}} \mathpzc{M}_{\mathrm{st}} \overset{\mathrm{AW}_{\text{st}}}{\xrightarrow{\hspace{12mm}}} \mathpzc{Z}_{\mathrm{st}}
\qquad
\mathpzc{E}_{\mathrm{st}}^\dagger \overset{\mathrm{TR}_{\text{st}}^\dagger}{\xrightarrow{\hspace{12mm}}} \mathpzc{M}_{\mathrm{st}}^\dagger \overset{\mathrm{AW}_{\text{st}}^\dagger}{\xrightarrow{\hspace{12mm}}} \mathpzc{Z}_{\mathrm{st}}^\dagger.
\]

\section{The (Co)homologies, in Individual Arities, of the Stable Operads}\label{subsec:cohomologies_of_stable_operads}

We have constructed our stabilizations of $\mathbb{E}_\infty$ operads, and now begin a study of these stable operads. More specifically, we constructed two stabilizations of $\mathbb{E}_\infty$ operads, that of the McClure-Smith operad and that of the Barratt-Eccles operad. Henceforth, we shall fix the stable Barratt-Eccles operad as our model for a stable operad, though all results which we mention also hold with the stable McClure-Smith operad (see, e.g., Proposition~\ref{prop:EstAlg_MstAlg_equivalence}, where we demonstrate that the homotopy theories of algebras over these operads are equivalent). As our first result, we compute the non-equivariant (co)homologies of the terms of the stable Barratt-Eccles operad. To begin, we have a result regarding the stabilization maps for the Barratt-Eccles operad.

\begin{Proposition}\label{prop:stabmapontoML}
For each $n \ge 0$, the towers
\[
\cdots \to (\Sigma^{2}\mathpzc{E})(n) \to (\Sigma\mathpzc{E})(n) \to \mathpzc{E}(n)
\qquad
\cdots \to (\Sigma^{2}\mathpzc{E}^\dagger)(n) \to (\Sigma\mathpzc{E}^\dagger)(n) \to \mathpzc{E}^\dagger(n)
\]
satisfy the Mittag-Leffler condition. In fact, if $n \ge 1$, the maps in the towers are onto.
\end{Proposition}

\begin{proof}
We shall give a proof of the case of the chain operad; the case of the cochain operad follows by reindexing. First, suppose that $n = 0$. In this case, the Mittag-Leffler property holds because $(\Sigma^k\mathpzc{E})(0)$ is simply $\F_p[k]$, and so the stabilization maps are then necessarily zero maps. Now suppose that $n \ge 1$. We will prove surjectivity of the map $(\Sigma\mathpzc{E})(n) \to \mathpzc{E}(n)$; the surjectivity of the remaining maps is entirely analogous. Let $d \ge 0$. Then $\mathpzc{E}(n)_d$ is generated by tuples $(\rho_0,\dots,\rho_d)$ where the $\rho_i$ are permutations in $\Sigma_n$. On the other hand, $(\Sigma\mathpzc{E})(n)_d = \mathpzc{E}(n)[1-n]_d = \mathpzc{E}(n)_{d+n-1}$ is generated by tuples $(\rho_0',\dots,\rho'_{d+n-1})$ where the $\rho'_i$ are once again permutations in $\Sigma_n$. Given a particular tuple $(\rho_0,\dots,\rho_d)$ in $\mathpzc{E}(n)_d$, we can of course find permutations $\rho'_1,\dots,\rho'_{n-1}$ in $\Sigma_n$ such that $(\rho'_1(1),\dots,\rho'_{n-1}(1),\rho_0(1))$ is an even permutation of $(1,\dots,n)$. Then, by definition of the stabilization map, we have that $(\rho_0,\dots,\rho_d)$ is the image of $(\rho'_1,\dots,\rho'_{n-1},\rho_0,\dots,\rho_d)$, and so we have the desired surjectivity.
\end{proof}

The above result allows us to compute the non-equivariant (co)homology of the stable Barratt-Eccles operad. The result is that, non-equivariantly, the (co)homologies are simply zero, except that the unit is present in arity one. Later, we shall contrast this with a result which shows that the equivariant (co)homologies, on the other hand, are highly non-trivial; see Proposition~\ref{prop:stablefreehom} and Remark~\ref{rmk:eqhomnontrivial}.

\begin{Proposition}\label{prop:stabhom}
We have the following
\[
\emph{H}_\bullet\mathpzc{E}_{\emph{st}}(n) \cong \left\{
\begin{array}{ll}
0 & n \neq 1 \\
\F_p[0] & n=1
\end{array}
\right.
\qquad
\emph{H}^\bullet\mathpzc{E}^\dagger_{\emph{st}}(n) \cong \left\{
\begin{array}{ll}
0 & n \neq 1 \\
\F_p[0] & n=1.
\end{array}
\right.
\]
\end{Proposition}

\begin{proof}
We shall give a proof of the case of the chain operad; the case of the cochain operad follows by reindexing. By definition, $\mathpzc{E}_{\text{st}}(n)$ is the limit of the following tower:
\[
\cdots \to (\Sigma^2\mathpzc{E})(n) \to (\Sigma\mathpzc{E})(n) \to \mathpzc{E}(n)
\]
By Proposition~\ref{prop:stabmapontoML}, this tower satisfies the Mittag-Leffler condition, and so, for each $d \in \Z$, we have an induced short exact sequence
\[
0 \to {\lim_k}^1\, \text{H}_{d+1}((\Sigma^k\mathpzc{E})(n)) \to \text{H}_d(\mathpzc{E}_{\text{st}}(n)) \to \lim_k \text{H}_d((\Sigma^k\mathpzc{E})(n)) \to 0.
\]
(For this induced short exact sequence, see, e.g., Theorem 3.5.8 in~\cite{Weibel}.) As $\mathpzc{E}$ is $\mathbb{E}_\infty$, $\text{H}_{d+1}((\Sigma^k\mathpzc{E})(n))$ is simply $\F_p$ if $d + 1 = k - kn$, and zero otherwise. Thus the tower comprising the $\text{H}_{d+1}((\Sigma^k\mathpzc{E})(n))$ clearly satisfies the Mittag-Leffler condition itself, so that the ${\lim}^1$ term vanishes and the induced map
\[
\text{H}_d(\mathpzc{E}_{\text{st}}(n)) \to \lim_k \text{H}_d((\Sigma^k\mathpzc{E})(n))
\]
is in fact an isomorphism, for each $d \in \Z$. The result now follows immediately from the fact that $\mathpzc{E}$ is $\mathbb{E}_\infty$ and that $(\Sigma^k\mathpzc{E})(n) = \mathpzc{E}(n)[k-kn]$.
\end{proof}

\section{Comparison with the Commutativity and Constant Operads}\label{subsec:comm_const_operads}

In this section, we provide a comparison of the relation between the Barratt-Eccles operad and the commutativity operad with the relation between the stable Barrat-Eccles operad and the constant operad. Let us first recall the definitions of the commutativity and constant operads. In this section, we shall work over an arbitrary ground field $k$.

\begin{Definition}
The \textit{commutativity chain operad}\index{commutativity operad} is the operad $\mathpzc{Comm}$ which is such that, for each $n \ge 0$, $\mathpzc{Comm}(n) = k[0]$ (the remainder of the operad structure is clear). The \textit{commutativity cochain operad} is the operad $\mathpzc{Comm}^\dagger$.
\end{Definition}

As is well-known and easily verified, an algebra over the commutativity operad is exactly a commutative dg algebra.

\begin{Definition}
The \textit{constant chain operad}\index{constant operad} is the operad $\mathpzc{Const}$ which is such that $\mathpzc{Const}(1) = k[0]$ and, for each $n \neq  1$, $\mathpzc{Const}(n) = 0$ (the remainder of the operad structure is clear). The \textit{constant cochain operad} is the operad $\mathpzc{Const}^\dagger$.
\end{Definition}

An algebra over the constant operad is simply a dg module. \\

The commutativity and constant operads are in fact related, in the sense of a stabilization, as we now demonstrate. One can easily verify that there is a stabilization map
\[
\Psi \colon \Sigma \mathpzc{Comm} \to \mathpzc{Comm}
\]
which is zero in each arity, except arity one, in which it is the identity on $k[0]$. (As for the other stabilization maps which we have encountered so far, we have denoted this map by $\Psi$, and as we've mentioned in the other cases, the context will always make clear which map we intend.) Using the maps $\Sigma^r\Psi$, for $r \ge 0$, we can construct a stabilization of the commutativity operad.

\begin{Definition}
The \textit{stable commutativity chain operad}\index{stable commutativity operad}, denoted $\mathpzc{Comm}_{\text{st}}$, is the operad defined as follows
\[
\mathpzc{Comm}_{\text{st}} := \underset{\longleftarrow}{\text{lim}}(\cdots \overset{\Sigma^2\Psi}\longrightarrow \Sigma^2\mathpzc{Comm} \overset{\Sigma\Psi}\longrightarrow \Sigma\mathpzc{Comm} \overset{\Psi}\longrightarrow \mathpzc{Comm}).
\]
The \textit{stable commutativity cochain operad} is the operad $\mathpzc{Comm}^\dagger_{\text{st}}$.
\end{Definition}

\begin{Remark}
Just as in the case of the other stabilized cochain operads (see Remarks~\ref{rmk:stable_ez_cochain},~\ref{rmk:stable_ms_cochain} and~\ref{rmk:stable_be_cochain}), note that we also have maps $\Sigma^r\Psi^\dagger \colon \Sigma^{r+1}\mathpzc{Comm}^\dagger \to \Sigma^{r}\mathpzc{Comm}^\dagger$, for $r \ge 0$, and that the stabilized commutativity cochain operad can also be constructed as the inverse limit
\[
\mathpzc{Comm}^\dagger_{\text{st}} = \underset{\longleftarrow}{\text{lim}}(\cdots \overset{\Sigma^2\Psi^\dagger}\longrightarrow \Sigma^2\mathpzc{Comm}^\dagger \overset{\Sigma\Psi^\dagger}\longrightarrow \Sigma\mathpzc{Comm}^\dagger \overset{\Psi^\dagger}\longrightarrow \mathpzc{Comm}^\dagger).
\] \customendremark
\end{Remark}

The next result shows that the stabilization of the commutativity operad is precisely the constant operad.

\begin{Proposition}\label{prop:stabilization_of_comm}
The stabilized operads $\mathpzc{Comm}_{\emph{st}}$ and $\mathpzc{Comm}_{\emph{st}}^\dagger$ are isomorphic to $\mathpzc{Const}$ and $\mathpzc{Const}^\dagger$, respectively.
\end{Proposition}

\begin{proof}
We shall demonstrate the case of the chain operad; the case of the cochain operad is analogous. For $n \ge 0$ and $r \ge 0$, we have that $(\Sigma^r\mathpzc{Comm})(n) = \mathpzc{Comm}(n)[r-rn] = k[r-rn]$. It follows that, if $n \neq 1$, the maps $\Sigma^r\Psi$ are all necessarily zero, and so $\mathpzc{Comm}_{\text{st}}(n) = 0$. On the other hand, in arity one, we have $(\Sigma^r\mathpzc{Comm})(1) = k[0]$ for all $r \ge 0$, and each map $\Sigma^r\Psi$ is the identity on $k[0]$, so that $\mathpzc{Comm}_{\text{st}}(1) = k[0]$, as desired. 
\end{proof}

Now, as is standard, there is a map 
\[
\mathpzc{E} \overset{\sim}{\relbar\joinrel\twoheadrightarrow} \mathpzc{Comm}
\]
given by mapping chains to sums of coefficients. This map is clearly surjective in each arity, and moreover, is a quasi-isomorphism in each arity (this follows from properties of the spaces $\mathrm{E}\Sigma_n$; see Proposition 2.3 in~\cite{BarrattEccles}). Thus $\mathpzc{E}$ is a $\Sigma_*$-projective resolution of $\mathpzc{Comm}$. We can also phrase this in model theoretic language as follows. The category whose objects are chain operads $\mathpzc{P}$ equipped with augmentation maps $\alpha \colon \mathpzc{P} \to \mathpzc{Comm}$ such that $\alpha(0)$ is an isomorphism, and whose morphisms are operad maps compatible with the augmentations, has a model structure with the componentwise quasi-isomorphisms as weak equivalences and componentwise epimorphisms as fibrations (see Theorem 3.1 in~\cite{BergerMoerdijk} and Remark 3.4 in~\cite{Hinich1err}). In this category, via the map above, $\mathpzc{E}$ is a cofibrant replacement of $\mathpzc{Comm}$. (Analogous remarks apply to the cochain operads $\mathpzc{E}^\dagger$ and $\mathpzc{Comm}^\dagger$.) \\

As we have seen in Proposition~\ref{prop:stabilization_of_comm} above, the constant operad is the stabilization of the commutativity operad. As such, we might guess that the stable Barratt-Eccles operad is a kind of resolution of the constant operad\index{stable Barratt-Eccles operad!relation to the constant operad}. \index{stable Barratt-Eccles operad!in arities zero and one}An easy check shows that $\mathpzc{E}_{\text{st}}(0) = 0$ and $\mathpzc{E}_{\text{st}}(1) = k[0]$. We then certainly do have a map
\[
\mathpzc{E}_{\text{st}} \overset{\sim}{\relbar\joinrel\twoheadrightarrow} \mathpzc{Const}
\]
which is the identity on $k[0]$ in arity one, and zero in all other arities. This map is clearly surjective in each arity, and, by Proposition~\ref{prop:stabhom}, is a quasi-isomorphism in each arity. Now, the category whose objects are chain operads $\mathpzc{P}$ where $\mathpzc{P}(0) = 0$, and whose morphisms are operad maps, has a model structure with the componentwise quasi-isomorphisms as weak equivalences and componentwise epimorphisms as fibrations (see Theorem 3.1 in~\cite{BergerMoerdijk} and Theorem 3.3 in~\cite{Hinich1err}). In this category then, the above map from $\mathpzc{E}_{\text{st}}$ to $\mathpzc{Const}$ is a trivial fibration. We do not know whether or not $\mathpzc{E}_{\text{st}}$ is cofibrant in this category (for this to be true, one might need to first restrict to a sub model category of additive operads; that is, a sub model category of operads whose associated monads are additive). Nonetheless, this discussion serves to illustrate that we ought to think of $\mathpzc{E}_{\text{st}}$ as a kind of resolution of $\mathpzc{Const}$. (Analogous remarks apply to the cochain operads $\mathpzc{E}_{\text{st}}^\dagger$ and $\mathpzc{Const}^\dagger$.)


\chapter{Homotopy Theory of Algebras Over the Stable Operads}

We have constructed the stable Barratt-Eccles operad. In this section, we shall develop a homotopy theory of algebras over this operad, in the sense of Quillen model structures.

\section{Cell Algebras}\label{sec:cell_alg}

In order to develop our homotopy theory, we first need to recall the standard notion of cell algebras over operads. This notion is as follows.

\begin{Definition}\label{def:cell_alg}
Let $\mathpzc{P}$ be a dg operad over $k$. A \textit{cell $\mathpzc{P}$-algebra}\index{cell!algebra} is a $\mathpzc{P}$-algebra $A$ for which there exists a cotower of $\mathpzc{P}$-algebras
\[
A_0 \to A_1 \to A_2 \to \cdots
\]
and a universal cocone from this cotower to $A$, such that:
\begin{itemize}
	\item $A_0$ is the initial $\mathpzc{P}$-algebra, namely $\mathpzc{P}(0)$.
	\item For each $n \ge 0$, the map $A_n \to A_{n+1}$ fits into an algebra pushout square
\begin{center}
\begin{tikzpicture}[node distance=1.5cm]
\node(A){$\mathbf{P}M$};
\node[below of = A](C){$\mathbf{P}\text{C}M$};
\node[right of = A](B){$A_n$};
\node[below of = B,yshift=0.5mm](D){$A_{n+1}$};
	
\draw[->] (A) -- (B) node[midway,anchor=south]{};
\draw[->] (A) -- (C) node[midway,anchor=east]{};
\draw[->] (C) -- (D) node[midway,anchor=north]{};
\draw[->] (B) -- (D) node[midway,anchor=west]{};
		
\begin{scope}[shift=($(A)!.2!(D)$)]
\draw +(0,-0.25) -- +(0,0)  -- +(0.25,0);
\end{scope}
\end{tikzpicture}
\end{center}
where $M$ is a dg module which is degreewise free and has zero differentials, and where $\text{C}M$ is the cone on $M$.
\end{itemize}
\end{Definition}

\begin{Remark}
The condition that the dg module $M$ be degreewise free and have zero differentials is equivalent to that it be a sum of copies of the standard sphere complexes. Moreover, the cone on such a complex, denoted by $\text{C}M$ above, is then a sum of the standard disk complexes. (See Section~\ref{sec:nots_convs} for the standard sphere and disk complexes.) \customendremark
\end{Remark}

More generally, we also have cell maps as follows.

\begin{Definition}\label{def:cell_map}
Let $\mathpzc{P}$ be a dg operad over $k$. A \textit{cell map}\index{cell!map} $A \to B$ of $\mathpzc{P}$-algebras is a map such that there exists a cotower of $\mathpzc{P}$-algebras
\[
A_0 \to A_1 \to A_2 \to \cdots
\]
and a universal cocone from this cotower to $B$, such that:
\begin{itemize}
	\item $A_0 = A$ and the map $A_0 \to B$ is the given map $A \to B$.
	\item For each $n \ge 0$, the map $A_n \to A_{n+1}$ fits into an algebra pushout square
\begin{center}
\begin{tikzpicture}[node distance=1.5cm]
\node(A){$\mathbf{P}M$};
\node[below of = A](C){$\mathbf{P}\text{C}M$};
\node[right of = A](B){$A_n$};
\node[below of = B,yshift=0.5mm](D){$A_{n+1}$};
	
\draw[->] (A) -- (B) node[midway,anchor=south]{};
\draw[->] (A) -- (C) node[midway,anchor=east]{};
\draw[->] (C) -- (D) node[midway,anchor=north]{};
\draw[->] (B) -- (D) node[midway,anchor=west]{};
		
\begin{scope}[shift=($(A)!.2!(D)$)]
\draw +(0,-0.25) -- +(0,0)  -- +(0.25,0);
\end{scope}
\end{tikzpicture}
\end{center}
where $M$ is a dg module which is degreewise free and has zero differentials.
\end{itemize}
\end{Definition}

\begin{Remark}
Looking at the definitions, we see that a $\mathpzc{P}$-algebra is a cell algebra if and only if the unique map $\mathpzc{P}(0) \to A$ is a cell map. \customendremark
\end{Remark}

We now describe well-known concrete models of cell algebras. Let $\mathpzc{P}$ be a dg operad over $k$ and let $A$ be a cell $\mathpzc{P}$-algebra. Let also 
\[
A_0 \to A_1 \to A_2 \to \cdots
\]
be a cell filtration of $A$ and fix some choices $M_1, M_2, \dots$ for the dg modules which appear in the attachment squares above. For each $n \ge 0$, let $N_n = \oplus_{i \le n} M_i$, where $N_0 = 0$, and let also $N = \oplus_{i \ge 1} M_i$. Then one can construct models for the $A_n$ and $A$, denoted say by $B_n$ and $B$, as follows. Let $\mathpzc{P}^\#$ denote the operad in graded modules formed by forgetting the differentials present in $\mathpzc{P}$. For $n \ge 0$, we have that, as a graded module
\[
B_n = \mathbf{P}^\#(N_n[1]) = \bigoplus_{k \ge 0} \mathpzc{P}(k) \otimes_{\Sigma_k} (N_n[1])^{\otimes k}
\]
and the differentials of the $B_n$ are induced inductively, via the Leibniz rule and the attachment maps $\mathbf{P}M_n \to A_{n+1}$, together with the operadic composition maps of $\mathpzc{P}$. In the limit, we have that, as a graded module
\[
B = \mathbf{P}^\#(N[1]) = \bigoplus_{k \ge 0} \mathpzc{P}(k) \otimes_{\Sigma_k} (N[1])^{\otimes k}
\]
and in this case the differential is of course induced by those of the $B_n$. The precise analogue of the statement that the $B_n$ and $B$ are models, respectively, for the $A_n$ and $A$ is that there exists a diagram of isomorphisms of $\mathpzc{P}$-algebras as follows
\begin{center}
\begin{tikzpicture}
\node [] (A) {$A_0$};
\node [right of = A, xshift=1cm] (B) {$A_1$};
\node [right of = B, xshift=1cm] (C) {$A_2$};
\node [right of = C, xshift=1cm] (D) {$\cdots$};
\node [right of = D, xshift=1cm] (E) {};

\node [below of = A] (AA) {$B_0$};
\node [right of = AA, xshift=1cm] (BB) {$B_1$};
\node [right of = BB, xshift=1cm] (CC) {$B_2$};
\node [right of = CC, xshift=1cm] (DD) {$\cdots$};
\node [right of = DD, xshift=1cm] (EE) {};

\draw [->] (A) -- (B);
\draw [->] (B) -- (C);
\draw [->] (C) -- (D);

\draw [->] (AA) -- (BB);
\draw [->] (BB) -- (CC);
\draw [->] (CC) -- (DD);

\draw [->] (A) -- (AA) node[midway,anchor=west]{$\cong$};
\draw [->] (B) -- (BB) node[midway,anchor=west]{$\cong$};
\draw [->] (C) -- (CC) node[midway,anchor=west]{$\cong$};
\end{tikzpicture}
\end{center}
and this diagram induces, in the limit, an isomorphism $A \to B$. See, for example,~\cite{Mandell} and~\cite{FresseBook}.

\section{Enveloping Operads}\label{sec:env_op}

In the process of developing our homotopy theory, we shall have a need to consider algebra coproducts of the form $A \amalg \mathbf{P}X$ (where $\mathpzc{P}$ is an operad, $\mathbf{P}$ the associated free algebra functor and $A$ a $\mathpzc{P}$-algebra). Such coproducts can be formed with the help of a general construction on $A$, that of the enveloping operad of $A$ (this construction, as the name suggests, also has relations to other notions, such as that of representations of $A$).

\begin{Definition}\label{def:env_op}
Let $\mathpzc{P}$ be a dg operad over $k$ and $A$ a $\mathpzc{P}$-algebra. The \textit{enveloping operad}\index{enveloping operad} of $A$, denoted $\mathpzc{U}^A$, is defined as follows. For each $j \ge 0$, the dg $k[\Sigma_j]$-module $\mathpzc{U}^A(j)$ is defined to be the dg module coequalizer
\[
\bigoplus_{i \ge 0} \mathpzc{P}(i+j) \otimes_{\Sigma_i} (\mathbf{P}A)^{\otimes i} \rightrightarrows \bigoplus_{i \ge 0} \mathpzc{P}(i+j) \otimes_{\Sigma_i} A^{\otimes i} \to \mathpzc{U}^A(j)
\]
where one of the two parallel maps is induced by the $\mathpzc{P}$-algebra structure map $\mathbf{P}A \to A$ of $A$ and the other by the composition product of $\mathpzc{P}$. Morever, the operadic structure maps of $\mathpzc{U}^A$ are induced by those of $\mathpzc{P}$.
\end{Definition}

We now summarise some of the facts about enveloping operads which we shall need:
\begin{itemize}
	\item[(1)] In operadic degree $0$, we have
\[
\mathpzc{U}^A(0) \cong A
\]
where the universal coequalizer map is given by the $\mathpzc{P}$-algebra structure map of $A$, a map $\mathbf{P}A \to A$.
	\item[(2)] In operadic degree one, as usual, $\mathpzc{U}^A(1)$ forms a unital associative algebra via the composition product $\mathpzc{U}^A(1) \otimes \mathpzc{U}^A(1) \to \mathpzc{U}^A(1)$. By definition, this algebra $\mathpzc{U}^A(1)$ is the \textit{enveloping algebra} of $A$. See~\cite{GinzburgKapranov} and~\cite{FresseBook}.
	\item[(3)] In the construction of the $\mathpzc{U}^A(j)$, note that the two parallel maps preserve the $i=0$ summands, and moreover that, the coequalizer of just these two summands is simply $\mathpzc{P}(j)$. It follows that, for each $A$ and $j \ge 0$, $\mathpzc{U}^{A}(j)$ is equipped with a canonical map $\mathpzc{P}(j) \to\mathpzc{U}^{A}(j)$; in fact, these assemble into a canonical operad map
\[
\mathpzc{P} \to \mathpzc{U}^{A}.
\]
	\item[(4)] An easy check of the definitions and universal properties demonstrates that the enveloping operad $\mathpzc{U}^{\mathpzc{P}(0)}$ of the initial $\mathpzc{P}$-algebra $\mathpzc{P}(0)$ is simply $\mathpzc{P}$; that is, we have
\[
\mathpzc{U}^{\mathpzc{P}(0)} \cong \mathpzc{P}.
\]
	\item[(5)] In the case of a free algebra $\mathbf{P}X$, in forming the enveloping operad, we can simply generate on $X$ rather than $\mathbf{P}X$ and dispense with the relations imposed by the parallel maps in the coequalizer, so that
\[
\mathpzc{U}^{\mathbf{P}X}(j) \cong \bigoplus_{i \ge 0} \mathpzc{P}(i+j) \otimes_{\Sigma_i} X^{\otimes i}.
\]
	\item[(6)] We have an equivalence of categories $\mathpzc{U}^A\text{-}\mathsf{Alg} \simeq \mathpzc{P}\text{-}\mathsf{Alg}_{A/}$ between the category of $\mathpzc{U}^A$-algebras and the category of $\mathpzc{P}$-algebras under $A$. See \cite{GetzlerJones}, \cite{Mandell} or \cite{FresseBook} for details. For example, if $B$ is a $\mathpzc{U}^A$-algebra, it is also endowed with a $\mathpzc{P}$-algebra structure by pulling back across the canonical map $\mathpzc{P} \to \mathpzc{U}^{A}$ mentioned above, and moreover, $B$ carries a canonical map from the initial $\mathpzc{U}^A$-algebra, which we have seen above is $\mathpzc{U}^A(0) \cong A$. Note also that this equivalence of categories follows easily in the case of the enveloping operad $\mathpzc{U}^{\mathpzc{P}(0)}$ from the observation above that $\mathpzc{U}^{\mathpzc{P}(0)}$ is simply $\mathpzc{P}$.
	\item[(7)] Given any dg module $X$, we have a natural, in $X$, isomorphism of $\mathpzc{P}$-algebras under $A$ as follows
\[
\mathbf{U}^AX \cong A \amalg \mathbf{P}X.
\]
See~\cite{Mandell} for this result.
\end{itemize}

In addition to the above, we shall also need certain facts about the enveloping operad of a cell algebra. There are well-known concrete models for the enveloping operad of a cell algebra, and a resulting filtration of such an enveloping operad -- see, for example,~\cite{Mandell}. We record these models and filtrations here, for future reference. Let $\mathpzc{P}$ be a dg operad over $k$ and let $A$ be a cell $\mathpzc{P}$-algebra. Let also $A_0 \to A_1 \to A_2 \to \cdots$ be a cell filtration of $A$ and fix some choices $M_1, M_2, \dots$ for the dg modules which appear in the attachment squares. For each $n \ge 0$, let $N_n = \oplus_{i \le n} M_i$, where $N_0 = 0$, and let also $N = \oplus_{i \ge 1} M_i$. Then the models for the $\mathpzc{U}^{A_n}(j)$ and $\mathpzc{U}^A(j)$, for $n, j \ge 0$, are as follows. For each $n, j \ge 0$, we have that, as a $k[\Sigma_j]$-module
\[
\mathpzc{U}^{A_n}(j) = \bigoplus_{i \ge 0} \mathpzc{P}(i + j) \otimes_{\Sigma_i} (N_n[1])^{\otimes i}.
\]
Similarly, for each $j \ge 0$, we have that, as a $k[\Sigma_j]$-module
\[
\mathpzc{U}^A(j) = \bigoplus_{i \ge 0} \mathpzc{P}(i + j) \otimes_{\Sigma_i} (N[1])^{\otimes i}.
\]
The differential on the $\mathpzc{U}^{A_n}(j)$ and on $\mathpzc{U}^A(j)$ are given by the Leibniz rule, the attachment maps and the operadic composition. Thus we see that the operad $\mathpzc{U}^A$ is filtered by the operads $\mathpzc{U}^{A_n}$. Now, as $A_0$ is the initial $\mathpzc{P}$-algebra $\mathpzc{P}(0)$, we have $\mathpzc{U}^{A_0} = \mathpzc{U}^{\mathpzc{P}(0)} = \mathpzc{P}$. The terms of the operad $\mathpzc{U}^{A_1}$ then arise from the terms of $\mathpzc{U}^{A_0}$ by attachment of cells; more generally, for $n \ge 1$, the terms of the operad $\mathpzc{U}^{A_n}$ arise from the terms of $\mathpzc{U}^{A_{n-1}}$ by attachment of cells. This allows us to define, for $n \ge 1$, a filtration on the terms of the operad $\mathpzc{U}^{A_n}$ as follows. Fix such an $n$. For any $j \ge 0$, we let $\text{F}_m\mathpzc{U}^{A_n}(j)$, where $m \ge 0$, denote the sub graded module of $\mathpzc{U}^{A_n}(j)$ generated by the elements $\sigma \otimes a_1 \otimes \cdots \otimes a_i$ where at most $m$ of the factors $a_1,\dots,a_i \in N_n[1]$ project to a non-zero element in $M_n[1]$ (which comprises the ``most recently added cells''); note that, since, when computing the differential of $\sigma \otimes a_1 \otimes \cdots \otimes a_i$ via the Leibniz rule, if $a_r \in M_n[1]$, we map it to the corresponding element of $\bigoplus_{i \ge 0} \mathpzc{P}(i) \otimes_{\Sigma_i} (N_{n-1}[1])^{\otimes i}$ via the attachment map $M_n \to A_{n-1}$, we have that the differential preserves the sub graded module $\text{F}_m\mathpzc{U}^{A_n}(j)$ so that we in fact have a sub dg module. Now, given $n \ge 1$ and $j \ge 0$, as graded right $k[\Sigma_j]$-modules, note that we have that
\[
\resizebox{.95\hsize}{!}{$\mathpzc{U}^{A_n}(j) = \bigoplus_{i \ge 0} \mathpzc{P}(i + j) \otimes_{\Sigma_i} (N_n[1])^{\otimes i} = \bigoplus_{i \ge 0} \bigoplus_{l=0}^{i} \mathpzc{P}(i+j) \otimes_{\Sigma_{i-l} \times \Sigma_l} N_{n-1}[1]^{\otimes (i-l)} \otimes M_n[1]^{\otimes l}$}
\]
and, for any $m \ge 0$, the submodule $\text{F}_m\mathpzc{U}^{A_n}(j)$ is then given by
\[
\text{F}_m\mathpzc{U}^{A_n}(j) = \bigoplus_{i \ge 0} \bigoplus_{l=0}^{\text{min}(i,m)} \mathpzc{P}(i+j) \otimes_{\Sigma_{i-l} \times \Sigma_l} N_{n-1}[1]^{\otimes (i-l)} \otimes M_n[1]^{\otimes l}.
\]
Thus we see that, for $m \ge 1$, the inclusions $\text{F}_{m-1}\mathpzc{U}^{A_n}(j) \to \text{F}_m\mathpzc{U}^{A_n}(j)$ are, at the level of the underlying graded modules, split monomorphisms. We also have that
\begin{align*}
\text{F}_m\mathpzc{U}^{A_n}(j)/\text{F}_{m-1}\mathpzc{U}^{A_n}(j) &\cong \bigoplus_{i \ge m} \mathpzc{P}(i+j) \otimes_{\Sigma_{i-m} \times \Sigma_m} N_{n-1}[1]^{\otimes (i-m)} \otimes M_n[1]^{\otimes m} \\
&\cong \left(\bigoplus_{i \ge m} \mathpzc{P}(i+j) \otimes_{\Sigma_{i-m}} N_{n-1}[1]^{\otimes (i-m)}\right) \otimes_{\Sigma_m} M_n[1]^{\otimes m} \\
&= \left(\bigoplus_{i \ge 0} \mathpzc{P}(i+m+j) \otimes_{\Sigma_i} N_{n-1}[1]^{\otimes i}\right) \otimes_{\Sigma_m} M_n[1]^{\otimes m} \\
&= \mathpzc{U}^{A_{n-1}}(m+j) \otimes_{\Sigma_m} M_n[1]^{\otimes m}.
\end{align*}
Moreover, recalling that when we compute the differential of $\sigma \otimes a_1 \otimes \cdots \otimes a_i$ via the Leibniz rule, if $a_r \in M_n[1]$, we map it to the corresponding element of $\bigoplus_{i \ge 0} \mathpzc{P}(i) \otimes_{\Sigma_i} N_{n-1}[1]^{\otimes i}$ via the attachment map $M_n \to A_{n-1}$, and so in particular we map to zero in the quotient $\text{F}_m\mathpzc{U}^{A_n}(j)/\text{F}_{m-1}\mathpzc{U}^{A_n}(j)$, we see that the isomorphism
\[
\text{F}_m\mathpzc{U}^{A_n}(j)/\text{F}_{m-1}\mathpzc{U}^{A_n}(j) \cong \mathpzc{U}^{A_{n-1}}(m+j) \otimes_{\Sigma_m} M_n[1]^{\otimes m}
\]
is in fact one of dg modules, not only of graded modules.

\section{Admissibility and Semi-admissibility of Operads}

In this section, we consider some generalities on the homotopy theories, in the sense of Quillen model structures, of operads and their algebras. Let $\mathpzc{P}$ be a dg operad over $k$. The most common fashion in which one tries to place a model structure on $\mathpzc{P}\text{-}\mathsf{Alg}$ is by pull back, of the projective model structures, across the forgetful functor $\mathpzc{P}\text{-}\mathsf{Alg} \to \mathsf{DG}_k$. This motivates the following standard definition.

\begin{Definition}\label{def:admissibility}
Let $\mathpzc{P}$ be a dg operad over $k$. We say that $\mathpzc{P}$ is \textit{admissible}\index{admissible operad} if $\mathpzc{P}\text{-}\mathsf{Alg}$ admits a model structure where the weak equivalences and fibrations are the quasi-isomorphisms and degreewise epimorphisms, respectively.
\end{Definition}

For example, it is known that cofibrant $\mathbb{E}_\infty$ operads are admissible -- see~\cite{BergerMoerdijk}. \\

In the case of our stable operads, we don't quite have admissibility, but rather a weakened form of it. To describe it, we first need to describe a weakening of model structures, to semi-model structures, which have also been considered in~\cite{HoveyPaper},~\cite{White},~\cite{Spitzweck} and~\cite{FresseBook}. In short, a semi-model category is exactly a model category except that the factorization $\rightarrow \, = \, \overset{\sim}\hookrightarrow\twoheadrightarrow$ and the lifting property $(\overset{\sim}\hookrightarrow) \boxslash (\twoheadrightarrow)$ (which is to say, the factorization and lifting properties that involve trivial cofibrations) are required to hold only in the case where the source is cofibrant. 

\begin{Definition}\label{def:category_semi_model_str}
A \textit{Quillen semi-model structure}\index{Quillen semi-model structure} is a category $\mathsf{E}$, together with three specified classes of morphisms, $\mathcal{W}$, $\mathcal{C}$ and $\mathcal{F}$, such that the following hold:
\begin{itemize}[leftmargin=12mm]
	\item[SM1:] $\mathsf{E}$ is bicomplete.
	\item[SM2:] The class $\mathcal{W}$ satisfies 2-out-of-3.
	\item[SM3:] The classes $\mathcal{W}, \mathcal{C}$ and $\mathcal{F}$ are closed under retracts.
	\item[SM4:] Given a cofibration $i \colon A \hookrightarrow B$ and a fibraiton $p \colon X \twoheadrightarrow Y$ we have that $i \boxslash p$ if either $p$ is weak equivalence, or if $i$ is a weak equivalence and $A$ is a cofibrant.
	\item[SM5:] Given any map $A \to B$, it can be factored as a cofibration followed by a trivial fibration, and, if $A$ is cofibrant, it can also be factored as a trivial cofibration followed by a fibration.
\end{itemize}
We also add in the following requirement, which does not automatically follow from the above:
\begin{itemize}[leftmargin=12mm]
	\item[SM6:] Fibrations are closed under composition, products and base change.
\end{itemize}
\end{Definition}

With this definition, one can run through the the standard arguments for model categories to verify that we can still perform analogous constructions of the derived category and derived functors, with appropriate modifications. In particular, one can construct the derived category via bifibrant replacements of cofibrant objects -- see Theorem 2.13 in~\cite{Mandell}, and moreover, as for derived functors, the relevant result which we will need later is the following, for which we refer to Theorems 2.14 and 2.15 in~\cite{Mandell}.

\begin{Proposition}\label{prop:semimodelqadj}
Let $L \colon \mathsf{E} \to \mathsf{M}$ and $R \colon \mathsf{M} \to \mathsf{E}$ be left and right adjoints between a semi-model category $\mathsf{E}$ and a model category $\mathsf{M}$. Then we have the following:
\begin{itemize}
	\item[(i)] If $L$ preserves cofibrations between cofibrant objects and $R$ preserves fibrations, then the left derived functor of $L$ and the right derived functor of $R$ exist and are adjoint. Moreover, $L$ converts weak equivalences between cofibrant objects to weak equivalences, and the restriction of the left derived functor of $L$ to the cofibrant objects is naturally isomorphic to the derived functor of the restriction of $L$.
	\item[(ii)] Suppose that (i) holds and in addition for any cofibrant object $A$ in $\mathsf{E}$ and any fibrant object $Y$ in $M$, a map $A \to RY$ is a weak equivalence if and only if the adjoint $LA \to Y$ is a weak equivalence. Then the left derived functor of $L$ and the right derived functor of $R$ are inverse equivalences.
\end{itemize}
Moreover, we also have the following:
\begin{itemize}
	\item[(iii)] The hypothesis in (i) above is equivalent to each of the following:
\begin{itemize}
	\item $L$ preserves cofibrations between cofibrant objects and acyclic cofibrations between cofibrant objects.
	\item $R$ preserves fibrations and acyclic fibrations.
\end{itemize}
\end{itemize}
\qed
\end{Proposition}

We can now define our weakening of admissibility for dg operads, which we call semi-admissibility.

\begin{Definition}\label{def:semiadmissibility}
Given a dg operad $\mathpzc{P}$ over $k$, we say that it is \textit{semi-admissible}\index{semi-admissible operad} if $\mathpzc{P}\text{-}\mathsf{Alg}$ admits a semi-model structure where the weak equivalences and fibrations are the quasi-isomorphisms and degreewise epimorphisms, respectively.
\end{Definition}

Finally, we mention criteria for the admissibility and semi-admissibility of a dg operad. See~\cite{Hinich1} for the admissibility criterion and~\cite{Mandell} for the semi-admissibility criterion (Mandell does not use the term ``semi-admissibility''; however, the results developed in Section 2 of~\cite{Mandell} are what we have codified in our definition of semi-model structures). Below, by ``disk complex'', we mean that which is referred to as such in Section~\ref{sec:nots_convs}.

\begin{Proposition}\label{prop:amenable_implies_admissible}
Let $\mathpzc{P}$ be a dg operad over $k$, $\mathbf{P}$ the associated free algebra functor. Then we have the following:
\begin{itemize}
	\item[(i)] If, for any $\mathpzc{P}$-algebra $A$, the natural map
\[
A \to A \amalg \mathbf{P}(\D^n)
\]
where $\D^n$ is a disk complex, is a quasi-isomorphism, $\mathpzc{P}$ is admissible.
	\item[(ii)] If the above condition holds for any cell $\mathpzc{P}$-algebra, $\mathpzc{P}$ is semi-admissible.
\end{itemize}
Moreover, in either case, the cofibrations are exactly the retracts of cell maps.
\qed
\end{Proposition}

\section{The Homotopy Theory of Algebras over the Stable Operads}

We can now develop our homotopy theory for algebras over the stable Barratt-Eccles chain and cochain operads. First, we wish to show that we have some homotopical control over these operads by showing that the associated monads preserve quasi-isomorphisms. To show this, we have a couple preliminary definitions and a lemma. 

\begin{Definition}\label{def:fin_mod}
Given a ring $R$, say that a dg left $R$-module is \textit{finite}\index{finite module} if it is bounded above and below and degreewise finitely generated.
\end{Definition}

\begin{Definition}\label{def:fin_flat}
Given a ring $R$ and a dg right $R$-module $C$, say that $C$ is \textit{semi-flat}\index{semi-flat module} if, as a functor on dg left $R$-modules, $C \otimes_R -$ preserves quasi-isomorphisms between finite modules.
\end{Definition}

In the case of the $\mathbb{E}_\infty$ operads $\mathpzc{E}$ and $\mathpzc{E}^\dagger$, we have that, for each $n$, $\mathpzc{E}_{\text{st}}(n)$ and $\mathpzc{E}^\dagger_{\text{st}}(n)$ are flat over $\F_p[\Sigma_n]$, as they are free over $\F_p[\Sigma_n]$. In the case of our stable operads, we have the following.

\begin{Lemma}\label{lem:finflatEst}
For each $n \ge 0$, $\mathpzc{E}_{\emph{st}}(n)$ and $\mathpzc{E}^\dagger_{\emph{st}}(n)$ are semi-flat over $\F_p[\Sigma_n]$.
\end{Lemma}

\begin{proof}
We shall demonstrate the case of the chain operad; the case of the cochain operad is entirely analogous. Fix $n \ge 0$. Let $Z$ be a finite chain complex over $\F_p[\Sigma_n]$. For each $k$, $(\Sigma^k\mathpzc{E})(n)$ is degreewise of finite dimension over $\F_p[\Sigma_n]$. Moreover, $Z$ is degreewise finitely presented over $\F_p[\Sigma_n]$ (we have finite generation by assumption and then finite presentation follows because $\F_p[\Sigma_n]$ is Noetherian). As a result, we can commute tensor product and inverse limit to conclude that
\[
\mathpzc{E}_{\text{st}}(n) \otimes_{\Sigma_n} Z = (\lim_k \: (\Sigma^k\mathpzc{E})(n)) \otimes_{\Sigma_n} Z = \lim_k \: ((\Sigma^k\mathpzc{E})(n) \otimes_{\Sigma_n} Z).
\]
Next, given a map $f \colon Z \to Z'$ between finite $\F_p[\Sigma_n]$-complexes $Z$ and $Z'$, we can write the induced map $\mathpzc{E}_{\text{st}}(n) \otimes_{\Sigma_n} Z \to \mathpzc{E}_{\text{st}}(n) \otimes_{\Sigma_n} Z'$ as the map induced on inverse limits by the maps $(\Sigma^k\mathpzc{E})(n) \otimes_{\Sigma_n} Z \to (\Sigma^k\mathpzc{E})(n) \otimes_{\Sigma_n} Z'$. If $f$ if a quasi-isomorphism, each of the latter maps $(\Sigma^k\mathpzc{E})(n) \otimes_{\Sigma_n} Z \to \Sigma^k\mathpzc{E}(n) \otimes_{\Sigma_n} Z'$ is also a quasi-isomorphism. Thus we have the following diagram of quasi-isomorphisms

\begin{center}
\begin{tikzpicture}[node distance = 2.5cm]
\node [] (A) {$(\Sigma \mathpzc{E})(n) \otimes_{\Sigma_n} Z$};
\node [right of = A,xshift=1.25cm] (B) {$\mathpzc{E}(n) \otimes_{\Sigma_n} Z$};
\node [below of = A,yshift=6mm] (C) {$(\Sigma \mathpzc{E})(n) \otimes_{\Sigma_n} Z'$};
\node [below of = B,yshift=6mm] (D) {$\mathpzc{E}(n) \otimes_{\Sigma_n} Z'$};
\node [left of = A,xshift=-1.25cm] (E) {$(\Sigma^2 \mathpzc{E})(n) \otimes_{\Sigma_n} Z$};
\node [left of = E,xshift=-1cm] (F) {$\cdots$};
\node [left of = C,xshift=-1.25cm] (G) {$(\Sigma^2 \mathpzc{E})(n)\otimes_{\Sigma_n} Z'$};
\node [left of = G,xshift=-1cm] (H) {$\cdots$};

\draw [->] (A) -- (B) node[midway,anchor=south]{};
\draw [->] (A) -- (C) node[midway,anchor=west]{$\sim$};
\draw [->] (B) -- (D) node[midway,anchor=west]{$\sim$};
\draw [->] (C) -- (D) node[midway,anchor=south]{};
\draw [->] (F) -- (E) node[midway,anchor=south]{};
\draw [->] (E) -- (A) node[midway,anchor=south]{};
\draw [->] (H) -- (G) node[midway,anchor=south]{};
\draw [->] (G) -- (C) node[midway,anchor=south]{};
\draw [->] (E) -- (G) node[midway,anchor=west]{$\sim$};
\end{tikzpicture}
\end{center}

and the map $\mathpzc{E}_{\text{st}}(n) \otimes_{\Sigma_n} Z \to \mathpzc{E}_{\text{st}}(n) \otimes_{\Sigma_n} Z'$ is the map induced on the limits of the towers by the vertical arrows in this diagram. Now, it follows easily from Proposition~\ref{prop:stabmapontoML} that both the upper and lower towers satisfy the Mittag-Leffler condition. Thus, by the $\text{lim}{}^1$ short exact sequence and the five lemma, the map induced on the limits is itself a quasi-isomorphism.
\end{proof}

For the next result, recall the standard sphere and disk complexes defined in Section~\ref{sec:nots_convs}, denoted by $\Sph^n$ and $\D^n$.

\begin{Lemma}\label{lem:complexes_over_k_are_sums_of_spheres_and_disks}
Given a (co)chain complex $X$ over a field $k$, we have that 
\[
X \cong (\bigoplus_{i \in I} \Sph^{n_i}) \oplus (\bigoplus_{j \in J} \D^{n_j})
\]
for some index sets $I$ and $J$.
\end{Lemma}

Note that the above isomorphism is not natural. Note also that, under this isomorphism, the (co)homology of $X$ is given exactly by the sum of the spherical summands $\bigoplus_{i \in I} \Sph^{n_i}$.

\begin{proof}
This is standard. See, e.g., Exercise 1.1.3 in~\cite{Weibel}.
\end{proof}

We can now show that the monads associated to the stable Barratt-Eccles chain and cochain operads preserve quasi-isomorphisms.

\begin{Proposition}\label{prop:MS_st_pres_w_eqs}
The monads $\mathbf{E}_{\normalfont{\textbf{st}}}$ and $\mathbf{E}_{\normalfont{\textbf{st}}}^\dagger$ associated to the stable Barratt-Eccles chain and cochain operads preserve quasi-isomorphisms.
\end{Proposition}

\begin{proof}
We shall demonstrate the case of the chain operad; the case of the cochain operad is entirely analogous. First, recall that the monad $\mathbf{E}$, associated to the unstable Barratt-Eccles chain operad, preserves quasi-isomorphisms, which follows immediately from the fact that, for each $n \ge 0$, $\mathpzc{E}(n)$ is $\F_p[\Sigma_n]$-free. For each $k \ge 0$, let $\Sigma^k\mathbf{E}$ denote the monad associated to the operadic suspension $\Sigma^k\mathpzc{E}$ (note that, despite the notation, the monad is not being suspended, only the operad is). For the same reason as for $\mathbf{E}$, each $\Sigma^k\mathbf{E}$ also preserves quasi-isomorphisms. We first note that $\mathbf{E}_{\textbf{st}}$ preserves quasi-isomorphisms between finite $\F_p$-chain complexes. This follows from Lemma~\ref{lem:finflatEst} and the fact that if $X \to Y$ is a quasi-isomorphism between finite complexes over $\F_p$ that then the induced map $X^{\otimes n} \to Y^{\otimes n}$ is a quasi-isomorphism between finite complexes over $\F_p[\Sigma_n]$. It remains to show that $\textbf{E}_{\textbf{st}}$ preserves quasi-isomorphisms between not necessarily finite $\F_p$-chain complexes. To deduce this from the case of finite complexes, recall that any monad associated to an operad preserves filtered colimits (see~\cite{Rezk}) and also that filtered colimits of complexes are exact. Next, given any chain complex $X$, note that
\[
X = \underset{S \subseteq_{\text{fin}} X}{\text{colim}} \: S
\]
where $S \subseteq_{\text{fin}} X$ denotes the category of finite subcomplexes of $X$, the category of which is clearly filtered. Let $f \colon X \to Y$ be a quasi-isomorphism, where $X$ and $Y$ are arbitrary $\F_p$-chain complexes. Because $\textbf{E}_{\textbf{st}}$ preserves filtered colimits, $\textbf{E}_{\textbf{st}}(f)$ is isomorphic to a map
\[
\underset{S \subseteq_{\text{fin}} X}{\text{colim}} \: \textbf{E}_{\textbf{st}}(S) \to \underset{T \subseteq_{\text{fin}} Y}{\text{colim}} \: \textbf{E}_{\textbf{st}}(T)
\]
which we then need to show to be a quasi-isomorphism. We have that these colimits can be taken to be in chain complexes because filtered colimits of dg operad algebras are created in the category of dg modules. We wish to use the fact that filtered colimits of complexes are exact, but are unable to do so at the moment because there are no induced maps between the summands in the colimits; in fact, the indexing categories for the colimits are not even the same. We remedy this as follows. As per Lemma~\ref{lem:complexes_over_k_are_sums_of_spheres_and_disks}, any chain complex over $\F_p$ is isomorphic to a direct sum $(\bigoplus_{i \in I} \Sph^{n_i}) \oplus (\bigoplus_{j \in J} \D^{n_j})$, where $\Sph^n$ and $\D^n$ denote the standard sphere and disk complexes (see Section~\ref{sec:nots_convs}). We now split the proof into two cases. \\

Case 1: Suppose that $X = (\bigoplus_{i \in I} \Sph^{n_i})$, $Y = (\bigoplus_{i \in I} \Sph^{n_i}) \oplus (\bigoplus_{j \in J} \D^{n_j})$ and $f$ is the inclusion $X \hookrightarrow Y$, which is obviously a quasi-isomorphism. Note that every subcomplex $T$ of $Y$ is necessarily a sum of the summands in $(\bigoplus_{i \in I} \Sph^{n_i}) \oplus (\bigoplus_{j \in J} \D^{n_j})$. For each finite subcomplex $T$ of $Y$, let $S_T$ denote the finite subcomplex of $X$ which contains only the spherical summands which occur in $T$. We thus have that, for each finite subcomplex $T$ of $Y$, $f$ restricts to a map $i_T \colon S_T \to T$ and that this map is itself a quasi-isomorphism. Moreover, we clearly have that
\[
X = \underset{T \subseteq_{\text{fin}} Y}{\text{colim}} \: S_T
\]
as the change of index category simply causes some repeats in the summands. We have thus decomposed the map $f \colon X \to Y$ into the map induced on colimits by the maps $i_T$
\[
X = \underset{T \subseteq_{\text{fin}} Y}{\text{colim}} \: S_T \to \underset{T \subseteq_{\text{fin}} Y}{\text{colim}} \: T = Y.
\]
Moreover, the map $\mathbf{E}_{\textbf{st}}X \to \mathbf{E}_{\textbf{st}}Y$ induced by $f$ is then decomposed as the following
\[
\mathbf{E}_{\textbf{st}}X = \underset{T \subseteq_{\text{fin}} Y}{\text{colim}} \: \textbf{E}_{\textbf{st}}(S_T) \to \underset{T \subseteq_{\text{fin}} Y}{\text{colim}} \: \textbf{E}_{\textbf{st}}(T) = \mathbf{E}_{\textbf{st}}Y.
\]
Finally, this map induced on colimits is a quasi-isomorphism by what we have shown above in the case of finite complexes and the exactness of filtered colimits. \\

Case 2: Now consider general $X$ and $Y$ and a quasi-isomorphism $f \colon X \to Y$. Let $X \cong (\bigoplus_{i \in I_1} \Sph^{n_i}) \oplus (\bigoplus_{j \in J_1} \D^{n_j})$ and let $Y \cong (\bigoplus_{i \in I_2} \Sph^{n_i}) \oplus (\bigoplus_{j \in J_2} \D^{n_j})$. Since $f$ is a quasi-isomorphism, and since the homologies of $X$ and $Y$ are given by $\bigoplus_{i \in I_1} \Sph^{n_i}$ and $\bigoplus_{i \in I_2} \Sph^{n_i}$, respectively, it follows that $f$, upon restriction, must induce an isomorphism $\bigoplus_{i \in I_1} \Sph^{n_i} \to \bigoplus_{i \in I_2} \Sph^{n_i}$. We then get a commutative diagram as follows.

\begin{center}
\begin{tikzpicture}[node distance = 1.5cm]
\node [] (A) {$\bigoplus_{i \in I_1} \Sph^{n_i}$};
\node [right of = A, xshift = 4.5cm] (B) {$\bigoplus_{i \in I_2} \Sph^{n_i}$};
\node [below of = A] (C) {$(\bigoplus_{i \in I_1} \Sph^{n_i}) \oplus (\bigoplus_{j \in J_1} \D^{n_j})$};
\node [below of = B] (D) {$(\bigoplus_{i \in I_2} \Sph^{n_i}) \oplus (\bigoplus_{j \in J_2} \D^{n_j})$};
\node [below of = C] (E) {$X$};
\node [below of = D] (F) {$Y$};

\draw [->] (A) -- (B) node[midway,anchor=south]{$\cong$};
\draw [->] (A) -- (C) node[midway,anchor=east]{$\subseteq$};
\draw [->] (B) -- (D) node[midway,anchor=west]{$\subseteq$};
\draw [->] (C) -- (E) node[midway,anchor=east]{$\cong$};
\draw [->] (D) -- (F) node[midway,anchor=west]{$\cong$};
\draw [->] (E) -- (F) node[midway,anchor=north]{$f$};
\end{tikzpicture}
\end{center}

Upon applying $\mathbf{E}_{\textbf{st}}$ to this diagram, having already established Case 1, we get the desired result.
\end{proof}

Next, to be able to do homotopy theory with algebras over our stable operads, we wish to show that the stable Barratt-Eccles chain and cochain operads $\mathpzc{E}_{\text{st}}$ and $\mathpzc{E}^\dagger_{\text{st}}$ are semi-admissible. We demonstrate this with the help of a few lemmas along the way. Due to the criterion in Proposition~\ref{prop:amenable_implies_admissible}, we are interested in the coproducts $A \amalg \mathbf{E}_{\text{st}}(\D^n)$ and $A \amalg \mathbf{E}_{\text{st}}^\dagger(\D^n)$ for cell algebras $A$. The construction, which we saw earlier in Section~\ref{sec:env_op}, of such coproducts via enveloping operads will lead us to consider the enveloping operads for cell algebras $A$.

\begin{Lemma}\label{lem:almost_splitunstable}
Let $R$ be a ring and let $i \colon C \to D$ be a map of dg right $R$-modules which is split as a morphism of graded right $R$-modules. Then, if any two of $C$, $D$ and $D/C$ are semi-flat, so is the third.
\end{Lemma}

This fact also holds with semi-flat replaced by flat, with essentially the same proof -- see Proposition 13.12 in~\cite{Mandell}.

\begin{proof}
We shall consider the case of chain complexes, the case of cochain complexes differing only in some notations. Given any chain complex $P$ of left $R$-modules, the sequence
\[
0 \to C \otimes_R P \to D \otimes_R P \to (D/C) \otimes_R P \to 0
\]
is exact, as tensor is always right exact and the given retraction $r \colon D \to C$ gives an induced retraction $r \otimes_R \text{id}_P \colon D \otimes_R P \to C \otimes_R P$ (at the level of graded modules). This yields a long exact sequence in homology:
\[
\cdots \to \text{H}_n(C \otimes_R P) \to \text{H}_n(D \otimes_R P) \to \text{H}_n((D/C) \otimes_R P) \to \cdots
\]
Given any quasi-isomorphism $P \to Q$ between finite chain complexes of left $R$-modules, we get a morphism of these long exact sequences:
\begin{center}
\begin{tikzpicture}[node distance = 2cm]
\node [] (A) {$\cdots$};
\node [right of = A,xshift=1cm] (B) {$\text{H}_n(C \otimes_R P)$};
\node [right of = B,xshift=1cm] (C) {$\text{H}_n(D \otimes_R P)$};
\node [right of = C,xshift=1cm] (D) {$\text{H}_n((D/C) \otimes_R P)$};
\node [right of = D,xshift=1cm] (E) {$\cdots$};

\node [below of = A] (a) {$\cdots$};
\node [below of = B] (b) {$\text{H}_n(C \otimes_R Q)$};
\node [below of = C] (c) {$\text{H}_n(D \otimes_R Q)$};
\node [below of = D] (d) {$\text{H}_n((D/C) \otimes_R Q)$};
\node [below of = E] (e) {$\cdots$};

\draw [->] (B) -- (b);
\draw [->] (C) -- (c);
\draw [->] (D) -- (d);
\draw [->] (A) -- (B);
\draw [->] (B) -- (C);
\draw [->] (C) -- (D);
\draw [->] (D) -- (E);
\draw [->] (a) -- (b);
\draw [->] (b) -- (c);
\draw [->] (c) -- (d);
\draw [->] (d) -- (e);
\end{tikzpicture}
\end{center}
The result now follows by the five lemma.
\end{proof}

\begin{Lemma}\label{lem:change_of_ringsunstable}
Let $m, n \ge 0$. Given a semi-flat dg right $\F_p[\Sigma_{m+n}]$-module $M$ and a finite dg left $\F_p[\Sigma_m]$-module $N$, $M \otimes_{\F_p[\Sigma_m]} N$ is semi-flat over $\F_p[\Sigma_n]$.
\end{Lemma}

This fact also holds with semi-flat replaced by flat (in which case $N$ needn't be finite), with essentially the same proof -- see Proposition 13.13 in~\cite{Mandell}.

\begin{proof}
Given a finite dg left $\F_p[\Sigma_n]$-module $P$, we have a natural isomorphism
\[
(M \otimes_{\F_p[\Sigma_{m}]} N) \otimes_{\F_p[\Sigma_{n}]} P \cong M \otimes_{\F_p[\Sigma_{m+n}]} (\F_p[\Sigma_{m+n}] \otimes_{\F_p[\Sigma_m] \otimes_{\F_p} \F_p[\Sigma_n]} (N \otimes_{\F_p} P))
\]
and from this the result follows immediately, noting that $\F_p[\Sigma_{m+n}]$ is flat over $\F_p[\Sigma_m] \otimes_{\F_p} \F_p[\Sigma_n]$, and that $\F_p[\Sigma_{m+n}] \otimes_{\F_p[\Sigma_m] \otimes_{\F_p} \F_p[\Sigma_n]} (N \otimes_{\F_p} P)$ is a finite complex over $\F_p[\Sigma_{m+n}]$ given the finiteness of $N$ and $P$.
\end{proof}

\begin{Lemma}\label{lem:Estenvopsflat}
Let $A$ be a cell $\mathpzc{E}_{\emph{st}}$-algebra or a cell $\mathpzc{E}_{\emph{st}}^\dagger$-algebra. Let also $\mathpzc{U}^A$ denote the associated enveloping operad. Then, for all $j \ge 0$, $\mathpzc{U}^A(j)$ is semi-flat over $\F_p[\Sigma_j]$.
\end{Lemma}

Note that in the case of the (unstable) Barratt-Eccles operad, an analogous result holds with semi-flat replaced by flat, with a similar proof -- see Lemma 13.6 in~\cite{Mandell}.

\begin{proof}
We shall demonstrate the result in the case of the chain operad $\mathpzc{E}_{\text{st}}$, the case of the cochain operad being entirely analogous. Let
\[
A_0 \to A_1 \to A_2 \to \cdots
\]
be a cell filtration of $A$ and fix some choices $M_1, M_2, \dots$ for the chain complexes which appear in the attachment squares. For each $n \ge 0$, let $N_n = \oplus_{i \le n} M_i$, where $N_0 = 0$, and let also $N = \oplus_{i \ge 0} M_i$. As in Section~\ref{sec:env_op}, we have that, for each $j \ge 0$, as a graded right $\F_p[\Sigma_j]$-module
\[
\mathpzc{U}^A(j) = \bigoplus_{i \ge 0} \mathpzc{E}_{\text{st}}(i + j) \otimes_{\Sigma_i} (N[1])^{\otimes i}.
\]
The differential on $\mathpzc{U}^A(j)$, we recall, is given by the Leibniz rule, the attachment maps and the operadic composition. Moreover, for each $n \ge 0$, and again for each $j \ge 0$, as a graded right $\F_p[\Sigma_j]$-module
\[
\mathpzc{U}^{A_n}(j) = \bigoplus_{i \ge 0} \mathpzc{E}_{\text{st}}(i + j) \otimes_{\Sigma_i} (N_n[1])^{\otimes i}.
\]
Moreover, from Section~\ref{sec:env_op}, recall that we have filtrations $\text{F}_m\mathpzc{U}^{A_n}$ of  the $\mathpzc{U}^{A_n}$. We shall prove the desired result by an induction. We show that, for each $m,j,n \ge 0$, $\text{F}_m\mathpzc{U}^{A_n}(j)$ is semi-flat over $\F_p[\Sigma_j]$, and we will do this by inducting on $n$. In the case $n=0$, as in Section~\ref{sec:env_op}, we have that $\mathpzc{U}^{A_0} = \mathpzc{E}_{\text{st}}$, and moreover that $\text{F}_m\mathpzc{U}^{A_0}(j) = \mathpzc{E}_{\text{st}}(j)$ for all $m, j \ge 0$. The required semi-flatness then follows by Lemma~\ref{lem:finflatEst}. Suppose now that, for some $n \ge 1$, we have that $\text{F}_m\mathpzc{U}^{A_{n-1}}(j)$ is finitely flat over $\F_p[\Sigma_j]$ for all $m,j \ge 0$. We wish to show that $\text{F}_m\mathpzc{U}^{A_n}(j)$ is finitely flat over $\F_p[\Sigma_j]$ for all $m, j \ge 0$. We shall do this by inducting over $m$. By definition of the filtration piece $\text{F}_0$, we have that, for each $j \ge 0$, $\text{F}_0\mathpzc{U}^{A_n}(j) = \mathpzc{U}^{A_{n-1}}(j) = \text{colim}_m\,\text{F}_m\mathpzc{U}^{A_{n-1}}(j)$ which, by invoking the inductive hypothesisis for the induction over $n$ and passing to the colimit, we see is semi-flat over $\F_p[\Sigma_j]$. Next, suppose that for some $m \ge 1$, $\text{F}_{m-1}\mathpzc{U}^{A_n}(j)$ is semi-flat over $\F_p[\Sigma_j]$. As in Section~\ref{sec:env_op}, we have that
\[
\text{F}_m\mathpzc{U}^{A_n}(j)/\text{F}_{m-1}\mathpzc{U}^{A_n}(j) \cong \mathpzc{U}^{A_{n-1}}(m+j) \otimes_{\Sigma_m} M_n[1]^{\otimes m}.
\]
Now, by invoking the inductive hypothesis for the induction over $n$ and passing to the colimit, we see that $\mathpzc{U}^{A_{n-1}}(j+m) = \text{colim}_{m'}\,\text{F}_{m'}\mathpzc{U}^{A_{n-1}}(m+j)$ is semi-flat over $\F_p[\Sigma_{m+j}]$. Moreover, by Lemma~\ref{lem:change_of_ringsunstable}, we have that $\mathpzc{U}^{A_{n-1}}(m+j) \otimes_{\Sigma_m} M_n[1]^{\otimes m}$ is then semi-flat over $\F_p[\Sigma_j]$ so long as $M_n$ is finite. In fact, this holds for arbitrary $M_n$ as a non-finite $M_n$ can be written as a filtered colimit of its finite subcomplexes and both the tensor product $\mathpzc{U}^{A_{n-1}}(m+j) \otimes_{\Sigma_m} -$ and the tensor power $(-)^{\otimes m}$ commute with filtered colimits. Next, recalling that the inclusion $\text{F}_{m-1}\mathpzc{U}^{A_n}(j) \to \text{F}_m\mathpzc{U}^{A_n}(j)$ is split at the level of the underlying graded modules (see Section~\ref{sec:env_op}), we may now invoke the inductive hypothesis for the induction over $m$ and apply Lemma~\ref{lem:almost_splitunstable} to conclude that $\text{F}_m\mathpzc{U}^{A_n}(j)$ is semi-flat over $\F_p[\Sigma_j]$, as desired. This completes the induction over $m$ so that we have that $\text{F}_m\mathpzc{U}^{A_n}(j)$ is semi-flat over $\F_p[\Sigma_j]$ for all $m,j \ge 0$. Moreover, this conclusion then completes the induction over $n$ so that we have that $\text{F}_m\mathpzc{U}^{A_n}(j)$ is semi-flat over $\F_p[\Sigma_j]$ for all $m,j,n \ge 0$. Finally then, if we fix a $j \ge 0$, upon passing to the colimit, we have that $\mathpzc{U}^{A_n}(j) = \text{colim}_m\text{F}_m\mathpzc{U}^{A_n}(j)$ is finitely flat over $\F_p[\Sigma_j]$, and then, passing to the colimit again, we have the desired result that $\mathpzc{U}^{A}(j) = \text{colim}_n\mathpzc{U}^{A_n}(j)$ is semi-flat over $\F_p[\Sigma_j]$, which completes the proof.
\end{proof}

We now use the above lemmas to demonstrate that the operads $\mathpzc{E}_{\text{st}}$ and $\mathpzc{E}_{\text{st}}^\dagger$ are semi-admissible.

\begin{Proposition}\label{prop:E_adm}
The Barratt-Eccles chain and cochain operads, $\mathpzc{E}_{\emph{st}}$ and $\mathpzc{E}^\dagger_{\emph{st}}$, are semi-admissible.
\end{Proposition}

\begin{proof}
We shall demonstrate the case of the chain operad, the case of the cochain operad being entirely analogous. By Proposition~\ref{prop:amenable_implies_admissible}, it suffices to show that, if $A$ is a cell $\mathpzc{E}_{\text{st}}$-algebra, then for each $n \in \Z$, the canonical map
\[
A \to A \amalg \mathbf{E}_{\textbf{st}}(\D^n)
\]
is a quasi-isomorphism; here $\D^n$ is a disk complex (see Section~\ref{sec:nots_convs}). As per the facts about enveloping operads in Section~\ref{sec:env_op}, we have that, as an algebra under $A$
\[
A \amalg \mathbf{E}_{\textbf{st}}(\D^n) \cong \mathbf{U}^A(\D^n) = \bigoplus_{j \ge 0} \mathpzc{U}^A(j) \otimes_{\Sigma_j} (\D^n)^{\otimes j} = A \oplus \left(\bigoplus_{j \ge 1} \mathpzc{U}^A(j) \otimes_{\Sigma_j} (\D^n)^{\otimes j}\right).
\]
Now, for $j \ge 1$, $(\D^n)^{\otimes j}$ has zero homology and is finite. Moreover, by Lemma~\ref{lem:Estenvopsflat}, $\mathpzc{U}^A(j)$ is semi-flat over $\F_p[\Sigma_j]$, so that this zero homology is preserved by the tensor, which gives us the desired result.
\end{proof}

Recalling the definition of semi-admissibility, and Proposition~\ref{prop:amenable_implies_admissible}, we get the following.

\begin{Corollary}\label{cor:stablesemimod}
\index{Quillen semi-model structure!on algebras over stable operads}The categories of algebras $\mathpzc{E}_{\emph{st}}\text{-}\mathsf{Alg}$ and $\mathpzc{E}^\dagger_{\emph{st}}\text{-}\mathsf{Alg}$ possess a Quillen semi-model structure where:
\begin{itemize}
	\item The weak equivalences are the quasi-isomorphisms.
	\item The fibrations are the surjective maps.
	\item The cofibrations are retracts of relative cell complexes, where the cells are the maps $\mathbf{E}_{\mathbf{st}}M \to \mathbf{E}_{\mathbf{st}}\emph{C}M$ in the chain case, and the maps $\mathbf{E}_{\mathbf{st}}^\dagger M \to \mathbf{E}_{\mathbf{st}}^\dagger\emph{C}M$ in the cochain case, where $M$ is a degreewise free complex with zero differentials.
\end{itemize}
\qed
\end{Corollary}


\chapter{Cohomology Operations for Algebras Over the Stable Operads}\label{sec:cohom_ops}

Given an algebra $A$ over the (unstable) Barratt-Eccles operad, the (co)homology of $A$ possesses natural operations. In the case of the chain operad, these are the generalized Dyer-Lashof operations, and in the case of the cochain operad, these are the generalized Steenrod operations (just like the operads, the generalized Dyer-Lashof operations and the generalized Steenrod operations are simply reindexed variants of one another). In this section, we shall demonstrate that algebras over the stable Barratt-Eccles operad also have natural (co)homology operations, though they are now stable operations, in a precise sense to be described below. In fact, for the sake of brevity, henceforth, we shall work with only the stable Barratt-Eccles cochain operad, though all that we say will have clear analogues for the stable Barratt-Eccles chain operad. Throughout this section, the ground field will be $\F_p$, for an unspecified but fixed prime $p$.

\section{The Cohomology Operations I}\label{sec:stabops}

Let $A$ be an algebra over $\mathpzc{E}_{\text{st}}^\dagger$. Our aim is to show that there are induced natural operations
\[
P^s \colon \text{H}^\bullet(A) \to \text{H}^{\bullet}(A)
\]
and also
\[
\beta P^s \colon \text{H}^\bullet(A) \to \text{H}^{\bullet}(A)
\]
in the case $p > 2$. Recall that such operations exist in the case of the unstable operad $\mathpzc{E}^\dagger$ (for a construction of them, see~\cite{May}). In fact, in our stable case, we will have some other operations on $\text{H}^\bullet(A)$ as well. The general construction of all of them will be found in Section~\ref{sec:stabops2} below. Here, we wish to give a restricted but more concrete construction of them. We shall restrict ourselves in this section alone to the case $p = 2$; analogous explicit considerations in the $p > 2$ case are also possible, though more cumbersome. \\

To begin, recall that, as in~\cite{May}, the operations, when $p = 2$, in the case of the unstable operad are defined with the help of the arity two part of the operad $\mathpzc{E}^\dagger$. As such, our first goal is to examine the arity two part of the stable operad $\mathpzc{E}_{\text{st}}^\dagger$. Examining the definition of the Barrat-Eccles cochain operad, one finds that the cochain complex $\mathpzc{E}^\dagger(2)$ is isomorphic to the standard $\F_2[\Sigma_2]$-free resolution of $\F_2$, namely
\begin{equation}\tag*{$\mathpzc{E}^\dagger(2):$}
\cdots \longrightarrow \underset{\text{deg}\, -2}{\F_2[\Sigma_2]} \overset{1+\tau}{\longrightarrow} \underset{\text{deg}\, -1}{\F_2[\Sigma_2]} \overset{1+\tau}{\longrightarrow} \underset{\text{deg}\, 0}{\F_2[\Sigma_2]} \longrightarrow 0 \longrightarrow \cdots .
\end{equation}
Here, $\tau$ denotes the non-trivial permutation of $\{1,2\}$. To be more specific, the isomorphism between the above complex and $\mathpzc{E}^\dagger(2)$, in degree $-d$, for $d \ge 0$, sends $1$ to the $(d+1)$-tuple $(1, \tau, 1, \tau, \dots)$ in $\mathpzc{E}^\dagger(2)^{-d} = \mathpzc{E}(2)_d = \mathrm{C}_d(\mathrm{E}\Sigma_2)$. Now let us see what the arity two part of our stable operad looks like.

\begin{Proposition}\label{prop:Est(2)}
\index{stable Barratt-Eccles operad!in arity two}The cochain complex $\mathpzc{E}_{\emph{st}}^\dagger(2)$ is isomorphic to the following.
\begin{equation}\tag*{$\mathpzc{E}^\dagger_{\text{st}}(2):$}
\cdots \longrightarrow \F_2[\Sigma_2] \overset{1+\tau}{\longrightarrow} \F_2[\Sigma_2] \overset{1+\tau}{\longrightarrow} \F_2[\Sigma_2] \overset{1+\tau}{\longrightarrow} \F_2[\Sigma_2] \longrightarrow \cdots
\end{equation}
\end{Proposition}

\begin{proof}
We have a description of $\mathpzc{E}^\dagger(2)$ above. Moreover, for each $k \ge 0$, we have that $(\Sigma^k\mathpzc{E}^\dagger)(2) = \mathpzc{E}^\dagger(2)[k]$. Now, given as input a tuple $(\rho_0,\dots,\rho_d)$ of permutations $\rho_i \in \Sigma_2$, for some $d \ge 0$, by definition, the stabilization map $(\Sigma^{k+1}\mathpzc{E}^\dagger)(2) \to (\Sigma^k\mathpzc{E}^\dagger(2)$ simply drops the first entry of the tuple. It follows that the inverse limit is then the desired complex -- to see this, note that, if we write the complexes $(\Sigma^k\mathpzc{E}^\dagger)(2)$ vertically, the tower $\cdots \to (\Sigma^2\mathpzc{E}^\dagger)(2) \to (\Sigma\mathpzc{E}^\dagger)(2) \to \mathpzc{E}^\dagger(2)$ looks as follows.
\begin{center}
	\begin{tikzpicture}[node distance=1.5cm]
	\node[](A){$\vdots$};
	\node[below of = A](B){$\F_2[\Sigma_2]$};
	\node[below of = B](C){$\F_2[\Sigma_2]$};
	\node[below of = C](D){$\F_2[\Sigma_2]$};
	\node[below of = D](E){$\F_2[\Sigma_2]$};
	\node[below of = E](F){$0$};
	\node[below of = F](G){$\vdots$};
	
	\node[right of = A,xshift=5mm](AA){$\vdots$};
	\node[below of = AA](BB){$\F_2[\Sigma_2]$};
	\node[below of = BB](CC){$\F_2[\Sigma_2]$};
	\node[below of = CC](DD){$\F_2[\Sigma_2]$};
	\node[below of = DD](EE){$0$};
	\node[below of = EE](FF){$0$};
	\node[below of = FF](GG){$\vdots$};
	
	\node[right of = AA,xshift=5mm](AAA){$\vdots$};
	\node[below of = AAA](BBB){$\F_2[\Sigma_2]$};
	\node[below of = BBB](CCC){$\F_2[\Sigma_2]$};
	\node[below of = CCC](DDD){$0$};
	\node[below of = DDD](EEE){$0$};
	\node[below of = EEE](FFF){$0$};
	\node[below of = FFF](GGG){$\vdots$};
	
	\node[left of = B,xshift=-5mm](BL){$\cdots$};
	\node[left of = C,xshift=-5mm](CL){$\cdots$};
	\node[left of = D,xshift=-5mm](DL){$\cdots$};
	\node[left of = E,xshift=-5mm](EL){$\cdots$};
	\node[left of = F,xshift=-5mm](FL){$\cdots$};
		
	\draw[->] (A) -- (B) node[midway,anchor=west]{};
	\draw[->] (B) -- (C) node[midway,anchor=west]{$1+\tau$};
	\draw[->] (C) -- (D) node[midway,anchor=west]{$1+\tau$};
	\draw[->] (D) -- (E) node[midway,anchor=west]{$1+\tau$};
	\draw[->] (E) -- (F) node[midway,anchor=west]{};
	\draw[->] (F) -- (G) node[midway,anchor=west]{};
	
	\draw[->] (AA) -- (BB) node[midway,anchor=west]{};
	\draw[->] (BB) -- (CC) node[midway,anchor=west]{$1+\tau$};
	\draw[->] (CC) -- (DD) node[midway,anchor=west]{$1+\tau$};
	\draw[->] (DD) -- (EE) node[midway,anchor=west]{};
	\draw[->] (EE) -- (FF) node[midway,anchor=west]{};
	\draw[->] (FF) -- (GG) node[midway,anchor=west]{};
	
	\draw[->] (AAA) -- (BBB) node[midway,anchor=west]{};
	\draw[->] (BBB) -- (CCC) node[midway,anchor=west]{$1+\tau$};
	\draw[->] (CCC) -- (DDD) node[midway,anchor=west]{};
	\draw[->] (DDD) -- (EEE) node[midway,anchor=west]{};
	\draw[->] (EEE) -- (FFF) node[midway,anchor=west]{};
	\draw[->] (FFF) -- (GGG) node[midway,anchor=west]{};
	
	\draw[->] (B) -- (BB) node[midway,anchor=south]{$\tau$};
	\draw[->] (BB) -- (BBB) node[midway,anchor=south]{$\tau$};
	
	\draw[->] (C) -- (CC) node[midway,anchor=south]{$\tau$};
	\draw[->] (CC) -- (CCC) node[midway,anchor=south]{$\tau$};
	
	\draw[->] (D) -- (DD) node[midway,anchor=south]{$\tau$};
	\draw[->] (DD) -- (DDD) node[midway,anchor=south]{};
	
	\draw[->] (E) -- (EE) node[midway,anchor=south]{};
	\draw[->] (EE) -- (EEE) node[midway,anchor=south]{};
	
	\draw[->] (F) -- (FF) node[midway,anchor=south]{};
	\draw[->] (FF) -- (FFF) node[midway,anchor=south]{};
	
	\draw[->] (BL) -- (B) node[midway,anchor=west]{};
	\draw[->] (CL) -- (C) node[midway,anchor=west]{};
	\draw[->] (DL) -- (D) node[midway,anchor=west]{};
	\draw[->] (EL) -- (E) node[midway,anchor=west]{};
	\draw[->] (FL) -- (F) node[midway,anchor=west]{};
	\end{tikzpicture}
\end{center}
\end{proof}

\begin{Remark}
We saw earlier, in Proposition~\ref{prop:stabhom}, that the non-equivariant homology of $\mathpzc{E}_{\text{st}}^\dagger(2)$ is zero. On the other hand, Proposition~\ref{prop:Est(2)} above and an easy calculation shows that the equivariant homology of $\mathpzc{E}_{\text{st}}^\dagger(2)$, by which we mean the homology of $\mathpzc{E}_{\text{st}}^\dagger(2)/\Sigma_2$, consists of exactly a unique $\F_2$ generator in each degree:
\begin{equation}\tag*{$\text{H}_\bullet(\mathpzc{E}_{\text{st}}^\dagger(2)/\Sigma_2):$} \cdots \qquad \underset{\text{deg}\,-1}{\F_2} \qquad \underset{\text{deg}\,0}{\F_2} \qquad \underset{\text{deg}\,1}{\F_2} \qquad \cdots .
\end{equation}
For comparison, in the unstable case, via the description of $\mathpzc{E}^\dagger(2)$ above and another easy calculation, we have that the homology of $\mathpzc{E}^\dagger(2)/\Sigma_2$, consists of exactly a unique $\F_2$ generator in each non-positive degree:
\begin{equation}\tag*{$\text{H}_\bullet(\mathpzc{E}^\dagger(2)/\Sigma_2):$}
\cdots \qquad \underset{\text{deg}\,-1}{\F_2} \qquad \underset{\text{deg}\, 0}{\F_2} \qquad \underset{\text{deg}\,1}{0} \qquad \cdots .
\end{equation}
In fact, we will see below that it is exactly the generators of these equivariant arity two homologies that give rise to the cohomology operations (see Definition~\ref{def:e_d^un_and_e_d^st} and Proposition~\ref{prop:opsalgEstconstrp=2}), and moreover, that the existence of generators in positive degrees in the stable case results in a stability property of these operations (see Remark~\ref{rmk:stable_vs_unstable_ops}). We will also see that the higher arity equivariant cohomologies correspond to iterated operations (see the filtrations of $\mathrm{H}^\bullet((\Sigma^k\mathbf{E}^\dagger)X)$ and $\mathcal{F}_k$ in Section~\ref{subsec:cohomologies_of_free_algebras_over_the_stable_operads}) and that the total equivariant cohomology is highly non-trivial, despite the trivial non-equivariant cohomology (see Remark~\ref{rmk:eqhomnontrivial}). \customendremark
\end{Remark}

Let us set in place some standardized notations to work with the unstable and stable Barratt-Eccles operads.

\begin{Definition}\label{def:e_d^un_and_e_d^st}
For each $d \in \Z$, we define an element $e_{d}^{\mathrm{un}} \in \mathpzc{E}^\dagger(2)^{d}$ by setting
\[
e_{d}^{\mathrm{un}} :=
\left\{
\begin{array}{ll}
  (1, \tau, 1, \tau, \dots) & d \le 0 \\
  0 & d > 0
\end{array}
\right.
\]
where, in the case $d \le 0$, the righthand side is a sequence of $|d|+1$ elements. In the stable case, for each $d \in \Z$, we define an element $e_d^{\mathrm{st}} \in \mathpzc{E}_{\mathrm{st}}^\dagger(2)^{d}$ by setting
\[
e_{d}^{\mathrm{st}} := (e_d^{\mathrm{un}}, e_{d-1}^{\mathrm{un}} \cdot \tau, e_{d-2}^{\mathrm{un}}, e_{d-3}^{\mathrm{un}} \cdot \tau, \dots)
\]
where, if $d > 0$, we note that the tuple on the righthand side has $d$ leading zeros.
\end{Definition}

\begin{Remark}\label{rmk:edtwice}
By the definition of $\mathpzc{E}_{\text{st}}^\dagger$ as an inverse limit, we have a canonical map
\[
\mathpzc{E}_{\text{st}}^\dagger \to \mathpzc{E}^\dagger
\]
from the stable Barratt-Eccles operad to the Barratt-Eccles operad, which, for any $n$, sends an infinite tuple in $\mathpzc{E}_{\text{st}}^\dagger(n)$ to its first entry. By the definition of the $e_d^{\text{un}}$ and $e_d^{\text{st}}$, we find that, in arity two, for each $d \in \Z$, this map sends $e_d^{\text{st}}$ to $e_d^{\text{un}}$. \customendremark
\end{Remark}

We set in place one more piece of notation: given an algebra $A$ over $\mathpzc{E}_{\text{st}}^\dagger$, $e \in \mathpzc{E}_{\text{st}}^\dagger(n)$ and $a_1, \dots, a_n \in A$, we let $e_*(a_1,\dots,a_n)$ denote the image of $e \otimes a_1 \otimes \cdots \otimes a_n$ under the algebra structure map $\mathpzc{E}_{\text{st}}^\dagger(n) \otimes A^{\otimes n} \to A$. With this in place, we can now define the cohomology operations for algebras over the stable Barratt-Eccles cochain operad, at least when $p = 2$. Given an algebra $A$ over $\mathpzc{E}_{\text{st}}^\dagger$, for each $s \in \Z$, we define $P^s \colon \text{H}^\bullet(A) \to \text{H}^\bullet(A)$ by setting, for $[a] \in \text{H}^q(A)$
\[
P^s([a]) = [(e_{s-q}^{\text{st}})_*(a,a)].
\]

\begin{Proposition}\label{prop:opsalgEstconstrp=2}
\index{generalized Steenrod operations!the $p=2$ case for the stable operad}The operations $P^s$, as defined above, are well-defined, linear over $\F_2$, of degree $s$ and natural in $A$.
\end{Proposition}

\begin{proof}
First, let us check that the operations are well-defined. This follows from the following two facts, which we shall demonstrate: (i) given a cocycle $a$ in $A$, $e_d^{\text{st}} \otimes a \otimes a$, for any $d$, is a cocycle in $\mathpzc{E}_{\text{st}}^\dagger(2) \otimes_{\Sigma_2} A^{\otimes 2}$ (ii) if $a$ and $a'$ are cohomologous cocycles in $A$, $e_d^{\text{st}} \otimes a \otimes a$ and $e_d^{\text{st}} \otimes a' \otimes a'$ are cohomologous cocycles in $\mathpzc{E}_{\text{st}}^\dagger(2) \otimes_{\Sigma_2} A^{\otimes 2}$. Consider (i) first. This follows from the following identities, which hold in $\mathpzc{E}_{\text{st}}^\dagger(2) \otimes_{\Sigma_2} A^{\otimes 2}$
\begin{align*}
\partial (e_d^{\text{st}} \otimes a \otimes a) &= (e_{d+1}^{\text{st}} \cdot (1 + \tau)) \otimes a \otimes a \\
&= e_{d+1}^{\text{st}} \otimes ((1+\tau) \cdot a \otimes a) \\
&= e_{d+1}^{\text{st}} \otimes a \otimes a + e_{d+1}^{\text{st}} \otimes (\tau \cdot a \otimes a) \\
&= e_{d+1}^{\text{st}} \otimes a \otimes a + e_{d+1}^{\text{st}} \otimes a \otimes a = 0.
\end{align*}
Next, consider (ii). We need to show that $e_d^{\text{st}} \otimes a \otimes a - e_d^{\text{st}} \otimes a' \otimes a'$ is a coboundary in $\mathpzc{E}_{\text{st}}^\dagger(2) \otimes_{\Sigma_2} A^{\otimes 2}$. By assumption, we know that $a - a'$ is a coboundary in $A$; let $a - a' = \partial b$. The desired result then follows from the following easily verifiable identity:
\[
\partial (e_d^{\text{st}} \otimes a \otimes b + e_d^{\text{st}} \otimes b \otimes a' + e_{d+1}^{\text{st}} \otimes b \otimes b) = e_d^{\text{st}} \otimes a \otimes a - e_d^{\text{st}} \otimes a' \otimes a'.
\]
Next, let us verify linearity over $\F_2$. First, we have homogeneity as follows:
\begin{align*}
P^s(\lambda [a]) &= P^s([\lambda a]) \\
&= [(e_{s-q}^{\text{st}})_*(\lambda a, \lambda a)] \\
& = \lambda^2 [(e_{s-q}^{\text{st}})_*(a, a)] \\
&= \lambda [(e_{s-q}^{\text{st}})_*(a, a)] \\
&= \lambda P^s([a]).
\end{align*}
As for additivity, let $[a],[b] \in \text{H}^q(A)$. Consider $e_{s-q}^{\text{st}} \otimes (a + b) \otimes (a + b) - e_{s-q}^{\text{st}} \otimes a \otimes a - e_{s-q} \otimes b \otimes b$ as an element of $\mathpzc{E}^\dagger_{\text{st}}(2) \otimes_{\Sigma_2} A^{\otimes 2}$. It suffices to show that this element is a coboundary in $\mathpzc{E}^\dagger_{\text{st}}(2) \otimes_{\Sigma_2} A^{\otimes 2}$. We have that $e_{s-q}^{\text{st}} \otimes (a + b) \otimes (a + b) - e_{s-q}^{\text{st}} \otimes a \otimes a - e_{s-q}^{\text{st}} \otimes b \otimes b = e_{s-q}^{\text{st}} \otimes a \otimes b + e_{s-q}^{\text{st}} \otimes b \otimes a$, and then the result follows by the following
\begin{align*}
e_{s-q}^{\text{st}} \otimes a \otimes b + e_{s-q}^{\text{st}} \otimes b \otimes a &= e_{s-q}^{\text{st}} \otimes a \otimes b +  e_{s-q}^{\text{st}} \tau \otimes a \otimes b \\
&= e_{s-q}^{\text{st}} (\tau + 1) \otimes a \otimes b \\
&= \partial (e_{s-q}^{\text{st}} \otimes a \otimes b).
\end{align*}

To see that $P^s$ is homogeneous of degree $s$, simply note that, given $[a] \in \text{H}^q(A)$, in $\mathpzc{E}_{\text{st}}^\dagger(2) \otimes A^{\otimes 2}$, $e_{s-q}^{\text{st}} \otimes a \otimes a$ has degree $(s-q)+q+q = q+s$. Finally, we verify naturality. Let $f \colon A \to B$ be a map of algebras over $\mathpzc{E}_{\text{st}}^\dagger$. We need to show that, for each $s \in \Z$, the following square commutes.

\begin{center}
\begin{tikzpicture}[node distance = 1.5cm]
\node [] (A) {$\text{H}^\bullet(A)$};
\node [below of = A] (B) {$\text{H}^\bullet(B)$};
\node [right of = A,xshift=1cm] (C) {$\text{H}^\bullet(A)$};
\node [right of = B,xshift=1cm] (D) {$\text{H}^\bullet(B)$};

\node [below right of = A,yshift=3mm,xshift=2mm] (E) {$\acts$ ?};

\draw [->] (A) -- (B) node[midway,anchor=east]{$f_*$};
\draw [->] (A) -- (C) node[midway,anchor=south]{$P^s$};
\draw [->] (C) -- (D) node[midway,anchor=west]{$f_*$};
\draw [->] (B) -- (D) node[midway,anchor=north]{$P^s$};
\end{tikzpicture}
\end{center}

This follows from the commutativity of the following diagram.

\begin{center}
\begin{tikzpicture}[node distance = 1.5cm]
\node [] (D) {$\text{H}_\bullet(\mathpzc{E}_{\text{st}}^\dagger(2) \otimes_{\Sigma_2} A^{\otimes 2})$};
\node [right of = D,xshift=4cm] (E) {$\text{H}_\bullet(A)$};

\node [below of = D] (DD) {$\text{H}_\bullet(\mathpzc{E}_{\text{st}}^\dagger(2) \otimes_{\Sigma_2} B^{\otimes 2})$};
\node [right of = DD,xshift=4cm] (EE) {$\text{H}_\bullet(B)$};

\draw [->] (D) -- (E) node[midway,anchor=south]{$\mathpzc{E}_{\text{st}}^\dagger$-action};
\draw [->] (DD) -- (EE) node[midway,anchor=north]{$\mathpzc{E}_{\text{st}}^\dagger$-action};
\draw [->] (D) -- (DD) node[midway,anchor=east]{$f_*$};
\draw [->] (E) -- (EE) node[midway,anchor=west]{$f_*$};
\end{tikzpicture}
\end{center}
\end{proof}

\begin{Remark}\label{rmk:stable_vs_unstable_ops}
We now make a few remarks regarding these stable operations and how they compare with the corresponding operations in the unstable case.
\begin{itemize}
	\item[(1)] In the case of the unstable Barratt-Eccles operad, the associated cohomology operations, which is to say the generalized Steenrod operations, are defined by exactly the same formula, $P^s([a]) = [(e_{s-q}^{\text{un}})_*(a,a)]$, except that we replace $e_{s-q}^{\text{st}}$ by $e_{s-q}^{\text{un}}$. See~\cite{May} for the construction of these operations; even the proof of the analogue of Proposition~\ref{prop:opsalgEstconstrp=2} is the same. What differs is that, since $e_d^{\text{un}}$ is zero for $d > 0$, these operations are zero when $s > q = |a|$. This property, that $P^s[a] = 0$ when $s > |a|$, is known as the \textit{instability} of the operations. In our stable case, we have non-zero generators $e_{d}^{\text{st}}$ even in positive degrees, and so our stable operations do not satisfy the instability property, as one would hope based on the terminology alone. In fact, we can see the disappearance of the instability of the operations in an iterative manner, as follows. We know that the operations in the case of algebras over $\mathpzc{E}^\dagger$ satisfy instability. By an analogous construction, or with the help of Proposition~\ref{prop:freealgsuspop}, one can show that one also has operations in the case of algebras over $\Sigma\mathpzc{E}^\dagger$ and moreover, that these satisfy a shifted instability condition, which says that $P^s[a] = 0$ when $s > |a|+1$. Similarly, one also has operations in the case of algebras over $\Sigma^2\mathpzc{E}^\dagger$, and these satisfy the shifted instability condition which says that $P^s[a] = 0$ when $s > |a|+2$. This continues, and eventually, in the limit, in the case of the stabilization $\mathpzc{E}^\dagger_{\text{st}}$, the instability disappears.
	\item[(2)] In the case of algebras over the unstable Barratt-Eccles operad, the cohomology operations satisfy the well-known Adem relations -- see~\cite{May}. These relations also hold for our operations above in the case of algebras over the stable Barratt-Eccles operad -- this follows, for example, from the computation, in Proposition~\ref{prop:stablefreehom} below, of the cohomologies of free algebras over $\mathpzc{E}_{\text{st}}^\dagger$. See also Remark~\ref{rmk:general_ops_vs_explicit_ops}.
	\item[(3)] In the case of an algebra $A$ over the unstable Barratt-Eccles operad, the cohomology $\text{H}^\bullet(A)$ not only possesses the operations $P^s$, but also a graded-commutative algebra structure -- see~\cite{May}. The definition of the products is given by the formula $[a] \cdot [b] = [(e_0^{\text{un}})_*(a,b)]$. In the case of algebras over the stable Barratt-Eccles operad, these products do not exist; we only have the additive structure provided by the operations, as one would expect in a stable situation. To see why these products do not exist in the stable situation, note that the obvious analogue of the above formula is $[a] \cdot [b] = [(e_0^{\text{st}})_*(a,b)]$. However, this does not make sense because, given cocycles $a$ and $b$, while $e_0^{\text{un}} \otimes a \otimes b$ is always a cocycle in $\mathpzc{E}^\dagger(2) \otimes A^{\otimes 2}$, $e_0^{\text{st}} \otimes a \otimes b$ is not always a cocycle in $\mathpzc{E}_{\text{st}}^\dagger(2) \otimes A^{\otimes 2}$, as $e_0^{\text{st}}$ is itself not a cocyle in $\mathpzc{E}_{\text{st}}^\dagger(2)$ (see Proposition~\ref{prop:Est(2)}). Put another way, $\mathpzc{E}^\dagger(2)$ has cohomology $\F_2[0]$, generated by $e_0^{\text{un}}$, and this generator yields the products in the unstable case, but $\mathpzc{E}_{\text{st}}^\dagger(2)$ has zero cohomology (see Proposition~\ref{prop:stabhom}). This argument shows that the products at least cannot be defined by the obvious analogous formula, though of course does not demonstrate that they necessarily do not exist. The latter follows however from the computation of the cohomology of free $\mathpzc{E}_{\text{st}}^\dagger$-algebras, to appear below in Proposition~\ref{prop:stablefreehom}. One can also see the lack of products in an iterative manner, as follows. We know that the cohomologies of algebras over $\mathpzc{E}^\dagger$ have a product operation. By an entirely analogous construction, or with the help of Proposition~\ref{prop:freealgsuspop}, one can show that the cohomologies of algebras over $\Sigma\mathpzc{E}^\dagger$ have a shifted product operation where the product of a degree $r$ element with a degree $s$ element lies in degree $r+s+1$. Similarly, the cohomologies of algebras over $\Sigma^2\mathpzc{E}^\dagger$ have a shifted product operation where the product of a degree $r$ element with a degree $s$ element lies in degree $r+s+2$. This continues, and eventually, in the limit, in the case of the stabilization $\mathpzc{E}^\dagger_{\text{st}}$, the product disappears, in the sense that it is undefinable.
	\item[(4)] Suppose that $A$ is an algebra over the unstable Barratt-Eccles operad $\mathpzc{E}^\dagger$. By pull back across the canonical map $\mathpzc{E}_{\text{st}}^\dagger \to \mathpzc{E}^\dagger$, $A$ is then also an algebra over the stable Barratt-Eccles operad $\mathpzc{E}^\dagger_{\text{st}}$. Each of the algebra structures induce operations on $\text{H}^\bullet(A)$, both of which are denote by $P^s$. The notation is consistent however, because, as noted in Remark~\ref{rmk:edtwice}, the canonical map $\mathpzc{E}_{\text{st}}^\dagger \to \mathpzc{E}^\dagger$ sends $e_d^{\text{st}}$ to $e_d^{\text{un}}$, for each $d \ge 0$.
\end{itemize} \customendremark
\end{Remark}

\section[The Completion $\widehat{\mathcal{B}}$]{The Completion $\widehat{\mathcal{B}}$ of the Algebra of Generalized Steenrod Operations}

Our next goal is to compute the cohomology of a free algebra $\textbf{E}_{\textbf{st}}^\dagger X$ over the stable Barratt-Eccles operad $\mathpzc{E}_{\text{st}}^\dagger$. In order to achieve this, we first need to recall the algebra $\mathcal{B}$ of generalized Steenrod operations and construct a completion $\widehat{\mathcal{B}}$ of it, which allows certain infinite sums of the iterated operations in $\mathcal{B}$. In the case of the unstable operad, it is well-known that $\text{H}^\bullet(\textbf{E}^\dagger X)$ is a free construction on $\text{H}^\bullet(X)$, that which first constructs the free unstable $\mathcal{B}$-module on $X$, and then applies a certain free algebra construction on this module to add in the products which exist in the unstable case (see Proposition~\ref{prop:cohomology_of_free_algebras_of_suspended_E}). In our stable case, we know that the products no longer exist and that the instability of the operations no longer holds (see Remark~\ref{rmk:stable_vs_unstable_ops}), and below (see Proposition~\ref{prop:stablefreehom}), we will prove that $\text{H}^\bullet(\textbf{E}_{\textbf{st}}^\dagger X)$ is exactly the free $\widehat{\mathcal{B}}$-module on $\text{H}^\bullet(X)$. \\

Now, let us begin constructing the completion $\widehat{\mathcal{B}}$. We first need to recall the definition of $\mathcal{B}$. In fact, we have an algebra $\mathcal{B}$ for each value of the prime $p$. In order to define these algebras, we need to discuss multi-indices and some associated definitions, both of which shall vary depending on whether $p = 2$ or $p > 2$. \\

First suppose that $p = 2$. In this case, a \textit{multi-index} is a sequence $I = (i_1,\dots,i_k)$ of integers $i_j \in \Z$, where $k \ge 0$ (if $k = 0$, we have the empty sequnce $()$). Given such a multi-index, we will associate to it the formal string $P^I = P^{i_1} \cdots P^{i_k}$. Given the multi-index $I = (i_1,\dots,i_k)$, we then set the following:
\begin{itemize}
	\item The \textit{length} $l(I)$ is $k$; if $I = ()$, this is to be interpreted as $0$.
	\item The \textit{degree} $d(I)$ is $i_1 + \cdots + i_k$; if $I = ()$, this is to be interpreted as $0$.
	\item The \textit{excess} $e(I)$ is $i_1 - i_2 - \cdots - i_k$; if $I = ()$, this is to be interpreted as $-\infty$.
	\item We say that $I$ is \textit{admissible} if $i_j \ge 2i_{j+1}$ for each $j$; the empty multi-index is admissible.
\end{itemize}

Now suppose that $p > 2$. In this case, a \textit{multi-index} is a sequence $I = (\varepsilon_1, i_1,\dots, \varepsilon_k, i_k)$ of integers $i_j \in \Z$ and $\varepsilon_j \in \{0,1\}$, where $k \ge 0$ (if $k = 0$, we have the empty sequnce $()$). Given such a multi-index, we associate to it the formal string $\beta^{\varepsilon_1} P^{i_1} \cdots \beta^{\varepsilon_k} P^{i_k}$, where $\beta^1$ here is to be intepreted as $\beta$ and $\beta^0$ as an empty symbol. Given the multi-index $I = (\varepsilon_1,i_1,\dots,\varepsilon_k,i_k)$, we then set the following:
\begin{itemize}
	\item The \textit{length} $l(I)$ is $k$; if $I = ()$, this is to be interpreted as $0$.
	\item The \textit{degree} $d(I)$ is $(2i_1(p-1)+\varepsilon_1) + \cdots + (2i_k(p-1)+\varepsilon_k)$; if $I = ()$, this is to be interpreted as $0$. 
	\item The \textit{excess} $e(I)$ is $(2i_1+\varepsilon_1) - (2i_2(p-1)+\varepsilon_2) - \cdots - (2i_k(p-1)+\varepsilon_k)$; if $I = ()$, this is to be interpreted as $-\infty$.
	\item We say that $I$ is \textit{admissible} if $i_j \ge pi_{j+1}+\varepsilon_{j+1}$ for each $j$; the empty multi-index is admissible.
\end{itemize}

\begin{Remark}\label{rmk:excess_S}
The excess of a multi-index $I$ is related to the instability of the cohomology operations on an algebra $A$ over $\mathpzc{E}^\dagger$, in that $P^I[a] = 0$ is zero whenever $e(I) > |a|$ (this is an easy consequence of the definition of excess and the instability property for a single operation $P^s$). Note also that the excess, in the case where $p = 2$, may also be written as $i_k - \sum_{j=2}^k(2i_j - i_{j-1})$, which gives a relation to the admissibility condition; moreover, it can also be written as $2i_1 - d(I)$, giving a relation to the degree. \customendremark
\end{Remark}

In order to construct $\mathcal{B}$, consider formal symbols $P^s$ and $\beta P^s$ for $s \in \Z$. We need to recall the Adem relations. If $p = 2$, the Adem relations consist of the relations
\[
P^rP^s = \sum_{i \in \Z} \binom{s-i-1}{r-2i}P^{r+s-i}P^i
\]
for all $r, s$ such that $r < 2s$. If $p > 2$, the Adem relations consist of the relations 
\[
P^rP^s=\sum_{i \in \Z}(-1)^{r+i}\binom{(p-1)(s-i)-1}{r-pi}P^{r+s-i}P^i
\]
\[
\beta P^rP^s = \sum_{i \in \Z} (-1)^{r+i}\binom{(p-1)(s-i)-1}{r-pi}\beta P^{r+s-i}P^i
\]
for all $r, s$ such that $r < ps$, and also the relations
\begin{multline*}
P^r \beta P^s = \sum_{i \in \Z} (-1)^{r+i}\binom{(p-1)(s-i)}{r-pi}\beta P^{r+s-i}P^i \\
+ \sum_{i \in \Z} (-1)^{r+i+1}\binom{(p-1)(s-i)-1}{r-pi-1} P^{r+s-i}\beta P^i
\end{multline*}
\[
\beta P^r \beta P^s = \sum_{i \in \Z} (-1)^{r+i}\binom{(p-1)(s-i)-1}{r-pi-1}\beta P^{r+s-i}\beta P^i
\]
for all $r, s$ such that $r \le ps$. In all cases, note that the Adem relations are homeogeneous.

\begin{Remark}
We can see that the admissibility of a multi-index is related to the Adem relations. For example, consider the case of $p = 2$, a multi-index $I = (i_1,\dots,i_k)$ and associated string $P^{i_1} \cdots P^{i_k}$. Certainly when $k = 2$, the term $P^{i_1}P^{i_2}$ admits an application of the Adem relations if and only if $i_1 < 2i_2$, which is to say if and only if $(i_1,i_2)$ is not admissible. More generally, the existence of the Cartan-Serre basis below (see Remark~\ref{rmk:facts_about_B}) implies that the relations apply to $P^{i_1} \cdots P^{i_k}$ if and only if $I$ is not admissible. \customendremark
\end{Remark}

Now we can define the algebra $\mathcal{B}$.

\begin{Definition}\label{def:alg_S_B}
The \textit{algebra of generalized Steenrod operations}\index{generalized Steenrod operations!algebra of} $\mathcal{B}$ is defined as follows. If $p = 2$, where $p$ is our fixed prime, we have
\[
\mathcal{B} := \mathbf{F}\{P^s \mid s \in \Z\}/I_{\text{Adem}}
\]
where $\mathbf{F}\{P^s \mid s \in \Z\}$ denotes the free graded algebra over $\F_2$ on the formal symbols $P^s$, for $s \in \Z$, where $P^s$ has degree $s$, and where $I_{\text{Adem}}$ denotes the two-sided ideal generated by the Adem relations. If $p > 2$, we have
\[
\mathcal{B} := \mathbf{F}\{P^s, \beta P^s \mid s \in \Z\}/I_{\text{Adem}}
\]
where $\mathbf{F}\{P^s, \beta P^s \mid s \in \Z\}$ denotes the free graded algebras over $\F_p$ on formal symbols $P^s, \beta P^s$, for $s \in \Z$, where $P^s, \beta P^s$ have degrees $2s(p-1), 2s(p-1)+1$ respectively, and where $I_{\text{Adem}}$ denotes the two-sided ideal generated by the Adem relations.
\end{Definition}

As the Adem relations are homogeneous, so is the ideal generated by them, and so $\mathcal{B}$ inherits the grading of the free algebra.

\begin{Remark}\label{rmk:facts_about_B}
We record some well-known facts about $\mathcal{B}$.
\begin{itemize}
	\item[(1)] There exists a well-known basis for $\mathcal{B}$, the \textit{Cartan-Serre basis}\index{Cartan-Serre basis}. This is an $\F_p$-basis and consists of the monomials $P^I$ where $I$ is admissible. See, for example,~\cite{CohenLadaMay}.
	\item[(2)] We will also have need to consider certain quotients of $\mathcal{B}$. If $p = 2$, for each $k \in \Z$, we set
\[
\mathcal{B}_{\le k} := \mathbf{F}\{P^s \mid s \in \Z\}/(I_{\text{Adem}} + I_{\text{exc} \, > \, k}) = \mathcal{B}/I_{> k}
\]
and if $p > 2$, for each $k \in \Z$, we set
\[
\mathcal{B}_{\le k} := \mathbf{F}\{P^s, \beta P^s \mid s \in \Z\}/(I_{\text{Adem}} + I_{\text{exc} \, > \, k}) = \mathcal{B}/I_{> k}
\]
where, in etiher case, $I_{\text{exc} \, > \, k}$ denotes the two-sided ideal generated by monomials of excess $> k$, and where $I_{> k} := (I_{\text{Adem}} + I_{\text{exc} \, > k})/I_{\text{Adem}}$. Sometimes, we will use the notations $I_{\text{exc} \, \ge \, k}$, $I_{\ge k}$ and $\mathcal{B}_{< k}$ in place of $I_{\text{exc} \, > \, k-1}$, $I_{> k-1}$ and $\mathcal{B}_{\le k-1}$. From the Cartan-Serre basis for $\mathcal{B}$, one can derive bases for these quotients; again, see~\cite{CohenLadaMay}. They are as follows: for each $k \in \Z$, the algebra $\mathcal{B}_{\le k}$ has an $\F_p$-basis given by the monomials $P^I$ where $I$ is admissible and $e(I) \le k$.
	\item[(3)] As is standard, we say that a $\mathcal{B}$-module $H$ is \textit{unstable}\index{unstable modules} if $P^Ih = 0$ is zero whenever $e(I) > |h|$. In the case of an algebra $A$ over the unstable Barratt-Eccles operad $\mathpzc{E}^\dagger$, we can sum up the cohomology operations and their instability property by noting that $\text{H}^\bullet(A)$ is then an unstable module over $\mathcal{B}$. For details, we refer the reader to~\cite{May}.
\end{itemize} \customendremark
\end{Remark}

We have recalled the definition of, and some facts about, the algebra $\mathcal{B}$. \index{generalized Steenrod operations!completion of the algebra of}We now move onto the construction of the completion $\widehat{\mathcal{B}}$ and shall begin by constructing the underlying graded module of $\widehat{\mathcal{B}}$. To define this graded module, consider functions
\[
f \colon \{\text{admissible multi-indices}\} \to \F_p.
\]
We have an addition and a scalar multiplication for such functions, computed pointwise. We think of such a function as a potentially infinite sum, and so use the suggestive notation
\[
\sum_{I \: \text{admissible}} a_IP^{I}
\]
where $a_I = f(I)$. Our graded module will consist of such sums, with particular finiteness properties in relation to the length and excess of multi-indices. Specifically, the underlying graded module of $\widehat{\mathcal{B}}$ is defined by setting that, in degree $d \in \Z$, the graded piece $\widehat{\mathcal{B}}_d$ is to consist of the potentially infinite sums above which satisfy the following requirements.
\begin{itemize}
	\item For all $I$, if $a_I \neq 0$, $d(I) = d$.
	\item The set of lengths $\{l(I) \mid a_I \neq 0\}$ is bounded above, or, equivalently, finite.
	\item For any $k \in \Z$, $\#\{I \mid a_I \neq 0, e(I) \le k\}$ is finite.
\end{itemize}

\begin{Remark}
Since, given any non-empty multi-index $I$ of degree $d$, we have $e(I) = 2i_1 - d(I) = 2i_1 - d$, where $I = (i_1,\dots,i_k)$, in the $p=2$ case and $e(I) = 2pi_1 + 2\varepsilon_1 - d(I) = 2pi_1 + 2\varepsilon_1 - d \le 2pi_1 + 2 - d$, where $I = (\varepsilon_1, i_1,\dots, \varepsilon_k, i_k)$, in the $p > 2$ case, we can rephrase the third condition as saying that, given any $k \in \Z$, there may exist at most finitely many $I$ with $a_I \neq 0$ which are such that the entry $i_1$ is not larger than $k$. We may then also, imprecisely though suggestively, package the condition as ``$i_1 \to +\infty$''. \customendremark
\end{Remark}

Note that we have an obvious embedding of graded modules
\[
\mathcal{B} \hookrightarrow \widehat{\mathcal{B}}.
\]
An example of an element, one in degree $0$, which is present in the completion $\widehat{\mathcal{B}}$ but not in $\mathcal{B}$, is the following infinite sum
\[
\sum_{k \ge 0} P^{k}P^{-k}.
\]
In fact, as the following proposition demonstrates, all elements of $\widehat{\mathcal{B}}$ which are not in $\mathcal{B}$ share the salient features of this example: that the initial entries of the multi-indices tend to $+\infty$ while the final entries tend to $-\infty$.

\begin{Proposition}\label{prop:i1posinf_ikneginf}
Let $\sum a_IP^I$ be an element of $\widehat{\mathcal{B}}$. We have the following:
\begin{itemize}
	\item[(i)] Given any $k \in \Z$, for all but finitely many $I$, the initial entry is greater than $k$.
	\item[(ii)] Given any $k \in \Z$, for all but finitely many $I$, the final entry is less than by $k$.
\end{itemize}
\end{Proposition}

Here, in the case $p > 2$, where multi-indices take the form $(\varepsilon_1,i_1,\dots,\varepsilon_r,i_r)$, where the $i_j$ lie in $\Z$ while the $\varepsilon_j$ lie in $\{0,1\}$, the first entry is taken to be $i_1$, and the final entry, $i_r$, which is to say we disregard the $\varepsilon_j$ for this particular purpose.

\begin{proof}
(i): Let the given element lie in degree $d$. When $p =2$, the result follows by the fact that $e(I) = 2i_1 - d$ for any $I = (i_1,\dots,i_r)$ of degree $d$ and that, in the sum, there can only be finitely many elements of excess below a given bound. When $p > 2$, the result follows in a similar fashion, using instead the identity $e(I) = 2pi_1 + 2\varepsilon_1 - d$ for any $I = (\varepsilon_1,i_1,\dots,\varepsilon_r,i_r)$ of degree $d$.  \\

(ii): Let the given element lie in degree $d$. Of course there can be only one length one monomial which occurs in the sum. Consider then monomials of length $r \ge 2$. First consider the case where $p = 2$. Given a multi-index $I = (i_1,\dots,i_r)$, by admissibility, we have $i_1 \ge 2i_2 \ge 2^2i_3 \ge \cdots \ge 2^{r-1}i_r$. Put another way, we have that $i_j \ge 2^{r-j}i_r$ for $j = 1,\dots,r$. Thus, we have
\begin{align*}
d &= i_1 + (i_2 + \cdots + i_r) \\
&\ge i_1 + (2^{r-2}i_r + 2^{r-3}i_r + \cdots + i_r) \\
&= i_1 + Ci_r
\end{align*}
where $C = 1 + 2 + \cdots + 2^{r-2} > 0$. Thus we have $i_r \le \frac{1}{C}(d -i_1)$ and so the result follows by part (i). Now consider the case $p > 2$. In this case, given a multi-index $(\varepsilon_1,i_1,\dots,\varepsilon_r,i_r)$, admissibility gives us that $i_j \ge pi_{j+1} + \varepsilon_{j+1}$ for each $j = 1,\dots,r-1$, and so, in particular, $i_j \ge pi_{j+1}$ for each $j = 1, \dots, r-1$. It follows that $i_j \ge p^{r-j}i_r$ for $j = 1,\dots,r$. Moreover, we have that $d(I) = 2(p-1)(i_1+ \cdots + i_k) + \varepsilon_1 + \cdots + \varepsilon_r \le 2(p-1)(i_1+ \cdots + i_k) + r$. The argument now is analogous to the one above for the $p = 2$ case.
\end{proof}

We now wish to endow our graded module $\widehat{\mathcal{B}}$ with an algebra structure. To do so, however, we first need some technical lemmas regarding the Adem relations. Given any multi-index $I$, via the Cartan-Serre basis which we mentioned above (see Remark~\ref{rmk:facts_about_B}), we know that $P^I$ can be written uniquely as a sum
\[
\sum a_KP^K
\]
where each $K$ is admissible. Let us call this the \textit{admissible monomials expansion} of $P^I$. Note that since the Adem relations either annihilate a monomial or preserve its length, any $K$ for which $a_K$ is non-zero must have the same length as $I$.

\begin{Lemma}\label{lem:first_entry_down_last_up}
Let $I$ be a non-empty multi-index. If $K$ is a multi-index which appears in the admissible monomials expansion of $P^I$, then the following hold.
\[
(\emph{initial entry of} \: K) \ge (\emph{initial entry of} \: I)
\quad
(\emph{final entry of} \: K) \le (\emph{final entry of} \: I)
\]
\end{Lemma}

As before, here we follow our convention that, in the case $p > 2$, where multi-indices take the form $(\varepsilon_1,i_1,\dots,\varepsilon_r,i_r)$, where the $i_j$ lie in $\Z$ while the $\varepsilon_j$ lie in $\{0,1\}$, the first entry is taken to be $i_1$, and the final entry, $i_r$, which is to say we disregard the $\varepsilon_j$ for this particular purpose.

\begin{proof}
We shall give a proof of the case where $p =2$; the case where $p > 2$ follows by a similar proof, upon appropriate modifications. Let $I = (i_1,\dots,i_n)$. In the case $n=1$, we already have admissibility and so the result is trivial. Consider the case where $n = 2$. Let $I = (r,s)$. If $r \ge 2s$, $P^rP^s$ is admissible and so the result is trivial. Suppose then that $r < 2s$. Then, by the Adem relations, we have that the admissible monomials expansion is:
\[
\sum_i \binom{s-i-1}{r-2i} P^{r+s-i}P^i
\]
(An easy check of when the binomial coefficient is non-zero shows that the terms which appear on the right-hand side are indeed admissible.) The terms on the right-hand side which appear are those with index $i$ satisfying $r-s+1 \le i \le r/2$. Thus the minimum first entry, say $k_{\text{min init}}$, of the multi-indices $(r+s-i,i)$ which occur satisfies $k_{\text{min init}} \ge r + s - r/2 = r/2 + s > r/2 + r/2 = r$, giving us the desired result. On the other hand, the maximum second entry, say $k_{\text{max final}}$, of the multi-indices $(a+b-i,i)$ which occur satisfies $k_{\text{max final}} \le r/2 < s$, once again giving us the desired result. \\

Now let us consider the case $n \ge 3$. We have that there exists a finite sequence of terms, say $T_1, \dots, T_m$, $m \ge 1$, in the free algebra over $\F_2$ on the $P^i$, $i \in \Z$, which is such that $T_1 = P^{i_1} \cdots P^{i_n}$, $T_m = \sum P^K$ is the admissible monomials expansion of $P^{i_1} \cdots P^{i_n}$, and, for each $j \ge 2$, $T_j$ is constructed from $T_{j-1}$ by taking some monomial summand $P^J$ and replacing a sub-monomial $P^rP^s$ of $P^J$ with the equivalent $\sum_i \binom{s-i-1}{r-2i} P^{r+s-i}P^i$ provided by the Adem relations. Now, if the move which is made in transitioning from $T_{j-1}$ to $T_j$ is applied to a sub-monomial $P^rP^s$ where $P^r$ is not the initial entry of the corresponding monomial $P^J$, there is no change made to any initial entry in any monomial summand. If, on the other hand, the move is applied to a sub-monomial $P^rP^s$ where $P^r$ is the initial entry of $P^J$, by the argument in the $n=2$ case above, the minimum of all the initial entries in the resulting multi-indices in $T_j$ is bounded below by the original such minimum in $T_{j-1}$. Thus, by a simple induction, we have that the minimum of the initial entries of all the multi-indices appearing in $\sum P^K$ is indeed bounded below by $i_1$. By an entirely analogous argument, considering instead the cases where $P^b$ is, or is not, the final entry of $P^J$, we have that the maximum of the final entries of all the multi-indices appearing in $\sum P^K$ is indeed bounded above by $i_k$.
\end{proof}

\begin{Lemma}\label{lem:excess_down}
Let $I$ be a multi-index. If $K$ is a multi-index which appears in the admissible monomials expansion of $P^I$, then $e(K) \ge e(I)$.
\end{Lemma}

\begin{proof}
The case of an empty $I$ is trivial, so suppose that it is non-empty. Suppose that $p = 2$. Let $I = (i_1, \dots, i_n)$ and $K = (k_1,\dots,k_n)$, where $n \ge 1$. We can write $e(I) = 2i_1 - d(I)$ and $e(K) = 2k_1 - d(K)$. The Adem relations preserve degree, so that $d(I) = d(K)$. The result then follows by Lemma~\ref{lem:first_entry_down_last_up}. Now suppose that $p > 2$. Let $I = (\varepsilon_1, i_1, \dots, \varepsilon_n, i_n)$ and $K = (\varepsilon'_1, k_1,\dots, \varepsilon'_n, k_n)$, where $n \ge 1$. We can write $e(I) = 2pi_1 + 2\varepsilon_1 - d(I)$ and $e(K) = 2pk_1 + 2\varepsilon'_1 - d(K)$. The Adem relations preserve degree, so that $d(I) = d(K)$. Moreover, an examination of the Adem relations shows that, if $\varepsilon_1 = 1$, then $\varepsilon'_1 = 1$, so that $\varepsilon_1 \le \varepsilon'_1$. The result now follows by Lemma~\ref{lem:first_entry_down_last_up}. 
\end{proof}

\begin{Lemma}\label{lem:QIQJprop}
Let $I$ and $J$ be admissible multi-indices. If $K$ is a multi-index which appears in the admissible monomials expansion of $P^IP^J$, then $e(K) \ge e(J)$.
\end{Lemma}

\begin{proof}
We shall give a proof of the case where $p =2$; the case where $p > 2$ follows by a similar proof, upon appropriate modifications. The proof will be via three inductions. \\

Consider the case when $I$ has length $1$. Let $I = (r)$. We will prove this case by induction on the length of $J$. If $J$ has length $0$, it is empty, the monomial in question is $P^r$, which is already admissible, and $e(J) = -\infty$, so that we have the desired result. Now suppose $J$ has length $1$. Let $J = (s)$. If $r \ge 2s$, the monomial in question, $P^rP^s$, is already admissible, and the excess is $r-s$, which is bounded below by $e(J) = s$ since $r \ge 2s$. On the other hand, if $r < 2s$, the admissible monomials expansion is given by
\[
\sum_i \binom{s-i-1}{r-2i}P^{r+s-i}P^i
\]
where $r-s+1 \le i \le r/2$. The excess of a generic term on the right-hand side is given by $r+s-2i$ and this is bounded below by $r+s-2(r/2) = s = e(J)$, giving us the desired result. Now suppose that we have the desired result for $J$ of length $< n$, where $n \ge 2$. Consider $P^rP^{j_1} \cdots P^{j_n} = (P^rP^{j_1} \cdots P^{j_{n-1}})P^{j_n}$. Let $\sum P^K$ be the admissible monomials expansion of $P^aP^{j_1} \cdots P^{j_{n-1}}$. By the induction hypothesis, for each $K$, we have $e(K) \ge e(J) + j_n$. We now have $P^aP^{j_1} \cdots P^{j_n} = \sum P^KP^{j_n}$. By Proposition~\ref{lem:excess_down}, for a given $K$, any multi-index which appears in the admissible monomials expansion of the term $P^KP^{j_n}$ has excess bounded below by $e(K,j_n) = e(K) - j_n \ge e(J) + j_n - j_n = e(J)$, giving us the desired result. We have thus established, by induction on the length of $J$, the case in which $I$ has length $1$. \\

Now consider the case where $J$ has length $1$. Let $J = (s)$. We will prove this case by induction on the length of $I$. If $I$ has length zero, it is empty, the monomial in question is $P^s$, which is already admissible and so the desired result is trivial. Suppose that $I$ has length $1$. Let $I = (r)$. The monomial in question is then $P^rP^s$ and the desired result follows by exactly the same argument as the one above which was already made for this monomial. Now suppose that we have the desired result for $I$ of length $< n$, where $n \ge 2$. Consider $P^{i_1} \cdots P^{i_n}P^s = P^{i_1}(P^{i_2} \cdots P^{i_{n-1}}P^{s})$. Let $\sum P^K$ be the admissible monomials expansion of $P^{i_2} \cdots Q^{i_{n-1}}P^{s}$. By the induction hypothesis, for each $K$, we have $e(K) \ge s$. We now have $P^{i_1} \cdots P^{i_n}P^s = \sum P^{i_1}P^K$. By the result established by the previous induction, that of the case in which $I$ has length $1$, we have that, for a given $K$, any multi-index which appears in the admissible monomials expansion of the term $P^{i_1}P^K$ has excess bounded below by $e(K)$, which as we saw already is itself below below by $s$, giving us the desired result. We have thus established, by induction on the length of $I$, the case in which $J$ has length $1$. \\

We will now prove the general statement in the proposition by induction on the length of $J$. First suppose that $J$ has length $0$. Then the monomial in question is $P^I$, the admissible monomials expansion is also simply $P^I$ (as $I$ is assumed to be admissible), and we have $e(I) \ge e(J)$ for any $I$ since $e(J) = -\infty$. If $J$ has length $1$, we have the desired result by the second of the two previous inductions. Now suppose that we have the desired result for $J$ of length $< n$, where $n \ge 2$. Consider $P^IP^{j_1} \cdots P^{j_n} = (P^IP^{j_1} \cdots P^{j_{n-1}})P^{j_n}$. Let $\sum P^K$ be the admissible monomials expansion of $P^IP^{j_1} \cdots P^{j_{n-1}}$. By the induction hypothesis, for each $K$, we have $e(K) \le e(J) + j_n$. We now have $P^IP^{j_1} \cdots P^{j_n} = \sum P^KP^{j_n}$. By Lemma~\ref{lem:excess_down}, for a given $K$, any multi-index which appears in the admissible monomials expansion of the term $P^KP^{j_n}$ has excess bounded below by $e(K,j_n) = e(K) - j_n \ge e(J) + j_n - j_n = e(J)$, giving us the desired result. We have thus established, by induction on the length of $J$, the completely general case.
\end{proof}

We are now ready to equip our graded module $\widehat{\mathcal{B}}$ with an algebra structure. For each $d_1, d_2 \in \Z$, we must construct maps
\[
\widehat{\mathcal{B}}_{d_1} \otimes \widehat{\mathcal{B}}_{d_2} \to \widehat{\mathcal{B}}_{d_1+d_2}.
\]
Consider two infinite sums, the product of which
\[
\left( \sum_I a_IP^I \right) \cdot \left( \sum_I b_IP^I \right)
\]
we wish to construct, where we suppose that the only $a_I$ and $b_I$ which are non-zero are those for which the degree is $d_1$, $d_2$ respectively. Given any two admissible $I$ and $J$, let
\[
P^IP^J = \sum_{K \: \text{admissible}} c^{I,J}_K P^K
\]
be the admissible monomials expansion of $P^IP^J$; note that, for any fixed $I$ and $J$, only finitely many of the $c^{I,J}_K$ may be non-zero. We then set
\begin{equation}\label{eqn:prodShat}
\left( \sum_I a_IP^I \right) \cdot \left( \sum_I b_IP^I \right) := \sum_K \left(\sum_{I,J} a_Ib_Jc^{I,J}_K\right) P^K.
\end{equation}

\begin{Proposition}\label{prop:prodwelldefScomplete}
The product on $\widehat{\mathcal{B}}$ as above is well-defined and equips $\widehat{\mathcal{B}}$ with an algebra structure over $\F_p$.
\end{Proposition}

\begin{proof}
We first show that the righthand side of (\ref{eqn:prodShat}) is well-defined as an infinite sum. To do this, having fixed the degrees $d_1$ and $d_2$ as above, we need to ensure that the sum $\sum_{I,J} a_Ib_Jc^{I,J}_K$ is finite for any given $K$. Fix such a $K$, say $K_0$. By the definition of $\widehat{\mathcal{B}}$ as a graded module, we know that, for all but finitely many $I$, we have that $a_I = 0$ or $e(I) > e(K_0) + d_2$. Note that any $I$ satisfying $e(I) > e(K_0) + d_2$ is necessarily non-empty. Now, for such an $I$, where $e(I) > e(K) + d_2$, we have that, for any $J$ for which $b_J \neq 0$, $e(IJ)  = e(I) - d_2 > e(K_0) + d_2 - d_2 = e(K_0)$, where $IJ$ denotes the concatenation of $I$ and $J$. Thus, by Lemma~\ref{lem:excess_down}, there are only finitely many $I$ for which $a_I \neq 0$ and for which there exists a $J$ such that $c^{I,J}_{K_0} \neq 0$ (note that $c^{I,J}_{K_0} \neq 0$ amounts to saying that $P^{K_0}$ appears in the admissible monomials expansion of $P^IP^J$). Fix such an $I$, say $I_0$. We know that, for all but finitely many $J$, we have that $b_J = 0$ or $e(J) > e(K_0)$. Thus, by Lemma~\ref{lem:QIQJprop}, there are only finitely many $J$ for which $b_J \neq 0$ and $c^{I_0,J}_{K_0} \neq 0$, where the latter amounts to saying that $P^{K_0}$ appears in the admissible monomials expansion of $P^{I_0}P^J$. All told, we have demonstrated that, for any given $K_0$, there are only finitely many terms in both the infinite sums
\[
\sum_I a_IP^I \qquad \text{and} \qquad \sum_I b_IP^I
\]
which make a non-zero contribution to the coefficient
\[
\sum_{I,J} a_Ib_Jc^{I,J}_{K_0}
\]
of $P^{K_0}$. Thus the product is indeed well-defined, at least as an infinite sum. In fact, it is an element of $\widehat{\mathcal{B}}_{d_1+d_2}$ for the following reasons: (i) the degree condition is satisfied because the Adem relations preserve degree (ii) the length condition is satisfied because the Adem relations either annihilate an element or preserve its length (iii) the excess condition is satisfied by the same argument as above, which showed not only well-definedness as an infinite sum, but more strongly that all but finitely many pairings of the $P^I$ and $P^J$ yield monomials with an associated excess above any given bound. Finally, because any given coefficient arises from a product of finite sums, it is clear that the requisite associativity, identity and bilinearity follow from the fact that the definition yields the product of $\mathcal{B}$ when restricted to finite sums and that these properties do indeed hold for the product of $\mathcal{B}$.
\end{proof}

We have now constructed $\widehat{\mathcal{B}}$ as a graded algebra over $\F_p$. Moreover, the embedding 
\[
\mathcal{B} \hookrightarrow \widehat{\mathcal{B}}
\]
is now clearly one of algebras. It remains to make precise in what sense $\widehat{\mathcal{B}}$ is a completion of $\mathcal{B}$. For $k \ge 0$, recall the quotients $\mathcal{B}_{\le k}$ of $\mathcal{B}$ which we defined earlier in Remark~\ref{rmk:facts_about_B}. Filter each of these algebras by length so that $\text{F}_t\mathcal{B}_{\le k}$ consists of those monomials $P^I$ satisfying the length bound $p^{l(I)} \le t$. Note that, for any $k \le l$, we have a canonical map
\[
\mathcal{B}_{\le l} \to \mathcal{B}_{\le k}
\]
and that this map is a filtered map, in that, for each $t \ge 0$, we have induced maps
\[
\text{F}_t\mathcal{B}_{\le l} \to \text{F}_t\mathcal{B}_{\le k}.
\]
Next, note that we can also filter $\widehat{\mathcal{B}}$ similarly by length by setting, for $t \ge 0$, $\text{F}_t\widehat{\mathcal{B}}$ to comprise the sums $\sum_{I \: \text{admissible}} a_IP^{I}$ which satisfy the degree, length and excess requirements in the definition of $\widehat{\mathcal{B}}$ and which also satisfy, more specifically regarding length, the bound $p^{l(I)} \le t$ for any $I$ where $a_I \neq 0$. For each $k \ge 0$, we have a map
\[
\widehat{\mathcal{B}} \to \mathcal{B}_{\le k}
\]
which projects an infinite sum to the sub-sum of the elements of excess $\le k$, of which there are finitely many by definition of $\widehat{\mathcal{B}}$. Moreover, this map is a filtered map in that we have, for each $t \ge 0$, an induced map
\[
\text{F}_t\widehat{\mathcal{B}} \to \text{F}_t\mathcal{B}_{\le k}.
\]
These maps are compatible with the maps $\text{F}_t\mathcal{B}_{\le l} \to \text{F}_t\mathcal{B}_{\le k}$ above in that they yield a left cone on the tower
\[
\cdots \to \text{F}_t\mathcal{B}_{\le 2} \to \text{F}_t\mathcal{B}_{\le 1} \to \text{F}_t\mathcal{B}_{\le 0}
\]
and so yield a map $\text{F}_t\widehat{\mathcal{B}} \to \lim_{k \ge 0}\text{F}_t\mathcal{B}_{\le k}$. The precise statement then regarding in what sense $\widehat{\mathcal{B}}$ is a completion of $\mathcal{B}$ is the following. 

\begin{Proposition}\label{prop:FtBhatcomplete}
The map $\emph{F}_t\widehat{\mathcal{B}} \to \lim_{k \ge 0}\emph{F}_t\mathcal{B}_{\le k}$ constructed above is an isomorphism of graded modules. Moreover, as graded modules, we have that
\[
\widehat{\mathcal{B}} \cong \underset{t \ge 0}{\emph{colim}} \: \underset{k \ge 0}{\emph{lim}} \: \emph{F}_t\mathcal{B}_{\le k}.
\]
\end{Proposition}

\begin{proof}
We noted above that $\mathcal{B}_{\le k}$ has a basis given by the admissible monomials of excess strictly smaller than $k$. Moreover, each $\text{F}_t\mathcal{B}_{\le k}$ then has a basis given by the admissible monomials $P^I$ satisfying both $e(I) \le k$ and the length bound $p^{l(I)} \le t$. Moreover, the map $\text{F}_t\mathcal{B}_{\le k+1} \to \text{F}_t\mathcal{B}_{\le k}$ simply kills those basis elements with excess exactly $k+1$. The first part of the result now follows by an easy verification of the necessary universal property. The second part then follows by noting that, under this established isomorphism, given $t \le t'$, the map induced on the limits by the natural maps $\text{F}_t\mathcal{B}_{\le k} \to \text{F}_{t'}\mathcal{B}_{\le k}$ corresponds exactly to the inclusion $\text{F}_t\widehat{\mathcal{B}} \to \text{F}_{t'}\widehat{\mathcal{B}}$.
\end{proof}

\section{Cohomologies of Free Algebras over the Stable Operads}\label{subsec:cohomologies_of_free_algebras_over_the_stable_operads}

We mentioned earlier that we wish to compute the cohomology of a free algebra $\textbf{E}_{\textbf{st}}^\dagger X$ over the stable Barratt-Eccles operad $\mathpzc{E}_{\text{st}}^\dagger$. Having constructed the completion $\widehat{\mathcal{B}}$, we are now able to do this. We shall show that, if $X$ is a cochain complex, the cohomology of the free $\mathpzc{E}_{\text{st}}^\dagger$-algebra on $X$ is precisely the free $\widehat{\mathcal{B}}$-module on $\text{H}^\bullet(X)$, namely $\widehat{\mathcal{B}} \otimes \text{H}^\bullet(X)$. In order to compare $\text{H}^\bullet(\mathbf{E}_{\textbf{st}}^\dagger X)$ and $\widehat{\mathcal{B}} \otimes \text{H}^\bullet(X)$, we shall introduce an intermediating construction. Thus, we define a functor $\text{A}$ on cochain complexes by setting
\begin{equation}\label{eqn:def_of_A}
\text{A}(X) := \underset{t \ge 0}{\text{colim}} \lim_{k \ge 0} \text{F}_t\text{H}^\bullet((\Sigma^k\mathbf{E}^\dagger)X).
\end{equation}
Here the maps in the limiting tower come from the stabilization maps $\Sigma^{k+1}\mathpzc{E}^\dagger \to \Sigma^k\mathpzc{E}^\dagger$ and the filtration $\{\text{F}_t\text{H}^\bullet((\Sigma^k\mathbf{E}^\dagger)X)\}_{t \ge 0}$ of the cohomology $\text{H}^\bullet((\Sigma^k\mathbf{E}^\dagger)X)$ is as follows:
\[
\text{F}_t\text{H}^\bullet((\Sigma^k\mathbf{E}^\dagger)X) = \bigoplus_{n \le t} \text{H}^\bullet((\Sigma^k\mathpzc{E}^\dagger)(n) \otimes_{\Sigma_n} X^{\otimes n}).
\]
We study this functor $\text{A}(-)$ by studying the steps in its construction, one at a time; recall that here $\Sigma^k\mathbf{E}^\dagger$ denotes the monad associated to the operadic suspension $\Sigma^k\mathpzc{E}^\dagger$, as opposed to a monadic suspension. Let $\Phi_{\mathrm{oper}}$ be the functor which takes a graded $\mathbb{F}_p$-module $M$ to the free unstable $\mathcal{B}$-module on $M$; thus $\Phi_{\mathrm{oper}}(M) = \bigoplus_{d \in \Z} \mathcal{B}_{\le d} \otimes M_d$, where $\mathcal{B}_{\le d}$ is defined as in Remark~\ref{rmk:facts_about_B}. Let also $\Phi_{\text{prod}}$ be the functor which takes a $\mathcal{B}$-module $N$ to the free graded-commutative algebra over $\mathbb{F}_p$ on $N$ (this is in fact an algebra over $\mathcal{B}$, where the $\mathcal{B}$-action is given by the internal Cartan formula; see Section 2 of Chapter I of~\cite{CohenLadaMay}). Finally, set $\Phi := \Phi_{\text{prod}}\Phi_{\text{oper}}$. Given an $\F_p$-module $M$, there is a filtration $\{\mathrm{F}_t\Phi(M)\}_{t \ge 0}$ of $\Phi(M)$ where $\mathrm{F}_t\Phi(M)$ is the $\F_p$-submodule of $\Phi(M)$ generated by the products $(P^{I_1} \otimes m_1) \cdots (P^{I_r} \otimes m_r)$ where $p^{l(I_1)} + \cdots + p^{l(I_r)} \le t$.

\begin{Proposition}\label{prop:cohomology_of_free_algebras_of_suspended_E}
Given a cochain complex $X$, we have that
\[
\mathrm{H}^\bullet((\Sigma^k\mathbf{E}^\dagger) X) \cong \Phi((\mathrm{H}^\bullet X)[k])[-k].
\]
In fact, for each $t \ge 0$, we have that
\[
\mathrm{F}_t\mathrm{H}^\bullet((\Sigma^k\mathbf{E}^\dagger) X) \cong \mathrm{F}_t\Phi((\mathrm{H}^\bullet X)[k])[-k].
\]
\end{Proposition}

\begin{proof}
For the case $k = 0$, see Theorem 4.1, and its proof, in Chapter I of~\cite{CohenLadaMay}, or Theorem 2.1 in Chapter IX of~\cite{BrunerMayMcClureSteinberger}. The general case then follows from Proposition~\ref{prop:freealgsuspop}.
\end{proof}

Given a cochain complex $X$, for brevity, for each $k \ge 0$, set
\[
\mathcal{F}_k := \Phi((\mathrm{H}^\bullet X)[k])[-k].
\]

\begin{Remark}\label{rmk:basis_for_Phi}
If $\{c_i\}$ is a basis of $\text{H}^\bullet(X)$, then $\mathcal{F}_k$ is the free graded-commutative algebra over $\F_p$ on the terms $P^Ic_i$ where $I$ is admissible and $e(I) \le |c_i| + k$. \customendremark
\end{Remark}

We have now identified the terms in the tower that appear in the definition of $\mathrm{A}(-)$. The next step is to identify the limit. Via the maps $\Sigma^{k+1}\mathpzc{E}^\dagger \to \Sigma^k\mathpzc{E}^\dagger$ and the isomorphisms in Proposition~\ref{prop:cohomology_of_free_algebras_of_suspended_E} above, we have commutative diagrams as follows.

\begin{center}
\begin{tikzpicture}[node distance = 1.5cm]
\node [] (A) {$\mathcal{F}_1$};
\node [right of = A,xshift=1.5cm] (B) {$\mathcal{F}_0$};
\node [below of = A] (C) {$\text{H}^\bullet((\Sigma\mathbf{E}^\dagger) X)$};
\node [below of = B] (D) {$\text{H}^\bullet(\mathbf{E}^\dagger X)$};
\node [left of = A,xshift=-1.5cm] (E) {$\mathcal{F}_2$};
\node [left of = E,xshift=-1.5cm] (F) {$\cdots$};
\node [left of = C,xshift=-1.5cm] (G) {$\text{H}^\bullet((\Sigma^2 \mathbf{E}^\dagger) X)$};
\node [left of = G,xshift=-1.5cm] (H) {$\cdots$};

\draw [->] (A) -- (B) node[midway,anchor=south]{};
\draw [->] (A) -- (C) node[midway,anchor=west]{$\cong$};
\draw [->] (B) -- (D) node[midway,anchor=west]{$\cong$};
\draw [->] (C) -- (D) node[midway,anchor=south]{};
\draw [->] (F) -- (E) node[midway,anchor=south]{};
\draw [->] (E) -- (A) node[midway,anchor=south]{};
\draw [->] (H) -- (G) node[midway,anchor=south]{};
\draw [->] (G) -- (C) node[midway,anchor=south]{};
\draw [->] (E) -- (G) node[midway,anchor=west]{$\cong$};
\end{tikzpicture}
\end{center}

\begin{center}
\begin{tikzpicture}[node distance = 1.5cm]
\node [] (A) {$\text{F}_t\mathcal{F}_1$};
\node [right of = A,xshift=1.5cm] (B) {$\text{F}_t\mathcal{F}_0$};
\node [below of = A] (C) {$\text{F}_t\text{H}^\bullet((\Sigma\mathbf{E}^\dagger) X)$};
\node [below of = B] (D) {$\text{F}_t\text{H}^\bullet(\mathbf{E}^\dagger X)$};
\node [left of = A,xshift=-1.5cm] (E) {$\text{F}_t\mathcal{F}_2$};
\node [left of = E,xshift=-1.5cm] (F) {$\cdots$};
\node [left of = C,xshift=-1.5cm] (G) {$\text{F}_t\text{H}^\bullet((\Sigma^2 \mathbf{E}^\dagger) X)$};
\node [left of = G,xshift=-1.5cm] (H) {$\cdots$};

\draw [->] (A) -- (B) node[midway,anchor=south]{};
\draw [->] (A) -- (C) node[midway,anchor=west]{$\cong$};
\draw [->] (B) -- (D) node[midway,anchor=west]{$\cong$};
\draw [->] (C) -- (D) node[midway,anchor=south]{};
\draw [->] (F) -- (E) node[midway,anchor=south]{};
\draw [->] (E) -- (A) node[midway,anchor=south]{};
\draw [->] (H) -- (G) node[midway,anchor=south]{};
\draw [->] (G) -- (C) node[midway,anchor=south]{};
\draw [->] (E) -- (G) node[midway,anchor=west]{$\cong$};
\end{tikzpicture}
\end{center}

\begin{Proposition}\label{prop:instab_prods}
Let $X$ be a cochain complex. For each $k \ge 0$, the map $\mathcal{F}_{k+1} \to \mathcal{F}_k$ kills any product, including the empty product, which is to say the multiplicative unit.
\end{Proposition}

One can sum up this proposition as ``the instability of products''.

\begin{proof}
Let us first deal with non-empty products. Recall, from Definition~\ref{def:e_d^un_and_e_d^st}, our notation $e_d^{\text{un}}$, for $d \le 0$, for particular elements of $\mathpzc{E}^\dagger(2)$, where $e_d^{\text{un}}$ has degree $d$. Let $e_d^{\text{un}}[k]$ denote the corresponding element of $(\Sigma^k\mathpzc{E}^\dagger)(2) = \mathpzc{E}^\dagger(2)[k]$. Note that $e_d^{\text{un}}[k]$ has degree $d+k$. In particular $e_0^{\text{un}}[k]$ has degree $k$, and, since $(\Sigma^k\mathpzc{E}^\dagger)(2)$ is zero above degree $k$, the stabilization map $\Sigma^{k+1}\mathpzc{E}^\dagger \to \Sigma^k\mathpzc{E}^\dagger$ must map $e_0^{\text{un}}[k+1]$ to $0$.  Now, fix $k \ge 0$ and consider some product $(P^I c) \cdot (P^{I'} c')$ in $\mathcal{F}_{k+1}$. Let $[\eta]$ and $[\xi]$, respectively, be the images of $P^I c$ and $P^{I'} c'$ under the isomorphism $\mathcal{F}_{k+1} \to \text{H}^\bullet((\Sigma^{k+1}\mathbf{E}^\dagger)X)$. Then, by definition of the products in the cohomologies (see~\cite{May}), $(P^I c) \cdot (P^{I'} c')$ maps, under this same isomorphism, to $[e_0^{\text{un}}[k+1]_*(\eta, \xi)]$, by which we mean the class of the image of $e_0^{\text{un}}[k+1] \otimes \eta \otimes \xi$ under the map $(\Sigma^{k+1}\mathpzc{E}^\dagger)(2) \otimes ((\Sigma^{k+1}\mathbf{E}^\dagger)X)^{\otimes 2} \to (\Sigma^{k+1}\mathbf{E}^\dagger)X$. Since, as noted above, $e_0^{\text{un}}[k+1]$ maps to $0$ under $\Sigma^{k+1}\mathpzc{E}^\dagger \to \Sigma^k\mathpzc{E}^\dagger$, we have that the product $(P^I c) \cdot (P^{I'} c')$ maps to zero under the composite $\mathcal{F}_{k+1} \to \text{H}^\bullet((\Sigma^{k+1}\mathbf{E}^\dagger)X) \to \text{H}^\bullet((\Sigma^{k}\mathbf{E}^\dagger)X)$, and this proves the required result for the map $\mathcal{F}_{k+1} \to \mathcal{F}_k$, except in the case of the empty product. For the empty product, the argument is similar: one notes that the multiplicative unit in $\mathcal{F}_{k+1}$ corresponds to the generator in degree $-k-1$ of $(\Sigma^{k+1}\mathpzc{E}^\dagger)(0)$, and this generator is mapped to zero in $(\Sigma^k\mathpzc{E}^\dagger)(0)$ since the latter is zero in all degrees except degree $-k$.
\end{proof}

The result above motivates us to focus on the submodules $\Phi_{\mathrm{oper}}(M)$ inside $\Phi(M)$; these submodules omit all products in $\Phi(M)$, including the empty product. Note that, for any $\F_p$-module $M$, we have a filtration $\{\mathrm{F}_t\Phi_{\mathrm{oper}}(M)\}_{t \ge 0}$ where $\mathrm{F}_t\Phi_{\mathrm{oper}}(M)$ is the submodule generated by the terms $P^I \otimes m$ which satisfy $p^{l(I)} \le t$. Moreover, the inclusion $\Phi_{\mathrm{oper}}(M) \hookrightarrow \Phi(M)$ is then a filtered one.

\begin{Remark}\label{rmk:Phi_oper_as_I_mod_I2}
Given any graded $\F_p$-module $M$, an alternative construction of $\Phi_{\mathrm{oper}}(M)$ from $\Phi(M)$ is as follows. For each such $M$, $\Phi(M)$ admits an augmentation $\Phi(M) \to \F_p$ which maps a polynomial to its constant term. Let $I(M)$ denote the corresponding augmentation ideal of $\Phi(M)$. Then $\Phi_{\mathrm{oper}}(M)$ is precisely $I(M)/I(M)^2$. \customendremark
\end{Remark}

Given a cochain complex $X$, for brevity, for each $k \ge 0$, set
\[
\mathcal{F}_k^+ := \Phi_{\mathrm{oper}}((\mathrm{H}^\bullet X)[k])[-k].
\]

\begin{Remark}\label{rmk:basis_for_Phi_oper}
If $\{c_i\}$ is a basis of $\text{H}^\bullet(X)$, then $\mathcal{F}_k^+$ is the free $\F_p$-module on the terms $P^Ic_i$ where $I$ is admissible and $e(I) \le |c_i| + k$. \customendremark
\end{Remark}

We now have commutative diagrams as follows.

\begin{center}
\begin{tikzpicture}[node distance = 1.5cm]
\node [] (A) {$\mathcal{F}_1$};
\node [right of = A,xshift=1.5cm] (B) {$\mathcal{F}_0$};
\node [below of = A] (C) {$\text{H}^\bullet((\Sigma\mathbf{E}^\dagger) X)$};
\node [below of = B] (D) {$\text{H}^\bullet(\mathbf{E}^\dagger X)$};
\node [left of = A,xshift=-1.5cm] (E) {$\mathcal{F}_2$};
\node [left of = E,xshift=-1.5cm] (F) {$\cdots$};
\node [left of = C,xshift=-1.5cm] (G) {$\text{H}^\bullet((\Sigma^2 \mathbf{E}^\dagger) X)$};
\node [left of = G,xshift=-1.5cm] (H) {$\cdots$};
\node [above of = B] (X) {$\mathcal{F}_0^{+}$};
\node [above of = A] (Y) {$\mathcal{F}_1^{+}$};
\node [above of = E] (Z) {$\mathcal{F}_2^{+}$};
\node [above of = F] (W) {$\cdots$};

\draw [->] (A) -- (B) node[midway,anchor=south]{};
\draw [->] (A) -- (C) node[midway,anchor=west]{$\cong$};
\draw [->] (B) -- (D) node[midway,anchor=west]{$\cong$};
\draw [->] (C) -- (D) node[midway,anchor=south]{};
\draw [->] (F) -- (E) node[midway,anchor=south]{};
\draw [->] (E) -- (A) node[midway,anchor=south]{};
\draw [->] (H) -- (G) node[midway,anchor=south]{};
\draw [->] (G) -- (C) node[midway,anchor=south]{};
\draw [->] (E) -- (G) node[midway,anchor=west]{$\cong$};
\draw [->] (W) -- (Z);
\draw [->] (Z) -- (Y);
\draw [->] (Y) -- (X);
\draw [right hook->] (X) -- (B);
\draw [right hook->] (Y) -- (A);
\draw [right hook->] (Z) -- (E);
\end{tikzpicture}
\end{center}

\begin{center}
\begin{tikzpicture}[node distance = 1.5cm]
\node [] (A) {$\text{F}_t\mathcal{F}_1$};
\node [right of = A,xshift=1.5cm] (B) {$\text{F}_t\mathcal{F}_0$};
\node [below of = A] (C) {$\text{F}_t\text{H}^\bullet((\Sigma\mathbf{E}^\dagger) X)$};
\node [below of = B] (D) {$\text{F}_t\text{H}^\bullet(\mathbf{E}^\dagger X)$};
\node [left of = A,xshift=-1.5cm] (E) {$\text{F}_t\mathcal{F}_2$};
\node [left of = E,xshift=-1.5cm] (F) {$\cdots$};
\node [left of = C,xshift=-1.5cm] (G) {$\text{F}_t\text{H}^\bullet((\Sigma^2 \mathbf{E}^\dagger) X)$};
\node [left of = G,xshift=-1.5cm] (H) {$\cdots$};
\node [above of = B] (X) {$\text{F}_t\mathcal{F}_0^{+}$};
\node [above of = A] (Y) {$\text{F}_t\mathcal{F}_1^{+}$};
\node [above of = E] (Z) {$\text{F}_t\mathcal{F}_2^{+}$};
\node [above of = F] (W) {$\cdots$};

\draw [->] (A) -- (B) node[midway,anchor=south]{};
\draw [->] (A) -- (C) node[midway,anchor=west]{$\cong$};
\draw [->] (B) -- (D) node[midway,anchor=west]{$\cong$};
\draw [->] (C) -- (D) node[midway,anchor=south]{};
\draw [->] (F) -- (E) node[midway,anchor=south]{};
\draw [->] (E) -- (A) node[midway,anchor=south]{};
\draw [->] (H) -- (G) node[midway,anchor=south]{};
\draw [->] (G) -- (C) node[midway,anchor=south]{};
\draw [->] (E) -- (G) node[midway,anchor=west]{$\cong$};
\draw [->] (W) -- (Z);
\draw [->] (Z) -- (Y);
\draw [->] (Y) -- (X);
\draw [right hook->] (X) -- (B);
\draw [right hook->] (Y) -- (A);
\draw [right hook->] (Z) -- (E);
\end{tikzpicture}
\end{center}

\begin{Proposition}\label{prop:proddisappear}
If $X$ is a cochain complex, and $t \ge 0$, the induced maps
\[
\lim_{k \ge 0} \mathcal{F}^{+}_k \to \lim_{k \ge 0} \emph{H}^\bullet((\Sigma^k \mathbf{E}^\dagger) X)
\qquad
\lim_{k \ge 0} \emph{F}_t\mathcal{F}^{+}_k \to \lim_{k \ge 0} \emph{F}_t\emph{H}^\bullet((\Sigma^k \mathbf{E}^\dagger) X)
\]
are isomorphisms.
\end{Proposition}

\begin{proof}
We shall demonstrate the case of the first map, from $\lim_k \mathcal{F}^{+}_k$ to $\lim_k \text{H}^\bullet((\Sigma^k \mathbf{E}^\dagger) X)$; the filtered case is analogous. We need to show that the induced map $\lim_k \mathcal{F}^{+}_k \to \lim_k \mathcal{F}_k$ is an isomorphism, as we already know that the map from $\lim_k \mathcal{F}_k$ to $\lim_k \text{H}^\bullet((\Sigma^k \mathbf{E}^\dagger) X)$ is an isomorphism. Injectivity is obvious from the standard description of inverse limits of towers via infinite tuples and, moreover, surjectivity follows from this description and the instability of products described in Proposition~\ref{prop:instab_prods}.
\end{proof}

\begin{Remark}
We mentioned earlier, in Section~\ref{sec:stabops}, that, while the cohomologies of algebras over $\mathpzc{E}^\dagger$ contain certain products, analogous products are not definable in the cohomologies of algebras over $\mathpzc{E}_{\text{st}}^\dagger$. Proposition~\ref{prop:proddisappear} allows us to see this lack of products in the limit concretely in the case of free algebras. \customendremark
\end{Remark}

We have now identified the limit which appears in the definition of $\mathrm{A}(-)$. The remaining piece in the definition of $\mathrm{A}(-)$ is the colimit. The next result computes this colimit and so computes $\mathrm{A}(-)$ itself.

\begin{Proposition}\label{prop:computation_of_A}
Given a cochain complex $X$ and a basis $\{c_i\}$ of $\mathrm{H}^\bullet(X)$, in degree say $d \in \Z$, $\mathrm{A}(X)$ is isomorphic to the $\F_p$-module which consists of the infinite sums
\[
\sum a_{I_i,c_i}(P^{I_i}c_i)
\]
satisfying the following requirements:
\begin{itemize}
	\item For all $(I_i,c_i)$, if $a_{I_i,c_i} \neq 0$, $d(I_i) + |c_i| = d$.
	\item The set of lengths $\{l(I_i) \mid a_{I_i,c_i} \neq 0\}$ is bounded above, or, equivalently, finite.
	\item For any $k \ge 0$, $\#\{(I_i, c_i) \mid a_{I_i,c_i} \neq 0, e(I_i) \le |c_i| + k\}$ is finite.
\end{itemize}

\end{Proposition}

Note that, by the infinite sum above, we mean the function $f \colon \{(I_i, c_i) \mid I_i \:\: \text{admissible}\} \to \F_p$, where $f(I_i,c_i) = a_{I_i,c_i}$.

\begin{proof}
By definition, we have $\mathrm{A}(X) = \mathrm{colim}_{t \ge 0} \lim_{k \ge 0} \mathrm{F}_t\mathrm{H}^\bullet((\Sigma^k\mathbf{E}^\dagger)X)$. By Proposition~\ref{prop:proddisappear}, we have that, for each $t \ge 0$, $\lim_{k \ge 0} \mathrm{F}_t\mathrm{H}^\bullet((\Sigma^k\mathbf{E}^\dagger)X)$ is isomorphic to $\lim_{k \ge 0} \mathrm{F}_t\mathcal{F}^{+}_k$. By construction, $\mathrm{F}_t\mathcal{F}^{+}_k$ is isomorphic to the free graded $\F_p$-module on the monomials $P^Ic_i$ which satisfy $e(I) \le |c_i|+k$ and $p^{l(I)} \le t$. Moreover, inspecting the tower in the proof of Proposition~\ref{prop:Est(2)} and recalling the definition of the operations $P^s$ as in Remark~\ref{rmk:stable_vs_unstable_ops}(1), under these isomorphisms, the map $\mathrm{F}_t\mathcal{F}^{+}_{k+1} \to \mathrm{F}_t\mathcal{F}^{+}_k$ simply kills each monomial $P^Ic_i$ where $e(I) = |c_i|+k+1$, while fixing all other monomials. The relevant limit and colimit are now easily computed and the result follows.
\end{proof}

The above computation of $\mathrm{A}(-)$ allows us to demonstrate the following a priori unexpected property of this functor. 

\begin{Proposition}\label{prop:ABadditive}
The functor $\mathrm{A}$ preserves arbitrary direct sums.
\end{Proposition}

\begin{proof}
This is clear from the computation in Proposition~\ref{prop:computation_of_A} since, if $\{X_i\}_{i \in I}$ is a family of cochain complexes, and $\{c^i_j\}_{j \in J}$ is a basis, for each $i \in I$, of $X_i$, then $\amalg_{i \in I} \{c^i_j\}_{j \in J}$ is a basis of $\bigoplus_{i \in I} X_i$. Alternatively, noting that we have, by definition, $
\mathrm{A}(X) = \mathrm{colim}_{t \ge 0} \lim_{k \ge 0} \mathrm{F}_t\mathrm{H}^\bullet((\Sigma^k\mathbf{E}^\dagger)X)$, and that colimits preserve arbitrary direct sums, we need only show that the limit $\lim_{k \ge 0} \mathrm{F}_t\mathrm{H}^\bullet((\Sigma^k\mathbf{E}^\dagger)X)$ preserves arbitrary direct sums, and this follows from Proposition~\ref{prop:proddisappear}, since, by definition, we have that $\mathcal{F}_k^+ = \Phi_{\mathrm{oper}}((\mathrm{H}^\bullet X)[k])[-k]$ and $\Phi_{\mathrm{oper}}$ clearly preserves arbitrary direct sums.
\end{proof}

Having completed our study of $\mathrm{A}(-)$, we now describe how, given a cochain complex $X$, $\mathrm{A}(X)$ is intermediate between $\text{H}^\bullet(\mathbf{E}_{\text{st}}^\dagger X)$ and $\widehat{\mathcal{B}} \otimes \text{H}^\bullet (X)$. Specifically, we shall construct maps
\[
\mathrm{H}^\bullet(\mathbf{E}_{\mathrm{st}}^\dagger X) \longrightarrow \mathrm{A}(X) \longleftarrow \widehat{\mathcal{B}} \otimes \mathrm{H}^\bullet (X).
\]
Via the canonical maps $\mathpzc{E}_{\text{st}}^\dagger \to \Sigma^k\mathpzc{E}^\dagger$, we get canonical maps $\mathbf{E}_{\textbf{st}}^\dagger X \to (\Sigma^k\mathbf{E}^\dagger)X$, for $k \ge 0$. These are in fact filtered maps, in that they decompose into maps $\text{F}_t\text{H}^\bullet(\mathbf{E}_{\textbf{st}}^\dagger X) \to \text{F}_t\text{H}^\bullet((\Sigma^k\mathbf{E}^\dagger)X)$. These latter maps are compatible with the maps from $\text{F}_t\text{H}^\bullet((\Sigma^{k+1}\mathbf{E}^\dagger)X)$ to $\text{F}_t\text{H}^\bullet((\Sigma^k\mathbf{E}^\dagger)X)$, so that we get an induced map $\text{F}_t\text{H}^\bullet(\mathbf{E}_{\textbf{st}}^\dagger X) \to \lim_{k \ge 0} \text{F}_t\text{H}^\bullet((\Sigma^k\mathbf{E}^\dagger)X)$. Upon taking colimits over $t$, we get an induced natural map
\begin{equation}\label{eqn:Phi_1}
\Phi_1 \colon \text{H}^\bullet(\mathbf{E}_{\text{st}}^\dagger X) \to \text{A}(X).
\end{equation}
Now consider $\widehat{\mathcal{B}} \otimes \text{H}^\bullet(X)$. We construct a natural map
\begin{equation}\label{eqn:Phi_2}
\Phi_2 \colon \widehat{\mathcal{B}} \otimes \text{H}^\bullet (X) \to \text{A}(X).
\end{equation}
This map is defined as follows; the idea is that the term $\text{H}^\bullet((\Sigma^k\mathbf{E}^\dagger)X)$ which arises in the construction of $\text{A}(X)$ inherits an action by the term $\mathcal{B}_{\le k}$ which arises in the construction of $\widehat{\mathcal{B}}$. Recall, as we saw in Propostion~\ref{prop:FtBhatcomplete}, that we have $\widehat{\mathcal{B}} = \text{colim}_{t \ge 0} \lim_{k \ge 0} \text{F}_t \mathcal{B}_{\le k}$. For each $d_1, d_2$, we need to specify a map $\widehat{\mathcal{B}}_{d_1} \otimes \text{H}^{d_2} (X) = \text{colim}_{t \ge 0} (\text{F}_{t}\widehat{\mathcal{B}}_{d_1} \otimes \text{H}^{d_2}(X)) \to \text{colim}_{t \ge 0} \lim_{k \ge 0} \text{F}_t\text{H}^{d_1 + d_2}((\Sigma^k\mathbf{E}^\dagger)X)$. Fix some $t_0 \ge 0$. For each $k \ge 0$, consider the following composite
\[
\text{F}_{t_0}\widehat{\mathcal{B}}_{d_1} \otimes \text{H}^{d_2}(X) \to \text{F}_{t_0}(\mathcal{B}_{\le d_2+k})_{d_1} \otimes \text{F}_1\text{H}^{d_2}((\Sigma^k\mathbf{E}^\dagger)X) \to \text{F}_{t_0+1}\text{H}^{d_1+d_2} ((\Sigma^k\mathbf{E}^\dagger)X).
\]
Here the first map is induced by the canonical map $\text{F}_{t_0}\widehat{\mathcal{B}}_{d_1} \to \text{F}_{t_0}(\mathcal{B}_{\le d_2+k})_{d_1}$ together with the unit of adjunction $X \to (\Sigma^k\mathbf{E}^\dagger)X$. The second map comes from Proposition~\ref{prop:cohomology_of_free_algebras_of_suspended_E}; or, more specifically, from the action of $\mathcal{B}$ on $\mathrm{H}^\bullet((\Sigma^k\mathbf{E}^\dagger)X)$. In short, given a degree $d$ cohomology class $c$ in $\text{H}^\bullet(X)$ and an infinite sum in $\widehat{\mathcal{B}}$, we pass from the former to a cohomology class $c'$ in $\text{H}^\bullet((\Sigma^k\mathbf{E}^\dagger)X)$, project the infinite sum onto the finite sub-sum comprising just the summands of excess at most $d+k$, and then act on $c'$ by the generalized Steenrod operations. Now, an easy check shows that the above composite maps are compatible with the maps $\text{F}_{t_0+1}\text{H}^{d_1+d_2}((\Sigma^{k+1}\mathbf{E}^\dagger)X) \to \text{F}_{t_0+1}\text{H}^{d_1+d_2}((\Sigma^k\mathbf{E}^\dagger)X)$, so that we get an induced map
\[
\text{F}_{t_0}\widehat{\mathcal{B}}_{d_1} \otimes \text{H}^{d_2}(X) \to \lim_{k \ge 0} \text{F}_{t_0+1}\text{H}^{d_1+d_2} ((\Sigma^k\mathbf{E}^\dagger)X) \hookrightarrow \underset{t \ge 0}{\text{colim}} \lim_{k \ge 0} \text{F}_t\text{H}^{d_1 + d_2}((\Sigma^k\mathbf{E}^\dagger)X).
\]
Finally, we note that the above composite maps are compatible with the inclusions $\text{F}_t \widehat{\mathcal{B}}_{d_1} \hookrightarrow \text{F}_{t+1} \widehat{\mathcal{B}}_{d_1}$, so that we get induced the desired map from $\widehat{\mathcal{B}}_{d_1} \otimes \text{H}^{d_2} (X)$ to $\text{colim}_{t \ge 0} \lim_{k \ge 0} \text{F}_t\text{H}^{d_1 + d_2}((\Sigma^k\mathbf{E}^\dagger)X)$, completing the construction of $\Phi_2$. \\

Given a cochain complex $X$, the maps $\Phi_1$ and $\Phi_2$ will allow us to relate $\text{H}^\bullet(\mathbf{E}_{\textbf{st}}^\dagger X)$ and $\widehat{\mathcal{B}} \otimes \text{H}^\bullet(X)$, via $\text{A}(X)$. In order to do this, we first need a lemma.
 
\begin{Lemma}\label{lem:stabmapontocx}
If $X$ is a cochain complex, the towers 
\[
\cdots \to (\Sigma^2\mathbf{E}^\dagger)X \to (\Sigma\mathbf{E}^\dagger)X \to \mathbf{E}^\dagger X
\]
\[
\cdots \to \emph{H}^\bullet((\Sigma^2\mathbf{E}^\dagger)X) \to \emph{H}^\bullet((\Sigma\mathbf{E}^\dagger)X) \to \emph{H}^\bullet(\mathbf{E}^\dagger X)
\]
satisfy the Mittag-Leffler condition. Moreover, for each $t \ge 0$, the towers
\[
\cdots \to \emph{F}_t(\Sigma^2\mathbf{E}^\dagger)X \to \emph{F}_t(\Sigma\mathbf{E}^\dagger)X \to \emph{F}_t\mathbf{E}^\dagger X
\]
\[
\cdots \to \emph{F}_t\emph{H}^\bullet((\Sigma^2\mathbf{E}^\dagger)X) \to \emph{F}_t\emph{H}^\bullet((\Sigma\mathbf{E}^\dagger)X) \to \emph{F}_t\emph{H}^\bullet(\mathbf{E}^\dagger X)
\]
also satisfy the Mittag-Leffler condition. 
\end{Lemma}

\begin{proof}
We shall show that, for each $n \ge 0$ and any cochain complex $X$, the towers
\[
\cdots \to (\Sigma^{2}\mathpzc{E}^\dagger)(n) \otimes_{\Sigma_n} X^{\otimes n} \to (\Sigma\mathpzc{E}^\dagger)(n) \otimes_{\Sigma_n} X^{\otimes n} \to \mathpzc{E}^\dagger(n) \otimes_{\Sigma_n} X^{\otimes n}
\]
\[
\cdots \to \text{H}^\bullet((\Sigma^{2}\mathpzc{E}^\dagger)(n) \otimes_{\Sigma_n} X^{\otimes n}) \to \text{H}^\bullet((\Sigma\mathpzc{E}^\dagger)(n) \otimes_{\Sigma_n} X^{\otimes n}) \to \text{H}^\bullet(\mathpzc{E}^\dagger(n) \otimes_{\Sigma_n} X^{\otimes n})
\]
satisfy the Mittag-Leffler condition, from which the desired results follow immediately. The Mittag-Leffler property for the first of the two towers above follows immediately from the Mittag-Leffler property of the tower $\cdots \to (\Sigma^{2}\mathpzc{E}^\dagger)(n) \to (\Sigma\mathpzc{E}^\dagger)(n) \to \mathpzc{E}^\dagger(n)$, which was established in Proposition~\ref{prop:stabmapontoML}. Consider then the second tower. Let $\{c_i\}$ denote a basis of $\text{H}^\bullet(X)$. As above, for each $k \ge 0$, $\mathrm{H}^\bullet((\Sigma^k\mathbf{E}^\dagger)X)$ can be identified with the free graded-commutative algebra over $\F_p$ on the terms $P^Ic_i$ where $I$ is admissible and $e(I) \le |c_i| + k$. Moreover, for each $t \ge 0$, $\text{F}_t\text{H}^\bullet((\Sigma^k\mathbf{E}^\dagger)X)$ can be identified with the $\F_p$-submodule of this free graded-commutative algebra generated by the products $(P^{I_1}c_1) \cdots (P^{I_r}c_r)$, where $r \ge 0$, each $I_i$ is admissible, $e(I_i) < |c_i| + k$ and $p^{l(I_1)} + \cdots + p^{l(I_r)} \le t$. Inspecting the definitions of the filtrations, we have that, under these identifications, $\text{H}^\bullet((\Sigma^k\mathpzc{E}^\dagger)(n) \otimes_{\Sigma_n} X^{\otimes n})$ is the $\F_p$-submodule generated by the terms $(P^{I_1}c_1) \cdots (P^{I_r}c_r)$ where $r \ge 0$, each $I_i$ is admissible, $e(I_i) < |c_i| + k$ for each $i$, and $p^{l(I_1)} + \cdots + p^{l(I_k)} = n$. Now, given a fixed $k$, the map $\text{H}^\bullet((\Sigma^{k+1}\mathpzc{E}^\dagger)(n) \otimes_{\Sigma_n} X^{\otimes n}) \to \text{H}^\bullet((\Sigma^k\mathpzc{E}^\dagger)(n) \otimes_{\Sigma_n} X^{\otimes n})$ sends any $P^Ic_i$ which satisfies, not only that $e(I) \le |c_i| + k + 1$, but also $e(I) \le |c_i| + k$, to itself, so that all monomials $P^Ic_i$ in $\text{H}^\bullet((\Sigma^k\mathpzc{E}^\dagger)(n) \otimes_{\Sigma_n} X^{\otimes n})$ occur in the image. On the other hand, by Proposition~\ref{prop:instab_prods}, all products are killed, and in particular, no products occur in the image. We have thus identified the image of the map $\text{H}^\bullet((\Sigma^{k+1}\mathpzc{E}^\dagger)(n) \otimes_{\Sigma_n} X^{\otimes n}) \to \text{H}^\bullet((\Sigma^k\mathpzc{E}^\dagger)(n) \otimes_{\Sigma_n} X^{\otimes n})$, and, moreover, essentially the same argument shows that all maps $\text{H}^\bullet((\Sigma^{k'}\mathpzc{E}^\dagger)(n) \otimes_{\Sigma_n} X^{\otimes n}) \to \text{H}^\bullet((\Sigma^k\mathpzc{E}^\dagger)(n) \otimes_{\Sigma_n} X^{\otimes n})$, for $k' \ge k+1$, have this same image, which gives us the desired result.
\end{proof}

We can now finally prove that $\text{H}^\bullet(\mathbf{E}_{\textbf{st}}^\dagger X)$ and $\widehat{\mathcal{B}} \otimes \text{H}^\bullet(X)$ are naturally isomorphic, thus computing the cohomologies of free algebras over the stable Barratt-Eccles cochain operad $\mathpzc{E}_{\text{st}}^\dagger$.

\begin{Proposition}\label{prop:stablefreehom}
Given any cochain complex $X$, the maps $\Phi_1$ and $\Phi_2$ above are isomorphisms. Thus, for cochain complexes $X$, we have a natural isomorphism
\[
\mathrm{H}^\bullet(\mathbf{E}_{\mathbf{st}}^\dagger X) \cong \widehat{\mathcal{B}} \otimes \mathrm{H}^\bullet (X).
\]
\end{Proposition}

\begin{proof}
Consider first the case of $\Phi_2$. As per Lemma~\ref{lem:complexes_over_k_are_sums_of_spheres_and_disks}, any cochain complex over $\F_p$ is isomorphic to a direct sum $(\bigoplus_{i \in I} \Sph^{n_i}) \oplus (\bigoplus_{j \in J} \D^{n_j})$, where $\Sph^n$ and $\D^n$ denote the standard sphere and disk complexes (see Section~\ref{sec:nots_convs} for these complexes). Because, by Proposition~\ref{prop:ABadditive}, $\mathrm{A}(X)$ commutes with arbitrary direct sums, and because $\widehat{\mathcal{B}} \otimes \mathrm{H}^\bullet (X)$ also clearly commutes with arbitrary direct sums, it suffices to demonstrate that $\Phi_2$ is an isomorphism in the case of a sphere complex $\Sph^n$ or a disk complex $\D^n$. The case of the disk complex follows immediately from the fact that the $\Sigma^k\mathbf{E}^\dagger$ preserve quasi-isomorphisms (as the operads $\Sigma^k\mathpzc{E}^\dagger$ are $\Sigma$-free). Thus, it remains to consider the case where $X = \Sph^n$ for some $n \in \Z$. Let $c_n$ denote the generator of $\text{H}^\bullet(X)$ in degree $n$. By Proposition~\ref{prop:FtBhatcomplete}, we have that
\[
\widehat{\mathcal{B}} \otimes \text{H}^\bullet(X) = \underset{t \ge 0}{\text{colim}} \lim_{k \ge 0} \text{F}_t \mathcal{B}_{\le k} \otimes \F_p\{c_n\}.
\]
Using the bases for the $\mathcal{B}_{\le k}$ in Remark~\ref{rmk:facts_about_B}, we see that $\widehat{\mathcal{B}} \otimes \text{H}^\bullet(X)$, in degree say $d \in \Z$, is isomorphic to the module which consists of infinite sums
\[
\sum a_{I}(P^{I}c_n)
\]
by which we mean functions $f \colon \{I \mid I \: \text{admissible}\} \to \F_p$, where $f(I) = a_I$, satisfying the following requirements:
\begin{itemize}
	\item For all $I$, if $a_I \neq 0$, $d(I) + n = d$.
	\item The set of lengths $\#\{l(I) \mid a_I \neq 0\}$ is bounded above, or, equivalently, finite.
	\item For any $k \ge 0$, $\#\{I \mid a_I \neq 0, e(I) \le k\}$ is finite.
\end{itemize}
On the other hand, by the computation in Proposition~\ref{prop:computation_of_A}, in degree $d$, $\mathrm{A}(X)$ is isomorphic to the module which consists of infinite sums
\[
\sum a_{I}(P^{I}c_n)
\]
satisfying the following requirements:
\begin{itemize}
	\item For all $I$, if $a_I \neq 0$, $d(I) + n = d$.
	\item The set of lengths $\#\{l(I) \mid a_I \neq 0\}$ is bounded above, or, equivalently, finite.
	\item For any $k \ge 0$, $\#\{I \mid a_I \neq 0, e(I) \le k+n\}$ is finite.
\end{itemize}
The only difference occurs in the third condition; however, the two conditions are obviously equivalent, and an easy check shows that the map $\Phi_2$, under these isomorphisms, corresponds to simply the identity, and so is itself an isomorphism. \\

Next, consider the composite $\Phi_2^{-1}\Phi_1 \colon \mathrm{H}^\bullet(\mathbf{E}_{\mathbf{st}}^\dagger X) \to \widehat{\mathcal{B}} \otimes \mathrm{H}^\bullet (X)$. As both sides clearly commute with filtered colimits, it suffices to demonstrate that $\Phi_2^{-1}\Phi_1$, or equivalently, $\Phi_1$, is an isomorphism for finite complexes, where, as per Definition~\ref{def:fin_mod}, a finite $\F_p$-cochain complex is one which is bounded above and below, and of finite dimension in each degree. Let $X$ be a finite cochain complex. We have filtrations $\{\text{F}_t\mathbf{E}_{\textbf{st}}^\dagger X\}_{t \ge 0}$ and $\{\text{F}_t(\Sigma^k\mathbf{E}^\dagger) X\}_{t \ge 0}$ of the free algebras $\mathbf{E}_{\textbf{st}}^\dagger X$ and $(\Sigma^k\mathbf{E}^\dagger) X$ by setting
\[
\text{F}_t\mathbf{E}_{\textbf{st}}^\dagger X = \bigoplus_{n \le t} \mathpzc{E}_{\textbf{st}}^\dagger(n) \otimes_{\Sigma_n} X^{\otimes n}
\]
\[
\text{F}_t(\Sigma^k\mathbf{E}^\dagger) X = \bigoplus_{n \le t} (\Sigma^k\mathpzc{E}^\dagger)(n) \otimes_{\Sigma_n} X^{\otimes n}.
\]
By taking cohomologies, we also get similar filtrations of $\text{H}^\bullet(\mathbf{E}_{\textbf{st}}^\dagger X)$ and $\text{H}^\bullet((\Sigma^k\mathbf{E}^\dagger) X)$. \\

Now, since $X$ is finite, we have
\begin{align*}
\text{F}_t\mathbf{E}_{\text{st}}^\dagger X &= \bigoplus_{n \le t} \mathpzc{E}_{\text{st}}^\dagger(n) \otimes_{\Sigma_n} X^{\otimes n} \\
&= \bigoplus_{n \le t} \lim_{k \ge 0}((\Sigma^k\mathpzc{E}^\dagger)(n)) \otimes_{\Sigma_n} X^{\otimes n} \\
&\cong \bigoplus_{n \le t} \lim_{k \ge 0}((\Sigma^k\mathpzc{E}^\dagger)(n) \otimes_{\Sigma_n} X^{\otimes n}) \\
&\cong \lim_{k \ge 0} (\bigoplus_{n \le t}(\Sigma^k\mathpzc{E}^\dagger)(n) \otimes_{\Sigma_n} X^{\otimes n}) \\
& = \lim_{k \ge 0} \text{F}_t(\Sigma^k\mathbf{E}^\dagger)X.
\end{align*}
Here it is the first (non-identity) isomorphism which requires finiteness of $X$. By Lemma~\ref{lem:stabmapontocx}, we have that
\[
\text{F}_t\text{H}^\bullet(\mathbf{E}_{\textbf{st}}^\dagger X) \cong \lim_{k \ge 0} \text{F}_t\text{H}^\bullet((\Sigma^k\mathbf{E}^\dagger)X)
\] 
and then, upon taking colimits over $t \ge 0$, we get that $\Phi_1 \colon \text{H}^\bullet(\mathbf{E}_{\textbf{st}}X) \to \text{A}(X)$ is an isomorphism, as desired. This completes the proof.
\end{proof}

\begin{Remark}\label{rmk:eqhomnontrivial}
We saw in Proposition~\ref{prop:stabhom} that the non-equivariant homologies, in individual arities, $\text{H}^\bullet(\mathpzc{E}^\dagger_{\text{st}}(n))$, $n \ge 0$, of the stable operad $\mathpzc{E}_{\text{st}}^\dagger$ are simply zero (except the unit in arity 1). On the other hand, the equivariant homologies, summed up, $\bigoplus_n \text{H}^\bullet(\mathpzc{E}^\dagger_{\text{st}}(n)/\Sigma_n)$, yield $\text{H}^\bullet(\mathbf{E}_{\textbf{st}}^\dagger\F_p[0])$. By Proposition~\ref{prop:stablefreehom} above, we have that this is exactly the completion $\widehat{\mathcal{B}}$ of the generalized Steenrod algebra:
\[
\bigoplus_{n \ge 0} \text{H}^\bullet(\mathpzc{E}^\dagger_{\text{st}}(n)/\Sigma_n) \:\: \cong \:\: \widehat{\mathcal{B}}.
\]
Thus, while the non-equivariant homologies are (almost) zero, the equivariant homologies $\text{H}^\bullet(\mathpzc{E}^\dagger_{\text{st}}(n)/\Sigma_n)$, $n \ge 0$, are highly non-trivial objects. \customendremark
\end{Remark}

\section{The Cohomology Operations II}\label{sec:stabops2}

Let $A$ be an algebra over the stable Barratt-Eccles operad $\mathpzc{E}_{\text{st}}^\dagger$. Earlier, at least in the case $p = 2$, we constructed cohomology operations for $A$, which could be viewed as elements of $\mathcal{B}$. In fact, these operations do not account for all the operations which are naturally induced on the cohomology of $A$. The algebra of cohomology operations for $A$ is instead the completion $\widehat{\mathcal{B}}$, which, as we have seen, allows certain infinite sums.

\begin{Proposition}\label{prop:(co)hom_ops_stab_op}
Given an algebra $A$ over $\mathpzc{E}_{\emph{st}}^\dagger$, $\emph{H}^\bullet(A)$ is naturally an algebra over $\widehat{\mathcal{B}}$.
\end{Proposition}

\begin{proof}
The map describing the action of $\widehat{\mathcal{B}}$ is the following composite
\[
\widehat{\mathcal{B}} \otimes \text{H}^\bullet(A) \overset{\cong}\longrightarrow \text{H}^\bullet(\mathbf{E}_{\textbf{st}}^\dagger A) \to \text{H}^\bullet(A).
\]
Here the first map is the isomorphism provided by Proposition~\ref{prop:stablefreehom}. The second map is that which arises by applying $\text{H}^\bullet(-)$ to the structure map $\alpha \colon \mathbf{E}_{\textbf{st}}^\dagger A \to A$ which defines the $\mathpzc{E}_{\text{st}}^\dagger$-algebra structure of $A$. The properties required by a $\widehat{\mathcal{B}}$-action follows from those required by an $\mathpzc{E}_{\text{st}}^\dagger$-action. For example, associativity can be derived as follows. Letting $m \colon \mathbf{E}_{\textbf{st}}^\dagger \mathbf{E}_{\textbf{st}}^\dagger \Rightarrow \mathbf{E}_{\textbf{st}}^\dagger$ denote the monadic multiplication, we have the following commutative square.
\begin{center}
\begin{tikzpicture}[node distance = 2cm]
\node(A){$\mathbf{E}_{\textbf{st}}^\dagger \mathbf{E}_{\textbf{st}}^\dagger A$};
\node[below of = A](C){$\mathbf{E}_{\textbf{st}}^\dagger A$};
\node[right of = A, xshift = 0.5cm](B){$\mathbf{E}_{\textbf{st}}^\dagger A$};
\node[below of = B](D){$A$};
	
\draw[->] (A) -- (B) node[midway,anchor=south]{$\mathbf{E}_{\textbf{st}}^\dagger\alpha$};
\draw[->] (A) -- (C) node[midway,anchor=east]{$m_A$};
\draw[->] (C) -- (D) node[midway,anchor=north]{$\alpha$};
\draw[->] (B) -- (D) node[midway,anchor=west]{$\alpha$};
\end{tikzpicture}
\end{center}
Upon applying $\text{H}^\bullet(-)$, and invoking Proposition~\ref{prop:stablefreehom}, we get the following commutative square.
\begin{center}
\begin{tikzpicture}[node distance = 2cm]
\node(A){$\widehat{\mathcal{B}} \otimes \widehat{\mathcal{B}} \otimes \text{H}^\bullet(A)$};
\node[below of = A](C){$\widehat{\mathcal{B}} \otimes \text{H}^\bullet(A)$};
\node[right of = A, xshift = 1cm](B){$\widehat{\mathcal{B}} \otimes \text{H}^\bullet(A)$};
\node[below of = B](D){$\text{H}^\bullet(A)$};
	
\draw[->] (A) -- (B) node[midway,anchor=south]{};
\draw[->] (A) -- (C) node[midway,anchor=east]{};
\draw[->] (C) -- (D) node[midway,anchor=north]{};
\draw[->] (B) -- (D) node[midway,anchor=west]{};
\end{tikzpicture}
\end{center}
A diagram chase now yields associativity of the $\widehat{\mathcal{B}}$-action.
\end{proof}

\begin{Remark}\label{rmk:general_ops_vs_explicit_ops}
Earlier, in Section~\ref{sec:stabops}, we constructed cohomology operations for algebras over $\mathpzc{E}_{\text{st}}^\dagger$ explicitly, at least in the case $p = 2$. If one unravels the general construction in Proposition~\ref{prop:(co)hom_ops_stab_op} above, one can see that, in the case $p = 2$, and in the case of the finite sums in $\widehat{\mathcal{B}}$, the operations coincide with those which we constructed earlier. In fact, the action of an infinite sum, in any given instance, reduces to an action by a finite sum in $\mathcal{B}$ (one projects to an appropriate finite sub-sum comprising iterated operations up to a certain upper bound on the excess). Moreover, in the earlier, more explicit construction of the operations for algebras over $\mathpzc{E}_{\text{st}}^\dagger$, we did not verify such things as the Adem relations. These now follow from the general construction above and from the definition of $\widehat{\mathcal{B}}$ itself. \customendremark
\end{Remark}

\section{Unstable Modules over $\widehat{\mathcal{B}}$}

We saw earlier, in Proposition~\ref{prop:cohomology_of_free_algebras_of_suspended_E}, that, in the case of an algebra over the unstable operad $\mathpzc{E}^\dagger$, the cohomology is not only a $\mathcal{B}$-module, but an unstable $\mathcal{B}$-module, where instability here means that $P^Ix = 0$ when $e(I) > |x|$. We can entirely analogously define instability for $\widehat{\mathcal{B}}$-modules: say that a $\widehat{\mathcal{B}}$-module $H$ is \textit{unstable}\index{unstable modules} if $P^Ix = 0$ whenever $e(I) > |x|$ (where $P^I$ is now really an infinite sum with all, but one, coefficients equal to zero). In the case of an algebra over the stable Barratt-Eccles operad $\mathpzc{E}_{\text{st}}^\dagger$ however, the cohomology is in fact not an unstable $\widehat{\mathcal{B}}$-module -- we saw this earlier, in Section~\ref{sec:stabops}, via the explicit construction of the operations in the case $p = 2$. \\

Now, given that $\widehat{\mathcal{B}}$-modules appear naturally in the stable case, it is natural to ask: why wasn't $\widehat{\mathcal{B}}$ seen in the unstable case? The following proposition gives an answer for this: restricting attention to unstable modules, we find that an unstable $\mathcal{B}$-module is canonically also an unstable $\widehat{\mathcal{B}}$-module so that, in particular, in the case of the unstable operad, we could equivalently have said that the cohomologies are unstable $\widehat{\mathcal{B}}$-modules.
	 
\begin{Proposition}\label{prop:BBhatactions}
Let $H$ be a graded module over $\F_p$. An unstable action by $\mathcal{B}$ on $H$ extends canonically to an unstable action by $\widehat{\mathcal{B}}$.
\end{Proposition}

\begin{proof}
Suppose given an unstable $\mathcal{B}$-action on $H$. We define a $\widehat{\mathcal{B}}$-action on $H$ as follows: given any element $x$, an infinite sum is to be projected to the finite sub-sum comprising operations of excess $< |x|$, and then act according to the given action by $\mathcal{B}$. We must simply check that this does indeed yield a $\widehat{\mathcal{B}}$-action, as instability and compatibility with the $\mathcal{B}$-action are clear. The only thing that is non-obvious is associativity. This can be verified as follows. Let $\Sigma_1$ and $\Sigma_2$ denote two infinite sums in $\widehat{\mathcal{B}}$, and let $x$ be an element of $H$. Then $(\Sigma_1)(\Sigma_2 x)$ can be computed as follows: consider only those operations $P^J$ in $\Sigma_2$ where $e(J) < |x|$, and then, given such an operation $P^J$, we only consider those operations $P^I$ in $\Sigma_1$ where $e(I) < |x| + d(J)$. On the other hand, $(\Sigma_1\Sigma_2)x$ can be computed as follows: consider only those terms in the product $\Sigma_1\Sigma_2$ coming from a product $P^I \cdot P^J$ which can contain a term of excess $< |x|$ (we don't need to restrict to exactly those terms with excesss $< |x|$ since any extraneous terms of excess $\ge |x|$ will act by zero due to instability). Due to Lemma~\ref{lem:QIQJprop}, we can restrict to products $P^I \cdot P^J$ where $e(J) < |x|$. Due to Lemma~\ref{lem:excess_down}, we can restrict further to those products which also satisfy $e(IJ) < |x|$, where $IJ$ denotes the concatenation of $I$ and $J$; if $I$ is non-empty, using $e(IJ) = e(I) - d(J)$, we can rephrase this condition to $e(I) < |x| + d(J)$, and moreover can do the same if $I$ is empty, as then both $e(I) < |x| + d(J)$ and $e(IJ) < |x|$ are true under the standing assumption that $e(J) < |x|$. These conditions, namely $e(J) < |x|$ and $e(I) < |x| + d(J)$, are exactly those which we saw in the consideration of $(\Sigma_1)(\Sigma_2 x)$, and so we have demonstrated that $(\Sigma_1)(\Sigma_2 x) = (\Sigma_1\Sigma_2)x$, as desired.
\end{proof}

\section{The Relation Between $\widehat{\mathcal{B}}$ and the Steenrod Algebra $\mathcal{A}$}

We have seen, in Proposition~\ref{prop:(co)hom_ops_stab_op}, that cohomologies of algebras over $\mathpzc{E}_{\text{st}}^\dagger$ yield modules over $\widehat{\mathcal{B}}$. We shall see later, in Proposition~\ref{prop:stabopactioncochains}, that spectral cochains yield algebras over $\mathpzc{E}_{\text{st}}^\dagger$, and as such, cohomologies of spectra possess actions by $\widehat{\mathcal{B}}$. In this special case of cochains on spectra, as in the case of cochains on spaces, we shall see, in Proposition~\ref{prop:P01sp}, that $P^0$ acts by the identity. For this reason, we shall now consider the quotient $\widehat{\mathcal{B}}/(1-P^0)$. Our goal is to show that this quotient is isomorphic to the Steenrod algebra $\mathcal{A}$. First, recall that the Steenrod algebra may be defined as follows. If $p = 2$, where $p$ is our fixed prime, we have
\[
\mathcal{A} = \mathbf{F}\{P^s \mid s \ge 0\}/(I_{\text{Adem}}, 1-P^0)
\]
where $\mathbf{F}\{P^s \mid s \ge 0\}$ denotes the free graded algebra over $\F_2$ on the formal symbols $P^s$, for $s \ge 0$, where $P^s$ has degree $s$, and where $I_{\text{Adem}}$ denotes the two-sided ideal generated by the Adem relations. If $p > 2$, we have
\[
\mathcal{A} = \mathbf{F}\{P^s, \beta P^s \mid s \ge 0\}/(I_{\text{Adem}}, 1-P^0)
\]
where $\mathbf{F}\{P^s, \beta P^s \mid s \in \Z\}$ denotes the free graded algebras over $\F_p$ on formal symbols $P^s, \beta P^s$, for $s \in \Z$, where $P^s, \beta P^s$ have degrees $2s(p-1), 2s(p-1)+1$ respectively, and where $I_{\text{Adem}}$ denotes the two-sided ideal generated by the Adem relations. \\

The Steenrod algebra has a basis, the Cartan-Serre basis, which is similar to the Cartan-Serre basis which we described earlier for $\mathcal{B}$. This is an $\F_p$-basis and is given by the monomials $P^I$ where $I$ is admissible and, if $p = 2$, $I = (i_1,\dots,i_k)$ satisfies $i_j > 0$ for each $j$, and if $p > 2$, $I = (\varepsilon_1, i_1,\dots,\varepsilon_k, i_k)$ satisfies, once again, $i_j > 0$ for each $j$. See, for example,~\cite{Milnor}. \\

Now, note that we have a map
\[
\widehat{\mathcal{B}} \to \mathcal{A}
\]
where, given the element $\sum a_IP^I$ of $\widehat{\mathcal{B}}$, we map it to the class of the sub-sum consisting of those multi-indices in which all entries are non-negative; this sub-sum is finite by Proposition~\ref{prop:i1posinf_ikneginf}. That this is indeed a map of algebras follows from the lemma below.

\begin{Lemma}\label{lem:PIPJnegentry}
Given admissible multi-indices $I$ and $J$, if either of them contains a negative entry, than all multi-indices in the admissible monomials expansion of $P^IP^J$ must contain a negative entry, or, equivalently by admissibility, must have a negative final entry.
\end{Lemma}

Here, as before, in the case where $p > 2$, in which case multi-indices take the form $(\varepsilon_1,i_1,\dots,\varepsilon_r,i_r)$, where the $i_j$ lie in $\Z$ while the $\varepsilon_j$ lie in $\{0,1\}$, the first entry is taken to be $i_1$, and the final entry, $i_r$, which is to say we disregard the $\varepsilon_j$ for this particular purpose.

\begin{proof}
We shall outline the case where $p = 2$; the $p > 2$ case is analogous. If $J$ contains a negative entry, or, equivalently by admissibility, has a negative final entry, the result follows by Lemma~\ref{lem:first_entry_down_last_up}. Assume then that it is $I$ that has a negative final entry. We shall prove the result by inducting on the length of $J$. If $J$ has length zero, it is empty and the result is trivial. Suppose that $J$ has length one, and say that it is equal to $(b)$, for some $b \in \Z$. Consider $P^IP^b$. We shall prove the result in this case by an induction on the length of $I$. As $I$ is required to contain an negative entry, it cannot have zero length. Suppose that $I$ has length one, and say that it is equal to $(a)$, where we must have $a < 0$. If $a \ge 2b$, the result is obvious as then $b < 0$. Suppose that $a < 2b$. Then the admissible monomials expansion of $P^aP^b$ is given by the Adem relations
\[
\sum \binom{b-i-1}{a-2i}P^{a+b-i}P^i.
\]
In order for the binomial coefficient to be non-zero, we must have $i \le a/2$, and so we see that the final entry $i$ in the multi-index $(a+b-i,i)$ is always negative, as desired. Now suppose that we have the result for terms $P^IP^b$ where $I$ has length $< n$, for some $n \ge 2$. Given an admissible multi-index $I = (i_1,\dots,i_n)$ of length $n$, where $i_n < 0$, we have $P^IP^b = P^{i_1}(P^{i_2} \cdots P^{i_n}P^b)$. Upon first forming the admissible monomials expansion of $P^{i_2} \cdots P^{i_n}P^b$, the inductive hypothesis for the induction on the length of $I$ and Lemma~\ref{lem:first_entry_down_last_up} give us the desired result. Now let us return to the induction on the length of $J$. We have demonstrated the result in the cases where $J$ has length zero or one. Now suppose that we have the result for terms $P^IP^J$ where $J$ has length $< n$, for some $n \ge 2$. Given an admissible multi-index $J = (j_1,\dots,j_n)$ of length $n$, we have $P^IP^J = (P^IP^{j_1} \cdots P^{j_{n-1}})P^{j_n}$. The result now follows by the inductive hypothesis and by the case where $J$ has length one.
\end{proof}

\begin{Proposition}\label{prop:BhatandA}
\index{generalized Steenrod operations!relation to the Steenrod algebra}We have the following.
\begin{itemize}
	\item[(i)] The following sequence is exact:
\[
0 \longrightarrow \widehat{\mathcal{B}} \overset{1-P^0}\longrightarrow \widehat{\mathcal{B}} \longrightarrow \mathcal{A} \longrightarrow 0.
\]
	\item[(ii)] The left ideal of $\widehat{\mathcal{B}}$ generated by $1-P^0$ coincides with the two-sided ideal and the above map $\widehat{\mathcal{B}} \to \mathcal{A}$ induces an algebra isomorphism
\[
\widehat{\mathcal{B}}/(1-P^0) \cong \mathcal{A}.
\]
\end{itemize}
\end{Proposition}

In the above sequence, the map denoted by $1-P^0$ is right multiplication by $1-P^0$. Note that, as shown in~\cite{Mandell}, analogues of both of these statements hold true if we replace $\widehat{\mathcal{B}}$ with $\mathcal{B}$.

\begin{proof}
Let us consider (i). Recall that, for each $k \ge 0$, we defined $\mathcal{B}_{\le k}$ as the quotient of $\mathcal{B}$ by $I_{\text{Adem}} + I_{\text{exc} \, > \, k}$, where $I_{\text{Adem}}$ denotes the ideal generated by the Adem relations and $I_{\text{exc} \, > \, k}$ the ideal generated by the monomials of excess $> k$. Let us let $\mathcal{A}_{\le k}$ denote the analogous quotients of $\mathcal{A}$. As per Proposition 12.5 in~\cite{Mandell}, for each $k \ge 1$, we have an exact sequence as follows
\[
0 \longrightarrow \mathcal{B}_{\le k} \overset{1-P^0}\longrightarrow \mathcal{B}_{\le k} \longrightarrow \mathcal{A}_{\le k} \longrightarrow 0.
\]
For each $t \ge 0$, we can consider a filtered version of the above sequence, where the filtrations are by length. That is, we can consider the sequence
\[
0 \longrightarrow \text{F}_t\mathcal{B}_{\le k} \overset{1-P^0}\longrightarrow \text{F}_{t+1}\mathcal{B}_{\le k} \longrightarrow \text{F}_{t+1}\mathcal{A}_{\le k} \longrightarrow 0
\]
and we note that this sequence is itself also exact: surjectivity at the righthand end is clear, injectivity at the lefthand end follows from the exactness of the previous sequence, and, just as in the case of the previous sequence, exactness in the middle follows by examination of the Cartan-Serre bases for $\mathcal{B}_{\le k}$ and $\mathcal{A}_{\le k}$ (these are the same as the bases for $\mathcal{B}$ and $\mathcal{A}$, made up of admissible multi-indices, together with the obvious restriction on the excess of the multi-indices). Upon taking limits over $k$, since the maps $\text{F}_t\mathcal{B}_{\le k+1} \to \text{F}_t\mathcal{B}_{\le k}$ are clearly onto, so that the corresponding $\lim^1$ vanishes, we get an exact sequence as follows
\[
0 \longrightarrow \lim_{k \ge 1}\text{F}_t\mathcal{B}_{\le k} \longrightarrow \lim_{k \ge 1}\text{F}_{t+1}\mathcal{B}_{\le k} \longrightarrow \lim_{k \ge 1}\text{F}_{t+1}\mathcal{A}_{\le k} \longrightarrow 0.
\]
Upon taking colimits over $t$, by Proposition~\ref{prop:FtBhatcomplete}, we get an exact sequence as follows
\[
0 \longrightarrow \widehat{\mathcal{B}} \longrightarrow \widehat{\mathcal{B}} \longrightarrow \underset{t \ge 0}{\text{colim}}\lim_{k \ge 1}\text{F}_{t+1}\mathcal{A}_{\le k} \longrightarrow 0.
\]
Moreover, an easy check shows that $\text{colim}_{t \ge 0}\lim_{k \ge 0}\text{F}_{t+1}\mathcal{A}_{\le k} \cong \mathcal{A}$ (in the case of $\mathcal{A}$, no infinite sums arise in this construction, unlike the case of $\mathcal{B}$; this follows by an argument similar to that in Proposition~\ref{prop:i1posinf_ikneginf}, where we saw that there must exist negative degree operations in all but finitely many terms of an infinite sum $\sum a_IP^I$ of iterated operations increasing in excess). Thus we have an exact sequence
\[
0 \longrightarrow \widehat{\mathcal{B}} \longrightarrow \widehat{\mathcal{B}} \longrightarrow \mathcal{A} \longrightarrow 0
\]
and, unravelling the identifications which we have made, we see that the maps in this sequence are exactly those in the proposition statement. This completes the proof of (i). \\

Now let us consider (ii). By the exactness in (i), the kernel of the map $\widehat{\mathcal{B}} \to \mathcal{A}$ coincides with the image of the map $\widehat{\mathcal{B}} \to \widehat{\mathcal{B}}$ which we have denoted by $1-P^0$. This latter image is the left ideal of $\widehat{\mathcal{B}}$ generated by $1-P^0$. By definition of the map $\widehat{\mathcal{B}} \to \mathcal{A}$, its kernel clearly contains the two-sided ideal of $\widehat{\mathcal{B}}$ generated by $1-P^0$, and so we have that this two-sided ideal is contained in the aforementioned left ideal, from which it follows that these two ideals coincide. Moreover, we have the isomorphism $\widehat{\mathcal{B}}/(1-P^0) \cong \mathcal{A}$ because the map $\widehat{\mathcal{B}} \to \mathcal{A}$ is onto and the kernel, as just established, is precisely the two-sided ideal generated by $1-P^0$.
\end{proof}

\section[Comparison of Algebras Over $\mathpzc{E}_{\mathrm{st}}^\dagger$ and $\mathpzc{M}_{\mathrm{st}}^\dagger$]{Comparison of the Homotopy Theories of Algebras Over $\mathpzc{E}_{\mathrm{st}}^\dagger$ and $\mathpzc{M}_{\mathrm{st}}^\dagger$}\label{subsec:comparison_EstAlg_MstAlg}

At the beginning of Section~\ref{subsec:cohomologies_of_stable_operads}, we chose the stable Barratt-Eccles operad as our standard model of stable operads, and mentioned that every result which we demonstrate for the stable Barratt-Eccles operad also holds for the stable McClure-Smith operad. In this section, we demonstrate that the homotopy theories of algebras over $\mathpzc{E}_{\text{st}}^\dagger$ and $\mathpzc{M}_{\text{st}}^\dagger$ are equivalent (an analogous result holds, of course, also for the chain operads $\mathpzc{E}_{\text{st}}$ and $\mathpzc{M}_{\text{st}}$). Recall the map $\mathrm{TR}^\dagger \colon \mathpzc{E}_{\text{st}}^\dagger \to \mathpzc{M}_{\text{st}}^\dagger$ defined in Section~\ref{subsec:stable_barratt_eccles_operad}. We get an induced functor
\[
(\mathrm{TR}^\dagger)_* \colon \mathpzc{M}_{\mathrm{st}}^\dagger\text{-}\mathsf{Alg} \to \mathpzc{E}_{\mathrm{st}}^\dagger\text{-}\mathsf{Alg}
\]
which, given an $\mathpzc{M}_{\mathrm{st}}^\dagger$-algebra $A$, with structure maps $\{ \alpha_n \colon \mathpzc{M}_{\mathrm{st}}^\dagger(n) \otimes A^{\otimes n} \to A\}_{n \ge 0}$, forgets part of the structure, yielding the same underlying cochain complex $A$, but with the composites $\alpha_n \circ (\mathrm{TR}^\dagger_n \otimes \mathrm{id}) \colon \mathpzc{E}_{\mathrm{st}}^\dagger(n) \otimes A^{\otimes n} \to \mathpzc{M}_{\mathrm{st}}^\dagger(n) \otimes A^{\otimes n} \to A$ as structure maps, which endow it with the structure of an algebra over $\mathpzc{E}_{\mathrm{st}}^\dagger$. This functor $(\mathrm{TR}^\dagger)_*$ has a left adjoint
\[
(\mathrm{TR}^\dagger)^* \colon \mathpzc{E}_{\mathrm{st}}^\dagger\text{-}\mathsf{Alg} \to \mathpzc{M}_{\mathrm{st}}^\dagger\text{-}\mathsf{Alg}
\]
which, given an $\mathpzc{E}_{\mathrm{st}}^\dagger$-algebra $A$, with structure map $\beta \colon \mathbf{E}_{\mathbf{st}}^\dagger A \to A$, sends $A$ to the quotient $\mathbf{M}_{\mathbf{st}}^\dagger A / (\mathbf{TR}^\dagger_A(\mathrm{ker}(\beta)))$, where $\mathbf{TR}^\dagger_A$ denotes the map $\mathbf{E}_{\mathbf{st}}^\dagger A \Rightarrow \mathbf{M}_{\mathbf{st}}^\dagger A$ induced by $\mathrm{TR}^\dagger$. (See Section 4 of~\cite{Hinich1} for this adjunction.) \\

As per Corollary~\ref{cor:stablesemimod}, $\mathpzc{E}_{\mathrm{st}}^\dagger\text{-}\mathsf{Alg}$ has a Quillen semi-model structure where the weak equivalences are quasi-isomorphisms and the fibrations are surjective maps. The semi-admissibility criterion proved in Proposition~\ref{prop:E_adm} for $\mathpzc{E}_{\mathrm{st}}^\dagger$ holds also for $\mathpzc{M}_{\mathrm{st}}^\dagger$, by the same proof, and so $\mathpzc{M}_{\mathrm{st}}^\dagger\text{-}\mathsf{Alg}$ also has a Quillen semi-model structure where the weak equivalences are quasi-isomorphisms and the fibrations are surjective maps.

\begin{Proposition}\label{prop:EstAlg_MstAlg_equivalence}
The adjunction $(\mathrm{TR}^\dagger)^* \dashv (\mathrm{TR}^\dagger)_*$ induces an equivalence between the homotopy categories $\mathsf{h}\mathpzc{E}_{\mathrm{st}}^\dagger\text{-}\mathsf{Alg}$ and $\mathsf{h}\mathpzc{M}_{\mathrm{st}}^\dagger\text{-}\mathsf{Alg}$.
\end{Proposition}

\begin{proof}
As the right adjoint $(\mathrm{TR}^\dagger)_*$ is simply forgetful, preserving the cochain complexes and the maps between them, it clearly preserves fibrations and acyclic fibrations. We thus have a derived adjunction $\mathbf{L}(\mathrm{TR}^\dagger)^* \dashv \mathbf{R}(\mathrm{TR}^\dagger)_*$. In order to show that these derived functors are equivalences, we wish to show that the unit of adjunction
\[
\eta_A \colon A \to (\mathrm{TR}^\dagger)_*(\mathrm{TR}^\dagger)^*(A)
\]
is a quasi-isomorphism for any cofibrant algebra $A$. In order to do this, we note a result, namely Theorem 4.7.4, from~\cite{Hinich1}. This says that if $\mathpzc{P}$ and $\mathpzc{P}'$ are ``$\Sigma$-split'' operads, and $F \colon \mathpzc{P} \to \mathpzc{P}^\prime$ is an operad map which is a quasi-isomorphism in each operadic degree and is compatible with the ``$\Sigma$-splittings'', that then $F$ induces a Quillen equivalence between the homotopy categories of algebras over $\mathpzc{P}$ and $\mathpzc{P}'$. The notions of $\Sigma$-split operads and $\Sigma$-splittings are defined in Section 4.2.4 of~\cite{Hinich1}, and, using the assumptions about them in the statement of the aforementioned result, Hinich demonstrates that the unit $\eta_A$ above is indeed a quasi-isomorphism. We, however, do not actually need these notions. Hinich's proof uses these notions for two purposes: (i) if $\mathpzc{P}$ is a $\Sigma$-split operad, the category of algebras $\mathpzc{P}\text{-}\mathsf{Alg}$ inherits a Quillen model structure where the weak equivalences are quasi-isomorphisms and the fibrations are surjective maps, and (ii) if the map $F \colon \mathpzc{P} \to \mathpzc{P}'$ is not only a quasi-isomorphism in each operadic degree, but also compatible with the $\Sigma$-splittings, then not only are the maps $\mathpzc{P}(n) \to \mathpzc{P}'(n)$, for each $n \ge 0$, quasi-isomorphisms, but so are the induced maps $\mathpzc{P}(n)/\Sigma_n \to \mathpzc{P}'(n)/\Sigma_n$, for each $n \ge 0$. We do not need (i) as we have already established that $\mathpzc{E}_{\mathrm{st}}^\dagger\text{-}\mathsf{Alg}$ and $\mathpzc{M}_{\mathrm{st}}^\dagger\text{-}\mathsf{Alg}$ have Quillen semi-model structures, which suffice for our purpose. We know from Section~\ref{subsec:stable_barratt_eccles_operad} that $\mathrm{TR}$ is a quasi-isomorphism in each operadic degree. In order to complete our proof, by the same argument as in Hinich's proof of Theorem 4.7.4 in~\cite{Hinich1}, it thus remains to establish that, for each $n \ge 0$, the map
\[
\mathpzc{E}_{\mathrm{st}}^\dagger(n) / \Sigma_n \to \mathpzc{M}_{\mathrm{st}}^\dagger(n) / \Sigma_n
\]
induced by $\mathrm{TR}$ is a quasi-isomorphism. Now, we have $\bigoplus_{n \ge 0} \mathpzc{E}_{\mathrm{st}}^\dagger(n) / \Sigma_n = \mathbf{E}_{\mathbf{st}}^\dagger\mathbb{F}_p[0]$, and, by Proposition~\ref{prop:stablefreehom}, we have a natural isomorphism between $\mathrm{H}^\bullet(\mathbf{E}_{\mathbf{st}}^\dagger\mathbb{F}_p[0])$ and $\widehat{\mathcal{B}}$. We also have that $\bigoplus_{n \ge 0} \mathpzc{M}_{\mathrm{st}}^\dagger(n) / \Sigma_n = \mathbf{M}_{\mathbf{st}}^\dagger\mathbb{F}_p[0]$, and, by a proof entirely analogous to that of Proposition~\ref{prop:stablefreehom}, we have a natural isomorphism between $\mathrm{H}^\bullet(\mathbf{M}_{\mathbf{st}}^\dagger\mathbb{F}_p[0])$ and $\widehat{\mathcal{B}}$. It suffices to demonstrate that the following diagram commutes.
\begin{center}
\begin{tikzpicture}
\node [] (A) {$\mathrm{H}^\bullet(\mathbf{E}_{\mathbf{st}}^\dagger\mathbb{F}_p[0])$};
\node [right of = A, xshift = 25mm] (B) {$\mathrm{H}^\bullet(\mathbf{M}_{\mathbf{st}}^\dagger\mathbb{F}_p[0])$};
\node [below of = A, xshift = 17mm] (C) {$\widehat{\mathcal{B}}$};

\draw [->] (A) -- (B);
\draw [->] (A) -- (C) node[midway, anchor = east, yshift = -1mm]{$\cong$};
\draw [->] (B) -- (C) node[midway, anchor = west, yshift = -1mm]{$\cong$};
\end{tikzpicture}
\end{center}
Here, the upper map is the obvious map induced by $\mathrm{TR}^\dagger$. By definition of the diagonal isomorphisms in the diagram, it suffices to demonstrate that the following diagram commutes.
\begin{center}
\begin{tikzpicture}
\node [] (A) {$\mathrm{H}^\bullet(\mathbf{E}_{\mathbf{st}}^\dagger\mathbb{F}_p[0])$};
\node [right of = A, xshift = 25mm] (B) {$\mathrm{A}^{\mathpzc{E}_{\mathrm{st}}^\dagger}(\mathbb{F}_p[0])$};
\node [below of = A] (C) {$\mathrm{H}^\bullet(\mathbf{M}_{\mathbf{st}}^\dagger\mathbb{F}_p[0])$};
\node [below of = B] (D) {$\mathrm{A}^{\mathpzc{M}_{\mathrm{st}}^\dagger}(\mathbb{F}_p[0])$};
\node [right of = B, xshift = 25mm, yshift = -5mm] (E) {$\widehat{\mathcal{B}}$};

\draw [->] (A) -- (B);
\draw [->] (A) -- (C);
\draw [->] (B) -- (D);
\draw [->] (C) -- (D);
\draw [->] (E) -- (B);
\draw [->] (E) -- (D);
\end{tikzpicture}
\end{center}
Here, $\mathrm{A}^{\mathpzc{E}_{\mathrm{st}}^\dagger}(-)$ is the functor which we called $\mathrm{A}(-)$ in Section~\ref{subsec:cohomologies_of_free_algebras_over_the_stable_operads}, defined in (\ref{eqn:def_of_A}), while $\mathrm{A}^{\mathpzc{M}_{\mathrm{st}}^\dagger}(-)$ is defined by the same formula, but with the stable Barratt-Eccles operad replaced by the stable McClure-Smith operad. The maps $\mathrm{H}^\bullet(\mathbf{E}_{\mathbf{st}}^\dagger\mathbb{F}_p[0]) \leftarrow \mathrm{A}^{\mathpzc{E}_{\mathrm{st}}^\dagger}(\mathbb{F}_p[0]) \to \widehat{\mathcal{B}}$ are as in Section~\ref{subsec:cohomologies_of_free_algebras_over_the_stable_operads}, and the maps $\mathrm{H}^\bullet(\mathbf{M}_{\mathbf{st}}^\dagger\mathbb{F}_p[0]) \leftarrow \mathrm{A}^{\mathpzc{M}_{\mathrm{st}}^\dagger}(\mathbb{F}_p[0]) \to \widehat{\mathcal{B}}$ are defined analogously, but with the stable Barratt-Eccles operad replaced by the stable McClure-Smith operad. The vertical maps are induced by $\mathrm{TR}^\dagger$. The lefthand square commutes because, as in Section~\ref{subsec:stable_barratt_eccles_operad}, the following square commutes.
\begin{center}
\begin{tikzpicture}[node distance = 1.5cm]
\node [] (A) {$\Sigma \mathpzc{E}^\dagger$};
\node [right of = A,xshift=1.5cm] (B) {$\mathpzc{E}^\dagger$};
\node [below of = A] (C) {$\Sigma \mathpzc{M}^\dagger$};
\node [below of = B] (D) {$\mathpzc{M}^\dagger$};

\draw [->] (A) -- (B) node[midway,anchor=south]{$\Psi$};
\draw [->] (A) -- (C) node[midway,anchor=east]{$\Sigma \text{TR}^\dagger$};
\draw [->] (B) -- (D) node[midway,anchor=west]{$\text{TR}^\dagger$};
\draw [->] (C) -- (D) node[midway,anchor=north]{$\Psi$};
\end{tikzpicture}
\end{center}
As for the righthand triangle, recall, as in Section~\ref{sec:stabops}, that, for $s \in \Z$, the operation $P^s$ on $\mathrm{H}^\bullet(\mathbf{E}^\dagger X)$ is defined using the elements $e_d^{\mathrm{un}}$, for $d \le 0$, of $\mathpzc{E}_{\mathrm{st}}^\dagger(2)$. As in Definition~\ref{def:e_d^un_and_e_d^st}, we have that $e_{d}^{\text{un}} = (1, \tau, 1, \tau, \dots)$ (a sequence of length $|d|+1$), where $\tau$ denotes the non-trivial permutation of $\{1,2\}$. The analogous operations $P^s$ on $\mathrm{H}^\bullet(\mathbf{M}^\dagger X)$ are defined by the same formula but with $e_d^{\mathrm{un}}$ replaced by elements, say $f_d^{\mathrm{un}}$, of $\mathpzc{M}_{\mathrm{st}}^\dagger(2)$. For each $d \le 0$, we have that $f_d^{\mathrm{un}}$ is given by the surjection $(1212\cdots)$ (a sequence of length $|d|+2$). It suffices to show that, for each $d \le 0$, $\mathrm{TR}^\dagger$ maps $e_d^{\mathrm{un}}$ to $f_d^{\mathrm{un}}$. This follows from the algorithmic procedure described in Section~\ref{subsec:stable_barratt_eccles_operad}. For example, consider the tuple $e_3^{\mathrm{un}} = (1,\tau,1,\tau)$. As per the algorithmic procedure, we must consider positive integers $r_0, r_1, r_2, r_3$ which sum to $2 + 3 = 5$. Each $r_i$ must necessarily be $1$, except for one, which must be $2$. An easy check shows that, unless $r_3 = 2$, the resulting surjection $f_{(r_0, r_1, r_2, r_3)}$ is degenerate. If indeed $r_3 = 2$, following the algorithm, we get the subsequences $(1)$, $(2)$, $(1)$ and $(21)$, which concatenate to form the surjection $(12121)$, which is precisely $f_3^{\mathrm{un}}$, as desired. (The proof for general $d$ is entirely analogous.)
\end{proof}


\chapter{The Homotopy Coherent, or $\infty$-, Additivity of the Stable Operads}

In this chapter, we wish to demonstrate the homotopy coherent, or $\infty$-, additivity of stable Barratt-Eccles operad, justifying the adjective ``stable'', in three successively more general forms. First, we will show that given free algebras $\mathbf{E}_{\textbf{st}}^\dagger X$ and $\mathbf{E}_{\textbf{st}}^\dagger Y$, the algebra coproduct $\mathbf{E}_{\textbf{st}}^\dagger X \amalg \mathbf{E}_{\textbf{st}}^\dagger Y$ is naturally quasi-isomorphic to the direct sum $\mathbf{E}_{\textbf{st}}^\dagger X \oplus \mathbf{E}_{\textbf{st}}^\dagger Y$. As $\mathbf{E}_{\textbf{st}}^\dagger$ is a left adjoint as a functor from dg modules to algebras, and so preserves colimits, we can also phrase this result as saying that $\mathbf{E}_{\textbf{st}}^\dagger$, as a monad on dg modules, is homotopy additive. Next, we shall generalize this result and show that, for cofibrant algebras $A$ and $B$, over $\mathpzc{E}_{\text{st}}^\dagger$, the coproduct $A \amalg B$ is naturally quasi-isomorphic to $A \oplus B$. Here the cofibrancy is in the sense of the Quillen semi-model structure provided by Corollary~\ref{cor:stablesemimod}. Finally, we shall generalize this one step further and show that if, given a diagram of algebras $A \leftarrow C \rightarrow B$, $A$ and $B$ are cofibrant and $C \to A$ is a cofibration, then $A \amalg_C B$ is naturally quasi-isomorphic to $A \oplus_C B$.

\section{Derived Coproducts of Algebras Over the Stable Operads}

We begin with the homotopy additivity of the monad $\mathbf{E}_{\textbf{st}}^\dagger$.

\begin{Proposition}\label{prop:additivity_of_monad}
If $X$ and $Y$ are cochain complexes, we have a natural quasi-isomorphism
\[
\mathbf{E}_{\normalfont{\textbf{st}}}^\dagger(X \oplus Y) \sim \mathbf{E}^\dagger_{\normalfont{\textbf{st}}}(X) \oplus \mathbf{E}^\dagger_{\normalfont{\textbf{st}}}(Y).
\]
\end{Proposition}

\begin{proof}
Given the cochain complexes $X$ and $Y$, we have a canonical map
\[
f \colon \mathbf{E}^\dagger_{\textbf{st}} (X) \oplus \mathbf{E}^\dagger_{\textbf{st}} (Y) \to \mathbf{E}^\dagger_{\textbf{st}} (X \oplus Y)
\]
and we claim that this map is a quasi-isomorphism. As in Proposition~\ref{prop:stablefreehom}, for cochain complexes $Z$, we have a natural isomorphism $\eta_Z \colon \mathrm{H}^\bullet(\mathbf{E}^\dagger_{\textbf{st}} (Z)) \cong \widehat{\mathcal{B}} \otimes \mathrm{H}^\bullet(Z)$. By naturality, we get the commutative diagram
\begin{center}
\begin{tikzpicture}[node distance = 1.5cm]
\node [] (A) {$\mathrm{H}^\bullet(\mathbf{E}^\dagger_{\textbf{st}} (X)) \oplus \mathrm{H}^\bullet(\mathbf{E}^\dagger_{\textbf{st}} (Y))$};
\node [right of = A, xshift = 30mm] (B) {$\mathrm{H}^\bullet(\mathbf{E}^\dagger_{\textbf{st}} (X) \oplus \mathbf{E}^\dagger_{\textbf{st}} (Y))$};
\node [right of = B, xshift = 25mm] (C) {$\mathrm{H}^\bullet(\mathbf{E}^\dagger_{\textbf{st}} (X \oplus Y))$};
\node [below of = A] (D) {$(\widehat{\mathcal{B}} \otimes \mathrm{H}^\bullet(X)) \oplus (\widehat{\mathcal{B}} \otimes \mathrm{H}^\bullet(Y))$};
\node [below of = C] (E) {$\widehat{\mathcal{B}} \otimes \mathrm{H}^\bullet(X \oplus Y)$};

\draw [->] (A) -- (B) node[midway,anchor=south]{$\cong$};
\draw [->] (B) -- (C) node[midway,anchor=south]{$\mathrm{H}^\bullet(f)$};
\draw [->] (A) -- (D) node[midway,anchor=east]{$\eta_X \oplus \eta_Y$};
\draw [->] (C) -- (E) node[midway,anchor=west]{$\eta_{X \oplus Y}$};
\draw [->] (D) -- (E) node[midway,anchor=north]{$\cong$};
\end{tikzpicture}
\end{center}
and from this it follows that $\mathrm{H}^\bullet(f)$ is an isomorphism, as desired.
\end{proof}

As we noted above, Propostion~\ref{prop:additivity_of_monad} can be regarded as a statement about coproducts of free algebras. We now consider, more generally, coproducts of cofibrant algebras. As we saw in Section~\ref{sec:env_op}, coproducts of cell algebras may be computed via enveloping operads, and so we shall be led to consider the enveloping operads associated to the stable Barrat-Eccles operad $\mathpzc{E}_{\text{st}}^\dagger$. First, however, we have the following lemma.

\begin{Lemma}\label{lem:stabilityigorsense}
For each $j \ge 2$ and each non-trivial partition $j = j_1 + \cdots + j_k$, we have that
\[
\mathpzc{E}_{\emph{st}}^\dagger(j)/\Sigma_{j_1} \times \cdots \times \Sigma_{j_k} \sim 0.
\]
\end{Lemma}

Here non-trivial means that the partition is not indiscrete. Moreover, by $\mathpzc{E}^\dagger_{\text{st}}(j)/\Sigma_{j_1} \times \cdots \times \Sigma_{j_k}$, we mean $\mathpzc{E}^\dagger_{\text{st}}(j) \otimes_{\Sigma_{j_1} \times \cdots \times \Sigma_{j_k}} (\Sph^0)^{\otimes j_1} \otimes \cdots \otimes (\Sph^0)^{\otimes j_k}$, where $\Sph^0$ is the complex $\F_p[0]$.

\begin{proof}
By Proposition~\ref{prop:additivity_of_monad}, we have that the canonical map
\[
\mathbf{E}^\dagger_{\textbf{st}}(\Sph^0) \oplus \cdots \oplus \mathbf{E}^\dagger_{\textbf{st}}(\Sph^0) \to \mathbf{E}^\dagger_{\textbf{st}}(\Sph^0 \oplus \cdots \oplus \Sph^0)
\]
where $\Sph^0$ denotes the complex $\F_p[0]$, and where we take $k$ copies of $\Sph^0$ in the case of partitions of size $k$, is a quasi-isomorphism. The result then follows from the fact that
\[
\mathbf{E}^\dagger_{\textbf{st}}(\Sph^0 \oplus \cdots \oplus \Sph^0) \cong \bigoplus_{j_1, \dots, j_k \ge 0} \mathpzc{E}_{\text{st}}^\dagger(j_1 +\cdots + j_k) \otimes_{\Sigma_{j_1} \times \cdots \times \Sigma_{j_k}} (\Sph^0)^{\otimes j_1} \otimes \cdots \otimes (\Sph^0)^{\otimes j_k}.
\]
\end{proof}

Now, given an algebra $A$ over $\mathpzc{E}_{\text{st}}^\dagger$, and the associated enveloping operad $\mathpzc{U}^A$, recall, as in Section~\ref{sec:env_op}, that we have a canonical map $\mathpzc{E}^\dagger_{\text{st}} \to \mathpzc{U}^A$.

\begin{Lemma}\label{lem:UandVforcofibA}
Given a cofibrant algebra $A$ over $\mathpzc{E}^\dagger_{\emph{st}}$, for each $j \ge 1$ and any partition $j = j_1 + \cdots + j_k$, the canonical map
\[
\mathpzc{E}_{\emph{st}}^\dagger(j)/\Sigma_{j_1} \times \cdots \times \Sigma_{j_k} \to \mathpzc{U}^A(j)/\Sigma_{j_1} \times \cdots \times \Sigma_{j_k}
\]
is a quasi-isomorphism.
\end{Lemma}

\begin{proof}
Without loss of generality, we may take $A$ to be a cell $\mathpzc{E}_{\text{st}}^\dagger$-algebra. Following the notation of Section~\ref{sec:cell_alg}, let
\[
A_0 \to A_1 \to A_2 \to \cdots
\]
be a cell filtration of $A$ and fix some choices $M_1, M_2, \dots$ for the cochain complexes which appear in the attachment squares. For each $n \ge 0$, let $N_n = \oplus_{i \le n} M_i$, where $N_0 = 0$, and let also $N = \oplus_{i \ge 0} M_i$. As per Section~\ref{sec:env_op}, we have that, for each $j \ge 0$, as a graded right $\F_p[\Sigma_j]$-module
\begin{equation}\label{eqn:cellU}
\mathpzc{U}^A(j) = \bigoplus_{i \ge 0} \mathpzc{E}^\dagger_{\text{st}}(i + j) \otimes_{\Sigma_i} (N[1])^{\otimes i}.
\end{equation}
The differential on $\mathpzc{U}^A(j)$, we recall, is given by the Leibniz rule, the attachment maps and the operadic composition. Moreover, for each $n \ge 0$, and again for each $j \ge 0$, as a graded right $\F_p[\Sigma_j]$-module
\begin{equation}\label{eqn:cellskelU}
\mathpzc{U}^{A_n}(j) = \bigoplus_{i \ge 0} \mathpzc{E}_{\text{st}}^\dagger(i + j) \otimes_{\Sigma_i} (N_n[1])^{\otimes i}.
\end{equation}
Moreover, from Section~\ref{sec:env_op}, recall that we have filtrations $\text{F}_m\mathpzc{U}^{A_n}$ of  the $\mathpzc{U}^{A_n}$. For each $j \ge 0$, the map $\mathpzc{E}_{\text{st}}^\dagger(j) \to \mathpzc{U}^{A}(j)$ corresponds to the inclusion of the $i = 0$ summand in (\ref{eqn:cellU}), and, similarly, for each $n \ge 0$, the map $\mathpzc{E}_{\text{st}}^\dagger(j) \to \mathpzc{U}^{A_n}(j)$ corresponds to the inclusion of the $i = 0$ summand in (\ref{eqn:cellskelU}). It follows that the map $\mathpzc{E}_{\text{st}}^\dagger(j) \to \mathpzc{U}^{A}(j)$ factors through $\mathpzc{U}^{A_n}(j)$ for each $n \ge 0$ and, moreover, the maps $\mathpzc{E}_{\text{st}}^\dagger(j) \to \mathpzc{U}^{A_n}(j)$ factor through $\text{F}_m\mathpzc{U}^{A_n}(j)$ for each $m \ge 0$. We shall now prove the desired result via an induction. Specifically, we shall show that, for each $m,n \ge 0$, the map
\[
\mathpzc{E}_{\text{st}}^\dagger(j)/\Sigma_{j_1} \times \cdots \times \Sigma_{j_k} \to \text{F}_m\mathpzc{U}^{A_n}(j)/\Sigma_{j_1} \times \cdots \times \Sigma_{j_k}
\]
is a quasi-isomorphism for all $j \ge 1$ and any partition $j = j_1 + \cdots + j_k$. The desired result then follows by passage to colimits. We will prove this statement by an induction on $n$. In the case $n = 0$, we have that $\text{F}_m\mathpzc{U}^{A_0}(j) = \mathpzc{E}_{\text{st}}^\dagger(j)$ for each $m \ge 0$ and $j \ge 1$, so that the result is obvious. Next suppose that, for some $n \ge 1$, the property holds for $\text{F}_m\mathpzc{U}^{A_{n-1}}$ holds for all $m \ge 0$. We shall show that this same property holds for $\text{F}_m\mathpzc{U}^{A_{n}}$ for $m \ge 0$, by an induction over $m$. We have that, for each $j \ge 1$, $\text{F}_0\mathpzc{U}^{A_n}(j) = \mathpzc{U}^{A_{n-1}}(j) = \text{colim}_m\,\text{F}_m\mathpzc{U}^{A_{n-1}}(j)$ which, by invoking the inductive hypothesisis for the induction over $n$ and passing to the colimit, we see satisfies the required property (recall that filtered colimits of complexes are exact). Next, suppose that the required property holds for $\text{F}_{m-1}\mathpzc{U}^{A_n}(j)$ for some $m \ge 1$. Fix some $j \ge 1$ and a partition $j = j_1 + \cdots + j _k$. We wish to show that the map $\mathpzc{E}_{\text{st}}(j)/\Sigma_{j_1} \times \cdots \times \Sigma_{j_k} \to \text{F}_m\mathpzc{U}^{A_n}(j)/\Sigma_{j_1} \times \cdots \times \Sigma_{j_k}$ is a quasi-isomorphism. Since we can factor this map as
\[
\mathpzc{E}_{\text{st}}(j)/\Sigma_{j_1} \times \cdots \times \Sigma_{j_k} \to \text{F}_{m-1}\mathpzc{U}^{A_n}(j)/\Sigma_{j_1} \times \cdots \times \Sigma_{j_k} \to \text{F}_m\mathpzc{U}^{A_n}(j)/\Sigma_{j_1} \times \cdots \times \Sigma_{j_k}
\]
it suffices, due to the inductive hypothesis for the induction over $m$, to show that the second map in this composition is a quasi-isomorphism. As in the proof of Lemma~\ref{lem:Estenvopsflat}, we have an exact sequence
\[
0 \to \text{F}_{m-1}\mathpzc{U}^{A_n}(j) \to \text{F}_m\mathpzc{U}^{A_n}(j) \to \text{F}_m\mathpzc{U}^{A_n}(j)/\text{F}_{m-1}\mathpzc{U}^{A_n}(j) \to 0
\]
which is split, at the lefthand end, at the level of graded modules. As in the proof of Lemma~\ref{lem:almost_splitunstable}, we thus have an induced exact sequence as follows
\begin{multline}\label{eqn:SES}
0 \to \text{F}_{m-1}\mathpzc{U}^{A_n}(j)/\Sigma_{j_1} \times \cdots \times \Sigma_{j_k} \to \text{F}_m\mathpzc{U}^{A_n}(j)/\Sigma_{j_1} \times \cdots \times \Sigma_{j_k} \\
\to (\text{F}_m\mathpzc{U}^{A_n}(j)/\text{F}_{m-1}\mathpzc{U}^{A_n}(j))/\Sigma_{j_1} \times \cdots \times \Sigma_{j_k} \to 0.
\end{multline}
By the long exact sequence in homology, it suffices to show that the righthand term, that is, the term $(\text{F}_m\mathpzc{U}^{A_n}(j)/\text{F}_{m-1}\mathpzc{U}^{A_n}(j))/\Sigma_{j_1} \times \cdots \times \Sigma_{j_k}$, has zero homology. From Section~\ref{sec:env_op}, we have that
\[
\text{F}_m\mathpzc{U}^{A_n}(j)/\text{F}_{m-1}\mathpzc{U}^{A_n}(j) \cong \mathpzc{U}^{A_{n-1}}(m+j) \otimes_{\Sigma_m} M_n[1]^{\otimes m}.
\]
It follows that $(\text{F}_m\mathpzc{U}^{A_n}(j)/\text{F}_{m-1}\mathpzc{U}^{A_n}(j))/\Sigma_{j_1} \times \cdots \times \Sigma_{j_k}$ is isomorphic to
\[
\mathpzc{U}^{A_{n-1}}(m+j) \otimes_{\Sigma_m \times \Sigma_{j_1} \times \cdots \times \Sigma_{j_k}} M_n[1]^{\otimes m} \otimes (\Sph^0)^{\otimes j_1} \otimes \cdots \otimes (\Sph^0)^{\otimes j_k}.
\]
By the inductive hypothesis for the induction over $n$, and by writing $M_n[1]$ (which is a sum of spheres) as a filtered colimit of its finite subcomplexes if it isn't already finite, we have that this term is quasi-isomorphic to
\[
\mathpzc{E}_{\text{st}}^\dagger(m+j) \otimes_{\Sigma_m \times \Sigma_{j_1} \times \cdots \times \Sigma_{j_k}} M_n[1]^{\otimes m} \otimes (\Sph^0)^{j_1} \otimes \cdots \otimes (\Sph^0)^{j_k}.
\]
Moreover, since $m+j \ge 2$, this has zero homology by Lemma~\ref{lem:stabilityigorsense}. This completes the induction over $m$ and then also the induction over $n$.
\end{proof}

\begin{Proposition}\label{prop:coproducts_of_algebras}
Let $A$ and $B$ be cofibrant algebras over $\mathpzc{E}^\dagger_{\emph{st}}$. Then we have a natural quasi-isomorphism
\[
A \amalg B \sim A \oplus B.
\]
\end{Proposition}

\begin{proof}
Without loss of generality, we may take $A$ and $B$ to be cell $\mathpzc{E}$-algebras. Following the notation of Section~\ref{sec:cell_alg}, let
\[
A_0 \to A_1 \to A_2 \to \cdots
\]
be a cell filtration of $A$ and fix some choices $M_1^A, M_2^A, \dots$ for the chain complexes which appear in the attachment squares. Let also
\[
B_0 \to B_1 \to B_2 \to \cdots
\]
be a cell filtration of $B$ and fix some choices $M_1^B, M_2^B, \dots$ for the chain complexes which appear in the attachment squares. From Section~\ref{sec:env_op}, recall the formulae for $\mathpzc{U}^{A_n}(j)$ and $\mathpzc{U}^{B_n}(j)$, and also their filtration pieces $\text{F}_m\mathpzc{U}^{A_n}(j)$ and $\text{F}_m\mathpzc{U}^{B_n}(j)$. It suffices to show that the canonical map
\[
A_n \oplus B_n \to A_n \amalg B_n
\]
is a quasi-isomorphism for each $n \ge 0$. In the case $n=0$, we get a map $\mathpzc{E}_{\text{st}}^\dagger(0) \oplus \mathpzc{E}_{\text{st}}^\dagger(0) \to \mathpzc{E}_{\text{st}}^\dagger(0) \amalg \mathpzc{E}_{\text{st}}^\dagger(0) \cong \mathpzc{E}_{\text{st}}^\dagger(0)$ (this isomorphism holds since $\mathpzc{E}_{\text{st}}^\dagger(0)$ is initial) and this map is necessarily a quasi-isomorphism since $\mathpzc{E}_{\text{st}}^\dagger(0)$ has zero homology as per~Proposition~\ref{prop:stabhom}. Suppose then that, for some $n \ge 1$, the map $A_{n-1} \oplus B_{n-1} \to A_{n-1} \amalg B_{n-1}$ is a quasi-isomorphism. Note that $A_n$, $B_n$, $A_n \amalg B_n$, may be identified, respectively, with $\mathpzc{U}^{A_n}(0)$, $\mathpzc{U}^{B_n}(0)$, $\mathpzc{U}^{A_n \amalg B_n}(0)$. Moreover, we have filtration pieces $\text{F}_m\mathpzc{U}^{A_n}(0), \text{F}_m\mathpzc{U}^{B_n}(0)$ and $\text{F}_m\mathpzc{U}^{A_n \amalg B_n}(0)$, and the map $A_n \oplus B_n \to A_n \amalg B_n$ is a filtered map, in that we have induced maps as follows
\[
\text{F}_m\mathpzc{U}^{A_n}(0) \oplus \text{F}_m\mathpzc{U}^{B_n}(0) \to \text{F}_m\mathpzc{U}^{A_n \amalg B_n}(0).
\]
Thus, from the map $A_n \oplus B_n \to A_n \amalg B_n$, we get an induced map of the strongly convergent spectral sequences associated to the aforementioned filtrations. Recalling, from Section~\ref{sec:env_op}, the computations of the associated graded pieces corresponding to the filtrations on the $\mathpzc{U}^{A_n}(j)$, we have that the map on the $E^1$-terms consists of the maps
\begin{multline*}
\Big(\mathpzc{U}^{A_{n-1}}(m) \otimes_{\Sigma_m} (M^A_n[1])^{\otimes m}\Big) \oplus \Big(\mathpzc{U}^{B_{n-1}}(m) \otimes_{\Sigma_m} (M^B_n[1])^{\otimes m}\Big) \\
\longrightarrow \mathpzc{U}^{A_{n-1} \amalg B_{n-1}}(m) \otimes_{\Sigma_m} (M^A_n[1] \oplus M^B_n[1])^{\otimes m}.
\end{multline*}
If $m = 0$, this map reduces to $A_{n-1} \oplus B_{n-1} \to A_{n-1} \amalg B_{n-1}$ and so is a quasi-isomorphism by the inductive hypothesis. Suppose then that $m \ge 1$. Then, by Lemma~\ref{lem:UandVforcofibA}, it suffices to show that the map
\begin{multline*}
\Big(\mathpzc{E}_{\text{st}}^\dagger(m) \otimes_{\Sigma_m} (M^A_n[1])^{\otimes m}\Big) \oplus \Big(\mathpzc{E}_{\text{st}}^\dagger(m) \otimes_{\Sigma_m} (M^B_n[1])^{\otimes m}\Big) \\
\longrightarrow \mathpzc{E}_{\text{st}}^\dagger(m) \otimes_{\Sigma_m} (M^A_n[1] \oplus M^B_n[1])^{\otimes m}
\end{multline*}
is a quasi-isomorphism. If $m = 1$, this is obvious, so suppose that $m \ge 2$. We have that
\[
\mathpzc{E}^\dagger_{\text{st}}(m) \otimes_{\Sigma_m} (M^A_n[1] \oplus M^B_n[1])^{\otimes m} \cong \bigoplus_{l = 0}^m \mathpzc{E}^\dagger_{\text{st}}(m) \otimes_{\Sigma_{m-l} \times \Sigma_l} (M^A_n[1])^{\otimes(m-l)} \otimes (M^B_n[1])^{\otimes l}.
\]
By Lemma~\ref{lem:stabilityigorsense}, the only two summands which have non-zero homology are those corresponding to $l = 0, m$, and so we once more have a quasi-isomorphism, as desired. It follows that the aforementioned map of spectral sequences is an isomorphism from $E^2$ onwards, and so we have that the map $A_n \oplus B_n \to A_n \amalg B_n$ is a quasi-isomorphism, completing the induction.
\end{proof}

\section{Derived Pushouts of Algebras Over the Stable Operads}

We shall now consider pushouts of algebras over the stable Barratt-Eccles operad $\mathpzc{E}_{\text{st}}^\dagger$. Our claim is that, for algebras $A$, $B$ and $C$ over $\mathpzc{E}^\dagger_{\text{st}}$, under suitable circumstances, we have that $A \amalg_C B \sim A \oplus_C B$. We shall compute the pushout via a bar construction. Given a diagram $A \leftarrow C \rightarrow B$ of algebras over $\mathpzc{E}_{\text{st}}^\dagger$, the bar construction $\beta_\bullet(A,C,B)$ is the simplicial $\mathpzc{E}_{\text{st}}^\dagger$-algebra that is given in simplicial degree $n$ as follows:
\[
\beta_n(A, C, B) := A \amalg \underbrace{C \amalg \cdots \amalg C}_{n \: \text{factors}} \amalg B.
\]
Given this simplicial algebra, we can consider its normalization $\text{N}(\beta_\bullet(A,C,B))$, which is once again an algebra over $\mathpzc{E}_{\text{st}}^\dagger$ via the shuffle map (see~\cite{KrizMay} for the normalization of a simplicial algebra and its algebra structure). Now, regarding $A \amalg_C B$ as a constant simplicial algebra, the canonical maps $A \amalg B \to A \amalg_C B$ and $C \to A \amalg_C B$ induce a map of simplicial algebras $\beta_\bullet(A, C, B) \to A \amalg_C B$ and therefore a map of algebras on their normalizations, $\text{N}(\beta_\bullet(A,C,B)) \to A \amalg_C B$.

\begin{Proposition}\label{prop:pushout_normalization}
Given a diagram $A \leftarrow C \rightarrow B$ of algebras over $\mathpzc{E}_{\emph{st}}^\dagger$, if each of $A$ and $B$ and $C$ are cofibrant and $C \to B$ is a cofibration, then the canonical map
\[
\emph{N}(\beta_\bullet(A,C,B)) \to A \amalg_C B
\]
is a quasi-isomorphism.
\end{Proposition}

In order to prove this result, we first need two lemmas.

\begin{Lemma}\label{lem:norm_fin_flat}
Given a diagram $A \leftarrow C \rightarrow B$ of algebras over $\mathpzc{E}_{\emph{st}}^\dagger$, if each of $A$, $B$ and $C$ is cofibrant, the normalization $\emph{N}(\mathpzc{U}^{\beta_\bullet(A,C,B)}(j))$ is semi-flat as a dg right $\F_p[\Sigma_j]$-module.
\end{Lemma}

\begin{proof}
Given a dg left $\F_p[\Sigma_j]$-module $X$, we have a natural isomorphism
\[
\text{N}(\mathpzc{U}^{\beta_\bullet(A,C,B)}(j)) \otimes_{\Sigma_j} X \cong \text{N}(\mathpzc{U}^{\beta_\bullet(A,C,B)}(j) \otimes_{\Sigma_j} X).
\]
The required flatness now follows from Lemma~\ref{lem:Estenvopsflat}.
\end{proof}

\begin{Lemma}\label{lem:Estweksforpowers}
Let $n \ge 0$, $P$ and $Q$ be dg right $\F_p[\Sigma_n]$-modules which are semi-flat, and let $P \to Q$ be a quasi-isomorphism. Then, for all dg $\F_p$-modules $X$, the induced map
\[
P \otimes_{\Sigma_n} X^{\otimes n} \to Q \otimes_{\Sigma_n} X^{\otimes n}
\]
is a quasi-isomorphism.
\end{Lemma}

\begin{proof}
We shall demonstrate this for finite $X$ (where finiteness over $\F_p$ means bounded above and below and of finite dimension in each degree). The case of a general $X$ then follows since any $X$ is a filtered colimit of its finite subcomplexes and filtered colimits commute with finite tensor powers and tensor products. In fact, we will show that, for any finite dg right $\F_p[\Sigma_n]$-module $Z$, the natural map
\[
P \otimes_{\Sigma_n} Z \to Q \otimes_{\Sigma_n} Z
\]
is a quasi-isomorphism. To see this, let $Z_{\text{cof}} \to Z$ be a cofibrant approximation of $Z$, in the projective Quillen model structure on dg right $\F_p[\Sigma_n]$-modules, and then consider the following commutative square:

\begin{center}
\begin{tikzpicture}[node distance = 1.5cm]
\node(A){$P \otimes_{\Sigma_n} Z$};
\node[below of = A](C){$P \otimes_{\Sigma_n} Z_{\text{cof}}$};
\node[right of = A, xshift = 1cm](B){$Q \otimes_{\Sigma_n} Z$};
\node[below of = B, yshift = 0mm](D){$Q \otimes_{\Sigma_n} Z_{\text{cof}}$};
	
\draw[->] (A) -- (B) node[midway,anchor=south]{};
\draw[->] (C) -- (A) node[midway,anchor=east]{};
\draw[->] (C) -- (D) node[midway,anchor=north]{};
\draw[->] (D) -- (B) node[midway,anchor=west]{};
\end{tikzpicture}
\end{center}

The bottom horizontal map is then a map
\[
P \otimes_{\Sigma_n}^{\mathbb{L}} Z \to Q \otimes_{\Sigma_n}^{\mathbb{L}} Z
\]
between the derived tensor products. To be precise, the derived functors are those of the functors $P \otimes_{\Sigma_n} -$ and $Q \otimes_{\Sigma_n} -$. Since, however, the two possible derived tensor products, achieved upon fixing one or the other variable, are naturally isomorphic, the above map can be identified with the image of $P \to Q$ under the derived functor of $- \otimes_{\Sigma_n}^{\mathbb{L}} Z$. As such, as $P \to Q$ is a quasi-isomorphism, the above map is also quasi-isomorphism. Now, if we can show that the vertical maps are also quasi-isomorphisms, we will have the desired result. The proofs for the two are identical, so we will describe just the one for the lefthand vertical map. Since $Z$ is bounded below, we can take $Z_{\text{cof}}$ to also be bounded below. Suppose that $Z$ is bounded above at degree $d$, in that $Z_{d'}$ is zero for $d' > d$. Consider the truncation $\tau_{\le d+1}Z_{\text{cof}}$ which is to say the complex which coincides with $Z_{\text{cof}}$ up to, and including, degree $d+1$, but is zero thereafter. More generally, we consider the truncations $\tau_{\le d+i}Z_{\text{cof}}$ for $i \ge 1$. Then we can write $Z_{\text{cof}}$ as the colimit of:
\[
\tau_{d+1}Z_{\text{cof}} \to \tau_{d+2}Z_{\text{cof}} \to \cdots .
\]
Moreover, since $Z$ is zero above degree $d$, we have maps $\tau_{\le d+i}Z_{\text{cof}} \to Z$, each of which is a quasi-isomorphism. Thus we get the following diagram.

\begin{center}
\begin{tikzpicture}[node distance = 1.5cm]
\node [] (A) {$\tau_{\le d+1}Z_{\text{cof}}$};
\node [right of = A,xshift=1.5cm] (B) {$\tau_{\le d+2}Z_{\text{cof}}$};
\node [right of = B,xshift=1.5cm] (E) {$\tau_{\le d+3}Z_{\text{cof}}$};
\node [right of = E,xshift=1.5cm] (F) {$\cdots$};
\node [below of = A] (C) {$Z$};
\node [below of = B] (D) {$Z$};
\node [right of = D,xshift=1.5cm] (G) {$Z$};
\node [right of = G,xshift=1.5cm] (H) {$\cdots$};

\draw [->] (A) -- (B) node[midway,anchor=south]{};
\draw [->] (B) -- (E) node[midway,anchor=south]{};
\draw [->] (E) -- (F) node[midway,anchor=south]{};
\draw [->] (A) -- (C) node[midway,anchor=west]{$\sim$};
\draw [->] (B) -- (D) node[midway,anchor=west]{$\sim$};
\draw [->] (E) -- (G) node[midway,anchor=west]{$\sim$};
\draw [->] (C) -- (D) node[midway,anchor=south]{};
\draw [->] (D) -- (G) node[midway,anchor=south]{};
\draw [->] (G) -- (H) node[midway,anchor=south]{};
\end{tikzpicture}
\end{center}

If we tensor this diagram with $P$, since $P$ is semi-flat, the vertical arrows remain quasi-isomorphisms. Moreover, the map induced on the colimits of the two cotowers in the resulting diagram is exactly the map $P \otimes_{\Sigma_n} Z_{\text{cof}} \to P \otimes_{\Sigma_n} Z$, so that we have our desired result, as sequential colimits are exact.
\end{proof}

\begin{proof}[Proof of Proposition~\ref{prop:pushout_normalization}]
We may assume, without loss of generality, that $A$ is a cell algebra and that $C \to B$ is a relative cell map. We shall in fact prove the more general fact that the map
\[
\text{N}(\mathpzc{U}^{\beta_\bullet(A,C,B)}(j)) \to \mathpzc{U}^{A \amalg_C B}(j)
\]
is a quasi-isomorphism for each $j \ge 0$. The desired result is the case $j=0$. Let
\[
B_0 \to B_1 \to B_2 \to \cdots
\]
be a factorization of $C \to B$ as a relative cell map, so that $B_0 = C$, and fix some choices $M_1, M_2, \dots$ for the cochain complexes which appear in the attachment squares. By passage to colimits, it suffices to show that, for all $n \ge 0$, the map
\[
\text{N}(\mathpzc{U}^{\beta_\bullet(A,C,B_n)}(j)) \to \mathpzc{U}^{A \amalg_C B_n}(j)
\]
is a quasi-isomorphism for each $j \ge 0$. Now, in the case where $n=0$, we get the map $\text{N}(\mathpzc{U}^{\beta_\bullet(A,C,C)}(j)) \to \mathpzc{U}^{A}(j)$. By a standard argument, the map of simplicial algebras $\beta_\bullet(A,C,C) \to A$ is a homotopy equivalence, and so, upon forming the arity $j$ parts of the enveloping operads, we have that the map $\mathpzc{U}^{\beta_\bullet(A,C,C)}(j) \to \mathpzc{U}^A(j)$ is a homotopy equivalence of simplicial dg right $\F_p[\Sigma_j]$-modules and, as simplicial homotopies induce chain homotopies on normalizations, upon taking normalizations, we get a chain homotopy equivalence of dg right $\F_p[\Sigma_j]$-modules. In particular, the map is a quasi-isomorphism, as desired. Now suppose that, for some $n \ge 1$, the map $\text{N}(\mathpzc{U}^{\beta_\bullet(A,C,B_{n-1})}(j)) \to \mathpzc{U}^{A \amalg_C B_{n-1}}(j)$ is a quasi-isomorphism for each $j \ge 0$. Recall the filtrations on the enveloping operads of cell algebras, for example, as in Section~\ref{sec:env_op}. The simplicial map $\mathpzc{U}^{\beta_\bullet(A,C,B_n)}(j) \to \mathpzc{U}^{A \amalg_C B_n}(j)$ is in fact a filtered map, in that, for each $m \ge 0$, we have an induced map $\text{F}_m\mathpzc{U}^{\beta_\bullet(A,C,B_n)}(j) \to \text{F}_m\mathpzc{U}^{A \amalg_C B_n}(j)$. We now take the normalization and consider the induced map on the strongly convergent spectral sequences associated to these filtrations. Recalling, from Section~\ref{sec:env_op}, the computations of the associated graded pieces corresponding to the filtrations on the enveloping operads, we have that the map on the $E^1$-terms consists of the following maps:
\[
\text{N}(\mathpzc{U}^{\beta_\bullet(A,C,B_{n-1})}(m+j)) \otimes_{\Sigma_m} M_n[1]^{\otimes m} \to \mathpzc{U}^{A \amalg_C B_{n-1}}(m+j) \otimes_{\Sigma_m} M_n[1]^{\otimes m}.
\]
By Lemma~\ref{lem:norm_fin_flat}, Lemma~\ref{lem:Estenvopsflat}, the inductive hypothesis and Lemma~\ref{lem:Estweksforpowers}, this map is a quasi-isomorphism for all $m \ge 0$ (to see this, first write $M_n[1]$, which is a sum of sphere complexes, as a filtered colimit of its finite subcomplexes). It follows that the map of spectral sequences is an isomorphism from $E^2$ onwards, and so the map $\text{N}(\mathpzc{U}^{\beta_\bullet(A,C,B_{n})}(j)) \to \mathpzc{U}^{A \amalg_C B_{n}}(j)$ is a quasi-isomorphism. This completes the induction.
\end{proof}

Now, with the help of the computation of the pushout in Proposition~\ref{prop:pushout_normalization}, we can now prove the desired result.

\begin{Proposition}\label{prop:pushouts_alg}
Given a diagram $A \leftarrow C \rightarrow B$ of algebras over $\mathpzc{E}_{\emph{st}}^\dagger$, if each of $A$, $B$ and $C$ are cofibrant, and $C \to B$ is a cofibration, then we have that
\[
A \amalg_C B \sim A \oplus_C B.
\]
\end{Proposition}

\begin{proof}
Let $\beta^{\text{dg}}_\bullet(A,C,B)$ denote the bar construction in dg modules, so that, in simplicial degree $n$, we have
\[
\beta_n^{\text{dg}}(A, C, B) := A \oplus \underbrace{C \oplus \cdots \oplus C}_{n \: \text{factors}} \oplus B.
\]
Then we have a composite quasi-isomorphism
\[
\text{N}(\beta^{\text{dg}}_\bullet(A,C,B)) \overset{\sim}\to \text{N}(\beta_\bullet(A,C,B)) \overset{\sim}\to A \amalg_C B.
\]
Here the first map is a quasi-isomorphism by Proposition~\ref{prop:coproducts_of_algebras}, and the second is a quasi-isomorphism by Proposition~\ref{prop:pushout_normalization}. Moreover, since cofibrations of algebras, being retracts of relative cell maps, are necessarily cofibrations of complexes (in the standard projective model structure on complexes), we have a natural quasi-isomorphism $\text{N}(\beta^{\text{dg}}_\bullet(A,C,B)) \sim A \oplus_C B$, and this gives us the desired result.
\end{proof}


\chapter{A Spectral Cochains Adjunction}

Hitherto, we have dealt with algebras over the stable Barratt-Eccles operad $\mathpzc{E}_{\text{st}}^\dagger$ at a general level. In this chapter, we shall consider an application of the stable Barratt-Eccles operad, and show that cochains on spectra yield examples of such algebras.

\section{Spectra and Their Model Structure}

We begin by fixing what it is that we mean by a spectrum, which is to say we need to fix a model for spectra. We adopt the following model. A spectrum\index{spectra} $E$ is a sequence of based simplicial sets $E_0, E_1, E_2, \dots$ together with a collection of structure maps $\rho_n \colon \Sigma E_n \to E_{n+1}$, where the suspension is the Kan suspension as defined in Definition~\ref{def:kan_susps}, or equivalently, structure maps $\sigma_n \colon E_n \to \Omega E_{n+1}$, where the loop space is the Moore loop space as defined in Definition~\ref{def:moore_loopings}. A map of spectra $f \colon E \to F$ is given by a collection of maps $f_n \colon E_n \to F_n$ which are compatible with the structure maps. We thus have a category of spectra, and we denote this category by $\mathsf{Sp}$. \\
	
We now also make some standard definitions regarding stable homotopy groups. Given a spectrum $E$, for $i \in \Z$, the $i^{\text{th}}$ \textit{stable homotopy group}\index{stable homotopy group} of $E$ is the colimit $\pi_i^{\text{st}}(E) := \text{colim}_{k \ge 0} \pi_{i+k}(|E_k|)$; here, given $k \ge 0$, the map $\pi_{i+k}(|E_k|) \to \pi_{i+k+1}(|E_{k+1}|)$ is that which sends the class of a based map $\Sph^{i+k}_{\text{top}} \to |E_k|$ to the class of the composite $\Sph^{i+k+1}_{\text{top}} \cong \Sigma \Sph^{i+k}_{\text{top}} \to |\Sigma E_k| \to |E_{k+1}|$ where we use the fact that, upon geometric realization, the Kan suspension can be identified, up to natural isomorphism, with the usual topological suspension, as shown for example in Proposition 2.16 in~\cite{MarcStephan}. Moreover, given a map $f \colon E \to F$ of spectra, we call it a \textit{stable weak homotopy equivalence} if the induced maps $\pi_i^{\text{st}}(E) \to \pi_i^{\text{st}}(F)$ are isomorphisms for each $i \in \Z$. 

\begin{Proposition}\label{prop:spec_mod_str}
\index{spectra!Quillen model structure on}On $\mathsf{Sp}$, there is a Quillen model structure such that:
\begin{itemize}
	\item[(i)] A map $f \colon E \to F$ is a weak equivalence if and only if it is a stable weak homotopy equivalence.
	\item[(ii)] A map $f \colon E \to F$ is a cofibration if and only if the map $f_0 \colon E_0 \to F_0$ and the maps $E_{n+1} \, \amalg_{\Sigma E_n} \, \Sigma F_n \to F_{n+1}$, for $n \ge 0$, are cofibrations of based simplicial sets; they are, in particular, levelwise cofibrations.
	\item[(iii)] A map $f \colon E \to F$ is a fibration if and only if it has the right lifting property with respect to maps which are both cofibrations and weak equivalences; they are, in particular, levelwise fibrations.
\end{itemize}
\end{Proposition}

\begin{proof}
All except that the cofibrations are levelwise monomorphisms is immediate from Theorem 2.29 in~\cite{MarcStephan}. To see that cofibrations are monomorphisms, let $f \colon E \to F$ be a cofibration of spectra $E$ and $F$. By definition, $f_0$ is a monomorphism, and moreover, so is $E_{1} \, \lamalg_{\Sigma E_0} \, \Sigma F_0 \to F_{1}$. Consider the following pushout square.
\begin{center}
\begin{tikzpicture}[node distance = 1.5cm]
\node(A){$\Sigma E_0$};
\node[below of = A](C){$E_{1}$};
\node[right of = A, xshift = 1cm](B){$\Sigma F_0$};
\node[below of = B, yshift = 0.5mm](D){$E_{1} \, \lamalg_{\Sigma E_0} \, \Sigma F_0$};
	
\draw[->] (A) -- (B) node[midway,anchor=south]{$\Sigma f_0$};
\draw[->] (A) -- (C) node[midway,anchor=east]{$\rho_0$};
\draw[->] (C) -- (D) node[midway,anchor=north]{$i$};
\draw[->] (B) -- (D) node[midway,anchor=west]{};

\begin{scope}[shift=($(A)!.2!(D)$)]
\draw +(0,-0.25) -- +(0,0)  -- +(0.25,0);
\end{scope}
\end{tikzpicture}
\end{center}
By Proposition~\ref{prop:susp_monos}, $\Sigma f_0$ is a monomorphism. Since in simplicial sets, pushouts of monomorphisms are once again monomorphisms, we see that the map $i$ is a monomorphism, and thus so is the composite $E_1 \to E_{1} \, \lamalg_{\Sigma E_0} \, \Sigma F_0 \to F_1$, which is exactly $f_1$. Repeating this argument, we see by induction that each $f_n \colon E_n \to F_n$ is a monomorphism.
\end{proof}

\begin{Remark}\label{rmk:comparison_comb_sp}
\index{spectra!relation to combinatorial spectra}As per Bousfield and Friedlander in~\cite[\S 2.5]{BousfieldFriedlander} and Stephan in~\cite{MarcStephan}, there is a Quillen equivalence
\[
\text{Sp} \colon \mathsf{Sp} \rightleftarrows \mathsf{CSp} \colon \text{Ps}
\]
between our category of spectra and $\mathsf{CSp}$, the category of Kan's combinatorial spectra, equipped with the model structure of Brown in~\cite{Brown}. \customendremark
\end{Remark}

We can single out some special spectra with some standard definitions, as follows. Given a spectrum $E$, say that it is an \textit{$\Omega$-spectrum}\index{spectra!$\Omega$-spectra} if each $E_n$ is a Kan complex and the maps $\sigma_n$ are weak homotopy equivalences of based simplicial sets. Given a spectrum $E$, say that it is a \textit{strict $\Omega$-spectrum}\index{spectra!strict $\Omega$-spectra} if each $E_n$ is a Kan complex and the maps $\sigma_n$ are isomorphisms of based simplicial sets. Given a spectrum $E$, say that it is \textit{$\Sigma$-cofibrant}\index{spectra!$\Sigma$-cofibrant spectra} if the maps $\rho_n$ are cofibrations of based simplicial sets.

\begin{Remark}\label{rmk:htpy_grps_fib_sp}
Given an $\Omega$-spectrum $E$, we clearly have that $\pi_i^{\text{st}}(E) \cong \pi_i(E_0)$ for $i \ge 0$, and $\pi_i^{\text{st}}(E) \cong \pi_0(E_{|i|})$ for $i < 0$. \customendremark
\end{Remark}

\begin{Proposition}\label{prop:fibcofibsp}
We have the following:
\begin{itemize}
	\item[(i)] The fibrant spectra are exactly the $\Omega$-spectra. The cofibrant spectra are exactly the $\Sigma$-cofibrant spectra. The strict $\Omega$-spectra are bifibrant.
	\item[(ii)] A map between $\Omega$-spectra is a weak equivalence if and only if it is a levelwise weak homotopy equivalence of based simplicial sets.
	\item[(iii)] A map between strict $\Omega$-spectra is a fibration if and only if it is levelwise fibration of based simplicial sets.
\end{itemize}
\end{Proposition}

\begin{proof}
(i): The case of fibrant objects follows from Theorem 2.29 in~\cite{MarcStephan}. For the cofibrant objects, note that a map $f \colon E' \to E$ is a cofibration if $E'_0 \to E_0$ and $E'_{n+1} \, \lamalg_{\Sigma E'_n} \, \Sigma E_n \to E_{n+1}$ are cofibrations. Taking the $E'_n$ to be $*$, we are left with the map $* \to E_0$ and the maps $\Sigma E_n \to E_{n+1}$. The former is of course always a cofibration of based simplicial sets. Now consider the case of strict $\Omega$-spectra. If each map $E_n \to \Omega E_{n+1}$ is an isomorphism, the maps $\Sigma E_n \to E_{n+1}$ are monomorphisms as they may be written as $\Sigma E_n \to \Sigma \Omega E_{n+1} \to E_{n+1}$ where the first map is an isomorphism and the second is a monomorphism by part (iii) of Proposition~\ref{prop:SigmaOmegadj}. \\

(ii): See Lemma 2.28 in~\cite{MarcStephan}. \\

(iii): Let $E \to F$ be a levelwise fibration between strict $\Omega$-spectra $E$ and $F$. Consider the adjunction between spectra and Kan's combinatorial spectra in Remark~\ref{rmk:comparison_comb_sp} above. It is immediate from the definitions (in, e.g., \cite{MarcStephan}) that if $E$ is a strict $\Omega$-spectra, then the unit of adjunction $E \to \text{Ps}\,\text{Sp}E$ is an isomorphism. Thus the induced map $\text{Ps}\,\text{Sp}E$ is a levelwise fibration. By Proposition 3.18 in~\cite{MarcStephan}, given a map $f$ between combinatorial spectra, $\text{Ps}(f)$ is a levelwise fibration if and only if it is a fibration. Thus $\text{Ps}\,\text{Sp}E$ is a fibration. It follows that the map $E \to F$ is itself a fibration, as desired.
\end{proof}

Next, let us consider some examples of spectra, under our model for spectra. The most obvious is that of suspension spectra, which are spectra freely generated on a space. Given a based simplicial set $S$, the \textit{suspension spectrum} $\Sigma^\infty S$ is defined by setting $(\Sigma^\infty S)_n = \Sigma^nS$, and the structure maps, in the form via suspensions, are identities. More generally, for each $n \ge 0$, we define $\Sigma^{\infty - n}X$ by setting $(\Sigma^{\infty - n}X)_m = \Sigma^{m-n}X$ for $m \ge n$, and $*$ for $m < n$ and the structure maps to be the obvious ones in suspension form. If $n = 0$, we recover $\Sigma^\infty S$. Note that the structure maps $\rho_n \colon \Sigma (\Sigma^\infty S)_n \to (\Sigma^\infty S)_{n+1}$ for the corresponding suspension spectrum $\Sigma^\infty$ are identities. It follows that the structure maps $\sigma _n \colon (\Sigma^\infty S)_n \to \Omega (\Sigma^\infty S)_{n+1 }$ are components of the unit of the $(\Sigma, \Omega)$-adjunction and so are, by Proposition~\ref{prop:SigmaOmegadj}, isomorphisms. The component simplicial sets $\Sigma^n S$, however, are not necessarily Kan complexes, even if $S$ is (e.g., consider the case of the $0$-sphere), and so the suspension spectrum is not necessarily a strict $\Omega$-spectrum, or an $\Omega$-spectrum at all for that matter. \\

Now let us consider Eilenberg-MacLane spectra in our model. Let $A$ be any abelian group. Then a standard model for the Eilenberg-MacLane space $\text{K}(A,n)$, for each $n \ge 0$, as a simplicial set, is the simplicial set whose $d$-simplices are given by the cocycles $\text{Z}^n(\Delta_d; A)$. Consider the simplicial set $\text{K}(A,n)$ as a based simplicial set, with zero as the basepoint. The based simplicial sets $\text{K}(A,0), \text{K}(A,1), \text{K}(A,2), \dots$ then assemble into a spectrum, and this is the \textit{Eilenberg-MacLane spectrum}\index{Eilenberg-MacLane spectra} $\text{H}A$. To see this, we need structure maps
\[
\text{K}(A,n) \to \Omega \text{K}(A,n+1).
\]
These are as follows. Let $\alpha$ be an $n$-cocycle on $\Delta_d$. Then, we may act on the chains on $\Delta_{d+1}$ by stipulating that, given a simplex $[n+1] \to [d+1]$ in $\Delta_{d+1}$, if no entry maps to zero, we send it to zero, or if exactly one entry maps to zero, we drop $0$ from both the source and target and then reindex to get a map $[n] \to [d]$, a simplex in $\Delta_d$, and then act by $\alpha$. (We needn't concern ourselves with the case of maps $[n+1] \to [d+1]$ which send more than one entry to zero as those yield degenerate simplices of $\Delta_{d+1}$.) An easy check shows that this action on chains defines an $(n+1)$-cocycle $\beta$ on $\Delta_{d+1}$, that this cocycle lies in $\Omega \text{K}(A,n+1)$ and moreover that the assignment $\alpha \mapsto \beta$ yields a map of based simplicial sets $\text{K}(A,n) \to \Omega \text{K}(A,n+1)$, as desired. \\

We can also consider shifts of Eilenberg-MacLane spectra under our model. Let $A$ be an abelian group as above. Given any $k \in \Z$, the \textit{shifted, or generalized, Eilenberg-MacLane spectrum}\index{Eilenberg-MacLane spectra!shifted}\index{Eilenberg-MacLane spectra!generalized} $\Sigma^k\text{H}A$ is the spectrum where $(\Sigma^k\text{H}A)_n = \text{K}(A,k+n)$. Here, for $n < 0$, we interpret $\text{K}(A,n)$ to be $*$, by which we mean the based $\Delta_0$. In general, we have that $(\Sigma^k\text{H}A)_{n,d} = \text{Z}^{k+n}(\Delta_d; A)$.

\begin{Proposition}\label{prop:EM_spec_fib}
Given any abelian group $A$, the Eilenberg-MacLane spectrum $\emph{H}A$, and more generally the shifted Eilenberg-MacLane spectrum $\Sigma^k\emph{H}A$ for any $k \in \Z$, is a strict $\Omega$-spectrum, and so is bifibrant.
\end{Proposition}

\begin{proof}
The case for the shifted Eilenberg-MacLane spectra follows immediately from that of the unshifted $\text{H}A$. Moreover, we only need to verify that $\text{H}A$ is a strict $\Omega$-spectrum as the bifibrancy then follows by Proposition~\ref{prop:fibcofibsp}. That each $(\text{H}A)_n$ is a Kan complex follows by the fact that it is the underlying simplicial set of a simplicial group. It then remains to show that the structure maps $\text{K}(A,n) \to \Omega \text{K}(A,n+1)$ above are isomorphisms. First, note that the graded pieces are in fact abelian groups and that the structure maps are clearly maps of abelian groups. Thus, for injectivity, we can simply check that only zero maps to zero. Consider some $n$-cocycle on $\Delta_p$ which maps to the zero $(n+1)$-cocycle on $\Delta_{p+1}$. Given $[n] \to [p]$, abut $0 \mapsto 0$ (where by ``abut $0 \mapsto 0$'', we mean construct the map which sends $0$ to $0$ and otherwise, for $i \ge 1$, sends $i$ to one more than the image of $i-1$ under the provided map). On the resulting map $[n+1] \to [p+1]$ the image cocycle acts by the original one. Thus the original one must be zero. This demonstrates injectivity. Now we show surjectivity. Consider some $(n+1)$-cocycle on $\Delta_{p+1}$. We define an $n$-cocycle on $\Delta_p$ as follows: given any $[n] \to [p]$, abut $0 \mapsto 0$ and then act by the given cocycle. This is indeed a cocycle: given $[n+1] \to [p]$, if we take faces then abut $0 \mapsto 0$ we get the same as first abutting to $[n+2] \to [p+1]$ and then taking the faces $d_i$ for $1 \le i \le n+2$, so that we must be getting the same final result as if we acted upon $d_0$ of the abutment to $[n+2] \to [p+1]$ by the original cocycle, but then this will map nothing to $0$ (since in the abutment only $0$ mapped to $0$), and thus the final result will be zero, as desired. Now we note that our given $(n+1)$-cocycle on $\Delta_{p+1}$ is exactly the image of this newly constructed $n$-cocycle on $\Delta_p$: it certainly is if exactly one entry maps to $0$; moreover, if no entry maps to $0$, it must be mapped to $0$ by our cocycle due to the ``$d_0 = *$'' condition, and we can ignore the cases where more than one entry maps to zero since we are taking normalized cochains.
\end{proof}

\begin{Remark}\label{rmk:EM_htpy_grps}
As a result of Proposition~\ref{prop:EM_spec_fib} above and Remark~\ref{rmk:htpy_grps_fib_sp}, we have that $\pi_0^{\text{st}}(\text{H}A) \cong \pi_0(\text{K}(A,0)) \cong A$, whereas $\pi_i^{\text{st}}(\text{H}A) \cong \pi_i(\text{K}(A,0)) \cong *$ for $i > 0$ and $\pi_i^{\text{st}}(\text{H}A) \cong \pi_0(\text{K}(A,|i|)) \cong *$ for $i < 0$. More generally, similar case by case considerations show that $\pi_k^{\text{st}}(\Sigma^k\text{H}A) \cong \pi_k(\text{K}(A,k)) \cong A$ whereas $\pi_i^{\text{st}}(\Sigma^k\text{H}A) \cong *$ for $i \neq k$. \customendremark
\end{Remark}

\section{Spectral Cochains as Algebras Over the Stable Operads}\label{subsec:spec_cochains}

Our goal in this section is to construct explicit models for the mod $p$ spectral (co)chains for our model of spectra and show that these (co)chains possess an algebraic structure given by a (co)action of the stable Barratt-Eccles operad. Let $E$ be a spectrum, with structure maps $\rho_n \colon \Sigma E_n \to E_{n+1}$. If we apply the mod $p$ chains functor, we get maps $\text{C}_\bullet(E_n;\F_p)[1] \cong \text{C}_\bullet(\Sigma E_n;\F_p) \to \text{C}_\bullet(E_{n+1};\F_p)$ where the first map is that which is in Proposition~\ref{prop:n_r_chains_susp_X}. Equivalently, we have a map
\[
\text{C}_\bullet(E_n;\F_p) \to \text{C}_\bullet(E_{n+1};\F_p)[-1].
\] 
Moreover, upon applying the dualization operator $(-)^\vee$ from Section~\ref{sec:nots_convs}, we get maps $\text{C}_\bullet(E_{n+1};\F_p)^\vee[1] = \text{C}_\bullet(E_{n+1};\F_p)[-1]^\vee \to \text{C}_\bullet(E_n;\F_p)^\vee$, and so, upon applying the reindexing operator $(-)^{\dagger}$, and moving the shift from the source to the target, we get maps
\[
\text{C}^\bullet(E_{n+1};\F_p)[-1] \to \text{C}^\bullet(E_n;\F_p).
\]
The model for the mod $p$ spectral (co)chains is then as follows. Let $E$ be a spectrum. The chains\index{spectra!chains on} on $E$ with coefficients in $\F_p$, denoted $\text{C}_\bullet(E;\F_p)$, are as follows:
\[
\text{C}_\bullet(E;\F_p) := \text{colim}(\text{C}_\bullet(E_0;\F_p) \to \text{C}_\bullet(E_{1};\F_p)[-1] \to \text{C}_\bullet(E_{2};\F_p)[-2] \to \cdots).
\]
The cochains\index{spectra!cochains on} on $E$ with coefficients in $\F_p$, denoted $\text{C}^\bullet(E;\F_p)$, are as follows:
\[
\text{C}^\bullet(E;\F_p) := \text{lim}(\cdots \to \text{C}^\bullet(E_2;\F_p)[-2] \to \text{C}^\bullet(E_1;\F_p)[-1] \to \text{C}^\bullet(E_0;\F_p)).
\]

\begin{Remark}\label{rmk:spec_chain_notation}
We can relate the (co)chains above to those on combinatorial spectra. We have that the (co)chains on $E$ as described above are exactly the (co)chains, in the usual sense, on the associated combinatorial spectrum $\text{Ps}(E)$, where $\text{Ps}$ is as in Remark~\ref{rmk:comparison_comb_sp}. Moreover, we can describe the spectral (co)chains above more explicitly. Let $E$ be a spectrum, with structure maps $\rho_n \colon \Sigma E_n \to E_{n+1}$. Given a $d$-simplex $x$ in $E_n$, we have a corresponding $(d+1)$-simplex $\Sigma x$ in $\Sigma E_n$ (the notation $\Sigma x$ here is as in Definition~\ref{def:susp_simps}), and thus, upon applying $\rho$, a $(d+1)$-simplex $\rho(\Sigma x) \in E_{n+1}$. Pictorially:
\[
x \in (E_n)_d \leadsto \Sigma x \in (\Sigma E_n)_{d+1} \leadsto \rho(\Sigma x) \in (E_{n+1})_{d+1}.
\]
We then clearly have that
\[
\text{C}_\bullet(E;\F_p) = \bigoplus_{n \ge 0} \text{C}_\bullet(E_n;\F_p)[-n] \Big/ (x - \rho(\Sigma x)).
\]
To be even more explicit, an easy check, using Proposition~\ref{prop:faces_of_susp_simps}, shows that $\text{C}_d(E;\F_p)$ is the free $\F_p$-module on $\amalg_{e - n = d} (E_n)_e$, modulo the submodule generated by the basepoint, the degenerate simplices and the terms $x - \rho(\Sigma x)$ where $x \in (E_n)_e$ for some $n,e$ such that $e - n = d$ (and so then $\rho(\Sigma x) \in (E_{n+1})_{e+1}$, and we also have that $(e+1)-(n+1) = d$). To keep the bookkeeping precise, in the future, given an element $x \in \amalg_{e - n = d} (E_n)_e$, we shall let $[n,e,x]$ denote the corresponding element of the spectral chains $\text{C}_d(E;\F_p)$. Also, on the cochains: an easy degreewise check shows that the internal hom complex functor $\text{F}(-,\F_p[0])$ converts the colimit appearing in the definition of the chains into a limit, and it then follows that the cochains are exactly the cochain complex formed by application of $(-)^\dagger \circ (-)^\vee$ to the chains. \customendremark
\end{Remark}

Next, we wish to show that the spectral cochains yield algebras over the stable Barratt-Eccles cochain operad $\mathpzc{E}_{\text{st}}^\dagger$. We begin by showing that chains on spectra naturally form coalgebras over the stable Eilenberg-Zilber chain operad $\mathpzc{Z}_{\text{st}}$.

\begin{Proposition}\label{prop:spec_chains_coalg}
Let $E$ be a spectrum. The chains $\emph{C}_\bullet(E;\F_p)$ naturally form a coalgebra over the stable Eilenberg-Zilber chain operad $\mathpzc{Z}_{\emph{st}}$.
\end{Proposition}

\begin{proof}
We wish to produce a coaction of the stable Eilenberg-Zilber operad on the chains, which, for brevity, we shall denote by $\mathrm{C}_\bullet(E)$. Fix $n \ge 0$. We then want a map
\[
\mu \colon \mathpzc{Z}_{\text{st}}(n) \otimes \text{C}_\bullet(E) \to \text{C}_\bullet(E)^{\otimes n}.
\]
The idea for this coaction is as follows. As in Section~\ref{subsec:stable_ez_operad}, an element $\alpha \in \mathpzc{Z}_{\text{st}}(n)$ is a sequence $(\alpha_0, \alpha_1, \dots)$ of natural transformations $\alpha_k \colon \mathrm{C}_\bullet(-) \Rightarrow \mathrm{C}_\bullet(-)^{\otimes n}$ over spaces. The chains on $E$ are a colimit of the chains on the spaces $E_k$; if a spectral chain $c$ on $E$ is a chain on $E_0$, we act on it by $\alpha_0$; if it is a chain on $E_1$, we act on it by $\alpha_1$, and so on. \\

Now, in order to construct the above map precisely, we will construct a bilinear map $\bar\mu \colon \mathpzc{Z}_{\text{st}}(n) \times \text{C}_\bullet(E) \to \text{C}_\bullet(E)^{\otimes n}$ which preserves the degrees and then check that the resulting map on the tensor product commutes with the differentials. To do this, we first produce a map $\bar{\bar\mu} \colon \mathpzc{Z}_{\text{st}}(n) \times \text{S}_\bullet(E) \to \text{C}_\bullet(E)^{\otimes n}$ where $\text{S}_\bullet(E) := \F_p[\amalg_{k,d} (E_k)_d]$ (with degrees so that $\text{C}_\bullet(E)$ is a quotient of $\text{S}_\bullet(E)$,  as in Remark~\ref{rmk:spec_chain_notation}). Let $\alpha = (\alpha_0,\alpha_1,\dots) \in \mathpzc{Z}_{\text{st}}(n)$ be of degree $d$ and let $x \in (E_k)_e$, which is of degree $e-k$ in $\text{S}_\bullet(E)$. We set $\bar{\bar\mu}(\alpha,x) := \alpha_k(x)$; or more precisely, we set $\bar{\bar\mu}(\alpha,x)$ to be the image of $x$ under the composite
\[
\text{C}_\bullet(E_k) \overset{(\alpha_k)_{E_k}}\longrightarrow \text{C}_\bullet(E_k)^{\otimes n} \longrightarrow \text{C}_\bullet(E)^{\otimes n}.
\]
Next we extend this definition to all of $\mathpzc{Z}_{\text{st}}(n) \times \text{S}_\bullet(E)$ by linearity in the second variable, recalling that $\text{S}_\bullet(E)$ is the free $\F_p$-module on $\lamalg_{k,e} (E_k)_e$. The map is then clearly bilinear in both variables. We need to check that $\bar{\bar\mu}(\alpha,x)$ is of degree $d+e-k$. Recall, as in Section~\ref{subsec:stable_ez_operad}, that $\alpha_k$ is of degree $d+k(n-1)$. As an element of $\text{C}_\bullet(E_k)$, $x$ has degree $e$. Upon application of $(\alpha_k)_{E_k}$, we get an element of degree $e+d+k(n-1) = d+e-k+nk$. By the definition of spectral chains, the map $\mathrm{C}_\bullet(E_k) \to \mathrm{C}_\bullet(E)$ has degree $-k$, and so the map $\text{C}_\bullet(E_k)^{\otimes n} \longrightarrow \text{C}_\bullet(E)^{\otimes n}$ has degree $-nk$. It follows that $\bar{\bar\mu}(\alpha, x)$ is of degree $d+e-n+nk-nk = d+e-n$, as desired. \\

Now we show that our map $\mathpzc{Z}_{\text{st}}(n) \times \text{S}_\bullet(E) \to \text{C}_\bullet(E)^{\otimes n}$ descends to a map $\mathpzc{Z}_{\text{st}}(n) \times \text{C}_\bullet(E) \to \text{C}_\bullet(E)^{\otimes n}$. Suppose first that $x \in (E_k)_e$ is degenerate, the basepoint or a degeneracy of a basepoint. Then, it represents zero in $\mathrm{C}_\bullet(E_k)$, and so $\bar{\bar\mu}(\alpha,x)$ is zero for any $\alpha$. Next, consider an $x \in (E_k)_e$ and $\rho_k(\Sigma x) \in (E_{k+1})_{e+1}$. We need $\bar{\bar\mu}(\alpha, x) = \bar{\bar\mu}(\alpha, \rho_k(\Sigma x))$ for any $\alpha$. Thus, for any given $\alpha$, we need the image of $x$ under
\begin{equation}\label{eq:comp1}
\text{C}_\bullet(E_k) \overset{(\alpha_k)_{E_k}}\longrightarrow \text{C}_\bullet(E_k)^{\otimes n} \longrightarrow \text{C}_\bullet(E)^{\otimes n}
\end{equation}
to coincide with the image of $\rho_k(\Sigma x)$ under
\begin{equation}\label{eq:comp2}
\text{C}_\bullet(E_{k+1}) \overset{(\alpha_{k+1})_{E_{k+1}}}\longrightarrow \text{C}_\bullet(E_{k+1})^{\otimes n} \longrightarrow \text{C}_\bullet(E)^{\otimes n}.
\end{equation}
Consider the following diagram.
\begin{center}
\begin{tikzpicture}[node distance=1.5cm]
\node[](A){$\text{C}_\bullet(\Sigma E_k)$};
\node[below of = A](C){$\text{C}_\bullet(E_{k+1})$};
\node[right of = A, xshift = 2.5cm](B){$\text{C}_\bullet(\Sigma E_k)^{\otimes n}$};
\node[below of = B, yshift = 0mm](D){$\text{C}_\bullet(E_{k+1})^{\otimes n}$};
\node[right of = D, xshift = 2cm](E){$\text{C}_\bullet(E)^{\otimes n}$};
\node[above of = A](F){$\mathrm{C}(E_k)$};
\node[above of = B](G){$\mathrm{C}(E_k)^{\otimes n}$};
	
\draw[->] (A) -- (B) node[midway,anchor=south]{$(\alpha_{k+1})_{\Sigma E_k}$};
\draw[->] (A) -- (C) node[midway,anchor=east]{$\text{C}_\bullet(\rho_k)$};
\draw[->] (C) -- (D) node[midway,anchor=north]{$(\alpha_{k+1})_{E_{k+1}}$};
\draw[->] (B) -- (D) node[midway,anchor=west]{$\text{C}_\bullet(\rho_k)^{\otimes n}$};
\draw[->] (D) -- (E);
\draw[->] (F) -- (G) node[midway,anchor=south]{$(\alpha_k)_{E_k}$};
\draw[->] (F) -- (A) node[midway,anchor=east]{$\mathrm{deg} \: 1 \cong$};
\draw[->] (G) -- (B) node[midway,anchor=west]{$\mathrm{deg} \: n \cong$};
\end{tikzpicture}
\end{center}
Here, the two vertical isomorphisms come from Proposition~\ref{prop:n_r_chains_susp_X}, the upper square commutes because $\alpha_k = \Psi(\alpha_{k+1})$ and the lower square commutes by naturality of $\alpha_{k+1}$. Starting with $x$ at the topleft corner, applying the sequence of maps given by $\downarrow, \downarrow, \rightarrow, \rightarrow$, we get the image of $\rho_k(\Sigma x)$ under the composite in (\ref{eq:comp2}). On the other hand, applying the sequence of maps given by $\rightarrow, \downarrow, \downarrow, \rightarrow$, we get the image of $x$ under the composite in (\ref{eq:comp1}). Thus the two images coincide, as desired. \\

We have now demonstrated that the map $\bar{\bar\mu} \colon \mathpzc{Z}_{\text{st}}(n) \times \text{S}_\bullet(E) \to \text{C}_\bullet(E)^{\otimes n}$ descends to a bilinear map $\bar\mu \colon \mathpzc{Z}_{\text{st}}(n) \times \text{C}_\bullet(E) \to \text{C}_\bullet(E)^{\otimes n}$ which preserves the degrees, and we let the associated map $\mathpzc{Z}_{\text{st}}(n) \otimes \text{C}_\bullet(E) \to \text{C}_\bullet(E)^{\otimes n}$ be $\mu$. To complete the construction of the coaction, it remains to check that $\mu$ commutes with the differentials. Consider $\alpha \in \mathpzc{Z}_{\text{st}}(n)_d$ and $x \in (E_k)_e$. Recall the notation $[k, e, x]$ introduced for spectral chains in Remark~\ref{rmk:spec_chain_notation}. We have $\partial(\alpha \otimes [k,e,x]) = \partial \alpha \otimes [k,e,x] + (-1)^d \alpha \otimes \partial [k,e,x] = (\partial \alpha_0, \partial \alpha_1, \dots) \otimes [k,e,x] + (-1)^d\alpha \otimes (\sum_i [k,e,d_i(x)])$ and this, under $\mu$, has image
\begin{align*}
\begin{split}
(\partial \alpha_k)_{E_k}(x) + (-1)^d \alpha_k(\sum_i d_i(x)) ={}& \partial_{\text{C}_\bullet(E)^{\otimes n}}(\alpha_k(x)) - (-1)^{d}\alpha_k(\sum_i d_i(x)) \\ & + (-1)^d\alpha_k(\sum_i d_i(x))
\end{split} \\
\begin{split}
={}& \partial_{\text{C}_\bullet(E)^{\otimes n}}(\alpha_k(x))
\end{split} \\
\begin{split}
={}& \partial_{\text{C}_\bullet(E)^{\otimes n}}(\mu(\alpha,x))
\end{split}
\end{align*}
as desired. We have now constructed the coaction maps $\mu$, and one can verify that the compatibility conditions required for an operad coaction are indeed satisfied.  \\

Finally, we demonstrate the functoriality. Let $E$ and $F$ be spectra and $f \colon E \to F$ a map of spectra. We then have an induced map $f_* \colon \text{C}_\bullet(E) \to \text{C}_\bullet(F)$ of chain complexes, and we need to check that this map is compatible with the coaction of $\mathpzc{Z}_{\text{st}}$. Consider $[k,e,x]$ in $\text{C}_\bullet(E)$. Applying $f_*$ and then $\mu^F \colon \mathpzc{Z}_{\text{st}}(n) \otimes \text{C}_\bullet(F) \to \text{C}_\bullet(F)^{\otimes n}$, we get $\alpha_k(f_k(x))$, which by naturality of $\alpha_k$ is equal to $f_k^{\otimes n}(\alpha_k(x))$, and this is exactly the result upon instead first applying $\mu^E$ and then $f_*^{\otimes n}$. This gives us the desired compatibility.
\end{proof}

\begin{Proposition}\label{prop:stabopactioncochains}
Given a spectrum $E$, the cochain complex $\emph{C}^\bullet(E; \F_p)$ is naturally an algebra over the stable Barratt-Eccles cochain operad $\mathpzc{E}_{\emph{st}}^\dagger$.
\end{Proposition}

\begin{proof}
By Propositions~\ref{prop:spec_chains_coalg}, the chains $\text{C}_\bullet(E;\F_p)$ are a coalgebra over $\mathpzc{Z}_{\text{st}}$. As in Section~\ref{sec:nots_convs}, if we apply $(-)^\vee$ to the chains, we get an algebra over $\mathpzc{Z}_{\text{st}}$. Thus, again as in Section~\ref{sec:nots_convs}, if we apply $(-)^\dagger \circ (-)^{\vee}$ to the chains, yielding the cochains as per Remark~\ref{rmk:spec_chain_notation}, we get an algebra over $\mathpzc{Z}_{\text{st}}^\dagger$. Finally, we get an $\mathpzc{E}_{\text{st}}^\dagger$-algebra structure by pulling back across the map $\mathpzc{E}_{\text{st}}^\dagger \to \mathpzc{Z}_{\text{st}}^\dagger$ constructed in Section~\ref{subsec:stable_barratt_eccles_operad}.
\end{proof}

As a result of the above, we can interpret cochains on spectra as a functor to algebras over $\mathpzc{E}_{\text{st}}^\dagger$:

\[
\text{C}^\bullet \colon \mathsf{Sp}^{\text{op}} \to \mathpzc{E}_{\text{st}}^\dagger\text{-}\mathsf{Alg}.
\]

\begin{Remark}
We can in fact view the action of the stable operad on spectral cochains in an iterative manner, as follows. By definition, we have that
\[
\text{C}_\bullet(E;\F_p) := \text{colim}(\text{C}_\bullet(E_0;\F_p) \to \text{C}_\bullet(E_{1};\F_p)[-1] \to \text{C}_\bullet(E_{2};\F_p)[-2] \to \cdots)
\]
\[
\text{C}^\bullet(E;\F_p) := \text{lim}(\cdots \to \text{C}^\bullet(E_2;\F_p)[-2] \to \text{C}^\bullet(E_1;\F_p)[-1] \to \text{C}^\bullet(E_0;\F_p)).
\]
As is well-known in the unstable case, we have that $\text{C}_\bullet(E_0;\F_p)$ and $\text{C}^\bullet(E_0;\F_p)$ form, respectively, a coalgebra over $\mathpzc{E}$ and an algebra over $\mathpzc{E}^\dagger$. Thus, by Proposition~\ref{prop:chainopsuspalg}, we have that the second terms, $\text{C}_\bullet(E_0;\F_p)[-1]$ and $\text{C}^\bullet(E_0;\F_p)[-1]$ form, respectively, a coalgebra over $\Sigma\mathpzc{E}$ and an algebra over $\Sigma\mathpzc{E}^\dagger$. Similarly, we have that the third terms, $\text{C}_\bullet(E_0;\F_p)[-2]$ and $\text{C}^\bullet(E_0;\F_p)[-2]$ form, respectively, a coalgebra over $\Sigma^2\mathpzc{E}$ and an algebra over $\Sigma^2\mathpzc{E}^\dagger$. In the limit, we get (co)algebra structures over the stable operads. \customendremark
\end{Remark}

\begin{Remark}\label{stable_EZ_operad_as_spectral_ops}
\index{stable Eilenberg-Zilber operad!as natural transformations on spectral (co)chains}The Eilenberg-Zilber chain operad $\mathpzc{Z}$ was defined in Definition~\ref{def:ez_op} using natural transformations. In particular, in operad degree $n \ge 0$ and chain degree $d \in \Z$, we have that
\[
\mathpzc{Z}(n)_d := \left\{
\begin{array}{ll}
\text{degree $d$ natural transformations from $\mathrm{C}_{\bullet}(-)$ to $\mathrm{C}_{\bullet}(-)^{\otimes n}$} \\
\text{where $\mathrm{C}_{\bullet}(-)$ and $\mathrm{C}_{\bullet}(-)^{\otimes n}$ are viewed as functors $\mathsf{Spc} \to \mathsf{Gr}_{\F_p}$}
\end{array}
\right\}.
\]
The construction in Proposition~\ref{prop:spec_chains_coalg} allows a similar interpretation of the stable Eilenberg-Zilber chain operad, using spectral chains. We can define a chain operad $\tilde{\mathpzc{Z}}_{\mathrm{st}}$ by setting, in operadic degree $n$ and chain degree $d \in \Z$, that
\[
\tilde{\mathpzc{Z}}_{\mathrm{st}}(n)_d = \left\{
\begin{array}{ll}
\text{degree $d$ natural transformations from $\mathrm{C}_{\bullet}(-)$ to $\mathrm{C}_{\bullet}(-)^{\otimes n}$} \\
\text{where $\mathrm{C}_{\bullet}(-)$ and $\mathrm{C}_{\bullet}(-)^{\otimes n}$ are viewed as functors $\mathsf{Sp} \to \mathsf{Gr}_{\F_p}$}
\end{array}
\right\}.
\]
We find that the stable Eilenberg-Zilber chain operad is, up to isomorphism, exactly this operad; that is, 
$\mathpzc{Z}_{\mathrm{st}} \cong \tilde{\mathpzc{Z}}_{\mathrm{st}}$. The correspondence is as follows. Given $\alpha \in \mathpzc{Z}_{\mathrm{st}}(n)_d$, for each spectrum $E$, using the coaction $\mathpzc{Z}_{\mathrm{st}}(n) \otimes \mathrm{C}_\bullet(E) \to \mathrm{C}_\bullet(E)^{\otimes n}$ defined in Proposition~\ref{prop:spec_chains_coalg}, we get a degree $d$ map $\mathrm{C}_\bullet(E) \to \mathrm{C}_\bullet(E)^{\otimes n}$. These maps assemble into the desired natural transformation $\mathrm{C}_\bullet(-) \Rightarrow \mathrm{C}_\bullet(-)^{\otimes n}$. On the other hand, given $\beta \in \tilde{\mathpzc{Z}}_{\mathrm{st}}(n)_d$, we need to define a tuple $(\alpha_0, \alpha_1, \dots)$ where $\alpha_0$ is a degree $d$ natural transformation  $\mathrm{C}_\bullet(-) \Rightarrow \mathrm{C}_\bullet(-)^{\otimes n}$ over spaces, $\alpha_1$ is a degree $d-n+1$ natural transformation $\mathrm{C}_\bullet(-) \Rightarrow \mathrm{C}_\bullet(-)^{\otimes n}$ over spaces such that $\Psi(\alpha_1) = \alpha_0$, and so on (here, $\Psi$ is the stabilization map which occurs in Definition~\ref{def:stable_ez_operad}). For each $i \ge 0$, the transformation $\alpha_i$ is given by the maps $\beta_{\Sigma^{\infty - i}X} \colon \mathrm{C}_\bullet(\Sigma^{\infty - i}X) \to \mathrm{C}_\bullet(\Sigma^{\infty - i}X)^{\otimes n}$. These maps have degree $d$, and we note that $\mathrm{C}_\bullet(\Sigma^{\infty - i}X) = \mathrm{C}_\bullet(X)[-i]$, $\mathrm{C}_\bullet(\Sigma^{\infty - i}X)^{\otimes n} = \mathrm{C}_\bullet(X)^{\otimes n}[-in]$, so that the maps can equivalently be considered as maps $\mathrm{C}_\bullet(X) \to \mathrm{C}_\bullet(X)^{\otimes n}$ of degree $d + i - in$, as desired. An analogous interpretation can also be made for the stable Eilenberg-Zilber cochain operad $\mathpzc{Z}_{\mathrm{st}}^\dagger$. \customendremark
\end{Remark}

Now, by Proposition~\ref{prop:stabopactioncochains} and Proposition~\ref{prop:(co)hom_ops_stab_op}, the cohomologies $\text{H}^\bullet(E;\F_p)$ inherit operations $P^s$, and also $\beta P^s$ in the case $p > 2$. As in the case of spaces, as shown below, they satisfy an important property which does not hold in general. This property is of course well-known to hold; we include it as a demonstration of work with our explicit model.

\begin{Proposition}\label{prop:P01sp}
\index{generalized Steenrod operations!the operation $P^0$ on spectral cochains}Given a spectrum $E$, the operation $P^0$ acts by the identity on $\emph{H}^\bullet(E;\F_p)$.
\end{Proposition}

\begin{proof}
We shall outline the $p = 2$ case; a similar proof is possible in the $p > 2$ case. For brevity, let $\text{C}^\bullet(E)$ denote the mod $p$ cochains. The operation $P^0$ is computed via images under the map $\mathpzc{E}_{\text{st}}^\dagger(2) \otimes \text{C}^\bullet(E)^{\otimes 2} \to \text{C}^\bullet(E)$. From Definition~\ref{def:e_d^un_and_e_d^st}, recall the notations $e_d^{\text{un}}$ and $e_d^{\text{st}}$, for certain elements of $\mathpzc{E}^\dagger(2)$ and $\mathpzc{E}^\dagger_{\text{st}}(2)$, respectively. In particular, recall that
\begin{equation}\label{eqn:P0_e_un_vs_e_st}
e_{d}^{\mathrm{st}} := (e_d^{\mathrm{un}}, e_{d-1}^{\mathrm{un}} \cdot \tau, e_{d-2}^{\mathrm{un}}, e_{d-3}^{\mathrm{un}} \cdot \tau, \dots).
\end{equation}
Consider a cocycle $\alpha \in \mathrm{C}^{d}(E)$. Let $\beta \in \mathrm{C}^d(E)$ denote the image of $e^{\mathrm{st}}_{-d} \otimes \alpha \otimes \alpha$ under the map $\mathpzc{E}^\dagger_{\mathrm{st}}(2) \otimes \mathrm{C}^\bullet(E)^{\otimes 2} \to \mathrm{C}^\bullet(E)$. By definition of $P^0$, $P^0\alpha$ is given by the class of $\beta$. Thus we need to show that $\beta = \alpha$. Consider the $d$-dimensionsal chain in $\mathrm{C}_d(E)$ determined by a simplex $x \in (E_n)_{d'}$, where $d'-n = d$. As the $\mathpzc{E}_{\text{st}}^\dagger$-action on cochains is dual to an $\mathpzc{E}_{\text{st}}$-coaction on chains, $\beta(x)$ is given by $(\alpha \otimes \alpha)(y)$ where $y$ is the image of $e_{-d}^{\text{st}} \otimes x$ under the coaction map $\mathpzc{E}_{\text{st}}(2) \otimes \text{C}_\bullet(E) \to \text{C}_\bullet(E)^{\otimes 2}$. By the definition of the $\mathpzc{E}_{\text{st}}$-coaction on chains and by (\ref{eqn:P0_e_un_vs_e_st}) above, we have that, since $x$ lies in $E_n$, computation of $y$ reduces to a computation of the image of $e^{\mathrm{un}}_{-d-n}\tau^{\varepsilon} \otimes x$, where $\varepsilon = (1+(-1)^n)/2$, under the unstable coaction map $\mathpzc{E}(2) \otimes \mathrm{C}_\bullet(E_n) \to \mathrm{C}_\bullet(E_n)^{\otimes 2}$. An explicit computation of this coaction in the unstable case shows that $y = x \otimes x$. One possible method of doing this computation is as follows:
\begin{itemize}
	\item As $x \otimes x$ is invariant under the action of $\tau$, it suffices to demonstrate that the coaction is $x \otimes x$ in the case of $e^{\mathrm{un}}_{-d-n} \otimes x$. Pass, under the map $\mathpzc{E}^\dagger \to \mathpzc{M}^\dagger$ described in Section~\ref{subsec:stable_barratt_eccles_operad}, from $e_{-d-n}^{\mathrm{un}}$ to the corresponding surjection in the McClure-Smith operad. An easy check shows that this is the surjection $(d+n+2) \to (2)$ given by the sequence $(1212 \cdots)$. Let us denote this surjection by $f$. Note that, since $d'- n = d$, $f$ is a surjection from $(d'+2)$ to $(2)$. 
	\item We now need to compute $\mathrm{AW}(f)(x)$ where $\mathrm{AW}$ is as in Section~\ref{subsec:stable_MS_operad}. As per Definition~\ref{defn:seqcoop}, we need to compute a sum, over overlapping partitions of $\{0,1,\dots,d'\}$, of terms of the form $x' \otimes x''$, where both $x'$ and $x''$ are restrictions of $x$. Since, upon computing the image $y$, we shall be computing $\beta(x) = (\alpha \otimes \alpha)(y)$, and $\alpha$ is a $d$-cocycle, we need only compute the terms where both $x'$ and $x''$ are of dimension $d$. An easy check shows that there is only one overlapping partition, namely $\{0\}, \{0,1\}, \{1,2\}, \dots, \{d'-1, d'\}. \{d'\}$, for which this is true, and in this case, we have $x' = x'' = x$, so that $y = x \otimes x$, as desired.
\end{itemize}
Now, as $y = x \otimes x$, we have that $\beta(x) = (\alpha \otimes \alpha)(x \otimes x) = \alpha(x)^2 = \alpha(x)$, and so $\beta = \alpha$, as desired.
\end{proof}

\section{Change of Coefficients from $\F_p$ to $\overline{\F}_p$}

Hitherto, we have worked with coefficients in $\F_p$, whether it was with stable operads or with spectral cochains. In order to construct algebraic models of $p$-adic stable homotopy types however, we must pass to the algebraic closure $\overline{\F}_p$. We describe the necessary modifications in this section. First, we define a new operad, one over $\overline{\F}_p$.

\begin{Definition}\label{def:operadFpbar}
The operad $\widebar{\mathpzc{E}}_{\text{st}}^{\dagger}$, an operad in cochain complexes over $\overline{\F}_p$, is as follows:
\[
\widebar{\mathpzc{E}}_{\text{st}}^\dagger(n) = \mathpzc{E}_{\text{st}}^\dagger(n) \otimes_{\F_p[\Sigma_n]} \overline{\F}_p[\Sigma_n].
\]
\end{Definition}

We now have three tasks to complete, tasks which we completed in the case of the operad $\mathpzc{E}_{\text{st}}^\dagger$:
\begin{itemize}
	\item[(T1)] Show that one can do homotopy theory over $\widebar{\mathpzc{E}}_{\text{st}}^{\dagger}$.
	\item[(T2)] Compute the cohomology of free algebras over $\widebar{\mathpzc{E}}_{\text{st}}^{\dagger}$ and develop a theory of cohomology operations.
	\item[(T3)] Demonstrate homotopy additivity properties of $\widebar{\mathpzc{E}}_{\text{st}}^{\dagger}$.
\end{itemize}

Let us first consider (T1). Our goal is to show that, just like $\mathpzc{E}_{\text{st}}^\dagger$, the monad associated to $\widebar{\mathpzc{E}}_{\text{st}}^{\dagger}$ preserves quasi-isomorphisms, and moreover, that $\widebar{\mathpzc{E}}_{\text{st}}^{\dagger}$ is semi-admissible, in the sense of Definition~\ref{def:semiadmissibility}. As usual, we denote the monad, and also the free algebra functor, associated to the operad $\widebar{\mathpzc{E}}_{\text{st}}^{\dagger}$, by $\overline{\mathbf{E}}_{\text{st}}^\dagger$.

\begin{Lemma}\label{lem:Estbar_on_free_Fpbar_complexes}
For cochain complexes over $X$ over $\F_p$, we have a natural isomorphism
\[
\overline{\mathbf{E}}_{\mathrm{st}}^\dagger (\overline{\F}_p \otimes_{\F_p} X) \cong \overline{\F}_p \otimes_{\F_p} \mathbf{E}_{\mathrm{st}}^\dagger X.
\]
\end{Lemma}

\begin{proof}
We have:
\begin{align*}
\overline{\mathbf{E}}_{\mathrm{st}}^\dagger (\overline{\F}_p \otimes_{\F_p} X) &= \bigoplus_{n \ge 0} \widebar{\mathpzc{E}}_{\text{st}}^{\dagger}(n) \otimes_{\overline{\F}_p[\Sigma_n]} (\overline{\F}_p \otimes_{\F_p} X)^{\otimes_{\overline{\F}_p} n} \\
&\cong \bigoplus_{n \ge 0} (\mathpzc{E}_{\text{st}}^\dagger(n) \otimes_{\F_p[\Sigma_n]} \overline{\F}_p[\Sigma_n]) \otimes_{\overline{\F}_p[\Sigma_n]} (\overline{\F}_p[\Sigma_n] \otimes_{\F_p[\Sigma_n]} X^{\otimes_{\F_p} n}) \\
&\cong \bigoplus_{n \ge 0} \mathpzc{E}_{\text{st}}^\dagger(n) \otimes_{\F_p[\Sigma_n]} (\overline{\F}_p[\Sigma_n] \otimes_{\overline{\F}_p[\Sigma_n]} (\overline{\F}_p[\Sigma_n] \otimes_{\F_p[\Sigma_n]} X^{\otimes_{\F_p} n})) \\
&\cong \bigoplus_{n \ge 0} \mathpzc{E}_{\text{st}}^\dagger(n) \otimes_{\F_p[\Sigma_n]} (\overline{\F}_p[\Sigma_n] \otimes_{\F_p[\Sigma_n]} X^{\otimes_{\F_p} n}) \\
&\cong \bigoplus_{n \ge 0} \mathpzc{E}_{\text{st}}^\dagger(n) \otimes_{\F_p[\Sigma_n]} \overline{\F}_p[\Sigma_n] \otimes_{\F_p[\Sigma_n]} X^{\otimes_{\F_p} n} \\
&\cong \bigoplus_{n \ge 0} \mathpzc{E}_{\text{st}}^\dagger(n) \otimes_{\F_p} \overline{\F}_p \otimes_{\F_p} X^{\otimes_{\F_p} n} \\
&\cong \overline{\F}_p \otimes_{\F_p} \Big(\bigoplus_{n \ge 0} \mathpzc{E}_{\text{st}}^\dagger(n) \otimes_{\F_p}  X^{\otimes_{\F_p} n}\Big) \\
&= \overline{\F}_p \otimes_{\F_p} \mathbf{E}_{\mathrm{st}}^\dagger X.
\end{align*}
\end{proof}

\begin{Proposition}\label{prop:Ebarmonad}
We have the following:
\begin{itemize}
	\item[(i)] The monad corresponding to the operad $\widebar{\mathpzc{E}}_{\text{st}}^{\dagger}$ preserves quasi-isomorphisms.
	\item[(ii)] The operad $\widebar{\mathpzc{E}}_{\text{st}}^{\dagger}$ is semi-admissible, which is to say the category of algebras $\widebar{\mathpzc{E}}_{\emph{st}}^\dagger\text{-}\mathsf{Alg}$ possesses a Quillen semi-model structure where:
\begin{itemize}
	\item The weak equivalences are the quasi-isomorphisms.
	\item The fibrations are the surjective maps.
	\item The cofibrations are retracts of relative cell complexes, where the cells are the maps $\overline{\mathbf{E}}_{\emph{st}}^\dagger M \to \overline{\mathbf{E}}_{\emph{st}}^\dagger\emph{C}M$, where $M$ is a degreewise free $\overline{\F}_p$-complex with zero differentials.
\end{itemize}
\end{itemize}
\end{Proposition}

\begin{proof}
(i): As in the proof of Proposition~\ref{prop:MS_st_pres_w_eqs}, it suffices to demonstrate the preservation of quasi-isomorphisms $f \colon X \to Y$ where, for some index sets $I$ and $J$, $X = \bigoplus_{i \in I} \Sph^{n_i}_{\overline{\F}_p}$, $Y = (\bigoplus_{i \in I} \Sph^{n_i}_{\overline{\F}_p}) \oplus (\bigoplus_{j \in J} \D^{n_j}_{\overline{\F}_p})$ and $f$ is the obvious inclusion. (For the notations $\Sph^{n_i}_{\overline{\F}_p}$ and $\D^{n_j}_{\overline{\F}_p}$, see the standard sphere and disk complexes in Section~\ref{sec:nots_convs}; here we've added the $\overline{\F}_p$ subscripts in the notations for clarity.) Define cochain complexes over $\F_p$ by setting $\underline{X} := \bigoplus_{i \in I} \Sph^{n_i}_{\F_p}$ and $\underline{Y} := (\bigoplus_{i \in I} \Sph^{n_i}_{\F_p}) \oplus (\bigoplus_{j \in J} \D^{n_j}_{\F_p})$, and let $\underline{f}$ be the obvious inclusion $\underline{X} \to \underline{Y}$. We then get a commutative square as follows.
\begin{center}
\begin{tikzpicture}[node distance = 1.5cm]
\node [] (A) {$X$};
\node [right of = A,xshift=1cm] (B) {$\overline{\F}_p \otimes_{\F_p} \underline{X}$};
\node [below of = A] (C) {$Y$};
\node [right of = C,xshift=1cm] (D) {$\overline{\F}_p \otimes_{\F_p} \underline{Y}$};

\draw [->] (A) -- (B) node[midway,anchor=south]{$\cong$};
\draw [->] (A) -- (C) node[midway,anchor=east]{$f$};
\draw [->] (C) -- (D) node[midway,anchor=north]{$\cong$};
\draw [->] (B) -- (D) node[midway,anchor=west]{$\overline{\F}_p \otimes_{\F_p} \underline{f}$};
\end{tikzpicture}
\end{center}
Upon applying $\overline{\mathbf{E}}_{\text{st}}^\dagger$ to this diagram, the result follows by Lemma~\ref{lem:Estbar_on_free_Fpbar_complexes} and Proposition~\ref{prop:MS_st_pres_w_eqs}. \\

(ii): Let $A$ be a cell $\widebar{\mathpzc{E}}_{\text{st}}^\dagger$-algebra. As in the proof of Proposition~\ref{prop:E_adm}, it suffices to show that, for each $j \ge 0$, $\mathpzc{U}^A(j)$ is semi-flat over $\overline{\F}_p[\Sigma_j]$; that is, it suffices to show that the obvious analogue of Lemma~\ref{lem:Estenvopsflat} holds. In fact, an easy inspection show that, by entirely analogous proofs, the obvious analogues of Lemmas~\ref{lem:finflatEst} and~\ref{lem:change_of_ringsunstable} hold, and then as a result, so does the analogue of Lemma~\ref{lem:Estenvopsflat}, as desired.
\end{proof}

The semi-model structure constructed above yields also the derived category of $\widebar{\mathpzc{E}}_{\text{st}}^\dagger$-algebras. We have now completed our first task (T1). We now move on to the second task (T2), that of computing the cohomologies of free $\widebar{\mathpzc{E}}_{\text{st}}^\dagger$-algebras, and developing a theory of cohomology operations.

\begin{Lemma}\label{lem:free_algebras_over_Est_and_Estbar}
For cochain complexes over $X$ over $\overline{\F}_p$, we have a natural isomorphism
\[
\overline{\mathbf{E}}_{\mathrm{st}}^\dagger X \cong \bigoplus_{n \ge 0} \mathpzc{E}_{\mathrm{st}}^\dagger(n) \otimes_{\F_p[\Sigma_n]} X^{\otimes_{\overline{\F}_p} n}.
\]
\end{Lemma}

Note that the only difference between the righthand side above and $\mathbf{E}_{\mathrm{st}}^\dagger X$ is that the tensor power on $X$ is over $\overline{\F}_p$ instead of $\F_p$.

\begin{proof}
We have:
\begin{align*}
\overline{\mathbf{E}}_{\mathrm{st}}^\dagger X &= \bigoplus_{n \ge 0} \widebar{\mathpzc{E}}_{\text{st}}^{\dagger}(n) \otimes_{\overline{\F}_p[\Sigma_n]} X^{\otimes_{\overline{\F}_p} n} \\
&= \bigoplus_{n \ge 0} (\mathpzc{E}_{\text{st}}^\dagger(n) \otimes_{\F_p[\Sigma_n]} \overline{\F}_p[\Sigma_n]) \otimes_{\overline{\F}_p[\Sigma_n]} X^{\otimes_{\overline{\F}_p} n} \\
&\cong \bigoplus_{n \ge 0} \mathpzc{E}_{\text{st}}^\dagger(n) \otimes_{\F_p[\Sigma_n]} (\overline{\F}_p[\Sigma_n] \otimes_{\overline{\F}_p[\Sigma_n]} X^{\otimes_{\overline{\F}_p} n}) \\
&\cong \bigoplus_{n \ge 0} \mathpzc{E}_{\text{st}}^\dagger(n) \otimes_{\F_p[\Sigma_n]} X^{\otimes_{\overline{\F}_p} n}.
\end{align*}
\end{proof}

\begin{Proposition}\label{prop:freealghomEbarpoint}
For $\overline{\F}_p$-complexes $X$, we have a natural isomorphism
\[
\mathrm{H}^\bullet(\overline{\mathbf{E}}_{\mathrm{st}}^\dagger X) \cong \widehat{\mathcal{B}} \otimes_{\F_p} \mathrm{H}^\bullet(X).
\]
\end{Proposition}

Note that the tensor in the righthand side is over $\F_p$, not over $\overline{\F}_p$.

\begin{proof}
Consider the composite
\[
\widehat{\mathcal{B}} \otimes_{\F_p} \mathrm{H}^\bullet(X) \overset{\cong}\longrightarrow \mathrm{H}^\bullet(\mathbf{E}_{\mathrm{st}}^\dagger X) \to \mathrm{H}^\bullet(\overline{\mathbf{E}}_{\mathrm{st}}^\dagger X)
\]
where the first map is the isomorphism in Proposition~\ref{prop:stablefreehom} and the second is induced by the map $\mathbf{E}_{\mathrm{st}}^\dagger X = \bigoplus_{n \ge 0} \mathpzc{E}_{\text{st}}^\dagger(n) \otimes_{\F_p[\Sigma_n]} X^{\otimes_{\F_p} n} \to \bigoplus_{n \ge 0} \mathpzc{E}_{\text{st}}^\dagger(n) \otimes_{\F_p[\Sigma_n]} X^{\otimes_{\overline{\F}_p} n} \cong \overline{\mathbf{E}}_{\mathrm{st}}^\dagger X$ (here, the final isomorphism comes from Lemma~\ref{lem:free_algebras_over_Est_and_Estbar}). We claim that the above composite is an isomorphism. As per Lemma~\ref{lem:complexes_over_k_are_sums_of_spheres_and_disks}, we have that, for some index sets $I$ and $J$, $X \cong (\bigoplus_{i \in I} \Sph^{n_i}_{\overline{\F}_p}) \oplus (\bigoplus_{j \in J} \D^{n_j}_{\overline{\F}_p})$. (For the notations $\Sph^{n_i}_{\overline{\F}_p}$ and $\D^{n_j}_{\overline{\F}_p}$, see the standard sphere and disk complexes in Section~\ref{sec:nots_convs}; here we've added the $\overline{\F}_p$ subscripts in the notations for clarity.) Define a cochain complex $\underline{X}$ over $\F_p$ by setting $\underline{X} := (\bigoplus_{i \in I} \Sph^{n_i}_{\F_p}) \oplus (\bigoplus_{j \in J} \D^{n_j}_{\F_p})$. We then have that $X \cong \overline{\F}_p \otimes_{\F_p} \underline{X}$. We also have a commutative diagram

\begin{center}
\begin{tikzpicture}[node distance = 1.5cm]
\node [] (A) {$\widehat{\mathcal{B}} \otimes_{\F_p} \mathrm{H}^\bullet(X)$};
\node [right of = A, xshift = 3.9cm] (B) {$\mathrm{H}^\bullet(\overline{\mathbf{E}}_{\mathrm{st}}^\dagger X)$};

\node [below of = A] (D) {$\overline{\F}_p \otimes_{\F_p} (\widehat{\mathcal{B}} \otimes_{\F_p} \mathrm{H}^\bullet(\underline{X}))$};
\node [below of = B] (E) {$\overline{\F}_p \otimes_{\F_p} \mathrm{H}^\bullet(\mathbf{E}_{\mathrm{st}}^\dagger \underline{X})$};

\draw [->] (A) -- (B) node[midway, anchor=south]{};
\draw [->] (A) -- (D) node[midway, anchor=east]{$\cong$};
\draw [->] (B) -- (E) node[midway, anchor=west]{$\cong$};
\draw [->] (D) -- (E) node[midway, anchor=south]{};
\end{tikzpicture}
\end{center}

where the upper horizontal map is the composite above, the righthand vertical map is induced by the isomorphism in Lemma~\ref{lem:Estbar_on_free_Fpbar_complexes} and where, an easy check shows, the lower horizontal map is the image under $\overline{\F}_p \otimes_{\F_p} -$ of the isomorphism in Proposition~\ref{prop:stablefreehom}. The desired result follows immediately.
\end{proof}

Next, we consider cohomology operations for $\widebar{\mathpzc{E}}_{\text{st}}^\dagger$-algebras.

\begin{Proposition}\label{prop:cohom_ops_Fpbar}
Given an algebra $A$ over $\widebar{\mathpzc{E}}_{\emph{st}}^\dagger$, the cohomology $\emph{H}^\bullet(A)$ possesses a natural $\F_p$-linear action by $\widehat{\mathcal{B}}$.
\end{Proposition}

Note that the operations are $\F_p$-linear, as opposed to $\overline{\F}_p$-linear.

\begin{proof}
As per Proposition~\ref{prop:freealghomEbarpoint}, we have a natural isomorphism $\text{H}^\bullet(\widebar{\mathbf{E}}_{\textbf{st}}^\dagger A) \cong \widehat{\mathcal{B}} \otimes_{\F_p} \text{H}^\bullet(A)$. The $\F_p$-linear action of $\widehat{\mathcal{B}}$ is then via the composite
\[
\widehat{\mathcal{B}} \otimes_{\F_p} \text{H}^\bullet(A) \overset{\cong}\longrightarrow \text{H}^\bullet(\widebar{\mathbf{E}}_{\textbf{st}}^\dagger A) \longrightarrow \text{H}^\bullet (A)
\]
where the second map is that which we achieve by applying $\text{H}^\bullet(-)$ to the algebra structure map $\widebar{\mathbf{E}}_{\textbf{st}}^\dagger A \to A$ of $A$.
\end{proof}

Finally, we shall now move on to the third task (T3), that of demonstrating homotopy additivity properties of $\widebar{\mathpzc{E}}_{\text{st}}^\dagger$.

\begin{Proposition}\label{prop:additivity_of_monad_Fpbar}
We have the following:
\begin{itemize}
	\item[(i)] Given $\overline{\F}_p$-complexes $X$ and $Y$, there is a natural quasi-isomorphism
\[
\overline{\mathbf{E}}_{\mathrm{st}}^\dagger(X \oplus Y) \sim \overline{\mathbf{E}}_{\mathrm{st}}^\dagger(X) \oplus \overline{\mathbf{E}}_{\mathrm{st}}^\dagger(Y).
\]
	\item[(ii)] Given cofibrant $\widebar{\mathpzc{E}}^\dagger_{\emph{st}}$-algebras $A$ and $B$, there is a natural quasi-isomorphism
\[
A \amalg B \sim A \oplus B.
\]
	\item[(iii)] Given a diagram $A \leftarrow C \rightarrow B$ of $\widebar{\mathpzc{E}}^\dagger_{\emph{st}}$-algebras, if each of $A$, $B$ and $C$ are cofibrant, and $C \to B$ is a cofibration, then we have that
\[
A \amalg_C B \sim A \oplus_C B.
\]
\end{itemize}
\end{Proposition}

\begin{proof}
(i): Given the $\overline{\F}_p$-complexes $X$ and $Y$, we have a canonical map
\[
\overline{\mathbf{E}}_{\text{st}}^\dagger(X) \oplus \overline{\mathbf{E}}_{\text{st}}^\dagger(Y) \to \overline{\mathbf{E}}_{\text{st}}^\dagger(X \oplus Y).
\]
As per Lemma~\ref{lem:complexes_over_k_are_sums_of_spheres_and_disks}, we have that, for some index sets $I_1$ and $J_1$, $X \cong (\bigoplus_{i \in I_1} \Sph^{n_i}_{\overline{\F}_p}) \oplus (\bigoplus_{j \in J_1} \D^{n_j}_{\overline{\F}_p})$. (For the notations $\Sph^{n_i}_{\overline{\F}_p}$ and $\D^{n_j}_{\overline{\F}_p}$, see the standard sphere and disk complexes in Section~\ref{sec:nots_convs}; here we've added the $\overline{\F}_p$ subscripts in the notations for clarity.) Similarly, for some index sets $I_2$ and $J_2$, $Y \cong (\bigoplus_{i \in I_2} \Sph^{n_i}_{\overline{\F}_p}) \oplus (\bigoplus_{j \in J_2} \D^{n_j}_{\overline{\F}_p})$. We then also get an induced isomorphism $X \oplus Y \cong (\bigoplus_{i \in I_1 \sqcup I_2} \Sph^{n_i}_{\overline{\F}_p}) \oplus (\bigoplus_{j \in J_1 \sqcup J_2} \D^{n_j}_{\overline{\F}_p})$. Define cochain complexes over $\F_p$ by setting $\underline{X} := (\bigoplus_{i \in I_1} \Sph^{n_i}_{\F_p}) \oplus (\bigoplus_{j \in J_1} \D^{n_j}_{\F_p})$ and $\underline{Y} := (\bigoplus_{i \in I_2} \Sph^{n_i}_{\F_p}) \oplus (\bigoplus_{j \in J_2} \D^{n_j}_{\F_p})$. We then have that $X \cong \overline{\F}_p \otimes_{\F_p} \underline{X}$, $Y \cong \overline{\F}_p \otimes_{\F_p} \underline{Y}$ and $X \oplus Y \cong \overline{\F}_p \otimes_{\F_p} (\underline{X} \oplus \underline{Y})$. We also have the following commutative diagram.

\begin{center}
\begin{tikzpicture}[node distance = 1.5cm]
\node [] (A) {$\overline{\mathbf{E}}_{\text{st}}^\dagger(X) \oplus \overline{\mathbf{E}}_{\text{st}}^\dagger(Y)$};
\node [right of = A, xshift = 1.5cm] (B) {};
\node [right of = B, xshift = 1.5cm] (C) {$\overline{\mathbf{E}}_{\text{st}}^\dagger(X \oplus Y)$};
\node [below of = A] (D) {$(\overline{\F}_p \otimes_{\F_p} \mathbf{E}_{\text{st}}^\dagger(\underline{X})) \oplus (\overline{\F}_p \otimes_{\F_p} \mathbf{E}_{\text{st}}^\dagger(\underline{Y}))$};
\node [below of = B] (E) {};
\node [below of = C] (F) {$\overline{\F}_p \otimes_{\F_p} \mathbf{E}_{\text{st}}^\dagger(\underline{X} \oplus \underline{Y})$};

\node [below of = E] (G) {$\overline{\F}_p \otimes_{\F_p} (\mathbf{E}_{\text{st}}^\dagger(\underline{X})) \oplus \mathbf{E}_{\text{st}}^\dagger(\underline{Y}))$};

\draw [->] (A) -- (C) node[midway, anchor=south]{};
\draw [->] (A) -- (D) node[midway, anchor=east]{$\cong$};
\draw [->] (C) -- (F) node[midway, anchor=west]{$\cong$};
\draw [->] (D) -- (G) node[midway, anchor=south]{$\cong$};
\draw [->] (G) -- (F) node[midway, anchor=south]{$\sim$};
\end{tikzpicture}
\end{center}

Here, the vertical isomorphisms are from Lemma~\ref{lem:Estbar_on_free_Fpbar_complexes} and the diagram commutes due to the naturality of those isomorphisms. The righthand diagonal map is a quasi-isomorphism due to Proposition~\ref{prop:additivity_of_monad}. It follows that the map $\overline{\mathbf{E}}_{\text{st}}^\dagger(X) \oplus \overline{\mathbf{E}}_{\text{st}}^\dagger(Y) \to \overline{\mathbf{E}}_{\text{st}}^\dagger(X \oplus Y)$ is also a quasi-isomorphism, as desired. \\

(ii): We are trying to show that the obvious analogue of Proposition~\ref{prop:coproducts_of_algebras} holds. An easy inspection shows that, by entirely analogous proofs, the obvious analogues of Lemmas~\ref{lem:stabilityigorsense} and~\ref{lem:UandVforcofibA} hold, and then, as a result, so does the analogue of Proposition~\ref{prop:coproducts_of_algebras}, as desired. \\

(iii): We are trying to show that the obvious analogue of Proposition~\ref{prop:pushouts_alg} holds. An easy inspection shows that, by entirely analogous proofs, the obvious analogues of Proposition~\ref{prop:pushout_normalization} and Lemmas~\ref{lem:norm_fin_flat} and~\ref{lem:Estweksforpowers} hold, and then, as a result, so does the analogue of Proposition~\ref{prop:pushouts_alg}, as desired.
\end{proof}

We have now completed the transition from coefficients in $\F_p$ to coefficients in $\overline{\F}_p$ at the level of the operad. Next, we consider spectral cochains with coefficients in $\overline{\F}_p$. Henceforth, we let $\text{C}^\bullet(-)$ denote cochains with coefficients in $\F_p$.

\begin{Definition}\label{def:speccochainsFpbar}
Given a spectrum $E$, we set the following:
\[
\widebar{\mathrm{C}}^\bullet(E) := \overline{\F}_p \otimes_{\F_p} \mathrm{C}^\bullet(E).
\]
\end{Definition}

\begin{Remark}\label{rmk:Fpbar_cochains_as_maps}
Given a spectrum $E$, we saw in Remark~\ref{rmk:spec_chain_notation} that $\F_p$-cochains on spectra, say in $\mathrm{C}^d(E)$, amount to the same thing as $\F_p$-linear maps $\mathrm{C}_d(E) \to \F_p$. Similarly, if we set $\widebar{\mathrm{C}}_\bullet(E) := \overline{\F}_p \otimes_{\F_p} \mathrm{C}_\bullet(E)$, $\overline{\F}_p$-cochains on spectra in $\widebar{\mathrm{C}}^d(E)$ amount to the same thing as $\overline{\F}_p$-linear maps $\widebar{\mathrm{C}}_d(E) \to \overline{\F}_p$, or, equivalently, $\F_p$-linear maps $\mathrm{C}_d(E) \to \overline{\F}_p$. To see this, note that $\widebar{\mathrm{C}}^d(E) = \overline{\F}_p \otimes_{\F_p} \mathrm{C}^d(E) \cong \overline{\F}_p \otimes_{\F_p} \mathrm{Hom}_{\F_p}(\mathrm{C}_d(E), \F_p) \cong \mathrm{Hom}_{\overline{\F}_p}(\overline{\F}_p \otimes_{\F_p} \mathrm{C}_d(E), \overline{\F}_p \otimes_{\F_p} \F_p) \cong \mathrm{Hom}_{\overline{\F}_p}(\widebar{\mathrm{C}}_d(E), \overline{\F}_p)$. \customendremark
\end{Remark}

Of couse, one would hope that the cochains $\widebar{\mathrm{C}}^\bullet(E)$ yield algebras over $\widebar{\mathpzc{E}}_{\mathrm{st}}^\dagger$. The following result confirms this.

\begin{Proposition}\label{prop:speccochainsalgoverFpbar}
Given a spectrum $E$, $\widebar{\mathrm{C}}^\bullet(E)$ is naturally an algebra over $\widebar{\mathpzc{E}}_{\mathrm{st}}^\dagger$.
\end{Proposition}

\begin{proof}
In general, if $X$ if an $\F_p$-complex which is an $\mathpzc{E}_{\text{st}}^\dagger$-algebra via a structure map $\alpha \colon \mathbf{E}_{\text{st}}^\dagger X \to X$, then $\overline{\F}_p \otimes_{\F_p} X$ is an $\overline{\F}_p$-complex which is an $\widebar{\mathpzc{E}}_{\text{st}}^\dagger$-algebra, via the structure map
\[
\overline{\mathbf{E}}_{\text{st}}^\dagger (\overline{\F}_p \otimes_{\F_p} X) \overset{\cong}\longrightarrow \overline{\F}_p \otimes_{\F_p} (\mathbf{E}_{\text{st}}^\dagger X) \longrightarrow \overline{\F}_p \otimes_{\F_p} X
\]
where the first  map is the isomorphism in Lemma~\ref{lem:Estbar_on_free_Fpbar_complexes} and the second map is $\overline{\F}_p \otimes_{\F_p} \alpha$.
\end{proof}

In Proposition~\ref{prop:P01sp}, we saw that, in the case of $\F_p$-cochains on spectra, the operation $P^0$ acts by the identity. The following result tells us the analogous facts in the case of $\overline{\F}_p$-cochains spectral cochains.

\begin{Proposition}\label{prop:P0_Fpbar_spectral_cochains}
Given a spectrum $E$, we have the following:
\begin{itemize}
	\item[(i)] The operation $P^0$ is the identity on elements of $\mathrm{H}^\bullet(E; \overline{\F}_p)$ in the image of $\mathrm{H}^\bullet(E; \F_p)$
	\item[(ii)] The operation $P^0$ is $\overline{\F}_p$-semilinear; that is, $P^0(\lambda\alpha) = \varphi(\lambda)P^0(\alpha)$, where $\varphi$ is the Frobenius automorphism of $\overline{\F}_p$.
\end{itemize}
\end{Proposition}

\begin{proof}
By the argument in the proof of Proposition~\ref{prop:P01sp}, we have that, if $\alpha$ is a $d$-cocycle, then $P^0[\alpha]$ is represented by a cocyle $\beta$ where $\beta(s) = \alpha(s)^p$ for all simplices $s$. Thus, if $\alpha \colon \mathrm{C}_d(E) \to \overline{\F}_p$ maps into $\F_p$, $\beta = \alpha$ and, in general, we have the semilinearity described in (ii).
\end{proof}

We have now completed the transition of all relevant previous material to coefficients in $\overline{\F}_p$.

\section{The Adjoint to Spectral Cochains}

Consider the spectral cochains functor
\[
\overline{\text{C}}^\bullet \colon \mathsf{Sp}^{\text{op}} \to \widebar{\mathpzc{E}}_{\text{st}}^\dagger\text{-}\mathsf{Alg}
\]
with coefficients in $\overline{\F}_p$, as defined in Definition~\ref{def:speccochainsFpbar}. We shall construct an adjoint to this spectral cochains functor. That is, we will construct an adjoint functor
\[
\text{U} \colon \widebar{\mathpzc{E}}_{\text{st}}^\dagger\text{-}\mathsf{Alg} \to \mathsf{Sp}^{\text{op}}.
\]
We define $\mathrm{U}$ by setting, given an $\widebar{\mathpzc{E}}_{\text{st}}^\dagger$-algebra $A$, the following in spectral degree $n \ge 0$ and simplicial degree $d \ge 0$:
\[
\mathrm{U}(A)_{n,d} := \widebar{\mathpzc{E}}_{\text{st}}^\dagger\mathsf{Alg}(A,\widebar{\text{C}}^\bullet(\Sigma^{\infty-n}\Delta_{d+})) = \widebar{\mathpzc{E}}_{\text{st}}^\dagger\mathsf{Alg}(A, \overline{\F}_p \otimes_{\F_p} \text{C}^\bullet(\Sigma^{\infty-n}\Delta_{d+})).
\]

\begin{Remark}
Before proceeding any further, we make a remark to give some context for what is to come. We are going to verify that $\mathrm{U}$, as defined above, is an adjoint to spectral cochains, and moreover that this adjunction is a Quillen adjunction. Although we don't develop this theory, one interpretation of these results is as follows. The category of algebras $\widebar{\mathpzc{E}}_{\text{st}}^\dagger\text{-}\mathsf{Alg}$ is enriched in spectra: tensored, cotensored, and with a mapping spectrum $\mathrm{F}(A,B)$ between any two algebras $A$ and $B$, all satisfying a version of Quillen’s Axiom SM7. Under this theory, the functor $\mathrm{U}$ is given by $\mathrm{U}(A) = \mathrm{F}(A, \widebar{\mathrm{C}}^\bullet\Sph)$ where $\Sph$ denotes the sphere spectrum. \customendremark
\end{Remark}

In order to check that, for each $A$, the above definition yields a spectrum $\mathrm{U}(A)$, and that this construction is functorial in $A$, we must carry out several verifications. Let us fix an $\widebar{\mathpzc{E}}_{\text{st}}^\dagger$-algebra $A$. Note that, with $n$ fixed, $\text{U}(A)_{n,d} $ is clearly contravariantly functorial in $d$, so that, for each $n \ge 0$, we have a simplicial set $\text{U}(A)_n$; moreover, it becomes a based simplicial set upon endowing it with the zero map as a basepoint. Next, for each $n \ge 0$, we need a spectral structure map $\mathrm{U}(A)_n \to \Omega\mathrm{U}(A)_{n+1}$. In simplicial dimension $d \ge 0$, this means a map $\mathrm{U}(A)_{n,d} \to \mathrm{U}(A)_{n+1,d+1}$, such that, for each algebra map $f \in \mathrm{U}(A)_{n+1,d+1}$, the image of $f$ under the simplicial operators $d_0$ and $d_1 \cdots d_{d+1}$ are both the zero map. Thus we need a map
\[
\widebar{\mathpzc{E}}_{\text{st}}^\dagger\text{-}\mathsf{Alg}(A, \overline{\F}_p \otimes_{\F_p} \mathrm{C}^\bullet(\Sigma^{\infty - n}\Delta_{d+})) \to \widebar{\mathpzc{E}}_{\text{st}}^\dagger\text{-}\mathsf{Alg}(A, \overline{\F}_p \otimes_{\F_p} \mathrm{C}^\bullet(\Sigma^{\infty - n-1}\Delta_{d+1,+}))
\]
which is such that the following requirements hold:
\begin{itemize}
	\item[(R1)] The algebra maps which lie in the image of this map satisfy the property that they yield the zero map upon postcomposition with the map $\overline{\F}_p \otimes_{\F_p} \mathrm{C}^\bullet(\Sigma^{\infty - n-1}\Delta_{d+1,+}) \to \overline{\F}_p \otimes_{\F_p} \mathrm{C}^\bullet(\Sigma^{\infty - n-1}\Delta_{d+})$ induced by $d_0$.
	\item[(R2)] The algebra maps which lie in the image of this map satisfy the property that they yield the zero map upon postcomposition with the map $\overline{\F}_p \otimes_{\F_p} \mathrm{C}^\bullet(\Sigma^{\infty - n-1}\Delta_{d+1,+}) \to  \overline{\F}_p \otimes_{\F_p} \mathrm{C}^\bullet(\Sigma^{\infty - n-1}\Delta_{0+})$ induced by $d_1 \cdots d_{d+1}$.
\end{itemize}

In order to construct this map, we first note that we have an isomorphism of differential graded $\F_p$-modules
\[
\mathrm{C}_\bullet(\Sigma^{\infty - n}\Delta_{d+}) \to \mathrm{C}_\bullet(\Sigma \Delta_{d+})
\]
of degree $n+1$ (note that here in the source we are taking chains on a spectrum while in the target we are taking chains on a based simplicial set). This isomorphism is given by sending $[n,e,x]$ (see Remark~\ref{rmk:spec_chain_notation} for this notation) in $\Delta_{d+}$, which is of degree $e-n$, to $[\Sigma x]$ (see Definition~\ref{def:susp_simps} for this notation) in $\Sigma \Delta_{d+}$, which is of degree $e+1$; that this is an isomorphism follows from Proposition~\ref{prop:nd_susp}. Next, note that we have another isomorphism of differential graded $\F_p$-modules
\[
\mathrm{C}_\bullet(\Sigma^{\infty - n-1}\Delta_{d+1,+}) \to \mathrm{C}_\bullet(\Delta_{d+1,+})
\]
again of degree $n+1$. This isomorphism is given by sending $[n+1,e,x]$ in $\Delta_{d+1,+}$, which is of degree $e-n-1$, to $[x]$ in $\Delta_{d,+}$, which is of degree $e$. Now, using the isomorphism in Example~\ref{examp:cones_Delta_k_+} and the canonical map $\text{C}(X) \to \Sigma X$, we have a canonical map $\Delta_{d+1} \to \Sigma \Delta_{d+}$, yielding a map $\Delta_{d+1,+} \to \Sigma \Delta_{d+}$, and so, using the above isomorphisms, we get a composite map
\begin{equation}\label{eqn:composite_for_U(A)_spectral_structure_maps}
\mathrm{C}_\bullet(\Sigma^{\infty - n-1}\Delta_{d+1,+}) \to \mathrm{C}_\bullet(\Delta_{d+1,+}) \to \mathrm{C}_\bullet(\Sigma\Delta_{d+}) \to \mathrm{C}_\bullet(\Sigma^{\infty - n}\Delta_{d+})
\end{equation}
which is a map of chain complexes since it is of degree $(n+1) + 0 - (n+1) = 0$. By Proposition~\ref{prop:spec_chains_coalg} (together with a pullback of structure across the map $\mathpzc{E}_{\mathrm{st}} \to \mathpzc{Z}_{\mathrm{st}}$ defined in Section~\ref{subsec:stable_barratt_eccles_operad}), the source and target, $\mathrm{C}_\bullet(\Sigma^{\infty - n-1}\Delta_{d+1,+})$ and $\mathrm{C}_\bullet(\Sigma^{\infty - n}\Delta_{d+})$, are $\mathpzc{E}_{\text{st}}$-coalgebras, and we claim that the composite above is in fact a map of $\mathpzc{E}_{\text{st}}$-coalgebras. Upon verifying this, by dualization, tensoring with $\overline{\F}_p$, and postcomposition, we shall have the desired map
\[
\mathpzc{E}_{\text{st}}\text{-}\mathsf{Alg}(A, \overline{\F}_p \otimes_{\F_p} \mathrm{C}^\bullet(\Sigma^{\infty - n}\Delta_{d+})) \to \mathpzc{E}_{\text{st}}\text{-}\mathsf{Alg}(A, \overline{\F}_p \otimes_{\F_p} \mathrm{C}^\bullet(\Sigma^{\infty - n-1}\Delta_{d+1,+})).
\]
To see that the composite in (\ref{eqn:composite_for_U(A)_spectral_structure_maps}) is an $\mathpzc{E}_{\text{st}}$-coalgebra map, consider some element $[n+1,e,x]$ in the source chains $\text{C}_\bullet(\Sigma^{\infty - n-1}\Delta_{d+1,+})$, where $x$ is a map $[e] \to [d+1]$. Let $x'$ be the corresponding map $[e'] \to [d]$ constructed by restricting to the preimage of the final $d+1$ elements of $[d+1]$ (if this preimage is empty, or if it is all of $[e]$, set $x'$ to be $*$). The image under the composite in (\ref{eqn:composite_for_U(A)_spectral_structure_maps}) of $[n+1,e,x]$, in the target chains $\text{C}_\bullet(\Sigma^{\infty - n}\Delta_{d+})$, is $[n+1,e',x']$ (we have $n+1$ here instead of $n$ since the third map in (\ref{eqn:composite_for_U(A)_spectral_structure_maps}) maps to the $(n+1)^{\text{st}}$ level, instead of the $n^{\text{th}}$ level, of the spectral chains $\text{C}_\bullet(\Sigma^{\infty - n}\Delta_{d+})$). Thus an element $\alpha = (\alpha_0,\alpha_1,\dots)$ of $\mathpzc{E}_{\text{st}}(k)$ coacts on both the element in the source and the element in the target by $\alpha_{n+1}$, yielding $\alpha_{n+1}(e)$ and $\alpha_{n+1}(e')$. As such, the verification that the composite in (\ref{eqn:composite_for_U(A)_spectral_structure_maps}) is an $\mathpzc{E}_{\text{st}}$-coalgebra map amounts to checking that the following square commutes.
\begin{center}
\begin{tikzpicture}[node distance = 1.5cm]
\node(A){$\text{C}_\bullet(\Delta_{d+1,+})$};
\node[below of = A](C){$\text{C}_\bullet(\Delta_{d+1,+})^{\otimes k}$};
\node[right of = A, xshift = 1cm](B){$\text{C}_\bullet(\Sigma \Delta_{d+})$};
\node[below of = B, yshift = 0mm](D){$\text{C}_\bullet(\Sigma \Delta_{d+})^{\otimes k}$};
	
\draw[->] (A) -- (B) node[midway,anchor=south]{};
\draw[->] (A) -- (C) node[midway,anchor=east]{$\alpha_{n+1}$};
\draw[->] (C) -- (D) node[midway,anchor=north]{};
\draw[->] (B) -- (D) node[midway,anchor=west]{$\alpha_{n+1}$};
\end{tikzpicture}
\end{center}
The commutativity of this square follows from the naturality of $\alpha_{n+1}$. \\

We have now constructed a map
\begin{equation}\label{eqn:U(A)_spectral_structure_map}
\mathpzc{E}_{\text{st}}\text{-}\mathsf{Alg}(A, \overline{\F}_p \otimes_{\F_p} \mathrm{C}^\bullet(\Sigma^{\infty - n}\Delta_{d+})) \to \mathpzc{E}_{\text{st}}\text{-}\mathsf{Alg}(A, \overline{\F}_p \otimes_{\F_p} \mathrm{C}^\bullet(\Sigma^{\infty - n-1}\Delta_{d+1,+})).
\end{equation}
but must still verify the requirements (R1) and (R2) above. Consider, first, the requirement (R1). Given a map
\[
A \to \overline{\F}_p \otimes_{\F_p} \mathrm{C}^\bullet(\Sigma^{\infty - n}\Delta_{d+}) \to \overline{\F}_p \otimes_{\F_p} \mathrm{C}^\bullet(\Sigma^{\infty - n - 1}\Delta_{d+1,+})
\]
where the second map is induced by the composite (\ref{eqn:composite_for_U(A)_spectral_structure_maps}), we must verify that the composite
\begin{multline*}
A \longrightarrow \overline{\F}_p \otimes_{\F_p} \mathrm{C}^\bullet(\Sigma^{\infty - n}\Delta_{d+}) \longrightarrow \overline{\F}_p \otimes_{\F_p} \mathrm{C}^\bullet(\Sigma^{\infty - n - 1}\Delta_{d+1,+}) \\
\longrightarrow \overline{\F}_p \otimes_{\F_p} \mathrm{C}^\bullet(\Sigma^{\infty - n-1}\Delta_{d,+})
\end{multline*}
where the final map is induced by $d_0$, is the zero map. Here the composite of the latter two maps can be constructed from the composite below by dualization and tensoring with $\overline{\F}_p$.

\begin{center}
\begin{tikzpicture}[node distance = 1cm]
\node [] (A) {$\mathrm{C}_\bullet(\Sigma^{\infty - n-1}\Delta_{d+})$};
\node [below of = A,yshift=-5mm] (B) {$\mathrm{C}_\bullet(\Sigma^{\infty - n-1}\Delta_{d+1,+})$};
\node [right of = B,xshift=2.5cm] (C) {$\mathrm{C}_\bullet(\Delta_{d+1,+})$};
\node [right of = C,xshift=2cm] (D) {$\mathrm{C}_\bullet(\Sigma\Delta_{d+})$};
\node [right of = D,xshift=2cm] (E) {$\mathrm{C}_\bullet(\Sigma^{\infty - n}\Delta_{d+})$};

\draw [->] (A) -- (B) node[midway,anchor=east]{};
\draw [->] (B) -- (C) node[midway,anchor=north]{};
\draw [->] (C) -- (D) node[midway,anchor=north]{};
\draw [->] (D) -- (E) node[midway,anchor=north]{};
\end{tikzpicture}
\end{center}

Thus, in order to verify (R1), it suffices to demonstrate that the above composite is zero. Start with some $q$-simplex $[e] \to [d]$ in $\mathrm{C}_\bullet(\Sigma^{\infty - n-1}\Delta_{d+})$. It gets postcomposed to a map $[e] \to [d+1]$ where the image doesn't contain $0$; then we get this same map again but with an altered degree; then we restrict to those entries which don't map to $0$ and so get back the original map $[e] \to [d]$ which, in the suspension, is killed (mapped to $*$) and thus, at the end, we get zero, as desired. \\

Next, let us verify the requirement (R2). In this case, we postcompose the algebra maps instead with the map from $\overline{\F}_p \otimes_{\F_p} \mathrm{C}^\bullet(\Sigma^{\infty - n-1}\Delta_{d+1,+})$ to $\overline{\F}_p \otimes_{\F_p} \mathrm{C}^\bullet(\Sigma^{\infty - n-1}\Delta_{0+})$ induced by $d_1 \cdots d_{d+1}$. The verification is analogous to that of (R1) except that, in the diagram above, the vertical map is instead the map $\mathrm{C}_\bullet(\Sigma^{\infty - n-1}\Delta_{0+}) \to \mathrm{C}_\bullet(\Sigma^{\infty - n-1}\Delta_{d+1,+})$ induced by $d_1 \cdots d_{d+1}$. The check that the resulting composite is indeed zero is as follows. Start with the identity on $[0]$; this is mapped to the inclusion $[0] \to [d+1]$ mapping $0$ to $0$; this, in turn, is killed by the map $\mathrm{C}_\bullet(\Delta_{d+1,+}) \to \mathrm{C}_\bullet(\Sigma\Delta_{d+})$ (see the definition of the map to the cone in Example~\ref{examp:cones_Delta_k_+}), and so, at the end, we again get zero, as desired. \\

We have now constructed, for each $n, d \ge 0$, a map $\mathrm{U}(A)_{n,d} \to \mathrm{U}(A)_{n+1,d+1}$, or, more specifically, a map $(\mathrm{U}(A)_{n})_{d} \to (\Omega\mathrm{U}(A)_{n+1})_{d}$. One can readily check that these maps commute with the simplicial operators, so that we have the desired spectral structure maps $\mathrm{U}(A)_n \to \Omega \mathrm{U}(A)_{n+1}$.\\

We now have a spectrum $\mathrm{U}(A)$ associated to each $\widebar{\mathpzc{E}}_{\mathrm{st}}^\dagger$-algebra $A$. Next, we show that this construction is functorial in the algebras $A$. This is easily seen from the fact that
\[
\mathrm{U}(A)_{n,d} = \widebar{\mathpzc{E}}_{\mathrm{st}}^\dagger\text{-}\mathsf{Alg}(A, \overline{\F}_p \otimes_{\F_p} \mathrm{C}^\bullet(\Sigma^{\infty - n}\Delta_{d+}))
\]
as, when given an algebra map from $A$ to $B$, we get, by precomposition, an induced map from $\widebar{\mathpzc{E}}_{\mathrm{st}}^\dagger\text{-}\mathsf{Alg}(B, \overline{\F}_p \otimes_{\F_p} \mathrm{C}^\bullet(\Sigma^{\infty - n}\Delta_{d+}))$ to $\widebar{\mathpzc{E}}_{\mathrm{st}}^\dagger\text{-}\mathsf{Alg}(A, \overline{\F}_p \otimes_{\F_p} \mathrm{C}^\bullet(\Sigma^{\infty - n}\Delta_{d+}))$. Considering these for a fixed $n$ but variable $d$, we get a simplicial set map $\mathrm{U}(B)_n \to \mathrm{U}(A)_n$ as the simplicial operators act by postcomposition and so commute with precomposition maps. Moreover, one can immediately verify that these simplicial set maps are compatible with the structure maps, defined above, of the spectra $\mathrm{U}(A)$ and $\mathrm{U}(B)$. Thus we have a functor
\[
\mathrm{U} \colon \widebar{\mathpzc{E}}_{\mathrm{st}}^\dagger\text{-}\mathsf{Alg} \to \mathsf{Sp}^{\mathrm{op}}.
\]

\begin{Proposition}\label{prop:U_is_right_adjoint}
The functor $\emph{U}$ is left adjoint to the cochains functor on spectra, so that we have an adjunction
\begin{center}
\begin{tikzpicture}[node distance=3cm]
\node[](A){$\mathsf{Sp}^{\mathrm{op}}$};
\node[right of = A](B){$\widebar{\mathpzc{E}}_{\mathrm{st}}^\dagger\text{-}\mathsf{Alg}$.};

\draw[->,transform canvas={yshift=2mm}] (A) -- (B) node[midway,anchor=south]{$\widebar{\mathrm{C}}^\bullet$};
\draw[->,transform canvas={yshift=-2mm}] (B) -- (A) node[midway,anchor=north]{$\mathrm{U}$} node[midway,anchor=south,yshift=-0.75mm]{$\top$};
\end{tikzpicture}
\end{center}
\end{Proposition}

\begin{proof}
Let $E$ be a spectrum and $A$ a $\mathpzc{E}_{\text{st}}$-algebra. We wish to construct the natural isomorphism between $\mathsf{Sp}(E,\text{U}(A))$ and $\widebar{\mathpzc{E}}^\dagger_{\text{st}}\text{-}\mathsf{Alg}(A,\widebar{\text{C}}^{\bullet}(E))$. This requires verifications which are not obvious but not too difficult, though they are rather lengthy. We shall provide here the part of the correspondence which yields an $\widebar{\mathpzc{E}}^\dagger_{\text{st}}$algebra map $g = \{g_n\} \colon A \to \widebar{\text{C}}^\bullet(E)$ when given a spectrum map $f =\{f_n\} \colon E \to \text{U}(A)$. Fix such a spectrum map $f = \{f_n\} \colon E \to \text{U}(A)$. We want an algebra map $g = \{g_n\} \colon A \to \widebar{\text{C}}^\bullet(E)$. Consider $a \in A$, of degree say $n$. We shall construct the image $g_n(a)$ in dualized form, as a map $g_n(a) \colon \mathrm{C}_n(E) \to \overline{\F}_p$. Consider some $[m,e,x]$ (see Remark~\ref{rmk:spec_chain_notation} for this notation) in $(E_m)_e$ where $e-m=n$. We have an element $f_m(x) \in \text{U}(A)_{m,e}$, which is to say a map $A \to \widebar{\text{C}}^\bullet(\Sigma^{\infty - m}\Delta_{e+})$, and we set $g_n(a)([m,e,x])$ to be image of $a$ under the composite
\[
A \to \widebar{\mathrm{C}}^\bullet(\Sigma^{\infty - m}\Delta_{e+}) \to \overline{\F}_p
\]
where the first map is the map $f_m(x)$ and the second is induced by the evaluation map $\mathrm{C}^\bullet(\Sigma^{\infty - m}\Delta_{e+}) \to \F_p$ which sends a cochain to its value on the chain $[m,e,\text{id}_{[e]}]$ (note that this chain is of degree $e-m = n$). Thus $g_n(a)([m,e,x]) = f_m(x)(a)([m,e,\text{id}_{[e]}])$. Linearity of $g_n$ follows from that of $f_m(x)$. Next, we must check that the differentials are preserved. Fix some $a \in A^n$. We have two $(n+1)$-cochains $g_{n+1}(\partial a), \partial g_n(a) \colon \mathrm{C}_{n+1}(E) \to \overline{\F}_p$ and we desire that these two be the same. Consider some $[m,e,x]$ where $x \in (E_m)_e$ and $e-m=n+1$. The latter cochain first forms $\partial [m,e,x] = \sum_i [m,e-1,d_i(x)]$ and then sends this to $\sum_i f_m(d_i(x))(a)([m,e-1,\text{id}_{[e-1]}])$. On the other hand, the former cochain sends $[m,e,x]$ to $f_m(x)(\partial a)([m,e,\text{id}_{[e]}])$. Now, since $f_m$ is a map of simplicial sets, we have $\sum_i f_m(d_i(x))(a)([m,e-1,\text{id}_{[e-1]}]) = \sum_i (d_if_m(x))(a)([m,e-1,\text{id}_{[e-1]}])$. For each $i$, the map $d_if_m(e)$ is the composite
\[
A \overset{f_m(x)}\longrightarrow \widebar{\text{C}}^\bullet(\Sigma^{\infty - m}\Delta_{e+}) \overset{d_i}\longrightarrow \widebar{\text{C}}^\bullet(\Sigma^{\infty - m}\Delta_{e-1,+}).
\]
It follows that $\sum_i (d_if_m(x))(a)([m,e-1,\text{id}_{[e-1]}]) = \sum_i (f_m(x))(a)([m,e-1,d_i])$. On the other hand, since $f_m(x)$ is a map of cochain complexes, we have that $f_m(x)(\partial a) = \partial f_m(x)(a)$. It follows that
\begin{align*}
f_m(x)(\partial a)([m,e,\text{id}_{[e]}]) &= (\partial f_m(x)(a))([m,e,\text{id}_{[e]}]) \\
&= f_m(x)(a)(\partial [m,e,\text{id}_{[e]}]) \\
&= f_m(x)(a)\left(\sum_i [m,e-1,d_i]\right) \\
&= \sum_i f_m(x)(a)([m,e-1,d_i])
\end{align*}
Thus, as desired, the two cochains coincide. As our last verification, we must check that our map $g = \{g_n\} \colon A \to \widebar{\text{C}}^\bullet(E)$ respects the actions by $\widebar{\mathpzc{E}}^\dagger_{\text{st}}$. Let $\alpha = (\alpha_0,\alpha_1,\dots) \in \widebar{\mathpzc{E}}^\dagger_{\text{st}}(k)$. Consider some $a_1, \dots, a_k \in A$, and assume without loss of generality (due to linearity), that the $a_i$ are homogeneous, say of degrees $n_1, \dots, n_k$. If we first act by $\alpha$ and then apply $g$, we get a cochain whose image at $[m,e,x]$, where $e-m = n_1 + \cdots + n_k$, is $f_m(x)(\alpha(a_1,\dots,a_k))([m,e,\text{id}_{[e]}])$. Since, by definition, $f_m(x)$ is a map of $\widebar{\mathpzc{E}}^\dagger_{\mathrm{st}}$-algebras, we have that this is in fact equivalent to $\alpha(f_m(x)(a_1),\dots,f_m(x)(a_k))([m,e,\text{id}_{[e]}])$. On the other hand, if we apply $g$ first and then act by $\alpha$, we first get cochains $g_{n_1}(a_1), \dots, g_{n_k}(a_k)$ and then the cochain $\alpha(g_{n_1}(a_1), \dots, g_{n_k}(a_k))$. Let the coaction of $\alpha$ on $[m,e,\text{id}_{[e]}]$ be given by $\sum [m,e_1,\theta_1] \otimes \cdots \otimes [m,e_k,\theta_k]$, where $\theta_i$ is a map $[e_i] \to [e]$. Then, considering the map $\Delta_e \to E_m$ corresponding to the same $x \in (E_m)_e$ as above, and using the naturality of $\alpha_m$, we find that the coaction of $\alpha$ on $[m,e,x]$ is given by $\sum [m,e_1,\theta_1^*x] \otimes \cdots \otimes [m,e_k,\theta_k^*x]$. Now, by definition of the $\widebar{\mathpzc{E}}^\dagger_{\text{st}}$-action on spectral cochains, if we evaluate the cochain $\alpha(g_{n_1}(a_1), \dots, g_{n_k}(a_k))$ at $[m,e,x]$, we get
\[
\alpha(g_{n_1}(a_1), \dots, g_{n_k}(a_k))([m,e,x]) = (g_{n_1}(a_1) \otimes \cdots \otimes g_{n_k}(a_k))(\alpha \cdot [m,e,x])
\]
which amounts to
\[
(g_{n_1}(a_1) \otimes \cdots \otimes g_{n_k}(a_k))\left(\sum [m,e_1,\theta_1^*x] \otimes \cdots \otimes [m,e_k,\theta_k^*x]\right).
\]
On the other hand, we have
\begin{multline*}
\alpha(f_m(x)(a_1),\dots,f_m(x)(a_k))([m,e,\text{id}_{[e]}]) \\
= (f_m(x)(a_1) \otimes \cdots \otimes f_m(x)(a_k))(\alpha \cdot [m,e,\text{id}_{[e]}])
\end{multline*}
which is to say
\[
(f_m(x)(a_1) \otimes \cdots \otimes f_m(x)(a_k))\left(\sum [m,e_1,\theta_1] \otimes \cdots \otimes [m,e_k,\theta_k]\right)
\]
and this amounts to
\[
\sum f_m(x)(a_1)([m,e_1,\theta_1]) \otimes \cdots \otimes f_m(x)(a_k)([m,e_k,\theta_k]).
\]
In either case, we only have to worry about summands where $e_i - m = n_i$ for each $i$. In this case, the former becomes
\[
\sum (f_m(\theta_1^*x)(a_1)[m,q_1,\text{id}_{[q_1]}] \otimes \cdots \otimes f_m(\theta_k^*x)(a_k)[m,q_k,\text{id}_{[q_k]}]
\]
which is to say
\[
\sum ((\theta_1^*f_m)(x)(a_1)[m,e_1,\text{id}_{[e_1]}] \otimes \cdots \otimes (\theta_k^*f_m)(x)(a_k)[m,e_k,\text{id}_{[e_k]}].
\]
Now, for each $i$, $\theta_i^*f_m(x)$ is the composite
\[
A \overset{f_m(x)}\longrightarrow \widebar{\text{C}}^\bullet(\Sigma^{\infty - m}\Delta_{e+}) \overset{\theta_i}\longrightarrow \widebar{\text{C}}^\bullet(\Sigma^{\infty - m}\Delta_{e_i,+}).
\]
It follows that, for each $i$, $((\theta_1^*f_m)(x)(a_1)[m,e_i,\text{id}_{[e_i]}] = f_m(x)(a_1)[m,e_i,\theta_i]$. Thus the two cochains coincide, as desired.
\end{proof}

Now we consider homotopical properties of the above spectral cochains adjunction.

\begin{Proposition}\label{prop:spec_coch_adj_qadj}
The spectral cochains adjunction
\begin{center}
\begin{tikzpicture}[node distance=3cm]
\node[](A){$\mathsf{Sp}^{\emph{op}}$};
\node[right of = A](B){$\widebar{\mathpzc{E}}_{\emph{st}}^\dagger\text{-}\mathsf{Alg}$};

\draw[->,transform canvas={yshift=2mm}] (A) -- (B) node[midway,anchor=south]{$\overline{\emph{C}}^\bullet$};
\draw[->,transform canvas={yshift=-2mm}] (B) -- (A) node[midway,anchor=north]{$\emph{U}$} node[midway,anchor=south,yshift=-0.75mm]{$\top$};
\end{tikzpicture}
\end{center}
is a Quillen adjunction.
\end{Proposition}

Note that here on the righthand side we have a Quillen semi-model category, as opposed to a Quillen model category. By a Quillen adjunction, we mean one which satisfies the conditions in Proposition~\ref{prop:semimodelqadj}(iii).

\begin{proof}
We first demonstrate that $\widebar{\mathrm{C}}^\bullet$ preserves fibrations, which is to say that $\widebar{\mathrm{C}}^\bullet$ sends a cofibration $i \colon E \to F$ of spectra to an epimorphism. Since $i$ is a cofibration, we have that $i_0 \colon E_0 \to F_0$ and, for $n \ge 0$, the maps
\[
E_{n+1} \, \amalg_{\Sigma E_n} \, \Sigma F_n \to F_{n+1}
\]
are cofibrations of based simplicial sets, which is to say that they are injective in each simplicial degree. In particular, by Proposition~\ref{prop:spec_mod_str}, each $i_n \colon E_n \to F_n$, for $n \ge 0$, is a monomorphism. Since monomorphisms of simplicial sets preserve non-degenerate simplices, each $i_n$, for $n \ge 0$, preserves non-degenerate simplices, and of course also the basepoints and their degeneracies. Thus, upon taking chains $\text{C}_\bullet(-)$, we get a sequential colimit of monomorphisms, which is once again a monomorphism since sequential colimits are exact. As we are working over a field, we have a split inclusion, so that, upon dualizing, reindexing and tensoring, we have the desired result for the cochains $\widebar{\mathrm{C}}^\bullet(-)$. \\

Next, we shall show that $\text{U}$ preserves cofibrations. Given a cofibration $A \to B$ of algebras, we wish to show that $\text{U}(A) \to \text{U}(B)$ is a cofibration in the opposite category of spectra. We know that all cofibrations of algebras may be written as retracts of cell maps. As $\text{U}$ is a left adjoint and so preserves colimits, we need only show that $\text{U}$ maps the cofibrations $\widebar{\mathbf{E}}_{\textbf{st}}^\dagger M \to \widebar{\mathbf{E}}_{\textbf{st}}^\dagger \text{C}M$ to cofibrations. Here $M$ is an $\overline{\F}_p$-complex with zero differentials. In fact, since for such $M$ the map $M \to \text{C}M$ decomposes as a direct sum of maps $\mathbb{S}^n \to \mathbb{D}^{n+1}$ for various $n$, we need only consider the case of the map $\mathbb{S}^n \to \mathbb{D}^{n+1}$ (see Section~\ref{sec:nots_convs} for the sphere and disk complexes $\Sph^n$ and $\D^n$). We shall show, more generally, that if $X \to Y$ is an inclusion of complexes where $X$ and $Y$ are of finite type (by which we mean that they are of finite dimension in each degree), then $\text{U}\widebar{\mathbf{E}}_{\textbf{st}}^\dagger Y \to \text{U} \widebar{\mathbf{E}}_{\textbf{st}}^\dagger X$ is a fibration of spectra. \\

To begin, we claim that, given any complex $X$ of finite type, $\text{U}\widebar{\mathbf{E}}_{\textbf{st}}^\dagger X$ is a strict $\Omega$-spectrum. First, note that, in each spectral degree $n$ and simplicial degree $d$, we have that
\begin{align*}
(\text{U}\widebar{\mathbf{E}}_{\textbf{st}}^\dagger X)_{n,d} &= \widebar{\mathpzc{E}}_{\text{st}}^\dagger\text{-}\mathsf{Alg}(\widebar{\mathbf{E}}_{\textbf{st}}^\dagger X, \overline{\F}_p \otimes_{\F_p} \mathrm{C}^\bullet(\Sigma^{\infty - n}\Delta_{d+})) \\
&\cong \mathsf{Co}_{\overline{\F}_p}(X, \overline{\F}_p \otimes_{\F_p} \mathrm{C}^\bullet(\Sigma^{\infty - n}\Delta_{d+})).
\end{align*}
As maps of complexes are closed under addition, we have that, for each $n \ge 0$, $(\text{U}\widebar{\mathbf{E}}_{\textbf{st}}^\dagger X)_n$ is the underlying simplicial set of a simplicial abelian group, and so a Kan complex. Next, we will show that the maps $\text{U}(A)_n \to \Omega\text{U}(A)_{n+1}$ are bijections in each simplicial degree $d$. To see this, first note that, since $X$ is of finite type, we may dualize to find that
\begin{align*}
(\text{U}\widebar{\mathbf{E}}_{\textbf{st}}^\dagger X)_{n,d} \cong \mathsf{Ch}_{\overline{\F}_p}(\overline{\F}_p \otimes_{\F_p} \mathrm{C}_\bullet(\Sigma^{\infty - n}\Delta_{d+}), \mathrm{D}X) \cong \mathsf{Ch}_{\F_p}(\mathrm{C}_\bullet(\Sigma^{\infty - n}\Delta_{d+}), \mathrm{D}X)
\end{align*}
where $\text{D}X$ is the chain complex given by $(\text{D}X)_e = \text{Hom}_{\overline{\F}_p}(X_e,\overline{\F}_p)$. Under this isomorphism, the map $\text{U}(A)_n \to \Omega\text{U}(A)_{n+1}$ is given by sending a complex map $\text{C}_\bullet(\Sigma^{\infty - n}\Delta_{d+}) \to \text{D}X$ to the composite 
\begin{center}
\begin{tikzpicture}[node distance = 1cm]
\node [] (B) {$\text{C}_\bullet(\Sigma^{\infty - n-1}\Delta_{d+1,+})$};
\node [right of = B,xshift=2.5cm] (C) {$\text{C}_\bullet(\Delta_{d+1,+})$};
\node [right of = C,xshift=2cm] (D) {$\text{C}_\bullet(\Sigma\Delta_{d+})$};
\node [right of = D,xshift=2cm] (E) {$\text{C}_\bullet(\Sigma^{\infty - n}\Delta_{d+})$};
\node [below of = E] (F) {$\text{D}X$};

\draw [->] (B) -- (C) node[midway,anchor=north]{};
\draw [->] (C) -- (D) node[midway,anchor=north]{};
\draw [->] (D) -- (E) node[midway,anchor=north]{};
\draw [->] (E) -- (F) node[midway,anchor=north]{};
\end{tikzpicture}
\end{center}
where the first three maps are the maps which arose in the definition of $\text{U}$. Consider a map $\text{C}_\bullet(\Sigma^{\infty - n}\Delta_{d+}) \to \text{D}X$ which is non-zero. Then there exists some simplex $\theta \colon [e] \to [d]$ in $\Delta_{d+}$ which is not mapped to zero. Consider now the map $\theta' \colon [e+1] \to [d+1]$ which maps $0$ to $0$ and maps $i$ to $\theta(i-1)+1$ for $i \ge 1$. This gives a simplex in $\Delta_{d+1,+}$ and so an element of the source of the composite above. Upon applying the first map, we get $\theta'$ again, and then upon applying the second map, we get the original $\theta \colon [e] \to [d]$ but in dimension $e+1$, then the original $\theta$ and then finally a non-zero element in $X$ by our assumption above. This shows that $\text{U}(X)_n \to \Omega\text{U}(X)_{n+1}$ is injective in each simplicial degree. It remains to demonstrate surjectivity. The proof is similar. Suppose given a map $f \colon \text{C}_\bullet(\Sigma^{\infty - n-1}\Delta_{d+1,+}) \to \text{D}X$ and suppose that it satisfies the ``$d_0 = d_1 \cdots d_{n+1} = *$'' condition required for membership in $\Omega\text{U}(X)_{n+1,d}$. We then need to define a map $g \colon \text{C}_\bullet(\Sigma^{\infty - n}\Delta_{d+}) \to \text{D}X$. Given $\theta \colon [e] \to [d]$, we map it to $f(\theta')$ where $\theta'$ is defined as above. One can check directly that this is indeed a map of complexes, and we see that, upon precomposition with the first three maps in the composite above, we get the original map $f$ since, as before, $\theta' \mapsto \theta$ under the composite of the first three maps. This completes the proof that $\text{U}\widebar{\mathbf{E}}_{\textbf{st}}^\dagger X$ is a strict $\Omega$-spectrum. \\

Now, invoking Proposition~\ref{prop:fibcofibsp}(iii), it remains to show that if $X \to Y$ is an inclusion of complexes where $X$ and $Y$ are of finite type, then $\text{U}\widebar{\mathbf{E}}_{\textbf{st}}^\dagger Y \to \text{U} \widebar{\mathbf{E}}_{\textbf{st}}^\dagger X$ is a levelwise fibration of spectra. Thus, for each $n \ge 0$, we desire lifts of the following squares.
\begin{equation}\label{eq:liftsqU1}
\begin{tikzpicture}[baseline=(current  bounding  box.center), node distance = 1.5cm]
\node [] (A) {$\Lambda_d^i$};
\node [right of = A,xshift=1cm] (B) {$(\text{U}\widebar{\mathbf{E}}_{\textbf{st}}^\dagger Y)_n$};
\node [below of = A] (C) {$\Delta_d$};
\node [right of = C,xshift=1cm] (D) {$(\text{U}\widebar{\mathbf{E}}_{\textbf{st}}^\dagger X)_n$};

\draw [->] (A) -- (B);
\draw [->] (A) -- (C);
\draw [->] (C) -- (D);
\draw [->] (B) -- (D);
\draw [->,dashed] (C) -- (B);
\end{tikzpicture}
\end{equation}
Via the earlier identifications, this amounts to a lift of the following square.
\begin{equation}\label{eq:liftsqU2}
\begin{tikzpicture}[baseline=(current  bounding  box.center), node distance = 1.5cm]
\node [] (A) {$\Lambda_d^i$};
\node [right of = A,xshift=1cm] (B) {$(\mathrm{V}(\mathrm{D} Y))_n$};
\node [below of = A] (C) {$\Delta_d$};
\node [right of = C,xshift=1cm] (D) {$(\mathrm{V}(\mathrm{D} X))_n$};

\draw [->] (A) -- (B);
\draw [->] (A) -- (C);
\draw [->] (C) -- (D);
\draw [->] (B) -- (D);
\draw [->,dashed] (C) -- (B);
\end{tikzpicture}
\end{equation}
Here $\mathrm{V}$ is the functor $\mathsf{Ch}_{\overline{\mathbb{F}}_p} \to \mathsf{Sp}$ given by setting
\[
\mathrm{V}(Z)_{n,d} := \mathsf{Ch}_{\overline{\F}_p}(\widebar{\mathrm{C}}_\bullet(\Sigma^{\infty - n}\Delta_{d+}), Z)
\]
which we note is a right adjoint to the spectral chains functor taken as a functor to simply chain complexes, forgetting the coalgebra structure. Considering the composite adjunction
\begin{center}
\begin{tikzpicture}[node distance=3cm]
\node[](A){$\mathsf{Sp}$};
\node[right of = A](B){$\mathsf{Ch}_{\overline{\F}_p}$};
\node[left of = A](C){$\mathsf{Spc}$};

\draw[->,transform canvas={yshift=2mm}] (A) -- (B) node[midway,anchor=south]{$\widebar{\mathrm{C}}_\bullet$};
\draw[->,transform canvas={yshift=-2mm}] (B) -- (A) node[midway,anchor=north]{$\mathrm{V}$} node[midway,anchor=south,yshift=-0.5mm]{$\bot$};

\draw[->,transform canvas={yshift=2mm}] (C) -- (A) node[midway,anchor=south]{$\Sigma^{\infty - n}_+$};
\draw[->,transform canvas={yshift=-2mm}] (A) -- (C) node[midway,anchor=north]{$(-)_n$} node[midway,anchor=south,yshift=-0.5mm]{$\bot$};
\end{tikzpicture}
\end{center}
we see that the above lifting problem is equivalent to one of the following form in chain complexes.
\begin{center}
\begin{tikzpicture}[node distance = 1.5cm]
\node [] (A) {$\widebar{\mathrm{C}}_\bullet(\Sigma^{\infty - n}\Lambda_{d+}^i)$};
\node [right of = A,xshift=1cm] (B) {$\mathrm{D}Y$};
\node [below of = A] (C) {$\widebar{\mathrm{C}}_\bullet(\Sigma^{\infty - n}\Delta_{d+})$};
\node [right of = C,xshift=1cm] (D) {$\mathrm{D}X$};

\draw [->] (A) -- (B) node[midway,anchor=south]{};
\draw [->] (A) -- (C);
\draw [->] (C) -- (D) node[midway,anchor=north]{};
\draw [->] (B) -- (D) node[midway,anchor=west]{$\text{D}f$};
\draw [->,dashed] (C) -- (B);
\end{tikzpicture}
\end{center}
Now, up to shifts, we have that $\widebar{\mathrm{C}}_\bullet(\Sigma^{\infty - n}\Delta_{d+}) \cong \widebar{\text{C}}_\bullet(\Delta_{d+})$ and $\widebar{\mathrm{C}}_\bullet(\Sigma^{\infty - n}\Lambda_{d+}^i) \cong \widebar{\mathrm{C}}_\bullet(\Lambda_{d+}^i)$. As a result, $\widebar{\mathrm{C}}_\bullet(\Sigma^{\infty - n}\Delta_{d+})$ and $\widebar{\mathrm{C}}_\bullet(\Sigma^{\infty - n}\Lambda_{d+}^i)$ are acyclic chain complexes. Moreover, $\widebar{\mathrm{C}}_\bullet(\Sigma^{\infty - n}\Lambda_{d+}^i) \to \widebar{\mathrm{C}}_\bullet(\Sigma^{\infty - n}\Delta_{d+})$ is clearly an inclusion between free complexes. Thus the lefthand vertical map is a trivial cofibration in the standard projective model structure on chain complexes. On the other hand, since $X \to Y$ was an inclusion of complexes, $\text{D}Y \to \text{D}X$ is an epimorphism. Thus the lift exists, as desired.
\end{proof}

As a result of the above, we have a derived adjunction
\begin{center}
\begin{tikzpicture}[node distance=2cm]
\node[](A){$\mathsf{hSp}^{\text{op}}$};
\node[right of = A](B){$\mathsf{h}\widebar{\mathpzc{E}}_{\text{st}}^\dagger\text{-}\mathsf{Alg}.$};

\draw[<-,transform canvas={yshift=2mm}] (B) -- (A) node[midway,anchor=south]{
};
\draw[<-,transform canvas={yshift=-2mm}] (A) -- (B) node[midway,anchor=north]{} node[midway,anchor=south,yshift=-0.5mm]{};
\end{tikzpicture}
\end{center}


\chapter{Algebraic Models of $p$-Adic Stable Homotopy Types}

In this section, we shall provide an application of our stable operads to $p$-adic stable homotopy theory. We will do this using the derived spectral cochains adjunction
\begin{center}
\begin{tikzpicture}[node distance=2cm]
\node[](A){$\mathsf{hSp}^{\text{op}}$};
\node[right of = A](B){$\mathsf{h}\widebar{\mathpzc{E}}_{\text{st}}^\dagger\text{-}\mathsf{Alg}$};

\draw[<-,transform canvas={yshift=2mm}] (B) -- (A) node[midway,anchor=south]{
};
\draw[<-,transform canvas={yshift=-2mm}] (A) -- (B) node[midway,anchor=north]{} node[midway,anchor=south,yshift=-0.5mm]{};
\end{tikzpicture}
\end{center}
constructed in the previous chapter. We have shown that cochains on spectra, appropriately defined, yield algebras over the stable Barratt-Eccles operad. We will now show that these cochains, endowed with this structure,  yield algebraic models for $p$-adic stable homotopy types, where $p$ here is a fixed but unspecified prime, as in previous chapters. More specifically, we will to show that, when restricted to the bounded below $p$-complete spectra of finite $p$-type, the map $\mathsf{hSp}^{\text{op}} \to \mathsf{h}\widebar{\mathpzc{E}}_{\text{st}}^\dagger\text{-}\mathsf{Alg}$ is fully faithful. This is equivalent, for formal reasons, to showing that, for any such spectrum $E$, the unit of the derived adjunction $E \to (\mathrm{der} \: \mathrm{U}) \circ (\mathrm{der} \: \widebar{\mathrm{C}}^\bullet(-)) (E)$ (``der'' indicates a derived functor) is an isomorphism. For this reason,  following~\cite{Mandell}, we make the following definition.

\begin{Definition}\label{def:res_sp}
A spectrum $E$ is said to be \textit{resolvable}\index{spectra!resolvable} if the unit of the derived spectral cochains adjunction above, evaluated at this spectrum, is an isomorphism.
\end{Definition}

\section{The Case of the Generalized Eilenberg-MacLane Spectra $\Sigma^n\mathrm{H}\mathbb{F}_p$}

To begin, we wish to prove that $\text{H}\F_p$, and more generally $\Sigma^n\text{H}\F_p$, $n \in \Z$, is resolvable. In order to do this,  first, note that, due to Proposition~\ref{prop:EM_spec_fib}, in computing the derived cochains functor, we need not perform any replacement of $\Sigma^n\text{H}\F_p$, though we do still need to perform a cofibrant replacement of $\widebar{\mathrm{C}}^\bullet(\Sigma^n\mathrm{H}\F_p)$. We will do this by constructing a cell model for $\widebar{\mathrm{C}}^\bullet(\Sigma^n\mathrm{H}\F_p)$. Fix $n \in \Z$. Intuitively, one expects that $\Sigma^n\text{H}\F_p$ on the spectral side ought to correspond to $\overline{\textbf{E}}_{\textbf{st}}^\dagger\overline{\F}_p[n]$, or something similar, on the algebraic side. As per Proposition~\ref{prop:cohom_ops_Fpbar}, we have an operation $P^0$ on the cohomology of the free algebra  $\overline{\textbf{E}}_{\textbf{st}}^\dagger\overline{\F}_p[n]$. As per Proposition~\ref{prop:P01sp}, $P^0$ always acts by the identity on spectral cochains. Moreover, we shall see that this is the only special circumstance which we need to take into account, in that we shall be able to construct our cell model for $\widebar{\mathrm{C}}^\bullet(\Sigma^n\mathrm{H}\F_p)$ by forcing this operation $P^0$ to be the identity. \\

First, recall that $\widebar{\mathrm{C}}^\bullet(\mathrm{H}\F_p)$ is given by applying $\overline{\F}_p \otimes_{\F_p} -$ to the following:
\[
\text{lim}(\cdots \to \text{C}^\bullet(\text{K}(\F_p,2))[-2] \to \text{C}^\bullet(\text{K}(\F_p,1))[-1] \to \text{C}^\bullet(\text{K}(\F_p,0))).
\]
Thus, in particular $\widebar{\mathrm{C}}^0(\mathrm{H}\F_p)$ is given by applying $\overline{\F}_p \otimes_{\F_p} -$ to the following:
\[
\text{lim}(\cdots \to \text{C}^2(\text{K}(\F_p,2)) \to \text{C}^1(\text{K}(\F_p,1)) \to \text{C}^0(\text{K}(\F_p,0))).
\]
Next, recall that, for each $m \ge 0$, $\text{K}(\F_p,m)$ is the based simplicial set whose $d$-simplices are given by $\text{Z}^m(\Delta_d;\F_p)$ (note that this is $*$ when $d < m$ and $\F_p$ when $d = m$). For each $m \ge 0$, we have a canonical fundamental class given by the cocyle $k_m$ in $\text{C}^m\text{K}(\F_p,m)$ which sends $\alpha \in \text{Z}^m(\Delta_m;\F_p)$ to $\alpha(\text{id}_{[m]})$. Upon unravelling the definition of the structure maps for Eilenberg-MacLane spectra, we find that, for each $m \ge 0$, the map $\text{C}^{m+1}(\text{K}(\F_p,m+1)) \to \text{C}^m(\text{K}(\F_p,m))$ sends $k_{m+1}$ to $k_m$. As such, we have a well-defined canonical element $(\cdots, k_2, k_1, k_0)$ in the inverse limit and so a well-defined canonical element $h_0 = 1 \otimes (\cdots, k_2, k_1, k_0)$ in $\widebar{\mathrm{C}}^0(\mathrm{H}\F_p)$. More generally, with $n$ as above, consider again $\widebar{\mathrm{C}}^\bullet(\Sigma^n\mathrm{H}\F_p)$, which is given by applying $\overline{\F}_p \otimes_{\F_p} -$ to the following:
\[
\text{lim}(\cdots \to \mathrm{C}^\bullet(\mathrm{K}(\F_p,n+2))[-2] \to \mathrm{C}^\bullet(\mathrm{K}(\F_p,n+1))[-1] \to \mathrm{C}^\bullet(\mathrm{K}(\F_p,n))).
\]
We have that, in particular, $\widebar{\mathrm{C}}^n(\Sigma^n\mathrm{H}\F_p)$ is given by applying $\overline{\F}_p \otimes_{\F_p} -$ to the following:
\[
\text{lim}(\cdots \to \text{C}^{n+2}(\text{K}(\F_p,n+2)) \to \text{C}^{n+1}(\text{K}(\F_p,n+1)) \to \text{C}^n(\text{K}(\F_p,n))).
\]
Once again, upon unravelling the definition of the structure maps for generalized Eilenberg-MacLane spectra, we find that, for each $m \ge n$, the map $\text{C}^{m+1}(\text{K}(\F_p,m+1)) \to \text{C}^m(\text{K}(\F_p,m))$ sends $k_{m+1}$ to $k_m$. As such, we have a well-defined canonical element $(\cdots, k_{n+2}, k_{n+1}, k_n)$ in the inverse limit, and can make the following definition for a canonically defined element of $\widebar{\mathrm{C}}^n(\Sigma^n\mathrm{H}\F_p)$.

\begin{Definition}\label{def:canonical_h_n}
For each $n \in \Z$, define $h_n \in \widebar{\mathrm{C}}^n(\Sigma^n\mathrm{H}\F_p)$ by setting
\[
h_n := 1 \otimes (\cdots, k_{n+2}, k_{n+1}, k_n)
\]
where the $k_i$ are as above.
\end{Definition}

Note that, for each $n$, $h_n$ is a cocycle, because each $k_m$, $m \ge 0$, is a cocycle. Now, we shall construct our cell model for $\widebar{\mathrm{C}}^\bullet(\Sigma^n\mathrm{H}\F_p)$ by attaching a cell to the free algebra $\widebar{\mathbf{E}}_{\textbf{st}}^\dagger\overline{\F}_p[n]$ to ensure that $P^0$ acts by the identity. Let $i_n$ denote the degree $n$ cocycle of $\widebar{\mathbf{E}}_{\textbf{st}}^\dagger\overline{\F}_p[n]$ given by the tensor $1 \otimes \mathrm{id}$. Let also $p_n$ be a representative of the class $(1-P^0)[i_n]$, and then denote also by the same symbol the map $\widebar{\mathbf{E}}_{\textbf{st}}^\dagger\overline{\F}_p[n] \to \widebar{\mathbf{E}}_{\textbf{st}}^\dagger\overline{\F}_p[n]$ induced by the map $\overline{\F}_p[n] \to \widebar{\mathbf{E}}_{\textbf{st}}^\dagger\overline{\F}_p[n] \colon 1 \mapsto p_n$. Now define an $\widebar{\mathpzc{E}}_{\mathrm{st}}^\dagger$-algebra $J_n$ via the following pushout diagram.

\begin{center}
\begin{tikzpicture}[node distance = 1.5cm]
\node(A){$\overline{\mathbf{E}}_{\textbf{st}}^\dagger\overline{\F}_p[n]$};
\node[below of = A](C){$\overline{\mathbf{E}}_{\textbf{st}}^\dagger\text{C}\overline{\F}_p[n]$};
\node[right of = A, xshift = 1.5cm](B){$\overline{\mathbf{E}}_{\textbf{st}}^\dagger\overline{\F}_p[n]$};
\node[below of = B](D){$J_n$};
	
\draw[->] (A) -- (B) node[midway,anchor=south]{$p_n$};
\draw[->] (A) -- (C) node[midway,anchor=east]{};
\draw[->] (C) -- (D) node[midway,anchor=north]{};
\draw[->] (B) -- (D) node[midway,anchor=west]{};
		
\begin{scope}[shift=($(A)!.2!(D)$)]
\draw +(0,-0.25) -- +(0,0)  -- +(0.25,0);
\end{scope}
\end{tikzpicture}
\end{center}

This algebra $J_n$ is our hopeful cell model for $\widebar{\mathrm{C}}^\bullet(\Sigma^n\mathrm{H}\F_p)$. In order to show that it is indeed a model for these cochains in an appropriate sense, we construct a comparison map $J_n \to \widebar{\mathrm{C}}^\bullet(\Sigma^n\mathrm{H}\F_p)$. First, let $f \colon \widebar{\mathbf{E}}_{\textbf{st}}^\dagger\overline{\F}_p[n] \to \widebar{\mathrm{C}}^\bullet(\Sigma^n\mathrm{H}\F_p)$ denote the map induced by the map $\overline{\F}_p[n] \to \widebar{\mathrm{C}}^\bullet(\Sigma^n\mathrm{H}\F_p) \colon 1 \mapsto h_n$. Next, let $q_n$ denote a degree $n+1$ element in $\widebar{\mathrm{C}}^\bullet(\Sigma^n\mathrm{H}\F_p)$ which is such that $\partial(q_n)$ is a representative of $(1-P^0)[h_n]$ (such an element $q_n$ exists since, as per Proposition~\ref{prop:P01sp}, $(1-P^0)[h_n]$ is zero). Denote by $g$ the map $\widebar{\mathbf{E}}_{\textbf{st}}^\dagger\mathrm{C}\overline{\F}_p[n] \to \widebar{\mathrm{C}}^\bullet(\Sigma^n\mathrm{H}\F_p)$ induced by the map $\mathrm{C}\overline{\F}_p[n] \to \widebar{\mathrm{C}}^\bullet(\Sigma^n\mathrm{H}\F_p)$ which sends the degree $n$ and $n+1$ generators, respectively, to $p_n$ and $q_n$. Now, by checking the images of $i_n$, we have that the following square commutes.

\begin{center}
\begin{tikzpicture}[node distance = 1.5cm]
\node(A){$\widebar{\mathbf{E}}_{\textbf{st}}^\dagger\overline{\F}_p[n]$};
\node[below of = A](C){$\widebar{\mathbf{E}}_{\textbf{st}}^\dagger\mathrm{C}\overline{\F}_p[n]$};
\node[right of = A, xshift = 1.5cm](B){$\widebar{\mathbf{E}}_{\textbf{st}}^\dagger\overline{\F}_p[n]$};
\node[below of = B](D){$\widebar{\mathrm{C}}^\bullet(\Sigma^n\mathrm{H}\F_p)$};
	
\draw[->] (A) -- (B) node[midway,anchor=south]{$p_n$};
\draw[->] (A) -- (C) node[midway,anchor=east]{};
\draw[->] (C) -- (D) node[midway,anchor=north]{$g$};
\draw[->] (B) -- (D) node[midway,anchor=west]{$f$};
		
\end{tikzpicture}
\end{center}

As such, we get an induced map
\[
a \colon J_n \to \widebar{\mathrm{C}}^\bullet(\Sigma^n\mathrm{H}\F_p).
\]
The following result now makes precise that $J_n$ is a cell model for $\widebar{\mathrm{C}}^\bullet(\Sigma^n\mathrm{H}\F_p)$.

\begin{Proposition}\label{prop:cell_model_EM}
For each $n \in \Z$, the map $a \colon J_n \to \widebar{\mathrm{C}}{}^\bullet(\Sigma^n\mathrm{H}\F_p)$ above is a quasi-isomorphism.
\end{Proposition}

\begin{proof}
Consider the composite
\[
\widebar{\mathbf{E}}_{\textbf{st}}^\dagger\overline{\F}_p[n] \oplus_{\widebar{\mathbf{E}}_{\textbf{st}}^\dagger\overline{\F}_p[n]} \widebar{\mathbf{E}}_{\textbf{st}}^\dagger\mathrm{C}\overline{\F}_p[n] \longrightarrow J_n \overset{a}\longrightarrow \widebar{\mathrm{C}}{}^\bullet(\Sigma^n\mathrm{H}\F_p).
\]
By Proposition~\ref{prop:pushouts_alg}, the first map is a quasi-isomorphism, and so it suffices to demonstrate that the composite, say $c$, is a quasi-isomorphism. Consider now instead the composite
\[
\widebar{\mathbf{E}}_{\textbf{st}}^\dagger\overline{\F}_p[n] \overset{b}\longrightarrow \widebar{\mathbf{E}}_{\textbf{st}}^\dagger\overline{\F}_p[n] \oplus_{\widebar{\mathbf{E}}_{\textbf{st}}^\dagger\overline{\F}_p[n]} \overline{\mathbf{E}}_{\textbf{st}}^\dagger\mathrm{C}\overline{\F}_p[n] \overset{c}\longrightarrow \overline{\mathrm{C}}{}^\bullet(\Sigma^n\mathrm{H}\F_p).
\]
Here $b$ is the canonical map from the first summand in the pushout. We claim that, upon taking cohomology, both $b$ and $c \circ b$ are onto and have the same kernel. It suffices to demonstrate this as then $c$ is clearly necessarily a quasi-isomorphism. Let $\iota$ denote the  map $\widebar{\mathbf{E}}_{\textbf{st}}^\dagger \overline{\F}_p[n] \to \widebar{\mathbf{E}}_{\textbf{st}}^\dagger\mathrm{C}\overline{\F}_p[n]$ and consider the following exact sequence:
\[
0 \to \widebar{\mathbf{E}}_{\textbf{st}}^\dagger\overline{\F}_p[n] \overset{p_n-\iota}\longrightarrow \widebar{\mathbf{E}}_{\textbf{st}}^\dagger\overline{\F}_p[n] \oplus \overline{\mathbf{E}}_{\textbf{st}}^\dagger\mathrm{C}\overline{\F}_p[n] \longrightarrow \overline{\mathbf{E}}_{\textbf{st}}^\dagger\overline{\F}_p[n] \oplus_{\widebar{\mathbf{E}}_{\textbf{st}}^\dagger\overline{\F}_p[n]} \overline{\mathbf{E}}_{\textbf{st}}^\dagger\mathrm{C}\overline{\F}_p[n] \to 0.
\]
By Proposition~\ref{prop:freealghomEbarpoint}, we can identify the cohomology of $\overline{\mathbf{E}}_{\textbf{st}}^\dagger\overline{\F}_p[n]$ with $\widehat{\mathcal{B}} \otimes_{\F_p} \overline{\F}_p[n]$, and by Propositions~\ref{prop:freealghomEbarpoint} and~\ref{prop:Ebarmonad}, we can identify the cohomology of $\overline{\mathbf{E}}_{\textbf{st}}^\dagger\overline{\F}_p[n] \oplus \overline{\mathbf{E}}_{\textbf{st}}^\dagger\text{C}\overline{\F}_p[n]$ also with $\widehat{\mathcal{B}} \otimes_{\F_p} \overline{\F}_p[n]$. Moreover, under this identification, the map corresponding to $p_n-b$ sends $1$ to $1-P^0$ and so, more generally, becomes right multiplication by $1-P^0$. Noting that this map is injective (which follows from the fact that the Adem relations preserve length), it follows from the long exact sequence in cohomology that, on cohomology, the map $b$ is onto with kernel the left ideal of $\widehat{\mathcal{B}} \otimes_{\F_p} \overline{\F}_p[n]$ generated by $1-P^0$, which we note, by Proposition~\ref{prop:BhatandA}, coincides with the two-sided ideal generated by $1-P^0$. \\

Now consider the composite $c \circ b$. Upon identifying once more the cohomology of $\overline{\mathbf{E}}_{\textbf{st}}^\dagger\overline{\F}_p[n]$ with $\widehat{\mathcal{B}} \otimes_{\F_p} \overline{\F}_p[n]$,  we have a map $\widehat{\mathcal{B}} \otimes_{\F_p} \overline{\F}_p[n] \to \mathrm{H}^\bullet(\widebar{\mathrm{C}}^\bullet(\Sigma^n\mathrm{H}\F_p))$. By Propositions~\ref{prop:P01sp} and~\ref{prop:BhatandA}, we get an induced map
\[
(\widehat{\mathcal{B}} \otimes_{\F_p} \overline{\F}_p[n])/(1-P^0) \cong \mathcal{A} \otimes_{\F_p} \overline{\F}_p[n] \to \text{H}^\bullet(\widebar{\mathrm{C}}^\bullet(\Sigma^n\mathrm{H}\F_p)).
\]
Noting that $1$ is mapped to the fundamental class $[h_n]$, by the standard calculation of the cohomology of Eilenberg-MacLane spectra, we have that this map is an isomorphism. As such, just as with $b$, at the level of cohomology, $c \circ b$ is onto with kernel the two-sided ideal generated by $1-P^0$, and this completes the proof.
\end{proof}

Having constructed our cofibrant replacement of $\widebar{\mathrm{C}}^\bullet(\Sigma^n\mathrm{H}\F_p)$, we now need to consider how this replacement transforms under application of $\mathrm{U}$. For this purpose, we have the following result.

\begin{Proposition}\label{prop:Utosquare}
We have the following:
\begin{itemize}
	\item[(i)] $\mathrm{U}\widebar{\mathbf{E}}_{\normalfont{\textbf{st}}}^\dagger\overline{\F}_p[n] \cong \Sigma^n\mathrm{H}\overline{\F}_p$ and, under this identification, $\mathrm{U}p_n$ induces on $\pi^{\mathrm{st}}_n$ the map $1 - \varphi$ where $\varphi$ is the Frobenius automorphism of $\overline{\F}_p$.
	\item[(ii)] $\mathrm{U}\widebar{\mathbf{E}}_{\normalfont{\textbf{st}}}^\dagger\mathrm{C}\overline{\F}_p[n] \sim *$ or, more specifically, $\mathrm{U}\overline{\mathbf{E}}_{\normalfont{\textbf{st}}}^\dagger\mathrm{C}\overline{\F}_p[n]$ is a contractible Kan complex in each spectral degree.
\end{itemize}
\end{Proposition}

\begin{proof}
(i): In spectral degree $m$ and simplicial degree $d$, we have
\begin{align*}
(\mathrm{U}\widebar{\mathbf{E}}_{\textbf{st}}^\dagger\overline{\F}_p[n])_{m,d} &= \widebar{\mathpzc{E}}_{\mathrm{st}}^\dagger\text{-}\mathsf{Alg}(\widebar{\mathbf{E}}_{\textbf{st}}^\dagger\overline{\F}_p[n],\widebar{\mathrm{C}}^\bullet(\Sigma^{\infty - m}\Delta_{d+})) \\
&\cong \mathsf{Co}_{\overline{\F}_p}(\overline{\F}_p[n], \widebar{\mathrm{C}}^\bullet(\Sigma^{\infty - m}\Delta_{d+})) \\
&\cong \text{Z}^n(\widebar{\text{C}}{}^\bullet(\Sigma^{\infty - m}\Delta_{d+})) \\
&= \mathrm{Z}^n(\overline{\F}_p \otimes_{\F_p} \mathrm{C}^\bullet(\Sigma^{\infty - m}\Delta_{d+})) \\
&\cong \mathrm{Z}^n(\overline{\F}_p \otimes_{\F_p} \mathrm{C}^\bullet(\Delta_{d+})[-m]) \\
&\cong \text{Z}^{n+m}(\Delta_{d};\overline{\F}_p) \\
&= (\Sigma^n\text{H}\overline{\mathbb{F}}_p)_{m,d}.
\end{align*}
One can readily verify directly that the action of the simplicial operators coincide and that so do the spectral structure maps. \\

By Proposition~\ref{prop:EM_spec_fib} and Remark~\ref{rmk:htpy_grps_fib_sp}, we can compute the $n^{\text{th}}$ stable homotopy group of $\Sigma^n\text{H}\overline{\F}_p$ via the $n^{\text{th}}$ unstable homotopy group of the space in spectral degree zero. Moreover, we find that, under the identification $\pi^{\text{st}}_n(\text{U}\widebar{\mathbf{E}}_{\textbf{st}}^\dagger\overline{\F}_p[n]) \cong \pi^{\text{st}}_n(\Sigma^n\text{H}\overline{\F}_p) \cong \overline{\F}_p$, an element $\lambda \in \overline{\F}_p$ corresponds to the class of the map
\[
\widebar{\mathbf{E}}_{\textbf{st}}^\dagger\overline{\F}_p[n] \longrightarrow \widebar{\mathrm{C}}{}^\bullet(\Sigma^{\infty}\Delta_{n+}) \cong \widebar{\mathrm{C}}{}^\bullet(\Delta_{n+})
\]
which sends $i_n$ to the cochain $\alpha$ which sends $\text{id}_{[n]} \in (\Delta_n)_n$ to $\lambda$. To act by $\text{U}p_n$, we precompose with $p_n \colon \widebar{\mathbf{E}}_{\textbf{st}}^\dagger\overline{\F}_p[n] \to \widebar{\mathbf{E}}_{\textbf{st}}^\dagger\overline{\F}_p[n]$. Thus, we need to compute the image of $i_n$ under the following composite:
\[
\widebar{\mathbf{E}}_{\textbf{st}}^\dagger\overline{\F}_p[n] \overset{p_n}\longrightarrow \widebar{\mathbf{E}}_{\textbf{st}}^\dagger\overline{\F}_p[n] \longrightarrow \widebar{\mathrm{C}}^\bullet(\Sigma^{\infty}\Delta_{n+}) \cong \widebar{\mathrm{C}}{}^\bullet(\Delta_{n+}).
\]
To do this, we need to act by $1-P^0$ on $\alpha \in \widebar{\mathrm{C}}{}^\bullet(\Sigma^{\infty}\Delta_{n+})$. As per Proposition~\ref{prop:P0_Fpbar_spectral_cochains}, $P^0$ fixes those cochains in $\widebar{\mathrm{C}}^\bullet(\Sigma^{\infty}\Delta_{n+})$ which are in the image of $\mathrm{C}^\bullet(\Sigma^{\infty}\Delta_{n+})$, and is $\overline{\F}_p$-semilinear, so that, under the composite above, $i_n$ maps to the cochain which sends $\text{id}_{[n]}$ to $1-\lambda^p$, as desired. \\

(ii): By Proposition~\ref{prop:Ebarmonad} (i), as $\text{C}\overline{\F}_p$ is acyclic, the canonical map $\widebar{\mathpzc{E}}^\dagger_{\text{st}}(0) \to \widebar{\mathbf{E}}^\dagger_{\textbf{st}}\text{C}\overline{\F}_p[n]$ is a quasi-isomorphism. Since $\widebar{\mathbf{E}}^\dagger_{\textbf{st}}\text{C}\overline{\F}_p[n]$ is a cell algebra, by Propositions~\ref{prop:semimodelqadj} and~\ref{prop:spec_coch_adj_qadj}, we have that $\text{U}\widebar{\mathbf{E}}^\dagger_{\textbf{st}}\text{C}\overline{\F}_p[n] \to \text{U}\widebar{\mathpzc{E}}^\dagger_{\text{st}}(0)$ is a weak equivalence of spectra. As $\widebar{\mathpzc{E}}^\dagger_{\text{st}}(0)$ is the initial algebra, by the definition of $\text{U}$, we have that $\mathrm{U}\widebar{\mathpzc{E}}^\dagger_{\text{st}}(0) = *$. Moreover, we saw in the proof of Proposition~\ref{prop:spec_coch_adj_qadj} that $\mathrm{U}\widebar{\mathbf{E}}^\dagger_{\textbf{st}}X$ is a strict $\Omega$-spectrum, and so a fibrant spectrum, when $X$ is a complex of finite type, so that $\mathrm{U}\widebar{\mathbf{E}}^\dagger_{\textbf{st}}\text{C}\overline{\F}_p[n]$ is fibrant spectrum. Thus, by Proposition~\ref{prop:fibcofibsp}(ii), $\text{U}\widebar{\mathbf{E}}^\dagger_{\textbf{st}}\text{C}\overline{\F}_p[n] \to \text{U}\widebar{\mathpzc{E}}^\dagger_{\text{st}}(0)$ is a levelwise weak equivalence, from which the desired result immediately follows.
\end{proof}

We can now demonstrate the resolvability of $\Sigma^n\text{H}\F_p$.

\begin{Proposition}\label{prop:HFp_resolv}
For each $n \in \Z$, $\Sigma^n\emph{H}\F_p$ is resolvable.
\end{Proposition}

\begin{proof}
Consider again the pushout square
\begin{center}
\begin{tikzpicture}[node distance = 1.5cm]
\node(A){$\widebar{\mathbf{E}}_{\textbf{st}}^\dagger\overline{\F}_p[n]$};
\node[below of = A](C){$\widebar{\mathbf{E}}_{\textbf{st}}^\dagger\text{C}\overline{\F}_p[n]$};
\node[right of = A, xshift = 1.5cm](B){$\widebar{\mathbf{E}}_{\textbf{st}}^\dagger\overline{\F}_p[n]$};
\node[below of = B](D){$J_n$};
	
\draw[->] (A) -- (B) node[midway,anchor=south]{$p_n$};
\draw[->] (A) -- (C) node[midway,anchor=east]{};
\draw[->] (C) -- (D) node[midway,anchor=north]{};
\draw[->] (B) -- (D) node[midway,anchor=west]{};
		
\begin{scope}[shift=($(A)!.2!(D)$)]
\draw +(0,-0.25) -- +(0,0)  -- +(0.25,0);
\end{scope}
\end{tikzpicture}
\end{center}
and the map $J_n \to \widebar{\text{C}}{}^\bullet(\Sigma^n\text{H}\F_p)$ which we constructed earlier in this section. Upon applying $\text{U}$ to the pushout square, as $\text{U}$ is a left adjoint and maps to the opposite category of spectra, we get a pullback square of spectra as follows.

\begin{center}
\begin{tikzpicture}[node distance = 1.5cm]
\node(A){$\text{U}J_n$};
\node[below of = A](C){$\text{U}\widebar{\mathbf{E}}^\dagger_{\textbf{st}}\overline{\F}_p[n]$};
\node[right of = A, xshift = 1.5cm](B){$\text{U}\widebar{\mathbf{E}}^\dagger_{\textbf{st}}\text{C}\widebar{\F}_p[n]$};
\node[below of = B](D){$\text{U}\overline{\mathbf{E}}^\dagger_{\textbf{st}}\overline{\F}_p[n]$};
	
\draw[->] (A) -- (B) node[midway,anchor=south]{};
\draw[->>] (A) -- (C) node[midway,anchor=east]{};
\draw[->] (C) -- (D) node[midway,anchor=north]{$\text{U}p_n$};
\draw[->>] (B) -- (D) node[midway,anchor=west]{};
		
\begin{scope}[shift=($(D)!.2!(A)$)]
\draw +(-0.25,0) -- +(0,0)  -- +(0,0.25);
\end{scope}
\end{tikzpicture}
\end{center}

Here the vertical maps are fibrations because $\widebar{\mathbf{E}}_{\textbf{st}}^\dagger\overline{\F}_p[n] \to \widebar{\mathbf{E}}_{\textbf{st}}^\dagger\text{C}\overline{\F}_p[n]$ is a cofibration between cell algebras and because, by Proposition~\ref{prop:spec_coch_adj_qadj}, $\text{U}$ maps cofibrations between cofibrant algebras to fibrations of spectra. By Proposition~\ref{prop:cell_model_EM}, the unit of the derived adjunction is represented by the composite
\[
\Sigma^n\text{H}\F_p \to \text{U}\widebar{\text{C}}{}^\bullet(\Sigma^n\text{H}\F_p) \to \text{U}J_n
\]
so that we need to show that this map is weak equivalence. We saw in the proof of Proposition~\ref{prop:spec_coch_adj_qadj} that $\text{U}\overline{\mathbf{E}}^\dagger_{\textbf{st}}X$ is a strict $\Omega$-spectrum, and so a fibrant spectrum, when $X$ is a complex of finite type. This implies that all spectra in the square above are fibrant. By Proposition~\ref{prop:Utosquare} and the long exact sequence in stable homotopy groups, we have that $\pi_i^{\text{st}}(\text{U}J_n)$ is $\F_p$ when $i = n$, and zero otherwise. Thus, it suffices to show that the map $\Sigma^n\text{H}\F_p \to \text{U}J_n$, given by the composite above, and denoted by say $\eta$, is an isomorphism on $\pi^{\text{st}}_n$. This amounts to showing that the map
\begin{multline*}
\mathsf{hSp}(\Sigma^\infty\Sph^n,\Sigma^n\text{H}\F_p) \longrightarrow \mathsf{hSp}(\Sigma^\infty\Sph^n, \text{U}J_n) = \mathsf{hSp}^{\text{op}}(\text{U}J_n,\Sigma^\infty\Sph^n) \\
\cong \mathsf{h}\widebar{\mathpzc{E}}_{\text{st}}^\dagger\text{-}\mathsf{Alg}(J_n,\widebar{\text{C}}{}^\bullet(\Sigma^\infty\Sph^n))
\end{multline*}
induced by $\eta$ and the derived adjunction isomorphism is bijective, or equivalently, injective. For each $\lambda \in \F_p$, consider the map $\sigma_\lambda \colon \Sigma^\infty\Sph^n \to \Sigma^n\text{H}\F_p$ given by the map $\Sph^n \to \text{K}(\F_p,n)$ which sends the unique non-denenerate $n$-simplex to the $n$-cocycle $\alpha$ on $\Delta_n$ defined by $\alpha(\text{id}_{[n]}) = \lambda$. The images of these maps under the localization functor $\gamma_{\mathsf{Sp}} \colon \mathsf{Sp} \to \mathsf{hSp}$ give the $p$ distinct maps in $\mathsf{hSp}(\Sigma^\infty\Sph^n,\Sigma^n\text{H}\F_p)$. Fix $\lambda \in \F_p$ and consider $\sigma_\lambda$. Unravelling the definition of the above map, the image of $\gamma_{\mathsf{Sp}}(\sigma_\lambda)$ is computed as follows: form the composite $J_n \to \widebar{\text{C}}{}^\bullet(\Sigma^n\text{H}\F_p) \to \widebar{\text{C}}{}^\bullet(\Sigma^\infty\Sph^n)$, where the second map is $\sigma_\lambda$ and the first map is the adjoint of $\eta$, and then take the image of this map under the localization functor $\gamma_{\widebar{\mathpzc{E}}_{\text{st}}^\dagger\text{-}\mathsf{Alg}} \colon \widebar{\mathpzc{E}}_{\text{st}}^\dagger\text{-}\mathsf{Alg} \to \mathsf{h}\overline{\mathpzc{E}}_{\text{st}}^\dagger\text{-}\mathsf{Alg}$. We have that, for different values of $\lambda$, the maps $\widebar{\mathrm{C}}\bullet(\Sigma^n\mathrm{H}\F_p) \to \widebar{\mathrm{C}}^\bullet(\Sigma^\infty\Sph^n)$ differ on cohomology, and thus so must the composite maps $J_n \to \widebar{\text{C}}{}^\bullet(\Sigma^n\text{H}\F_p) \to \widebar{\text{C}}{}^\bullet(\Sigma^\infty\Sph^n)$, and as a result the images of these maps under $\gamma_{\overline{\mathpzc{E}}_{\text{st}}^\dagger\text{-}\mathsf{Alg}}$ must be distinct. Thus the above composite map between $\mathsf{hSp}(\Sigma^\infty\Sph^n,\Sigma^n\text{H}\F_p)$ and $\mathsf{h}\widebar{\mathpzc{E}}_{\text{st}}^\dagger\text{-}\mathsf{Alg}(J_n,\widebar{\text{C}}{}^\bullet(\Sigma^\infty\Sph^n))$ is injective, and so bijective, as desired.
\end{proof}

\section{Fibration Theorems}

In the previous section, we have demonstrated the resolvability of (generalized) Eilenberg-MacLane spectra. We now demonstrate results which will also allow us to induct up Postnikov towers for more general resolvability results.

\begin{Proposition}\label{prop:resolv_inv_limits}
Let $E$ be a spectrum and suppose that it can be described as the inverse limit of a diagram
\[
\cdots \to E_2 \to E_1 \to E_0
\]
such that:
\begin{itemize}
	\item Each map $E_{n+1} \to E_n$ is a fibration and $E_0$ is fibrant.
	\item The canonical map $\emph{colim}\,\widebar{\emph{H}}{}^\bullet E_n \to \widebar{\emph{H}}{}^\bullet E$ is an isomorphism.
\end{itemize}
Then $E$ is resolvable whenever the $E_n$,  for all $n \ge 0$,  are resolvable.
\end{Proposition}

\begin{proof}
Suppose that the $E_n$ are resolvable. We can factor maps of $\widebar{\mathpzc{E}}^\dagger_{\text{st}}$-algebras into relative cell inclusions followed by trivial fibrations. Applying this to the cotower $\widebar{\text{C}}{}^\bullet E_0 \to \widebar{\text{C}}{}^\bullet E_1 \to \widebar{\text{C}}{}^\bullet E_2 \to \cdots$ we get a diagram of $\widebar{\mathpzc{E}}^\dagger_{\text{st}}$-algebras as follows:

\begin{center}
\begin{tikzpicture}[node distance = 2cm]
\node [] (A) {$\overline{\F}_p$};
\node [right of = A,xshift=0cm] (B) {$A_0$};
\node [below of = B] (C) {$\widebar{\text{C}}{}^\bullet E_0$};
\node [right of = C] (D) {$\widebar{\text{C}}{}^\bullet E_1$};
\node [right of = B,xshift=0cm] (E) {$A_1$};
\node [right of = E,xshift=0cm] (F) {$\cdots$};
\node [right of = D,xshift=0cm] (G) {$\cdots$};

\draw [right hook->] (A) -- (B) node[midway,anchor=south]{};
\draw [right hook->] (B) -- (E) node[midway,anchor=west]{};
\draw [right hook->] (E) -- (F) node[midway,anchor=west]{};
\draw [->] (C) -- (D) node[midway,anchor=west]{};
\draw [->] (D) -- (G) node[midway,anchor=west]{};

\draw [->>] (B) -- (C) node[midway,anchor=west]{$\sim$};
\draw [->>] (E) -- (D) node[midway,anchor=west]{$\sim$};
\end{tikzpicture}
\end{center}

Set $A = \text{colim}\,A_n$. Then, by the assumption that $\widebar{\text{H}}{}^\bullet E \cong \text{colim}\,\widebar{\text{H}}{}^\bullet E_n$, we have that the canonical map $A \to \widebar{\text{C}}{}^\bullet E$ is a quasi-isomorphism. Applying $\text{U}$, we have that $\text{U}A$ is the inverse limit of the $\text{U}A_n$ and we have a commutative diagram as follows.

\begin{center}
\begin{tikzpicture}[node distance = 2cm]
\node [] (A) {$E_1$};
\node [right of = A,xshift=0cm] (B) {$E_0$};
\node [below of = A] (C) {$\text{U}A_1$};
\node [below of = B] (D) {$\text{U}A_0$};
\node [left of = A,xshift=0cm] (E) {$\cdots$};
\node [left of = C,xshift=0cm] (G) {$\cdots$};

\draw [->] (A) -- (B) node[midway,anchor=south]{};
\draw [->] (A) -- (C) node[midway,anchor=west]{$\sim$};
\draw [->] (B) -- (D) node[midway,anchor=west]{$\sim$};
\draw [->] (C) -- (D) node[midway,anchor=south]{};
\draw [->] (E) -- (A) node[midway,anchor=south]{};
\draw [->] (G) -- (C) node[midway,anchor=south]{};
\end{tikzpicture}
\end{center}

Here the vertical maps are the composites $E_n \to \text{U}\widebar{\text{C}}{}^\bullet E_n \to \text{U}A_n$ which are quasi-isomorphisms as the $E_n$ are resolvable and these composites represent components of the unit of the derived adjunction. Moreover, since $\text{U}$ maps cofibrations between cofibrant algebras to fibrations of spectra, each map in the bottom row is a fibration and each of the $\text{U}A_n$ are fibrant. As weak equivalences between fibrant spectra are simply levelwise weak equivalences (see Proposition~\ref{prop:fibcofibsp}(ii)), by the standard argument for inverse limits of weak equivalences of spaces along towers of fibrations, we find that the induced map on limits $E \to \text{U}A$ is a weak equivalence. This map is the composite $E \to \text{U}\widebar{\text{C}}{}^\bullet E \to \text{U}A$ and so represents the evaluation of the unit of the derived adjunction at $E$. This unit is thus an isomorphism at $E$, and so $E$ is resolvable, as desired.
\end{proof}

Next, we wish to consider resolvability of fibre products.

\begin{Proposition}\label{prop:resolv_fib_prods}
Let $E$ be a spectrum and suppose that it can be written as a fibre product
\begin{center}
\begin{tikzpicture}[node distance = 1.5cm]
\node(A){$E$};
\node[below of = A](C){$E_1$};
\node[right of = A](B){$E_2$};
\node[below of = B](D){$F$};
	
\draw[->] (A) -- (B) node[midway,anchor=south]{};
\draw[->] (A) -- (C) node[midway,anchor=east]{};
\draw[->] (C) -- (D) node[midway,anchor=north]{};
\draw[->] (B) -- (D) node[midway,anchor=west]{};

\begin{scope}[shift=($(D)!.2!(A)$)]
\draw +(-0.25,0) -- +(0,0)  -- +(0,0.25);
\end{scope}
\end{tikzpicture}
\end{center}
such that:
\begin{itemize}
	\item $E_1, E_2$ and $F$ are fibrant and $E_1, E_2$ are of finite $p$-type.
	\item The righthand vertical map $E_2 \to F$ is a fibration.
	\item There exists an $N$ such that, for $n > N$, $(E_1)_n, (E_2)_n$ are connected and $F_n$ is simply connected.
\end{itemize}
Then $E$ is resolvable whenever $E_1, E_2, F$ are resolvable.
\end{Proposition}

\begin{proof}
Suppose that $E_1, E_2$ and $F$ are resolvable. For the diagram $\widebar{\text{C}}{}^\bullet(E_1) \leftarrow \widebar{\text{C}}{}^\bullet(F) \rightarrow \widebar{\text{C}}{}^\bullet(E_2)$, we take a cofibrant approximation as below.

\begin{center}
\begin{tikzpicture}[node distance = 2cm]
\node [] (A) {$A$};
\node [right of = A,xshift=0cm] (B) {$C$};
\node [below of = A] (C) {$\widebar{\text{C}}{}^\bullet(F)$};
\node [below of = B] (D) {$\widebar{\text{C}}{}^\bullet(E_2)$};
\node [left of = A,xshift=0cm] (E) {$B$};
\node [left of = C,xshift=0cm] (G) {$\widebar{\text{C}}{}^\bullet(E_1)$};

\draw [right hook->] (A) -- (B) node[midway,anchor=south]{};
\draw [->] (A) -- (C) node[midway,anchor=west]{$\sim$};
\draw [->] (B) -- (D) node[midway,anchor=west]{$\sim$};
\draw [->] (C) -- (D) node[midway,anchor=south]{};
\draw [left hook->] (A) -- (E) node[midway,anchor=south]{};
\draw [->] (C) -- (G) node[midway,anchor=south]{};
\draw [->] (E) -- (G) node[midway,anchor=west]{$\sim$};
\end{tikzpicture}
\end{center}

Suppose, for the time being, that we have shown that the induced map $B \amalg_A C \to \widebar{\text{C}}{}^\bullet(E)$ is a quasi-isomorphism. Then, having formed a cofibrant replacement of the cochains $\widebar{\text{C}}{}^\bullet(E)$, the unit of the derived adjunction, evaluated at $E$, is represented by the composite $E \to \text{U}\widebar{\text{C}}^\bullet (E) \to \text{U}(B \amalg_A C)$. Moreover,  we have the following commutative diagram.

\begin{center}
\begin{tikzpicture}[node distance = 3cm]
\node [] (A) {$E$};
\node [below of = A] (B) {$E_1$};
\node [right of = A] (C) {$E_2$};
\node [below of = C] (D) {$F$};

\node [below right of = A,yshift=1cm,xshift=3cm] (AA) {$\text{U}(B \amalg_A C)$};
\node [below of = AA] (BB) {$\text{U}B$};
\node [right of = AA] (CC) {$\text{U}C$};
\node [below of = CC] (DD) {$\text{U}A$};

\draw[->] (A) -- (B);
\draw[->] (A) -- (C);
\draw[->] (B) -- (D);
\draw[->] (C) -- (D);

\draw[->] (AA) -- (BB);
\draw[->] (AA) -- (CC);
\draw[->] (BB) -- (DD);
\draw[->] (CC) -- (DD);

\draw[->] (A) -- (AA);
\draw[->] (B) -- (BB);
\draw[->] (C) -- (CC);
\draw[->] (D) -- (DD);
\end{tikzpicture}
\end{center}
If $E_1, E_2, F$ are resolvable, each of the maps $E_1 \to \text{U}B$, $E_2 \to \text{U}C$ and $F \to \text{U}A$ is a weak equivalence. Moreover, each of the lefthand and righthand squares are pullback squares, and the maps $E_2 \to F$ and $\text{U}C \to \text{U}A$ are fibrations (the latter because, as in the proof of Proposition~\ref{prop:spec_coch_adj_qadj}, $\text{U}$ maps cofibrations to fibrations of spectra). It follows that $E \to \text{U}(B \amalg_A C)$ is also weak equivalence,  and this is exactly what is desired to show that $E$ is resolvable. \\

Due to the argument just described, it remains only to show that the map $B \amalg_A C \to \widebar{\text{C}}{}^\bullet(E)$ is a quasi-isomorphism. Recall that the pushout may be computed via the bar construction
\[
\text{Bar}_n(B,A,C) = B \amalg \underbrace{A \amalg \cdots \amalg A}_{n \: \text{factors}} \amalg \, C
\]
in that, by Proposition~\ref{prop:pushout_normalization} (or really the analogue of it for $\widebar{\mathpzc{E}}^\dagger_{\text{st}}$),  the induced map from the normalization $\text{N}(\text{Bar}_\bullet(B,A,C)) \to B \amalg_A C$ is a quasi-isomorphism. We wish to relate this pushout to the fibre product. We first construct cochains on the fibre product via a cobar construction. The cobar construction is defined as follows:
\[
\text{Cobar}^n(E_1,F,E_2) = E_1 \times \underbrace{F \times \cdots \times F}_{n \: \text{factors}} \times E_2.
\]
This gives a cosimplicial spectrum with coface maps induced by diagonal maps and codegeneracies by projections. Applying cochains, we get a simplicial cochain complex
\[
\widebar{\text{C}}{}^\bullet(\text{Cobar}^\bullet(E_1,F,E_2)).
\]
Considering $E$ as a constant cosimplicial spectrum, we have an induced map
\[
\text{N}(\widebar{\text{C}}{}^\bullet(\text{Cobar}^\bullet(E_1,F,E_2))) \to \widebar{\text{C}}{}^\bullet(E)
\]
and we claim that this is a quasi-isomorphism. Expressing, as in Section~\ref{subsec:spec_cochains}, spectral cochains as an inverse limit of space level cochains, we have that the map 
\[
\widebar{\text{C}}{}^\bullet(\text{Cobar}^\bullet(E_1,F_,E_2)) \to \widebar{\text{C}}{}^\bullet(E)
\]
is an inverse limit of the maps
\[
\widebar{\text{C}}{}^\bullet(\text{Cobar}^\bullet((E_1)_n,F_n,(E_2)_n))[-n] \to \widebar{\text{C}}{}^\bullet(E_n)[-n].
\]
As in the proof of Lemma 5.2 in~\cite{Mandell} (a lemma proven in the course of demonstrating the unstable analogue of our result here),  for sufficiently large $n$, upon normalization, these maps are quasi-isomorphisms. Moreover, the maps forming the inverse limit tower are epimorphisms since $E_1, E_2, F$ are fibrant. Thus, by a $\text{lim}^1$ argument, we have that the map $\widebar{\text{C}}{}^\bullet((\text{Cobar}^\bullet(E_1,F_,E_2))) \to \widebar{\text{C}}{}^\bullet(E)$ between the inverse limits is also a quasi-isomorphism. \\

Now we relate the bar and cobar constructions. Using the various projection maps on $E_1 \times F \times \cdots \times F \times E_2$, we have maps
\begin{multline*}
B \amalg A \amalg \cdots \amalg A \amalg C \longrightarrow \widebar{\text{C}}{}^\bullet(E_1) \amalg \widebar{\text{C}}{}^\bullet(F) \amalg \cdots \amalg \widebar{\text{C}}{}^\bullet(F) \amalg \widebar{\text{C}}{}^\bullet(E_2) \\
\longrightarrow \widebar{\text{C}}{}^\bullet(E_1 \times F \times \cdots \times F \times E_2).
\end{multline*}
These maps are quasi-isomorphisms because if we postcompose with the map
\[
\widebar{\text{C}}{}^\bullet(E_1 \times F \times \cdots \times F \times E_2) \longrightarrow \widebar{\text{C}}{}^\bullet(E_1 \amalg F \amalg \cdots \amalg F \amalg E_2)
\]
induced by the canonical map
\[
E_1 \amalg F \amalg \cdots \amalg F \amalg E_2 \to E_1 \times F \times \cdots \times F \times E_2
\]
(given by a matrix with identity maps along the diagonal and zero maps elsewhere) and make the identification $\widebar{\text{C}}{}^\bullet(E_1 \amalg F \amalg \cdots \amalg F \amalg E_2) \cong \widebar{\text{C}}{}^\bullet(E_1) \amalg \widebar{\text{C}}{}^\bullet(F) \amalg \cdots \amalg \widebar{\text{C}}{}^\bullet(F) \amalg \widebar{\text{C}}{}^\bullet(E_2)$, we get a quasi-isomorphism by definition of $A, B$ and $C$, and because the canonical map $E_1 \amalg F \amalg \cdots \amalg F \amalg E_2 \to E_1 \times F \times \cdots \times F \times E_2$ is a weak equivalence of spectra by the standard argument (coproducts and products of fibrant spectra are weakly equivalent). Now, it follows that we get a quasi-isomorphism of simplicial $\widebar{\mathpzc{E}}^\dagger_{\text{st}}$-algebras
\[
\text{Bar}_\bullet(B,A,C) \to \widebar{\text{C}}{}^\bullet(\text{Cobar}^\bullet(E_1,F,E_2))
\]
and so a quasi-isomorphism
\[
\text{N}(\text{Bar}_\bullet(B,A,C)) \to \text{N}(\widebar{\text{C}}{}^\bullet(\text{Cobar}^\bullet(E_1,F,E_2))).
\]
Finally, we can make use of this by noting that we have a commutative square as follows.

\begin{center}
\begin{tikzpicture}[node distance = 1.5cm]
\node(A){$\text{N}(\text{Bar}_\bullet(B,A,C))$};
\node[below of = A](C){$B \amalg_C A$};
\node[right of = A, xshift = 3cm](B){$\text{N}(\widebar{\mathrm{C}}^\bullet(\text{Cobar}^\bullet(E_1,F,E_2)))$};
\node[below of = B](D){$\widebar{\mathrm{C}}^\bullet(E)$};
	
\draw[->] (A) -- (B) node[midway,anchor=south]{$\sim$};
\draw[->] (A) -- (C) node[midway,anchor=east]{$\sim$};
\draw[->] (C) -- (D) node[midway,anchor=north]{};
\draw[->] (B) -- (D) node[midway,anchor=west]{$\sim$};
\end{tikzpicture}
\end{center}

Here the bottom map is the aforementioned map, the map which we wished to show to be a weak equivalence, and so we are done.
\end{proof}

\section{Proof of the Main Theorem}

With the aid of the results in the previous two sections, we can now extend our earlier resolvability result to include other spectra.

\begin{Proposition}\label{prop:ZmodpmZphat}
The Eilenberg-MacLane spectra $\Sigma^n\emph{H}A$, for $n \in \Z$, with $A = \Z/p^m$ for some $m \ge 1$ or $A = \Z_p^\wedge$ are resolvable.
\end{Proposition}

\begin{proof}
For $m \ge 1$ and $n \in \Z$, recall that we have well-known commutative squares as follows.

\begin{center}
\begin{tikzpicture}[node distance = 1.5cm]
\node(A){$\text{K}(\Z/p^m,n)$};
\node[below of = A](C){$\text{K}(\Z/p^{m-1},n)$};
\node[right of = A, xshift = 2cm](B){$\text{P}\text{K}(\Z/p,n+1)$};
\node[below of = B](D){$\text{K}(\Z/p,n+1)$};
	
\draw[->] (A) -- (B) node[midway,anchor=south]{};
\draw[->] (A) -- (C) node[midway,anchor=east]{};
\draw[->] (C) -- (D) node[midway,anchor=north]{};
\draw[->>] (B) -- (D) node[midway,anchor=west]{};

\begin{scope}[shift=($(D)!.2!(A)$)]
\draw +(-0.25,0) -- +(0,0)  -- +(0,0.25);
\end{scope}
\end{tikzpicture}
\end{center}

(Here $\text{P}$ denotes a path space, and, given the description of the Eilenberg-MacLane spaces before, these maps can be given precise combinatorial descriptions.) An easy check shows that the maps in these squares in fact assemble together to yield maps of the Eilenberg-MacLane spectra, so that, for $m \ge 1$ and $n \in \Z$, we have commutative squares as follows.

\begin{center}
\begin{tikzpicture}[node distance = 1.5cm]
\node(A){$\Sigma^n\text{H}\Z/p^{m+1}$};
\node[below of = A](C){$\Sigma^n\text{H}\Z/p^m$};
\node[right of = A, xshift = 1cm](B){$\text{P}\Sigma^{n+1}\text{H}\Z/p$};
\node[below of = B, yshift = 0mm](D){$\Sigma^{n+1}\text{H}\Z/p$};
	
\draw[->] (A) -- (B) node[midway,anchor=south]{};
\draw[->] (A) -- (C) node[midway,anchor=east]{};
\draw[->] (C) -- (D) node[midway,anchor=north]{};
\draw[->>] (B) -- (D) node[midway,anchor=west]{};
\end{tikzpicture}
\end{center}

Moreover, the conditions of Proposition~\ref{prop:resolv_fib_prods} are satisfied, so that, by induction, we have the desired result for $\Z/p^m$ for $m \ge 1$. Next, Proposition~\ref{prop:resolv_inv_limits} gives us the case of $\Sigma^n\text{H}\Z_p^{\wedge}$ using the following tower:
\[
\Sigma^n\text{H}\Z_p^{\wedge} = \lim(\cdots \to \Sigma^n\text{H}\Z/p^m \to \cdots \to \Sigma^n\text{H}\Z/p).
\]
\end{proof}

We are now finally able to provide the desired algebraic models of $p$-adic stable homotopy types.

\begin{Proposition}\label{prop:algmodels}
All bounded below, $p$-complete spectra of finite $p$-type are resolvable. As a result, the cochains functor
\[
\widebar{\emph{C}}^\bullet \colon \mathsf{Sp}^{\emph{op}} \to \widebar{\mathpzc{E}}_{\emph{st}}^\dagger\emph{-}\mathsf{Alg}
\]
induces a full embedding of the homotopy category of spectra into the derived category of $\widebar{\mathpzc{E}}_{\emph{st}}^\dagger$-algebras when we restrict to bounded below, $p$-complete spectra of finite $p$-type.
\end{Proposition}

\begin{proof}
This follows from our resolvability results above, namely Propositions~\ref{prop:ZmodpmZphat},~\ref{prop:resolv_fib_prods} and~\ref{prop:resolv_inv_limits}, and the fact that bounded below, $p$-complete spectra of finite $p$-type admit Postnikov towers in which the fibres are $\Sigma^n\text{H}A$, for $n \in \Z$, with either $A = \Z/p^m$ for some $m \ge 1$ or $A = \Z_p^\wedge$.
\end{proof}

\backmatter

\nocite{*}
\bibliographystyle{amsalpha}
\bibliography{bib}

\providecommand{\bysame}{\leavevmode\hbox to3em{\hrulefill}\thinspace}
\providecommand{\MR}{\relax\ifhmode\unskip\space\fi MR }
\providecommand{\MRhref}[2]{%
  \href{http://www.ams.org/mathscinet-getitem?mr=#1}{#2}
}
\providecommand{\href}[2]{#2}
\begin{thebibliography}{BMMS86}

\bibitem[BE74]{BarrattEccles}
M.~Barratt and P.~Eccles, \emph{{On $\Gamma_+$ structures. I. A free group
  functor for stable homotopy theory}}, Topology \textbf{13} (1974), 25--45.

\bibitem[BF78]{BousfieldFriedlander}
A.~Bousfield and E.~Friedlander, \emph{{Homotopy theory of $\Gamma$-spaces,
  spectra, and bisimplicial sets}}, Geometric Applications of Homotopy Theory
  II, Lecture Notes in Mathematics \textbf{658} (1978), 80--130.

\bibitem[BF04]{BergerFresse}
C.~Berger and B.~Fresse, \emph{{Combinatorial operad actions on cochains}},
  Mathematical Proceedings of the Cambridge Philosophical Society \textbf{137}
  (2004), no.~1, 135--174.

\bibitem[BLS81]{BrownLenagan}
K.A. Brown, T.H. Lenaganm, and J.T. Stafford, \emph{{K-Theory and Stable
  Structure of Some Noetherian Group Rings}}, Proceedings of the London
  Mathematical Society \textbf{s3-42} (1981), no.~2, 193--230.

\bibitem[BM03]{BergerMoerdijk}
C.~Berger and I.~Moerdijk, \emph{{Axiomatic homotopy theory for operads}}, I.
  Comment. Math. Helv. \textbf{78} (2003), no.~805.

\bibitem[BMMS86]{BrunerMayMcClureSteinberger}
R.R. Bruner, J.P. May, J.E. McClure, and M.~Steinberger, \emph{{$H_\infty$ ring
  spectra and their applications}}, 1 ed., Lecture Notes in Mathematics, vol.
  1176, Springer Berlin, Heidelberg, 1986.

\bibitem[Bro73]{Brown}
K.S. Brown, \emph{{Abstract homotopy theory and generalized sheaf cohomology}},
  Trans. Amer. Math. Soc. \textbf{186} (1973), 419--458.

\bibitem[CKP17]{ChenKrizPultr}
R.~Chen, I.~Kriz, and A.~Pultr, \emph{{Kan's combinatorial spectra and their
  sheaves revisited}}, Theory and Applications of Categories \textbf{32}
  (2017), no.~39, 1363--1396.

\bibitem[CLM76]{CohenLadaMay}
F.~Cohen, T.~Lada, and J.P. May, \emph{{The homology of iterated loop spaces}},
  1 ed., Lecture Notes in Mathematics, vol. 533, Springer Berlin, Heidelberg,
  1976.

\bibitem[Fre98]{Fresse}
B.~Fresse, \emph{{Cogroups in algebras over an operad are free algebras}},
  Comment. Math. Helv. \textbf{73} (1998), no.~637.

\bibitem[Fre09a]{French}
J.~French, \emph{{A comparison of spectral sequences computing unstable
  homotopy groups of $p$-complete, nilpotent spaces}}, 2009.

\bibitem[Fre09b]{FresseBook}
B.~Fresse, \emph{{Modules over operads and functors}}, 1 ed., Lecture Notes in
  Mathematics, vol. 1967, Springer Berlin, Heidelberg, 2009.

\bibitem[GJ94]{GetzlerJones}
E.~Getzler and J.D.S. Jones, \emph{{Operads, homotopy algebra and iterated
  integrals for double loop spaces}}, 1994.

\bibitem[GJ09]{GoerssJardine}
P.~Goerss and R.~Jardine, \emph{{Simplicial homotopy theory}}, Birkh\"{a}user
  Basel, 2009.

\bibitem[GK94]{GinzburgKapranov}
V.~Ginzburg and M.~Kapranov, \emph{{Koszul duality for operads}}, Duke Math. J.
  \textbf{76} (1994), no.~1, 203--272.

\bibitem[Hin97]{Hinich1}
V.~Hinich, \emph{{Homological algebra of homotopy algebras}}, Communications in
  Algebra \textbf{25} (1997), no.~10, 3291--3323.

\bibitem[Hin01]{Hinich2}
\bysame, \emph{{Virtual operad algebras and realization of homotopy types}}, J.
  Pure Appl. Algebra \textbf{159} (2001), no.~2-3, 173--185.

\bibitem[Hin03]{Hinich1err}
V.~Hinich, \emph{{Erratum to "Homological algebra of homotopy algebras"}},
  2003.

\bibitem[Hov98]{HoveyPaper}
M.~Hovey, \emph{Monoidal model categories}, 1998.

\bibitem[Hov99]{Hovey}
\bysame, \emph{{Model categories}}, Mathematical Surveys and Monographs,
  vol.~63, American Mathematical Society, 1999.

\bibitem[HS87]{HinichSchechtman}
V.~Hinich and V.~Schechtman, \emph{{On homotopy limits of homotopy algebras}},
  Lecture Notes in Math \textbf{1289} (1987), 240--264.

\bibitem[iL23]{iLucio}
V.R. i~Lucio, \emph{Absolute algebras, contramodules, and duality squares},
  2023.

\bibitem[KM95]{KrizMay}
I.~Kriz and J.P. May, \emph{{Operads, Algebras, Modules, and Motives}},
  Ast\'{e}risque \textbf{233} (1995).

\bibitem[LGL22]{leGrignouLejay}
B.~Le~Grignou and D.~Lejay, \emph{Homotopy theory of linear coalgebras}, 2022.

\bibitem[Man01]{Mandell}
M.~Mandell, \emph{{$E_\infty$-Algebras and $p$-adic homotopy theory}}, Topology
  \textbf{40} (2001), no.~1, 43--94.

\bibitem[Mar06]{Markl}
M.~Markl, \emph{{Operads and PROPS}}, Handbook of Algebra \textbf{5} (2006),
  87--140.

\bibitem[May70]{May}
J.P. May, \emph{{A general algebraic approach to the Steenrod operations}}, The
  Steenrod Algebra and Its Applications: A Conference to Celebrate N.E.
  Steenrod's Sixtieth Birthday, Lecture Notes in Mathematics \textbf{168}
  (1970).

\bibitem[May03]{MayUnpublished}
\bysame, \emph{{Operads and sheaf cohomology}}, Unpublished notes (2003).

\bibitem[Mil58]{Milnor}
J.~Milnor, \emph{{The Steenrod algebra and its dual}}, The Annals of
  Mathematics \textbf{67} (1958), no.~1, 150--171.

\bibitem[MS03]{McClureSmith}
J.~McClure and J.~Smith, \emph{{Multivariable cochain operations and little
  $n$-cubes}}, Journal of the American Mathematical Society \textbf{16} (2003),
  no.~3, 681--704.

\bibitem[Rez96]{Rezk}
C.~Rezk, \emph{{Spaces of Algebra Structures and Cohomology of Operads}}, Ph.D.
  Thesis, Massachusetts Institute of Technology (1996).

\bibitem[Smi89]{Smirnov1}
V.~Smirnov, \emph{{On the chain complex of an iterated loop space}}, Izv. Akad.
  Nauk SSSR Ser. Mat. \textbf{53} (1989), 1108--1119, English translation in:
  Math. USSR-Izv. 35 (1990), 445–455.

\bibitem[Smi02]{Smirnov2}
\bysame, \emph{{The homology of iterated loop spaces}}, Forum Math. \textbf{14}
  (2002), 345--381.

\bibitem[Spi01]{Spitzweck}
M.~Spitzweck, \emph{{Operads, Algebras and Modules in Model Categories and
  Motives}}, Ph.D. Thesis, University of Bonn (2001).

\bibitem[Ste15]{MarcStephan}
M.~Stephan, \emph{{Kan spectra, group spectra and twisting structures}}, Ph.D.
  Thesis, \'{E}cole Polytechnique F\'{e}d\'{e}rale de Lausanne (2015).

\bibitem[Wei94]{Weibel}
C.~Weibel, \emph{{An Introduction to Homological Algebra}}, Cambridge
  University Press, 1994.

\bibitem[Whi17]{White}
D.~White, \emph{{Model structures on commutative monoids in general model
  categories}}, Journal of Pure and Applied Algebra \textbf{221} (2017),
  no.~12, 3124--3168.

\bibitem[Wu10]{Wu}
J.~Wu, \emph{{Simplicial objects and homotopy groups}}, Braids (A. J. Berrick,
  F. R. Cohen, E. Hanbury, Y.-L. Wong, J. Wu, editors), Lect. Notes Ser. Inst.
  Math. Sci. Natl. Univ. Singap. \textbf{19} (2010).

\bibitem[Yal14]{Yalin}
S.~Yalin, \emph{{The homotopy theory of bialgebras over pairs of operads}},
  Journal of Pure and Applied Algebra \textbf{218} (2014), no.~6, 973--991.

\end{thebibliography}

\printindex

\end{document}